\documentclass[12pt,english]{amsart}
\usepackage[T1]{fontenc}
\usepackage[latin9]{inputenc}
\usepackage{geometry}
\usepackage{nicefrac}
\usepackage{graphicx}
\usepackage{amssymb}

\makeatletter

\newcommand{\noun}[1]{\textsc{#1}}
\providecommand{\tabularnewline}{\\}

\numberwithin{equation}{section} 
\numberwithin{figure}{section} 
  \theoremstyle{plain}
  \newtheorem{thm}{Theorem}[section]
  \theoremstyle{definition}
  \newtheorem{defn}[thm]{Definition}
  \theoremstyle{remark}
  \newtheorem{rem}[thm]{Remark}
  \theoremstyle{plain}
  \newtheorem{lem}[thm]{Lemma}
  \theoremstyle{plain}
  \newtheorem{prop}[thm]{Proposition}
  \theoremstyle{plain}
  \newtheorem{cor}[thm]{Corollary}
 \theoremstyle{definition}
  \newtheorem{example}[thm]{Example}
  \theoremstyle{plain}
  \newtheorem{fact}[thm]{Fact}
  \theoremstyle{plain}
  \newtheorem{conjecture}[thm]{Conjecture}
  \theoremstyle{remark}
  \newtheorem*{rem*}{Remark}
  \theoremstyle{plain}
  \newtheorem*{conjecture*}{Conjecture}
  \theoremstyle{remark}
  \newtheorem{claim}[thm]{Claim}
  \theoremstyle{definition}
  \newtheorem*{problem*}{Problem}

\usepackage{amsmath}
\usepackage{amsfonts}
\usepackage{amssymb}
\usepackage{babel}
\usepackage{diagrams}
\input xy
\xyoption{all}
\def\bb{\bullet}
\def\ms{\mathcal{S}}

\def\CC{\mathbb{C}}
\def\RR{\mathbb{R}}
\def\NN{\mathbb{N}}
\def\UU{\mathbb{U}}
\def\QQ{\mathbb{Q}}

\def\ZZ{\mathbb{Z}}
\def\PP{\mathbb{P}}

\def\HH{\mathbb{H}}
\def\DD{\mathbb{D}}
\def\AA{\mathbb{A}}
\def\TT{\mathbb{T}}
\def\oTT{\buildrel \circ \over \TT {}^{^n}}
\def\II{\mathbb{I}}
\def\MM{\mathbb{M}}
\def\LL{\mathbb{L}}
\def\C{\mathcal{C}}
\def\A{\mathcal{A}}
\def\B{\mathcal{B}}
\def\D{\mathcal{D}}

\def\DT{\tilde{\Delta}}
\def\DDT{\tilde{\DD}}
\def\ST{\tilde{\sigma}}
\def\PD{\PP_{\Delta}}
\def\PDT{\PP_{\DT}}

\def\TXT{\tilde{X}_t}
\def\TXL{\tilde{X}^{\lambda}}
\def\TXM{\tilde{\X}_{-}}
\def\E{\mathcal{E}}
\def\F{\mathcal{F}}
\def\G{\mathcal{G}}
\def\H{\mathcal{H}}
\def\I{\mathcal{I}}
\def\J{\mathcal{J}}
\def\K{\mathcal{K}}
\def\L{\mathcal{L}}
\def\M{\mathcal{M}}

\def\N{\mathcal{N}}
\def\O{\mathcal{O}}
\def\Q{\mathcal{Q}}
\def\R{\mathcal{R}}
\def\T{\mathcal{T}}

\def\CP{\mathcal{P}}

\def\V{\mathcal{V}}
\def\W{\mathcal{W}}
\def\X{\mathcal{X}}
\def\Y{\mathcal{Y}}

\def\d{\partial}

\def\e{\epsilon}

\def\van{\tilde{\varphi}}

\def\v{\underline{v}}
\def\tor{\hat{\mathbb{T}}^n_{\v,\e}}
\def\w{\omega}
\def\twt{\tilde{\w}_t}

\def\l{\lambda}

\def\sR{\mathfrak{R}}
\def\ms{\mathcal{S}}

\def\sW{\mathfrak{W}}

\def\sZ{\mathfrak{Z}}

\def\rateq{\buildrel\text{rat}\over\equiv}

\def\homeq{\buildrel\text{hom}\over\equiv}
\def\nrateq{\mspace{7mu}/\mspace{-19mu}\rateq}
\def\nequiv{\mspace{7mu}/\mspace{-19mu}\equiv}

\def\bb{\bullet}

\def\mapf{\buildrel \cong \over \to}
\def\mapg{\buildrel \pi \over \to}

\def\maprho{\buildrel \rho \over \to}

\def\hm{Hom_{_{\text{MHS}}}}
\def\ext{Ext^1_{_{\text{MHS}}}}

\def\dlog{\text{dlog}}

\def\db{\partial_{\mathcal{B}}}
\def\m{\setminus}
\def\emn{\buildrel (N) \over \equiv}
\def\qneq{\buildrel \QQ(n) \over \equiv}
\def\qoeq{\buildrel \QQ(1) \over \equiv}
\def\qteq{\buildrel \QQ(2) \over \equiv}
\def\qtheq{\buildrel \QQ(3) \over \equiv}
\def\udb{\underline{\underline{\db}}}
\def\udt{\underline{\underline{\d_{\text{top}}}}}
\def\aeq{\buildrel *<0 \over \cong}
\def\beq{\buildrel *\leq -n \over \cong}
\def\eqe{\buildrel (4.4) \over =}
\def\eeq{\rTo^{W_{\bb}AJ^{n,r}_Y} }
\def\mapNN{\buildrel N \over \to}
\def\aequiv{\buildrel A \over \equiv}
\def\bcequiv{\buildrel B,C \over \equiv}
\def\ems{\buildrel (6) \over \equiv}
\def\emqt{\buildrel \QQ(3) \over \equiv}
\def\emql{\buildrel \QQ(\ell+1) \over \equiv}
\def\maybe{\buildrel ? \over =}
\def\sH{\mathfrak{H}}
\def\thda{\buildrel \theta \over \dashrightarrow}
\def\qhat{\buildrel (\wedge) \over  E}
\def\byt{\buildrel \text{by } \S 5.1.3 \over \cong}
\def\slt{\buildrel  =1 \over {\overbrace{(ad-bc)}} }
\def\amat{\tiny \begin{pmatrix} a & b \\ c & d \end{pmatrix} \normalsize}
\def\bmat{\tiny \begin{pmatrix} 0 & -1 \\ 1 & -1 \end{pmatrix} \normalsize}
\def\cmat{\tiny \begin{pmatrix} p & q \\ -s & r \end{pmatrix} \normalsize}
\def\cpmat{\tiny \begin{pmatrix} p & q \\ -s' & r' \end{pmatrix} \normalsize}
\def\dmat{\tiny \begin{pmatrix} r & -q \\ s & p \end{pmatrix} \normalsize}
\def\emat{\tiny \begin{pmatrix} 1 & m \\ 0 & 1 \end{pmatrix} \normalsize}
\def\fmat{\tiny \begin{pmatrix} -1 & -m \\ 0 & -1 \end{pmatrix} \normalsize}
\def\gmat{\tiny \begin{pmatrix} 1 &  & *  \\  &  \ddots &  \\  0 & & 1 \end{pmatrix} \normalsize}
\def\hmat{\tiny \begin{pmatrix} 1 & 0 \\ 0 & 1 \end{pmatrix} \normalsize}
\def\imat{\tiny \begin{pmatrix} \bar{1} & \bar{a} \\  \bar{0} & \bar{1} \end{pmatrix} \normalsize}
\def\jmat{\tiny \begin{pmatrix} * & * \\ m_0 & n_0 \end{pmatrix} \normalsize} 
\def\kmat{\tiny \begin{pmatrix} 1 & 1 \\ 0 & 1 \end{pmatrix} \normalsize}
\def\lmat{\tiny \begin{pmatrix} 1 & 0 \\ 1 & 1 \end{pmatrix} \normalsize}
\def\mmat{\tiny \begin{pmatrix} 0 & 1 \\ -1 & 0 \end{pmatrix} \normalsize}
\def\nmat{\tiny \begin{pmatrix} 0 & -\nicefrac{1}{\sqrt{N}} \\ \sqrt{N} & 0 \end{pmatrix} \normalsize}
\def\omat{\tiny \begin{pmatrix} 1 & j \\ 0 & 1 \end{pmatrix} \normalsize}
\def\pmat{\tiny \begin{pmatrix} 1 & 0 \\ j & 1 \end{pmatrix} \normalsize}
\def\qmat{\tiny \begin{pmatrix} -1 & 0 \\ 0 & -1 \end{pmatrix} \normalsize}
\def\rmat{\tiny \begin{pmatrix} 0 & -1 \\ 1 & 0 \end{pmatrix} \normalsize}
\def\smat{\tiny \begin{pmatrix} 1 & N \\ 0 & 1 \end{pmatrix} \normalsize}
\def\tmat{\tiny \begin{pmatrix} 1 & 0 \\  N & 1 \end{pmatrix} \normalsize}
\def\umat{\tiny \begin{pmatrix} 1 & 2 \\  0 & 1 \end{pmatrix} \normalsize}
\def\vmat{\tiny \begin{pmatrix} 1 & 0 \\  2 & 1 \end{pmatrix} \normalsize}
\def\wmat{\tiny \begin{pmatrix} -3 & -8 \\  -1 & -3 \end{pmatrix} \normalsize}
\def\xmat{\tiny \begin{pmatrix} -1 & -4 \\  0 & -1 \end{pmatrix} \normalsize}
\def\ymat{\tiny \begin{pmatrix} -4 & -9 \\  1 & 2 \end{pmatrix} \normalsize}
\def\label{\tiny \begin{matrix} \left\langle \sZ \right\rangle \\  \downarrow \\ \left\{ AJ^{\ell+1,\ell+1}(\left\langle \sZ \right\rangle_y ) \right\}_{y\in Y(N)} \end{matrix} \normalsize}
\def\alabel{\tiny \begin{matrix} \text{"Leray} \\ \{1,\ell\} \\ \text{part"} \end{matrix} \normalsize}
\def\blabel{\tiny \begin{matrix} \text{locally} \\ F d\underline{z}\wedge d\tau \\  \downarrow \\ (F d\underline{z})\otimes d\tau \end{matrix} \normalsize}
\def\FT{\buildrel \text{FT} \over \mapsto}
\def\incl{\buildrel \jmath(N) \over \subset}
\def\EE{\mathbb{E}}
\def\se{\mathfrak{e}}
\def\uhp{\mathfrak{H}}
\def\test{\buildrel \text{test forms} \over {\overbrace{\Gamma_{_{d \text{-closed}}} (F^2A^3_X)}} }
\def\desing{\begin{matrix} \text{desingularization} \\ \text{of }\,|Z| \end{matrix}}
\def\klma{\begin{matrix} \Omega_Z \\ R_Z  \end{matrix}}
\def\klmb{\begin{matrix} \Omega_n \\ R_n \end{matrix}}

\usepackage{babel}
\makeatother

\begin{document}

\title{Algebraic $K$-theory of toric hypersurfaces}

\author{Charles F. Doran and Matt Kerr}

\maketitle
\tableofcontents{}

$\pagebreak$

\section*{Introduction}

Writing in 1997 on vanishing of constant terms in powers of Laurent
polynomials%
\footnote{here $\mathbf{T}=\mathbb{G}_{m}$%
}\[
\phi\in\CC[\mathbf{T}^{n}]=\CC[x_{1},x_{1}^{-1},\ldots,x_{n},x_{n}^{-1}],\]
Duistermaat and van der Kallen \cite{DvK} proved the following\vspace{2mm}\\
\textbf{Completion Theorem:} \emph{Given $\phi\in\CC[\mathbf{T}^{n}]$
with $\underline{0}\in Int(Newton(\phi))$, $\exists$ good}%
\footnote{i.e., the complement $\overline{\X}\m\mathbf{T}^{n}$ is a NCD in
$\overline{\X}$%
}\emph{ compactification $\X\supset\mathbf{T}^{n}$ with}

\emph{(a) a holomorphic map $\X\to\mathbb{P}^{1}$ extending $\phi$,}

\emph{(b) $\Omega\in\Omega^{n}(\begin{array}[t]{c}
\underbrace{\X\m\phi^{-1}(\infty)}\\
=:\X_{-}\end{array})$ extending $\bigwedge^{n}\dlog\underline{x}$}.\vspace{2mm}\textbf{
}\\
For a simple example, take $n=2$ and \[
\phi=\prod_{i=1}^{2}\left(x_{i}-\frac{\mu^{2}+1}{\mu}+\frac{1}{x_{i}}\right),\,\,\,\,\,\,\,\,\,\,\mu\in\CC^{*}.\]
In the {}``initial compactification'' $\PP^{1}\times\PP^{1}$($\supset\CC^{*}\times\CC^{*}$),
the level sets $1-t\phi=0$ (see Figure 0.1,%
\begin{figure}
\caption{\protect\includegraphics[scale=0.7]{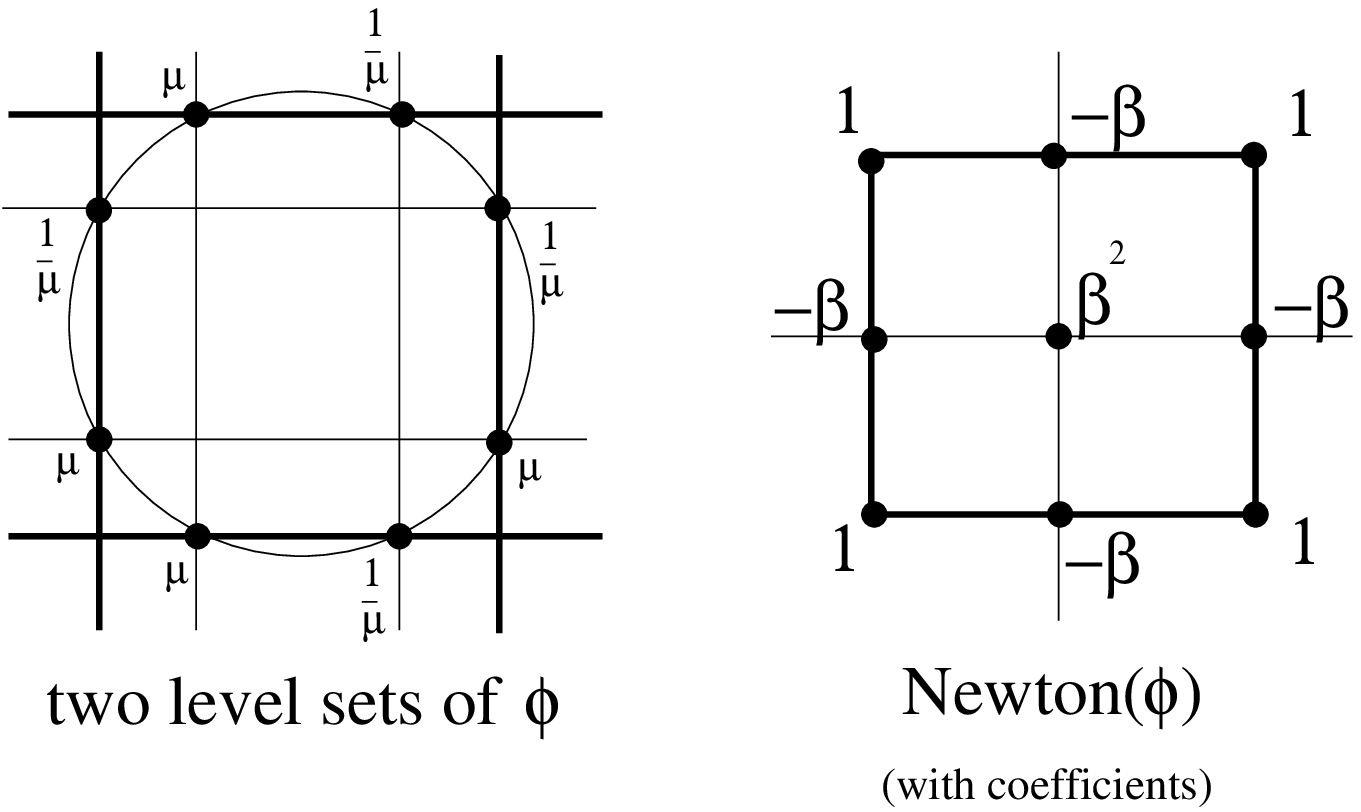}}

\end{figure}
 where $\beta:=\frac{\mu^{2}+1}{\mu}$) complete to a pencil of elliptic
curves, with generic member smooth. For $\phi$ to extend to a well-defined
function we must blow $\PP^{1}\times\PP^{1}$ up at the $8$ points
(marked in the Figure) in the base locus; this yields $\E\rTo_{\nicefrac{1}{\phi}}\PP_{t}^{1}$
as in the Completion Theorem.

What that result does not address at all is the periods of $\Omega$.
Since the Haar form $\frac{1}{(2\pi i)^{n}}\bigwedge^{n}\dlog\underline{x}\,:=\,\frac{dx_{1}}{2\pi ix_{1}}\wedge\cdots\wedge\frac{dx_{n}}{2\pi ix_{n}}$
has only rational periods, one might ask under what circumstances
this remains true for $\Omega$.\vspace{2mm}\\
$\underline{\text{Question 1 (Nori)}}:$ \emph{Write $Hg(\text{---}):=\hm{(}\QQ(0),\text{---})$;
we have $\bigwedge^{n}\dlog\underline{x}\in Hg(H^{n}(\mathbf{T}^{n},\QQ(n)))$.
Is $\Omega\in Hg(H^{n}(\X_{-},\QQ(n)))$?}\vspace{2mm}\\
In the above example, the easiest way to compute periods of $\Omega$
against topological $2$-cycles on $\E_{-}$ is to do a bit of homological
algebra. Writing $E_{0}:=\phi^{-1}(\infty)$, $E_{0}^{[0]}=\widetilde{E_{0}}=\amalg^{4}\PP^{1}$,
$E_{0}^{[1]}=\text{sing}(E_{0})$, we instead can pair $2$-cocycles
in the double-complex of currents\[
\D_{E_{0}^{[1]}}^{\bullet-4}\rTo_{\cdot(2\pi i)}^{\text{Gysin}}F^{1}\D_{E_{0}^{[0]}}^{\bullet-2}\rTo_{\cdot(2\pi i)}^{\text{Gysin}}\begin{array}[t]{c}
F^{2}\D_{\E}^{\bullet}\\
\tiny\text{(deg. 0)}\end{array}\]
against $2$-cycles in \[
C_{\bullet}^{\text{top}}(E_{0}^{[1]};\QQ)\lTo^{\text{intersect}}C_{\bullet}^{\text{top}}(E_{0}^{[0]};\QQ)_{\#}\lTo^{\text{intersect}}C_{\bullet}^{\text{top}}(\E;\QQ)_{\#}\]
(where {}``$\#$'' means chains and their boundaries properly intersect
relevant substrata). If $L_{1}=\{(x,y)=(\mu,0)\}$ and $L_{2}=\{(x,y)=(\frac{1}{\mu},0)\}$
are the sections of $\E$ and $\Gamma=\{\text{path from }(\mu,0)\text{ to }(\frac{1}{\mu},0)\text{ on }\widetilde{E_{0}}\}$,
then we can pair\[
\left\langle \left(\{1,-1,1,-1\},\{\frac{dx}{x},-\frac{dy}{y},-\frac{dx}{x},\frac{dy}{y}\},\Omega\right)\,,\,\left(\{0,0,0,0\},\{\Gamma,0,0,0\},L_{1}-L_{2}\right)\right\rangle \,=\]
\[
\int_{L_{1}-L_{2}}\Omega\,+\,2\pi i\int_{\Gamma}\frac{dx}{x}\,\,=\,\,-4\pi i\log\mu.\]
So the answer is yes precisely when $\E$ has no nontorsion section,
or equivalently when\[
\mu\text{ is a root of unity}.\]
This points the way toward some sort of arithmetic restriction on
$\phi$. (Indeed, the condition on $\mu$, not that on the sections,
is the one which generalizes.)

Now assume $K\subset\bar{\QQ}$ is a number field, and take $\phi\in K[\mathbf{T}^{n}]$.
If the celebrated Hodge and Bloch-Beilinson conjectures are assumed
to hold, an equivalent problem is\vspace{2mm} \\
$\underline{\text{Question 2}}:$ \emph{Does the {}``toric symbol''
$\{x_{1},\ldots,x_{n}\}\in H_{\M}^{n}(\mathbf{T}^{n},\QQ(n))$,}%
\footnote{\emph{$H_{\M}^{n}(\mathbf{T}^{n},\QQ(n))\cong K_{n}^{\text{alg}}(\mathbf{T}^{n})_{\QQ}^{(n)}\cong CH^{n}(\mathbf{T}^{n},n)_{(\QQ)}$}%
}\emph{ or some other symbol with fundamental class $[\bigwedge^{n}\dlog\underline{x}]\in H^{n}(\mathbf{T}^{n},\QQ(n))$,
extend to $\Xi\in H_{\M}^{n}(\X_{-},\QQ(n))$?}\vspace{2mm} \\
So the question about periods of the {}``extended Haar form'' is
replaced by a question about algebraic $K$-theory. If one doesn't
assume the conjectures then of course this is a stronger criterion
than that in Nori's question; but in fact there are very concrete
sufficient conditions for an affirmative answer.

To state these conditions we first fix the specific compactifications
we will use (for $n\leq4$). The \emph{Newton polytope} $\Delta:=Newton(\phi)$
is the convex hull in $\RR^{n}$ of the exponent vectors of all nonzero
monomials appearing in $\phi$. Assume this (hence $\phi$) is \emph{reflexive},
i.e. its polar polytope $\Delta^{\circ}\subset\RR^{n}$ has only integral
vertices; and demand that $1-t\phi(\underline{x})$ be $\Delta$-regular%
\footnote{a mild genericity condition (cf. \cite{Ba1} or $\S1.1$ below)%
} for general $t$. (We actually make a weaker, but more technical,
assumption in Theorem 1.7 for $n\leq3$.) Associated to the fan on
$\Delta^{\circ}$ is a (compact) toric Fano $n$-fold $\PP_{\Delta}\supset\mathbf{T}^{n}$
where the components of the {}``divisor at $\infty$'' $\mathbb{D}=\PP_{\Delta}\m\mathbf{T}^{n}$
correspond to the facets of $\Delta$. This is usually too singular,
and we replace it by $\PP_{\tilde{\Delta}}$,%
\footnote{$\tilde{\mathbb{D}}$ will denote the new divisor at infinity (not
a desingularization).%
} the toric variety associated to the fan on a maximal projective triangulation
of $\Delta^{\circ}$. (In the example, $\PP^{1}\times\PP^{1}=\PP_{\Delta}=\PP_{\tilde{\Delta}}$.)
Taking Zariski closure of the level sets \[
1-t\phi(\underline{x})=0\]
then leads to a $1$-parameter family of $\tilde{\X}$ of anticanonical
hypersurfaces $\tilde{X}_{t}\subset\PP_{\tilde{\Delta}}$, i.e. Calabi-Yau
$(n-1)$-folds. (Again, as in the example, $\tilde{\X}$ is nothing
but $\PP_{\tilde{\Delta}}$ blown up along {[}successive proper transforms
of] the components of the base locus.%
\footnote{Our actual definition of $\tilde{\X}$ in $\S\S1-2$ is slightly different
from that used here; note that $\tilde{\X}$ replaces $\X$ in Questions
1-2. %
}) If we define $\tilde{\pi}:=\frac{1}{\phi}:\,\tilde{\X}\to\PP_{t}^{1}$,
two more properties all these families have in common is:

\begin{itemize}
\item the local system $R^{n-1}\tilde{\pi}_{*}\QQ$ has maximal unipotent
monodromy%
\footnote{for $n=4$ an extra assumption is needed for this.%
} about $t=0$ (cf. $\S2.1$)
\item the relative dualizing sheaf $\omega_{\nicefrac{\tilde{\X}}{\PP^{1}}}:=K_{\tilde{\X}}\otimes\tilde{\pi}^{-1}\theta_{\PP^{1}}^{1}$
has \[
\deg\omega_{\nicefrac{\tilde{\X}}{\PP^{1}}}=1\,\,\,\,\,\,\text{(cf. }\S8.3\text{)}.\]

\end{itemize}
We write $\L\subset\PP^{1}$ for the discriminant locus of $\tilde{\pi}$,
and $\tilde{D}:=\tilde{\mathbb{D}}\cap\tilde{X}_{t}$ for the base
locus of the family.

Also writing in 1997, F. Rodriguez-Villegas \cite{RV} introduced
the arithmetic condition on $\phi$ for $n=2$, that forces the toric
symbol $\xi:=\{x_{1},x_{2}\}$ in Question 2 to extend. Namely, by
decorating the integral points in $\Delta$ with the corresponding
coefficients (in some field $K\subset\CC$) of monomials in $\phi$,
the coefficients along each edge of $\Delta$ yield a $1$-variable
polynomial. If these {}``edge polynomials'' are cyclotomic, then
all Tame symbols of $\xi$ are torsion and Villegas says $\phi$ is
\emph{tempered}. In $\S1$ of this paper, Villegas's definition is
extended to $n\leq4$ in order to prove Theorem $1.7$, which is a
stronger version of the following 

\begin{thm}
Let $\phi\in K[\mathbf{T}^{n}]$ ($n\leq4$, $K$ a number field)
be reflexive, tempered, and regular. (For $n=4$ assume also that
$K$ is totally real and that the components of the $1$-skeleton
of $\tilde{\DD}$ are rational $/K$.) Then Question $2$ (and therefore
Question 1) has a positive answer.
\end{thm}
For example, for $n=3,$ given a reflexive $\Delta\subset\RR^{3}$
with only triangular facets, $\phi:=$\{characteristic Laurent polynomial
of the vertex set of $\Delta$\} will satisfy the Theorem.

The upshot is that we get in each case a family $\Xi_{t}:=\Xi|_{\tilde{X}_{t}}\in CH^{n}(\tilde{X}_{t},n)$
of Milnor $K_{2}$ (resp. $K_{3},\, K_{4}$) classes on elliptic curves
(resp. $K3$ surfaces, $CY$ $3$-folds). In $\S2$ we show that these
classes are always nontorsion by evaluating their image under the
Abel-Jacobi map (or {}``rational regulator map'')\[
AJ^{n,n}:\begin{array}[t]{c}
H_{\M}^{n}(\tilde{X}_{t},\QQ(n))\\
\shortparallel\\
CH^{n}(\tilde{X}_{t},n)\end{array}\to\mspace{-20mu}\begin{array}[t]{c}
H_{\D}^{n}(\tilde{X}_{t},\QQ(n))\\
\shortparallel\\
H^{n-1}(\tilde{X}_{t},\CC/\QQ(n))\end{array}\]
against a family of topological cycles $\tilde{\varphi}_{t}$ vanishing
at $t=0$. This yields the formula (Theorem 2.2)\begin{equation}
\Psi(t):= \left\langle \tilde{\varphi}_t , AJ(\Xi_t ) \right\rangle \equiv (2\pi i)^{n-1}\left\{ \log(t)+\sum_{m\geq 1}\frac{[\phi^m]_0}{m}t^m \right\} \, \, \, \, \, \text{mod } \QQ(n)\\
\end{equation}(where $[\cdot]_{0}$ takes the constant term).

A fundamental goal of writing this paper has been to broaden the relevancy
of (generalized) algebraic cycles and (generalized) normal functions
beyond their traditional context of Hodge theory and motives. In particular,
we want to persuade the reader that higher cycles are not just to
be sought out in the context of the Beilinson conjectures, but instead
also are behind things like solutions of inhomogeneous Picard-Fuchs
(IPF) equations --- even ones arising in string theory. Already in
the context of open mirror symmetry in \cite{MW}, the domainwall
tension for $D$-branes wrapped on the quintic mirror has been interpreted
as the Poincar\'e normal function associated to a family of algebraic
$1$-cycles. This yields not only the solution of an IPF equation,
but also data on {}``counting holomorphic disks'' on the real quintic
$\subset\PP^{4}$. The higher cycles we consider in this paper are
instead related to the local mirror symmetry setting, and their associated
{}``regulator periods'' $\Psi(t)$ furnish the mirror map in that
context. Hence for $n=2$, assuming a conjectural {}``central charge
formula'' of Hosono \cite{Ho}, we obtain information on the asymptotics
of instanton numbers $\{n_{d}\}$ for $K_{\PP_{\Delta^{\circ}}}$.
This story is worked out in $\S3$, with explicit comuptations connecting
the exponential growth rate of the $\{n_{d}\}$ to limits of $AJ$
mappings in $\S4$.

The {}``higher normal functions'' $V(t)$ obtained from our generalized
cycles, on the other hand, provide solutions to certain IPF equations
(cf. $\S2.3$). While we don't know if these play any distinguished
role in local mirror symmetry, they do play a central part in the
Ap\'ery-Beukers irrationality proofs of $\zeta(2)$ and $\zeta(3)$,
and provide a missing link for completing the {}``algebro-geometrization''
of these proofs begun by Beukers, Peters, and Stienstra \cite{Bk,BP,Pe,PS}.
We will try to convey this link below, but for a complete discussion/proof
the reader is referred to \cite{Ke2}.

Another number-theoretic phenomenon on which our construction sheds
light is the {}``modularity'' of the logarithmic Mahler measure
\begin{equation}
m(t^{-1}-\phi):=\frac{1}{(2\pi i)^n}\int_{|x_1|=\cdots =|x_n|=1} \log |t^{-1}-\phi|\bigwedge^n \dlog \underline{x}. \\
\end{equation} Specifically, several authors \cite{RV,Be1,MOY,S} have noted computationally
that (for $n=2,3$) pullbacks of $(0.2)$ by the inverse of the mirror
map frequently yield Eisenstein-Kronecker-Lerch series. In Corollary
2.7, $\Psi(t)$ is related to $(0.2)$, and in $\S8$ we use $AJ$
computations (done in $\S\S5-7$) for Beilinson's Eisenstein symbol
to $prove$ a general result on pullbacks of $\Psi$ by automorphic
functions (Theorem 8.3). This completely explains the observations
on Mahler measures.

One more noteworthy application of Theorem 0.1 is to the splitting
of the MHS on the cohomology $H^{n-1}(\tilde{X}_{0})$ of the {}``large
complex structure'' singular fiber. In fact, whenever Question $1$
has a positive answer, taking Poincar\'e residue of $\Omega\in\hm{(}\QQ(0),H^{n}(\tilde{\X}_{-},\QQ(n)))$
yields \[
Res(\Omega)\in\hm{(}\QQ(0),H_{n-1}(\tilde{X}_{0},\QQ))\]
hence (dually) a morphism \begin{equation}
H^{n-1}(\tilde{X}_0, \QQ(j)) \to \QQ(j) \\
\end{equation} of MHS for any $j$. Now the cycle $\Xi$ produced by the Theorem
obviously does $not$ extend through $\tilde{X}_{0}$. Given a second
cycle $\mathfrak{Z}\in CH^{j}(\tilde{\X}\m\cup_{i}X_{t_{i}},2j-n)$
(all $t_{i}\in\L\m\{0\}$) which $does$ extend,%
\footnote{If $2j=n$ one must also $assume$ that $[\iota_{X_{0}}^{*}\mathfrak{Z}]=0\in H^{2j}(\tilde{X}_{0})$. %
} together with a family $\omega\in\Gamma\left(\PP^{1},\tilde{\pi}_{*}\omega_{\nicefrac{\tilde{\X}}{\PP^{1}}}\right)$
of holomorphic forms, one has the associated (multivalued) normal
function\[
\nu(t)=\left\langle AJ(\mathfrak{Z}|_{\tilde{X}_{t}}),\,\omega(t)\right\rangle \]
over $\PP^{1}\m\L$. If we normalize $\omega$ so that $\widehat{\omega(0)}:=\text{im}\{\omega(0)\}\in H_{n-1}(\tilde{X}_{0},\CC)$
is just $[Res(\Omega)]$,%
\footnote{e.g., one could just take $\omega=\nabla_{\delta_{t}}[AJ_{\tilde{X}_{t}}(\Xi_{t})]$%
} then the splitting $(0.3)$ gives {}``meaning'' to \begin{equation}
\lim_{t\to 0}\nu(t) \in \CC/\QQ(j) ;\\
\end{equation} that is, nontriviality of $(0.4)$ implies nontriviality of $AJ(\mathfrak{Z}|_{\tilde{X}_{t}})$
as a section of the sheaf of generalized Jacobians $J^{j,\,2j-n}(\tilde{X}_{t}).$
This {}``splitting principle'' will be elaborated upon in a future
work.

In the remainder of this Introduction, we want to convey some of the
main ideas behind these applications (including the ones not done
in this paper) through three key examples \begin{equation}
\phi = \frac{(x-1)^2 (y-1)^2}{xy} , \, \,\, n=2,\\
\end{equation}\begin{equation}
\phi = \frac{(x-1)(y-1)(z-1)[(x-1)(y-1)-xyz]}{xyz} , \, \, \, n=3,\\
\end{equation}\begin{equation}
\phi = \frac{x^5 + y^5 + z^5 + w^5 +1}{xyzw}, \, \, \, n=4,\\
\end{equation} all of which satisfy the strengthened version (Theorem 1.7) of Thm.
0.1.

Begin by considering the sequence\small\[
-4,-4,-12,-48,-240,-1356,-8428,-56000,-392040,-2859120,\ldots\]
\normalsize of genus zero local instanton numbers $\{n_{d}\}_{d\geq1}$
for $K_{\PP^{1}\times\PP^{1}}$ \cite{CKYZ}. The related Gromov-Witten
invariants $\{N_{d}\}$ count (roughly speaking) the contribution
to the {}``number of rational curves of degree $d$'' on a CY $3$-fold
made by an embedded $\PP^{1}\times\PP^{1}$ (when there is one). They
have, according to \cite{MOY}, exponential growth rate \begin{equation}
\lim_{d\to \infty } \left| \frac{n_{d+1}}{n_d} \right| = \lim_{d\to \infty} \left| \frac{N_{d+1}}{N_d} \right| = e^{\frac{8}{\pi} G}, \\
\end{equation} where $G:=1-\frac{1}{3^{2}}+\frac{1}{5^{2}}-\frac{1}{7^{2}}+\cdots$
is Catalan's constant. The exponent of $(0.8)$ also appears as a
special value of a hypergeometric integral in a formula \begin{equation}
\frac{8}{\pi} G = \log(16)-\sum_{n\geq 1}\frac{\binom{2n}{n}^2}{16^n n} = -\lim_{\e \to 0} \left\{ \int_{\e}^{\frac{1}{16}} \,_2 F_1(\frac{1}{2},\frac{1}{2};1;4t)\frac{dt}{t} - \log(\e) \right\}\\
\end{equation} essentially known to Ramanujan. The surprising fact is that a family
of higher cycles, in $K_{2}^{\text{alg}}$ of a family of elliptic
curves, is behind $(0.8)$ and $(0.9)$. In order to illustrate how
this works, we shall first offer a brief review of the relevant $AJ$
maps.

To begin with, recall Griffiths's $AJ$ map \cite{G} for $1$-cycles
homologous to zero on a smooth projective $3$-fold $X/\CC$. Writing%
\footnote{Here $q_{i}\in\QQ$, and except where otherwise indicated all cycle
groups and intermediate Jacobians in this paper are taken $\otimes\QQ$.
Also note that $Z^{p}(X)$ denotes complex codimension $p$ algebraic
cycles, while $Z_{\text{top}}^{p}(X)$ (resp. $C_{\text{top}}^{p}(X)$)
means real codimension $p$ (piecewise) smooth topological cycles
(resp. chains). %
}\[
Z=\sum q_{i}C_{i}\in Z_{\text{hom}}^{2}(X),\,\,\,\,\,\,\square:=\PP^{1}\m\{1\},\]
we want to know whether $Z$ is \emph{rationally equivalent to zero}:\[
Z\rateq0\,\,\,\Longleftrightarrow\,\,\,\begin{array}{c}
\exists W\in Z^{2}(X\times\square)\text{ (properly intersecting }X\times\{0,\infty\}\text{)}\\
\text{with }W\cdot(X\times\{0\})-W\cdot(X\times\{\infty\})=Z.\end{array}\]
The map%
\footnote{$A_{X}^{k}=\oplus_{p+q=k}A_{X}^{p,q}$ denotes $C^{\infty}$ $k$-forms
on $X$.%
}\[
Z_{\text{hom}}^{2}(X)\rTo^{\widetilde{AJ}}J^{2}(X):=\frac{H^{3}(X,\CC)}{F^{2}H^{3}(X,\CC)+H^{3}(X,\QQ(2))}\cong\frac{\{F^{2}H^{3}(X,\CC)\}^{\vee}}{\text{im}\{H_{3}(X,\QQ(2))\}}\]
\[
\mspace{200mu}\cong\frac{\{\test/d[\Gamma(F^{2}A_{X}^{2})]\}^{\vee}}{\left\{ \int_{Z_{3}^{\text{top}}(X;\QQ(2))}(\,\,\cdot\,\,)\right\} }\]
induced by\[
Z\longmapsto(2\pi i)^{2}\int_{\d^{-1}Z}(\,\,\cdot\,\,),\]
where $\d^{-1}Z\in C_{3}^{\text{top}}(X;\QQ)$ is any (piecewise smooth)
$3$-chain bounding on $Z$, descends modulo $\rateq$ to yield\[
AJ:\, CH_{\text{hom}}^{2}(X)\to J^{2}(X).\]
This is the type of $AJ$-map which yields the normal functions considered
in \cite{MW}, and detects classes in $K_{0}(X)^{(2)}\cong CH^{2}(X)$.

Now suppose we have an elliptic curve \[
E\subset\PP_{\Delta}=\text{ toric Fano surface,}\]
and would like to detect classes in \[
K_{2}(E)\begin{array}[t]{c}
\cong\\
\tiny\text{\text{"de-loop"}}\end{array}K_{0}(E\times\mspace{-10mu}\begin{array}[t]{c}
\underbrace{\check{C}}\\
\text{nodal}\\
\text{affine}\\
\text{curve}\end{array}\mspace{-10mu}\times\check{C})\,\,\cong\,\, CH^{2}(\begin{array}[t]{c}
\underbrace{E\times\square^{2},E\times\d\square^{2}})\\
X\end{array},\]
where the right-hand term is a $relative$ Chow group and \[
\d\square^{2}:=(\{0,\infty\}\times\square)\cup(\square\times\{0,\infty\})\subset\square^{2}.\]
The {}``relative cycles'' $Z=\sum q_{i}C_{i}\in Z^{2}(X)$ are just
those whose component curves $C_{i}$ properly intersect%
\footnote{all coskeleta of: i.e. components of $E\times\d\square^{2}$, and
intersections of these components.%
} $E\times\d\square^{2}$ and satisfy $Z\cdot(E\times\d\square^{2})=0$,
and relative rational equivalences are defined similarly.%
\footnote{$W\in Z^{2}(E\times\square^{3})$ must intersect $E\times\d\square^{3}$
properly and have $W\cdot(E\times\d\square^{2}\times\square)=0$. %
} Writing \[
\II^{2}\,\,:=\,\,(\{1\}\times\CC^{*})\cup(\CC^{*}\times\{1\})\,\,\subset\,\,(\CC^{*})^{2},\]
\[
X^{\vee}:=(E\times(\CC^{*})^{2},\, E\times\II^{2})\]
for the {}``Lefschetz dual'' variety, the test forms live on $X^{\vee}$;
and \[
J^{2}(X):=\frac{H^{3}(X,\CC)}{F^{2}H^{3}(X,\CC)+H^{3}(X,\QQ(2))}\cong\frac{\{F^{2}H^{3}(X^{\vee},\CC)\}^{\vee}}{\text{im}\{H_{3}(X^{\vee},\QQ(2))\}}\]
\[
\mspace{100mu}\cong\frac{\{H^{1}(E,\CC)\otimes\dlog z_{1}\wedge\dlog z_{2}\}^{\vee}}{\text{im}\{H_{1}(E,\QQ)\otimes S^{1}\times S^{1}\}}\cong Hom\left(H^{1}(E,\QQ),\,\CC/\QQ(2)\right).\]
To produce a map \[
AJ:\,\, CH^{2}(X)\to J^{2}(X),\]
one first notes that $H^{i}(\square,\d\square)=\left\{ \begin{array}{cc}
\QQ(0), & i=1\\
0, & \text{otherwise}\end{array}\right.$ $\implies$ \[
Hg^{2}(H^{4}(X))\cong Hg^{2}(H^{2}(E)\otimes\QQ(0)^{\otimes2})=\{0\}\,\,\,\,\implies\]
$CH^{2}(X)=CH_{\text{hom}}^{2}(X)$. Hence for any $Z\in Z^{2}(X)$,
we \emph{essentially}%
\footnote{for a more precise statement see \cite[sec. 5.8]{KLM} and references
cited therein.%
} have\[
Z=\d\Gamma\,\,\,\,\,\,\,\text{in}\,\,\,\,\,\,\,\, C_{\bullet}^{\text{top}}(E\times(\CC^{*})^{2},\, E\times\II^{2}).\]
We can then consider on test forms in $\Gamma_{_{d\text{-closed}}}(A_{E}^{1})$
\begin{equation}
AJ_X(Z):= \int_{\Gamma} (\, \cdot \, ) \wedge \frac{dz_1}{z_1} \wedge \frac{dz_2}{z_2} \in J^2(X), \
\end{equation} which we now turn to computing in one example.

The Laurent polynomial $(0.5)$ has Newton polytope as shown in Figure
0.2,%
\begin{figure}

\caption{\protect\includegraphics[scale=0.7]{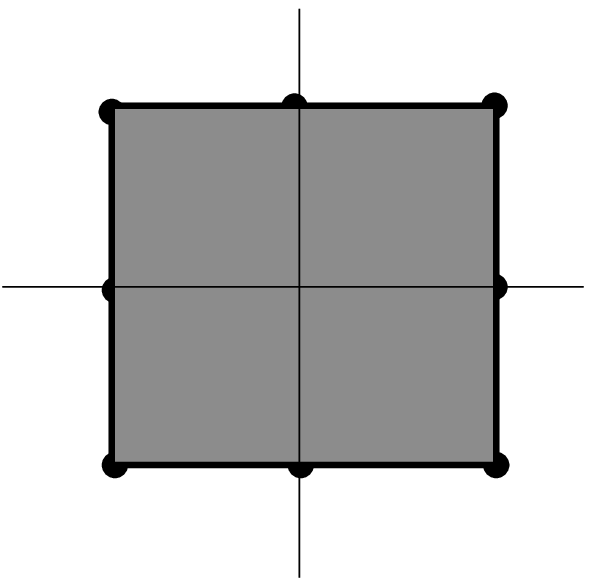}}

\end{figure}
 which corresponds to $\PP_{\Delta}=\PP^{1}\times\PP^{1}$. A projective
description of the fibers of $\X\rTo^{\pi}\PP_{t}^{1}$ is then \begin{equation}
E_t \, := \, \{\mathsf{XYZW}=t(\mathsf{X}-\mathsf{W})^2(\mathsf{Y}-\mathsf{W})^2\}\subset \PP^1_{\mathsf{X}:\mathsf{W}}\times \PP^1_{\mathsf{Y}:\mathsf{Z}}, \\
\end{equation} and after a minimal desingularization at $t=\infty$, $\pi$ has
singular fibers as in Figure 0.3.%
\begin{figure}

\caption{\protect\includegraphics[scale=0.7]{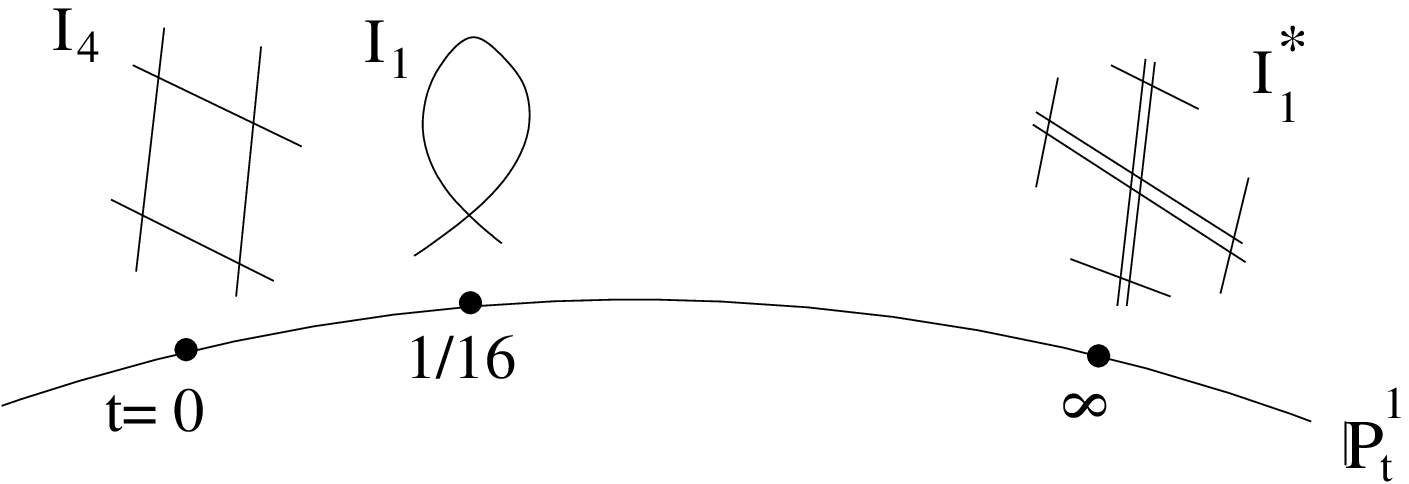}}

\end{figure}
 Now consider the pair of meromorphic functions \[
x:=\frac{\mathsf{X}}{\mathsf{W}},\,\,\, y:=\frac{\mathsf{Y}}{\mathsf{Z}}\,\,\in\CC(E_{t})^{*}\]
arising from the toric coordinates; their divisors\[
(x)=2[b]-2[d],\,\,\,\,\,\,(y)=2[a]-2[c]\]
are supported on marked $4$-torsion points (see Figure 0.4%
\begin{figure}

\caption{\protect\includegraphics[scale=0.7]{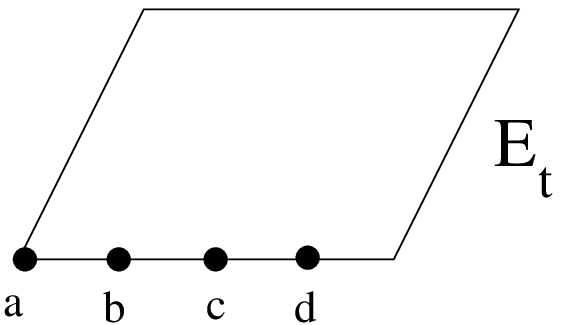}}

\end{figure}
), and in fact $\X$ is nothing but the modular family over $X_{1}(4)$.%
\footnote{we use the notation $\overline{Y}_{1}(4)$ for this in $\S\S5-8$.%
} Most importantly,\[
x=0\text{ or }\infty\,\,\,\,\,\implies\,\,\,\,\,\mathsf{X}\text{ or }\mathsf{W}=0\,\,\,\,\,\begin{array}{c}
\text{use}\\
\Longrightarrow\\
\text{(0.11)}\end{array}\,\,\,\,\,\mathsf{Y}=\mathsf{Z}\,\,\,\,\,\implies\,\,\,\,\, y=1,\]
\[
y=0\text{ or }\infty\,\,\,\,\,\implies\,\,\,\cdots\,\,\,\implies\,\,\,\,\, x=1.\]
Recalling that $1\notin\square$, if we consider the {}``graph''%
\footnote{in the sense of calculus, not combinatorics!%
} of the symbol $\{x,y\}$\[
Z_{t}:=\{(e,x(e),y(e))\,|\, e\in E_{t}\}\in Z^{2}(E\times\square_{z_{1}}\times\square_{z_{2}}),\]
then $Z_{t}\cdot(E\times\d\square^{2})=0$\[
\implies\,\,\,\, Z_{t}\in CH^{2}(X),\]
i.e. $Z_{t}$ is a relative cycle. Interestingly, this example appears
in \cite{Co} as the degeneration of a Ceresa cycle on the Jacobian
of a nonhyperelliptic genus $3$ curve, as that curve acquires $2$
successive nodes.

To construct an explicit $3$-chain $\Gamma_{t}$ bounding on $Z_{t}$,
we use a procedure similar to that in \cite{Bl3} which was generalized
in \cite{Ke1,KLM}. First look at the picture of $Z_{t}\subset E_{t}\times\square\times\square$
in Figure 5.%
\begin{figure}

\caption{\protect\includegraphics[scale=0.7]{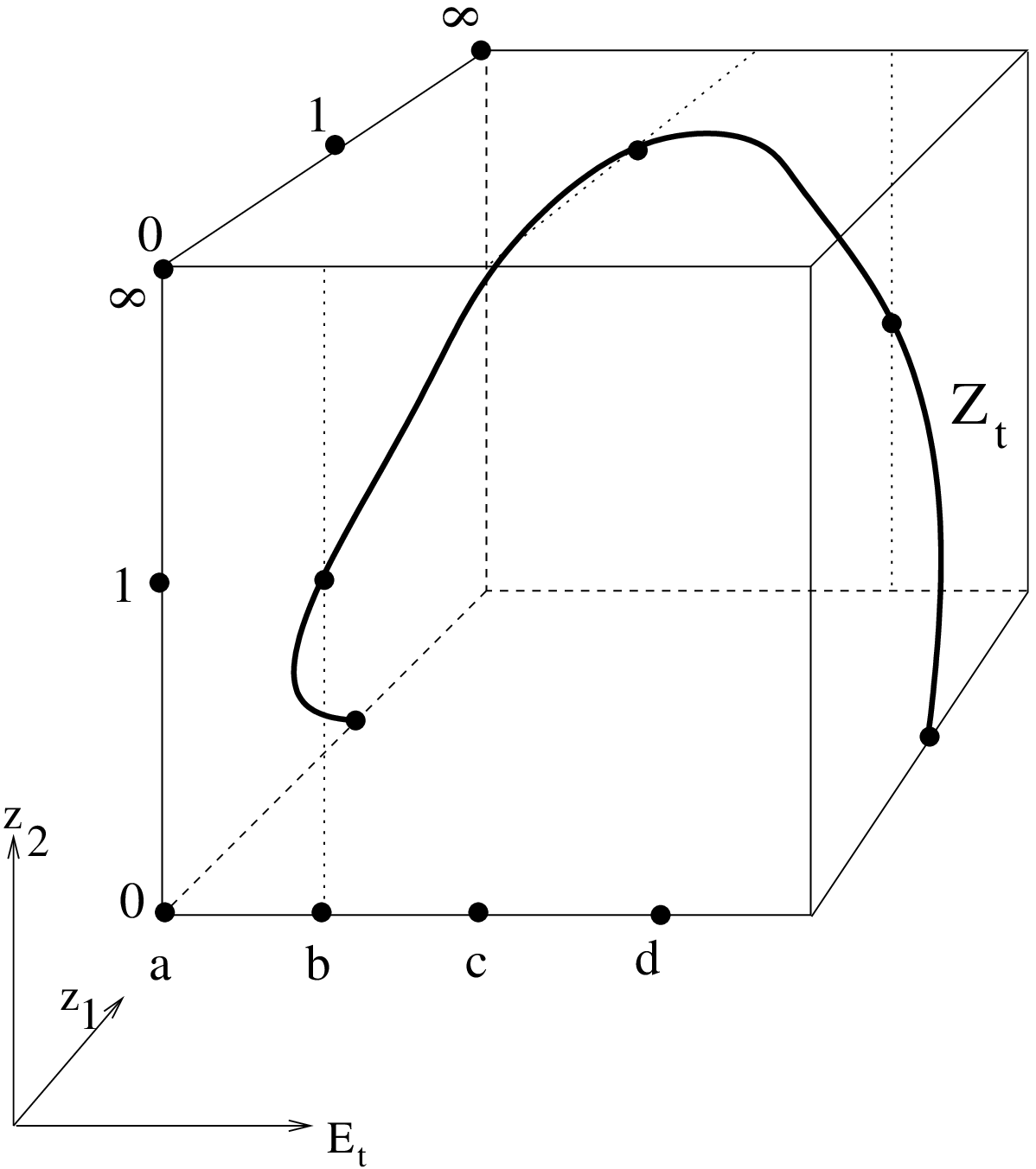}}

\end{figure}
 For a first approximation of $\Gamma$, {}``squash'' $Z_{t}$ to
$\{1\}$%
\footnote{recall that for purposes of bounding $Z_{t}$, $E_{t}\times\II^{2}$
is a sort of {}``topological trashcan''.%
} in the $z_{1}$-coordinate and write down the membrane \begin{equation}
\left\{ (e,\overrightarrow{1.x(e)},y(e)) \, | \, e\in E \right\}\\
\end{equation} which it traces out. The path $\overrightarrow{1.x(e)}\,\subset\,\PP^{1}\m T_{z_{1}}$
can be chosen continuously in $e\in E\m T_{x}$, where $T_{x}:=\{e\in E_{t}\,|\, x(e)\in\RR^{\leq0}\cup\{\infty\}\}$
is the cut in the branch of $\log(x)$. Along $T_{x}$ we have a problem,
namely that $(0.12)$ has $\{(e,S_{x}^{1},y(e))\,|\, e\in T_{x}\}$
as an additional (and unwanted) boundary component. So we squash this
component to $\{1\}$ in the $z_{2}$-coordinate and continue on,
obtaining at last\[
\Gamma_{t}=\left\{ (e,\overrightarrow{1.x(e)},y(e))\right\} _{e\in E_{t}}+\left\{ (e,S_{z}^{1},\overrightarrow{1.y(e)})\right\} _{e\in T_{x}}+\left\{ (e,S_{z_{1}}^{1},S_{z_{2}}^{1})\right\} _{e\in\d^{-1}(T_{x}\cap T_{y})}.\]
Thus $(0.10)$ becomes \[
\int_{\Gamma_{t}}\omega_{E}\wedge\dlog z_{1}\wedge\dlog z_{2}\,\,=\]
\[
\int_{E}\omega_{E}\wedge\log x\dlog y\,-\,2\pi i\int_{T_{x}}\omega_{E}\log y\,-4\pi^{2}\int_{\d^{-1}(T_{x}\cap T_{y})}\omega_{E}\,=\]
\[
(\begin{array}[t]{c}
\underbrace{\log x\dlog y-2\pi i\log y\delta_{T_{x}}}\\
=:R\{x,y\}\in\D^{1}(E_{t})\end{array}-4\pi^{2}\delta_{\d^{-1}(T_{X}\cap T_{y})})(\omega_{E}),\]
 where $\D^{1}$ denotes $1$-currents; in fact, there is nothing
preventing us from taking {[}Poincar\'e duals of] topological $1$-cycles
$\gamma$ as our test forms, and so\[
CH^{2}(E_{t},2):=CH^{2}(X_{t})\rTo^{AJ_{\text{(rel)}}}Hom(H_{1}(E_{t},\QQ),\,\CC/\QQ(2))\]
is induced (on our cycle) by \begin{equation}
Z_t \longmapsto \left\{ \gamma \mapsto \int_{\gamma} R\{x,y\} \right\} . \\
\end{equation} Explicit computation on a particular choice of $\gamma_{t}$ (using
not much more than residue theory; see $\S2.1$) yields $(0.1)$,
which in this case is \begin{equation}
\Psi(t) = \int_{\gamma_t} R\{x,y\} \qteq 2\pi i\left\{ \log t + \sum_{m\geq 1} \frac{\binom{2m}{m}^2}{m} t^m \right\}. \\
\end{equation} Nontriviality of the family of cycles then follows from non-constancy
of the {}``regulator period'' $\Psi$. Both $(0.8)$ and $(0.9)$
are obtained by computing its value $\Psi(\frac{1}{16})$ at the {}``conifold
point'', by pulling back the current $R\{x,y\}$ along a desingularization
of the nodal rational curve $E_{\frac{1}{16}}$. (See the {}``$D_{5}$''
computation in $\S4.3$.) In particular, the relation to the asymptotics
of the $\{N_{d}\}$ (cf. $(0.8)$) comes from the conjectural \emph{mirror
theorem}%
\footnote{the $N_{d}$ here is actually $N_{2d}^{\left\langle K_{\PP^{1}\times\PP^{1}}\right\rangle }$
in $\S3.3$.%
}\[
\frac{1}{(2\pi i)^{2}}-\sum_{d\geq1}d^{3}N_{d}Q^{d}=\frac{\Y(t)}{\left({}_{2}F_{1}(\frac{1}{2},\frac{1}{2};1;4t)\right)^{3}}\]
in which \begin{equation}
\text{the r.h.s. blows up at }\frac{1}{16}\text{, and}\\
\end{equation}\begin{equation}
\text{the mirror map }Q(t)=\exp\left\{\frac{\Psi(t)}{2\pi i}\right\}.\\
\end{equation} $(0.16)$ is based on an analysis ($\S3.1$) of periods on the (open
$CY$ 3-fold) mirror manifold of $K_{\PP^{1}\times\PP^{1}}$, which
generalizes nicely to higher dimensions (for periods on certain open
$CY$ 4- and 5-folds).

As suggested above, the family of cycles $\{Z_{t}\in CH^{2}(X_{t},2)\}$
can be canonically constructed on the universal family $\E_{1}(4)\to Y_{1}(4)=\Gamma_{1}(4)\diagdown\uhp$
of elliptic curves with a marked $4$-torsion point. (Similar constructions
are possible in any level $\geq3$ and even in higher dimension, by
working on \emph{Kuga varieties}, or fiber products of such universal
families; this construction is recalled in $\S5$.) Using fiberwise
double Fourier series for currents on $\E_{1}(4)$, we obtain a very
different expression for the regulator period $\left\langle \tilde{\varphi},AJ(Z)\right\rangle $
as a function of $\tau\in\uhp$,\[
\tilde{\Psi}(\tau)\qteq2\pi i\left\{ \frac{2\pi i}{4}\tau-4\sum_{\mu\geq1}\frac{q_{0}^{\mu}}{\mu}\left(\sum_{r|\mu}r^{2}\chi_{-4}(r)\right)\right\} ,\]
where $q_{0}=e^{\frac{2\pi i}{4}\tau}$. (See Theorem $7.7$ and formulas
$(7.11)$, $(7.16)$ for the general result.) This must coincide with
$(0.14)$ in the sense that\[
\tilde{\Psi}(\tau(t))\qteq\Psi(t),\]
where $\tau(t)=\frac{4}{2\pi i}\log t+t\CC[[t]]$ is the period map.
The rich interactions between the genus $0$ case of the modular/Kuga
construction and the toric construction, including a complete classification
of the elliptic curve families where the constructions coincide, are
explained in $\S8$.

Before turning to our next example Laurent polynomial $(0.6)$, we
give a brief outline of how the $AJ$-formulas $(0.10)$, $(0.13)$
for $CH^{2}(E,2)$ generalize to the setting\[
AJ_{X}^{p,n}:\, CH^{p}(X,n)\to\begin{array}[t]{c}
\underbrace{H_{\H}^{2p-n}(X,\QQ(p))}.\\
\tiny\text{absolute Hodge cohomology}\end{array}\]
(The reader is encouraged to look at \cite[sec. 5]{KLM} and \cite[sec. 8]{KL},
however.) Here $X$ is smooth (quasi-projective) and the higher Chow
groups satisfy \[
\begin{array}[t]{c}
\underbrace{H_{\M}^{2p-n}(X,\QQ(p))}\\
\tiny\text{motivic cohomology}\end{array}\cong\mspace{-10mu}\begin{array}[t]{c}
CH^{p}(X,n)\\
\shortparallel\\
CH^{p}(X\times\square^{n},X\times\d\square^{n})\end{array}\mspace{-10mu}\cong Gr_{\gamma}^{p}K_{n}(X)_{\QQ},\]
where $\d\square^{n}:=\{\underline{z}\in\square^{n}\,|\,\text{some }z_{i}=0\text{ or }\infty\}\subset\square^{n}$.
When $X$ is singular these isomorphisms fail, but one still has\[
AJ_{X}^{p,n}:\, H_{\M}^{2p-n}(X,\QQ(p))\to H_{\H}^{2p-n}(X,\QQ(p))\]
which is treated using hyper-resolutions in \cite[sec. 8]{KL}. 

Recall that the higher Chow groups were defined \cite{Bl4} as the
homology of the complex \[
Z^{p}(X,\bullet):=\frac{\left\{ \begin{array}{c}
\text{"admissible" cycles in }X\times\square^{\bullet}\text{: components}\\
\text{properly intersect all coskeleta of }X\times\d\square^{\bullet}\end{array}\right\} }{\left\{ \text{"degenerate" cycles}\right\} }\]
with differential $\db$ taking the alternating sum of the restrictions
to {}``facets'' of $X\times\d\square^{\bullet}$. The KLM formula
for $AJ^{p,n}$ on $X$ smooth projective (and some quasi-projective
cases) is given simply as a map of complexes \begin{equation}
Z_{\RR}^p(X,-\bullet)\to C_{\D}^{2p+\bullet}(X,\QQ(p)):=C^{2p+\bullet}_{\text{top}}(X;\QQ(p))\oplus F^p\D^{2p+\bullet}(X)\oplus \D^{2p+\bullet -1}(X),\\
\end{equation} where $Z_{\RR}^{p}(X,-\bullet)\subset Z^{p}(X,-\bullet)$ is a quasi-isomorphic
subcomplex.%
\footnote{The proper intersection condition is extended to include certain real
semi-algebraic subsets of $X\times\square^{\bullet}$ in order to
make the formulas $(0.18$-$20)$ well-defined (e.g., the intersections
of $T_{z_{i}}$'s). The (cone) differential on the r.h. complex in
$(0.17)$ sends $(a,b,c)\mapsto(-\d a,-d[b],d[c]-b+\delta_{a})$.%
} $(0.17)$ is defined on an irreducible $\RR$-admissible cycle $Z\subset X\times\square^{n}$
by%
\footnote{here $T_{Z}$ is a $C^{\infty}$ chain, while $\Omega_{Z}$ and $R_{Z}$
are currents.%
} \begin{equation}
Z \longmapsto (2\pi i)^{p-n}\left( (2\pi i)^nT_Z,\Omega_Z,R_Z\right) .\\
\end{equation} Writing \\
\xymatrix{
& & & 
\square^n_{(z_1,\ldots,z_n)}
& 
\left\{ \desing
\right\} 
\ar [l]_{\pi_{\square} \mspace{50mu}} 
\ar [d]^{\pi_X} 
\\
& & & &
X \, ,
}\\
\\
\begin{equation}
\begin{matrix}
T_n:=\bigcap_{i=1}^n T_{z_i} := \bigcap_{i=1}^n\left\{ z_i\in(\RR^{\leq 0}\cup \{\infty \})\right\}\in C^n_{\text{top}}(\square^n)
\\ \\
\Omega_n :=\bigwedge^n \dlog z_i := \frac{dz_1}{z_1} \wedge \cdots \wedge \frac{dz_n}{z_n} \in F^n\D^n(\square^n)
\\ \\
R_n:=R\{z_1,\ldots ,z_n\}:= \mspace{200mu}
\\
\sum_{i=1}^n(\pm 2\pi i)^{i-1}\log(z_i)\frac{dz_{i+1}}{z_{i+1}} \wedge \cdots \wedge \frac{dz_n}{z_n} \cdot \delta_{T_{z_1}\cap\cdots \cap T_{z_{i-1}}} \in \D^{n-1}(\square^n),
\end{matrix}
\\
\end{equation} the KLM (normal) currents are defined by \begin{equation}
T_Z := \pi_X \left\{ Z\cdot (X\times T_n) \right\} \, , \, \, \, \left\{ \klma \right\} :={\pi_X}_* {\pi_{\square}}^* \left\{ \klmb \right\} . \\
\end{equation}

Suppose we are given a higher Chow cycle, i.e. a $\db$-closed precycle
(=admissible cycle) $Z\in Z_{\RR}^{p}(X,n)$. Then\[
d[R_{Z}]=\Omega_{Z}-(2\pi i)^{n}\delta_{T_{Z}},\]
or just $-(2\pi i)^{n}\delta_{T_{Z}}$ if $\dim X<p$ or $p<n$. So
for a symbol $\{\mathbf{f}\}=\{f_{1},\ldots,f_{n}\}\in Z^{n}(U,n)$
(where $f_{i}\in\mathcal{O}^{*}(U)$ and $U$ is smooth quasiprojective
of $\dim<n$), $R_{\{\mathbf{f}\}}=R\{f_{1},\ldots,f_{n}\}$ (as in
$(0.19)$) satisfies \begin{equation}
d[R_{\{\mathbf{f}\}}]=-(2\pi i)^n \delta_{T_{f_1}\cap \cdots \cap T_{f_n}} =: -(2\pi i)^n \delta_{T_{\mathbf{f}}} .\\
\end{equation} In Theorem 0.1, $\Xi_{t}\in Z^{n}(\tilde{X}_{t},n)$ is $\db$-closed
and $\dim(\tilde{X}_{t})=n-1$; hence \[
R_{\Xi_{t}}':=R_{\Xi_{t}}+(2\pi i)^{n}\delta_{\d^{-1}T_{Z}}\in\D^{n-1}(\tilde{X}_{t})\]
is $d$-closed and defines a lift%
\footnote{multivalued if $t$ is allowed to vary%
} of $AJ(\Xi_{t})\in H^{n-1}(\tilde{X}_{t},\CC/\QQ(n))$ to $H^{n-1}(\tilde{X}_{t},\CC)$.
We are interested in the higher normal function \begin{equation}
V(t):= \left\langle [R_{\Xi_t}'], [\omega_t]\right\rangle \\
\end{equation} associated to $\Xi$ and a section $\omega\in\Gamma(\PP^{1},\omega_{\nicefrac{\tilde{\X}}{\PP^{1}}})$
of the dualizing sheaf. If $D_{\text{PF}}^{\omega}$ is the Picard-Fuchs
operator associated to $\omega$ (which kills its periods over topological
cycles), then nonvanishing of \[
D_{\text{PF}}^{\omega}V(t)=:g_{\Xi,\omega}(t)\in\CC(\PP^{1})\]
implies generic nontriviality of $AJ(\Xi_{t})$. This gives a connection
to inhomogeneous Picard-Fuchs equations, explained in $\S2.3$. One
way to evaluate $(0.22)$ is to observe that the restriction of $\Xi_{t}$
to $\tilde{X}_{t}^{*}:=\tilde{X}_{t}\cap\mathbf{T}^{n}$ is $\rateq$
(by a $\db$-coboundary) to the toric symbol $\{x_{1},\ldots,x_{n}\}|_{\tilde{X}_{t}^{*}}$,
and so \[
[R_{\Xi_{t}}'|_{\tilde{X}_{t}^{*}}]\equiv[R\{x_{1}|_{\tilde{X}_{t}^{*}},\ldots,x_{n}|_{\tilde{X}_{t}^{*}}\}+(2\pi i)^{n}\delta_{\Gamma_{t}}]\in H^{n-1}(\tilde{X}_{t}^{*},\CC)\]
for some $\Gamma_{t}\in C_{n-1}^{\text{top}}(\tilde{X}_{t},\tilde{D};\QQ)$.
When we can arrange for $\Gamma_{t}$ to vanish (which is true in
the calculation below), a careful analytic argument with KLM currents
demonstrates that \begin{equation}
V(t) = \int_{\tilde{X}_t} R\{ x_1|_{\tilde{X}_t},\ldots ,x_n|_{\tilde{X}_t} \}\wedge \omega_t.\\
\end{equation}

What originally got us thinking about higher normal functions was
the following integral from a paper \cite{Bk} of Beukers: \begin{equation}
\R(\lambda) = \int_0^1 \int_0^1 \int_0^1 \frac{dX\, dY\, dZ}{1-(1-XY)Z-\lambda XYZ(1-X)(1-Y)(1-Z)} , \\
\end{equation} with $\R(0)=2\zeta(3)$. This is the unique linear combination of
the generating series of the two sequences $\{a_{m}\},\{b_{m}\}$
used by Ap\'ery to prove irrationality of $\zeta(3)$, with larger
radius of convergence than those series. (This leads to Beukers's
simpler, geometrically motivated proof.) Substituting $X=\frac{x}{x-1}$,
$Y=\frac{y}{y-1}$, $Z=\frac{z}{z-1}$, $(0.24)$ becomes\[
\int\int\int_{T:=T_{x}\cap T_{y}\cap T_{z}}\frac{\dlog x\wedge\dlog y\wedge\dlog z}{\lambda-\frac{(x-1)(y-1)(z-1)(1-x-y+xy-xyz)}{xyz}}\,\,=\]
\begin{equation}
\int_T \frac{\bigwedge^3 \dlog x_i}{\lambda - \phi(\underline{x})} =: \int_T (2\pi i)^3 \hat{\omega}_{\lambda}, \\
\end{equation} where $\phi$ is as in $(0.6)$ and (writing $t=\lambda^{-1}$) $\hat{\omega}_{\lambda}\in\Omega^{3}(\PP_{\tilde{\Delta}})\left\langle \log\tilde{X}_{t}\right\rangle $
($\Delta$ is shown in the Figure%
\begin{figure}

\caption{\protect\includegraphics[scale=0.7]{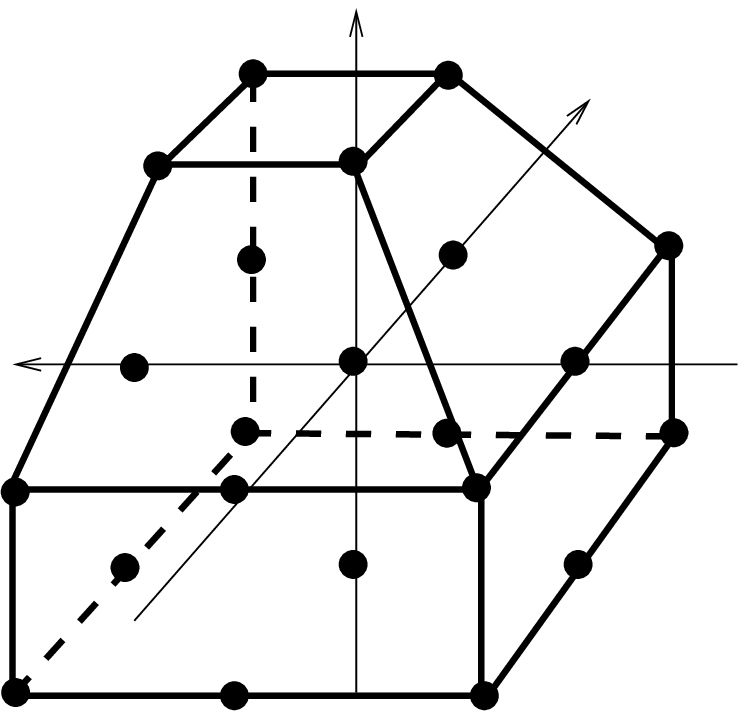}}

\end{figure}
). Differentiating $\hat{\omega}_{\lambda}$ as a $current$ on $\PP_{\tilde{\Delta}}$,\begin{equation}
d[\hat{\omega}_{\lambda}] = 2\pi i(\iota_{\tilde{X}_t})_* Res_{\tilde{X}_t}(\hat{\omega}_{\lambda}) =: (\iota_{\tilde{X}_t})_* \omega_{\lambda} \\
\end{equation} defines our section $\{\omega_{\lambda}\in\Gamma(K_{\tilde{X}_{t}})\}_{t\in\PP^{1}}$
of the dualizing sheaf. Using $(0.26)$ and the generalization\[
d[R_{\{\underline{x}\}}]\,\,=\,\,\sum\{\text{terms supported on }\tilde{\DD}\}\,+\,\bigwedge^{3}\dlog\underline{x}\,-\,(2\pi i)^{3}\delta_{T}\]
of $(0.21)$ to $\PP_{\tilde{\Delta}}$, $(0.25)$ becomes \[
\int_{\PP_{\tilde{\Delta}}}(2\pi i)^{3}\delta_{T}\wedge\hat{\omega}_{\lambda}\,\,=\,-\int_{\PP_{\tilde{\Delta}}}d[R_{\{\underline{x}\}}]\wedge\hat{\omega}_{\lambda}\]
\[
\rEq_{\text{parts}}^{\int\text{ by}}\,\int_{\PP_{\tilde{\Delta}}}R_{\{\underline{x}\}}\wedge\iota_{\tilde{X}_{t}}{}_{_{*}}\omega_{\lambda}\,\,=\,\,\int_{\tilde{X}_{t}}R_{\{\underline{x}\}|_{\tilde{X}_{t}}}\wedge\omega_{\lambda},\]
which is $(0.23)$.%
\footnote{Of course, much of the above needs more thorough justification, as
$R\{\underline{x}\}$ is not technically a current on $\PP_{\tilde{\Delta}}$,
and this will be done in \cite{Ke2}.%
} In fact, $\R(\lambda)$'s interpretation as a higher normal function
associated to a family of $K_{3}(K3)$-classes extending through singular
fibers%
\footnote{other than $\lambda=\infty$/$t=0$.%
} leads (almost) automatically to the {}``larger radius of convergence''
mentioned above, as well as to its satisfaction of an inhomogeneous
Picard-Fuchs equation (which then produces a recursion on the $\{b_{m}\}$).

One knows from \cite{Pe} that the family of $K3$ surfaces $\tilde{\X}$
associated to $(0.6)$ is the canonical family of Kummer surfaces
over $\Gamma_{0}(6)^{+6}\diagdown\uhp^{*}$. From the toric ($\S2.2$)
and modular ($\S7.3$) computations of the {}``fundamental regulator
period'' one gets two rather different expressions\[
\Psi(t)=(2\pi i)^{2}\left\{ \log t\,+\,\sum_{m\geq1}\frac{t^{m}}{m}\sum_{k=0}^{m}\binom{m}{k}^{2}\binom{m+k}{k}^{2}\right\} \]
\[
\tilde{\Psi}(\tau)=-12(2\pi i)^{3}\tau\,+\,\frac{(2\pi i)^{2}}{20}\left\{ 7\psi_{4}(q)-2\psi_{4}(q^{2})+3\psi_{4}(q^{3})-42\psi_{4}(q^{6})\right\} \]
(where $q:=e^{2\pi i\tau}$ and $\psi_{4}(q)=\sum_{M\geq1}\frac{q^{M}}{M}\{\sum_{r|M}r^{3}\}$)
which must coincide modulo $\QQ(3)$ under the {}``period map''
$\tau(t)=\frac{\int_{\varphi_{1}}\omega_{t}}{\int_{\varphi_{0}}\omega_{t}}$
(see $\S8.3$). 

In general when a toric-hypersurface pencil arising from Theorem 0.1
is modular (in a sense to be made precise in $\S8.3$), the limit
MHS at $t=0$ is trivialized by taking $q:=\exp(\frac{2\pi i}{N}\tau(t))$
(for some $N\in\ZZ$) as the local parameter (or more generally $t_{0}$
with $\lim_{t\to0}\frac{q(t)}{t_{0}(t)}$ a root of unity). An example
of a $non$modular case --- with nontrivial LMHS (see $\S8.6$) ---
is the mirror quintic family obtained from $\phi=x+y+z+w+\frac{1}{xyzw}$.
It follows that the Fermat quintic family $\tilde{\X}$ obtained from
$(0.7)$ (of which the mirror quintic is essentially a quotient) also
has extensions in $H_{\text{lim}}^{3}(\tilde{X}_{0})$ $not$ trivializable
by change of parameter. What is still true is that we have the splitting
$(0.3)$ of MHS\[
H^{3}(\tilde{X}_{0})\twoheadrightarrow\QQ(0)\]
induced by $\left\langle \,\cdot\,,\widehat{\omega(0)}\right\rangle $,
and inducing \[
J^{2}(\tilde{X}_{0})\rOnto^{\theta}\CC/\QQ(2).\]
This follows from the existence of $\Xi$ in the Theorem, and is $false$
if we change the coefficients in $(0.7)$ (e.g. writing instead $\phi=\frac{x^{5}+2y^{5}+7z^{5}+w^{5}+1}{xyzw}$)
without regard for the {}``generalized temperedness'' criterion.

Sticking with the Fermat family, here is why this is important. Let
$\mathsf{D}^{*}:=\mathsf{D}\m\{0\}\subset\PP^{1}$ be a punctured
disk about $t=0$, and suppose we are given a {}``local'' family
of cycles $\{Z_{t}\in Z_{\text{hom}}^{2}(\tilde{X}_{t})\}_{t\in\mathsf{D}^{*}}$
satisfying $\sZ^{*}:=\cup_{t\in\mathsf{D}^{*}}Z_{t}\homeq0$ on $\tilde{\pi}^{-1}(\mathsf{D}^{*})\subset\tilde{\X}$.
Then by \cite[sec. III.B]{GGK1} $\lim_{t\to0}AJ_{\tilde{X}_{t}}(Z_{t})\in J^{2}(\tilde{X}_{0})$
is well-defined,%
\footnote{as an invariant of the family of \emph{rational equivalence classes}.%
} and by applying $\theta$ so is $\theta(\lim_{t\to0}AJ_{\tilde{X}_{t}}(Z_{t}))=\lim_{t\to0}\nu(t)=:\nu(0)$
(cf. $(0.4)$). In {[}op. cit., $\S$IV.C] such a family is constructed,
with%
\footnote{$D_{2}=$Bloch-Wigner function%
}\[
\Im(\nu(0))=D_{2}(\sqrt{-3});\]
 and so the general $Z_{t}\nrateq0$.

To conclude, we comment on a few intriguing issues arising in the
present work, which might form the basis for later projects. We would
like to have a better understanding of the geometry of families of
$K3$ surfaces supporting $K_{3}$-classes which are $not$ Eisenstein
symbols. There are scores%
\footnote{corresponding to about a quarter of the 4319 reflexive polytopes in
$\RR^{3}$; see $\S1.3$.%
} of Laurent polynomials $\phi\in\QQ[\mathbf{T}^{3}]$ satisfying Theorem
1.7, but (for example) we are only able to show the generic Picard
number $\text{rk}(Pic(X_{\eta}))=19$ for a handful of these. While
there are techniques for obtaining $lower$ bounds on this number,
we are aware of no methods for (nontrivially) bounding it $above$.
Do any of the families have $generic$ Picard rank $<19$? Are any
of them not elliptic fibrations? In fact, on those that admit a $torically$
$defined$ elliptic fiber structure, we are able to construct (and
partially evaluate the regulator on) families of $K_{1}$-classes.

For CY 3-folds, it turns out that $none$ of the $K_{4}$-classes
constructed by Theorem 1.7 are Eisenstein symbols, because none of
the allowed CY families are classically modular (Prop. 8.15). This
would likely be remedied by generalizing the construction to admit
singularities on the generic fiber as we have done for $K3$'s; this
hard work has yet to be done.

The conjectural mirror theorem of $\S3.4$ relates Hodge theory of
the (open $CY$ $3$-fold) $B$-model family $Y_{t}:=\{1-t\phi(\underline{x})+u^{2}+v^{2}=0\}\subset(\CC^{*})^{2}\times\CC^{2}$
to enumerative geometry of the ($A$-model) total space of the canonical
bundle $K_{\PP_{\Delta^{\circ}}}$. But the mirror map and the VHS
$H^{3}(Y_{t})$ are determined from the data of the underlying elliptic
curve family $X_{t}^{*}=\{1-t\phi(\underline{x})=0\}\subset(\CC^{*})^{2}$
and the toric symbol $\{x_{1},x_{2}\}\in K_{2}(X_{t}^{*})$ (whose
$AJ$ class in $\ext(\QQ(0),H^{1}(X_{t}^{*},\QQ(2)))$ projects to
$H^{3}(Y_{t})$, cf. Prop. 3.5ff). The mirror $X^{\circ}$ of $\{X_{t}\}$
is the (elliptic curve) zero locus of a section of $K_{\PP_{\Delta^{\circ}}}^{\vee}$.
Is it possible to recast the Gromov-Witten invariants of $K_{\PP_{\Delta^{\circ}}}$
directly in terms of $X^{\circ}$, and thus rewrite the mirror theorem
in terms of $X_{t}\longleftrightarrow X^{\circ}$? A starting point
might be to think of $H^{\text{even}}(K_{\PP_{\Delta^{\circ}}})$
as an extension of $H^{\text{even}}(X^{\circ})$ by $\QQ(0)$ and
reduce the quantum product to one on $H^{\text{even}}(X^{\circ})$.

A. Collino \cite{Co} has studied the behavior of the Ceresa cycle
associated to a nonhyperelliptic genus $3$ curve as this curve acquires
two successive nodes. Working modulo $2$-isogenies, with each degeneration
a $\mathbb{G}_{m}$ splits off from the (Jacobian) abelian variety
on which the cycle sits. Under this process $CH^{2}\left(\begin{array}{c}
\text{abelian}\\
3\text{-fold}\end{array}\right)\rightsquigarrow CH^{2}\left(\begin{array}{c}
\text{abelian}\\
\text{surface}\end{array},\,1\right)\rightsquigarrow CH^{2}\left(\begin{array}{c}
\text{elliptic}\\
\text{curve}\end{array},\,2\right)$, the Ceresa cycle limits to the Eisenstein symbol over $Y_{1}(4)$,
which should be thought of as the intersection of two boundary components
in moduli space. Obviously this admits generalization, essentially
by considering moduli of genus $3$ Jacobians with level $N$ structure.
It is of great interest, therefore, to attempt a modular computation
of the normal function for such {}``modular Ceresa cycles'', which
should limit to an integral of an Eisenstein series. Certain singularities
of this normal function in the sense of Griffiths and Green \cite{GG}
(equivalently, the residues of the corresponding Hodge class {[}op.
cit.]), must then be given by the rational residues (in the sense
of $\S5.1.5$ below) of {}``$\QQ$-Eisenstein series'' $E_{3}^{\QQ}(N)$.
It is a fundamental property of Eisenstein series that they are determined
by their residues.

In fact, there is a beautiful analogy between the picture in $\S4$
of {[}op. cit.] and the Eisenstein situation reviewed in $\S\S5-6.1$.
Given a projective variety $X^{2p}$, a $(p,p)$-class $\zeta$, and
a sufficiently ample line bundle $\mathcal{L}\to X$, the infinitesimal
invariant of $\zeta$ (pulled back to the incidence variety $\X\subset X\times\PP H^{0}(\mathcal{O}_{X}(\mathcal{L}))$)
maps to certain {}``residues'' over higher-codimension substrata
of $X^{\vee}\subset\PP H^{0}(\mathcal{O}_{X}(\mathcal{L}))$. An explicit
form of Deligne's {}``Hodge $\Longrightarrow$ Absolute Hodge''
conjecture, is that this map should be injective on Hodge classes%
\footnote{The point is that the map preserves $\QQ$-structure and the target
$\QQ$-structure is {}``algebraic'' (in the sense of being Galois-invariant).%
} --- that is, that the $rational$ $(p,p)$ classes are {}``generalized
$\QQ$-Eisenstein series''. That all such should be motivated by
a {}``generalized Eisenstein symbol'' is, of course, the Hodge Conjecture.
In the context of Kuga varieties over modular curves (and higher cycles),
we have spelled out how Beilinson's work established the relevant
(Beilinson-)Hodge Conjecture in $\S\S5-6.1$ below.$\vspace{2mm}$\\
\textbf{Acknowledgments:} It gives us great pleasure to thank M.-J.
Bertin, S. Bloch, M. Lalin, V. Maillot, J. Stienstra, and F. Rodriguez-Villegas
for stimulating conversations and interest in this work; and especially
M. Nori for suggesting the view of our Theorem 1.7 with which this
Introduction began, and A. Scholl for explaining Beilinson's construction
to us. This paper has also benefitted from concurrent collaborations
of the first author with J. Morgan and the second author with P. Griffiths,
M. Green, and J. Lewis.

${}$

\section{\textbf{Constructing motivic cohomology classes on families of CY-varieties}}

The goal of this section is a combinatorial machine for producing
1-parameter families of Calabi-Yau $(n-1)$-folds%
\footnote{the small tilde does not denote a desingularization; $\tilde{X}_{t}$
can be singular.%
} $\tilde{X}_{t}$ that carry nontrivial elements $\Xi_{t}\in H_{\M}^{n}(\tilde{X}_{t},\QQ(n))$
$\forall t\in\PP^{1}\m\{0\}$, for $n=2,3,4.$ For $n=2$, our construction
is a slight extension of work \cite{RV} of Villegas. We remind the
reader that for $\tilde{X}_{t}$ smooth, working $\otimes\QQ$ (as
is our convention in this paper)\[
H_{\M}^{n}(\tilde{X}_{t},\QQ(n))\begin{array}[b]{c}
_{\cong}\\
\rightarrow\end{array}CH^{n}(\tilde{X}_{t},n)\begin{array}[b]{c}
_{\cong}\\
\leftarrow\end{array}Gr_{\gamma}^{n}K_{n}(\tilde{X}_{t}).\]
 Our construction still yields something in $H_{\M}^{n}$ for singular
members of the family, though in that case $CH^{n}(\tilde{X}_{t},n)\cong Gr_{\gamma}^{n}G_{n}(\tilde{X}_{t})$
and both isomorphisms above fail. However, by taking hyper-resolutions
as in \cite{L1}, $H_{\M}^{n}$ can still be represented by higher
Chow precycles, which allows for explicit computation \cite{KL} of
the Abel-Jacobi map\[
AJ^{n,n}:\, H_{\M}^{n}(\tilde{X}_{t},\QQ(n))\to H^{n-1}(\tilde{X}_{t},\CC/\QQ(n))\]
in terms of currents and $C^{\infty}$ chains. We will partially compute
$AJ$ in $\S2$, and deal with the singular fibers (in some cases)
in $\S4$.

\subsection{Toric data}

Our $\tilde{X}_{t}$'s will be hypersurfaces in toric Fano $n$-folds
$\PP_{\tilde{\Delta}}$. To start the construction, let\[
\sum_{\underline{m}\in\ZZ^{n}}\alpha_{\underline{m}}\underline{x}^{\underline{m}}=\phi\in K[x_{1}^{\pm1},\ldots,x_{n}^{\pm1}]\]
be a Laurent polynomial with coefficients in a number field $K\subset\CC$.
Denote by $\mathfrak{{M}}_{\phi}:=\{\underline{m}\in\ZZ^{n}|\,\alpha_{\underline{m}}\neq0\}$
the set of monomial exponents, with convex hull $\Delta\subset\RR^{n}$.
Define\[
\PP_{\Delta}:=Proj\left(\CC\left[\left\{ x_{o}^{\ell}\underline{x}^{\underline{m}}\left|\,\underline{m}\in\ell\Delta\cap\ZZ^{n},\,\ell\in\ZZ^{\geq0}\right.\right\} \right]\right)\mspace{100mu}\]
\[
\mspace{200mu}\hookleftarrow Proj\left(\CC\left[x_{0},x_{1}^{\pm1},\ldots,x_{n}^{\pm1}\right]\right)=\left(\CC^{*}\right)^{n},\]
 and in $\PP_{\Delta}$ hypersurfaces \[
X^{\lambda}:=\overline{\{\lambda-\phi=0\}}\rEq_{t=\lambda^{-1}}\overline{\{1-t\phi=0\}}=:X_{t}\,;\]
these fit together in the family\[
\PP^{1}\times\PP_{\Delta}\supset\X\rTo^{\pi}\PP^{1}.\]
The base locus is the intersection of $X_{t}$ (for any $t\neq0$)
with $\DD:=\PP_{\Delta}\m(\CC^{*})^{n}$. To describe this, denote
the codimension-$i$ faces of $\Delta$ by $(\sigma\in)\Delta(i)$,
with corresponding $(n-i)$-plane $\RR_{\sigma}.$ Take $x_{1}^{\sigma},\ldots,x_{n-i}^{\sigma}$
to be monomials corresponding to a set of generators for $\RR_{\sigma}\cap\ZZ^{n}$($\cong\ZZ^{n-i}$
after choosing an origin $o_{\sigma}$). The latter may be completed
to a generating set for $\ZZ^{n}$; denote the corresponding monomials
by $x_{n-i+1}^{\sigma},\ldots,x_{n}^{\sigma}$. Set\[
\DD_{\sigma}^{*}:=\left\{ x_{j}^{\sigma}=0\,\,\forall j>n-i,\,\,\,\neq0\,\,\forall j\leq n-i\right\} \subseteq\PP_{\Delta},\,\,\,\DD_{\sigma}:=\overline{\DD_{\sigma}^{*}}\]
and independently of $t\neq0$\[
D_{\sigma}^{(*)}:=\DD_{\sigma}^{(*)}\cap X_{t}.\]
That this intersection is proper follows from the construction; if
its multiplicity is $>1$, then $D_{\sigma}$ is non-reduced. We also
write \[
X_{t}^{*}:=X_{t}\cap(\CC^{*})^{n},\,\,\,\,\, D:=\cup_{\sigma\in\Delta(1)}D_{\sigma}=X_{t}\m X_{t}^{*},\,\,\,\,\,\DD:=\cup\DD_{\sigma}=X_{0}.\]
Now $D_{\sigma}^{(*)}$ may be described as follows: multiply $\phi$
by $\underline{x}^{-\underline{o}_{\sigma}}$ to translate the origin
of $\ZZ_{\sigma}^{n-i}$ to $\{\underline{0}\}$, rewrite the result
in terms of the $\{x_{j}^{\sigma}\}_{j=1}^{n}$, and set $x_{j>n-i}^{\sigma}=0$.
This defines a Laurent polynomial $\phi_{\sigma}(x_{1}^{\sigma},\ldots,x_{n-i}^{\sigma})$,
hence also a polytope $\Delta_{\sigma}\subset\RR^{n-i}$ (by taking
convex hull), which is just $\sigma\subset\RR_{\sigma}$ viewed as
a polytope. (We refer to $\phi_{\sigma}$ as a {}``$(n-i)$-face
polynomial'', with {}``$(n-i)$-face'' replaced by {}``edge''
or {}``facet'' if $n-i=1$ or $i=1$ respectively.) We have\[
D_{\sigma}=\overline{\{\phi_{\sigma}=0\}}\subset\PP_{\Delta_{\sigma}}\cong\DD_{\sigma}.\]
 Of course, $\PP_{\Delta}$ and $X_{t}$ (for $t$ general) may be
singular.

Given nonvanishing holomorphic functions $f_{1},\ldots,f_{\ell}\in\Gamma(Y,\mathcal{O}_{Y}^{*})$
on a quasi-projective variety $Y$, the symbol $\{f_{1},\ldots,f_{\ell}\}\in Z^{\ell}(Y,\ell)$
denotes the higher Chow cycle given by their graph in $Y\times(\PP^{1}\m\{1\})^{\ell}$.
Its class $\left\langle \{f_{1},\ldots,f_{\ell}\}\right\rangle \in CH^{\ell}(Y,\ell)$
maps to an element in Milnor $K$-theory $K_{\ell}^{M}(\CC(Y))\cong CH^{\ell}(\eta_{Y},\ell)$
which is also denoted $\{f_{1},\ldots,f_{\ell}\}$.

\begin{defn}
(i) $\phi$ \textbf{reflexive} $\Longleftrightarrow$ $\Delta$ reflexive
\cite{Ba1} (with $\{\underline{0}\}\in\RR^{n}$ as its unique integral
interior point).

(ii) $\phi$ \textbf{regular} $\Longleftrightarrow$ $\lambda-\phi$
is $\Delta$-regular \cite{Ba1} for general $\lambda\in\CC$ 

$\mspace{100mu}\Longleftrightarrow$ $D_{\sigma}^{*}$ nonsingular
and reduced ($\forall$ $i\geq1$, $\sigma\in\Delta(i)$).

(iii) $\phi$ \textbf{tempered} $\Longleftrightarrow$ coordinate
symbols $\{x_{1}^{\sigma},\ldots,x_{n-i}^{\sigma}\}$ give (working
$\otimes\QQ$ as always) trivial classes in $CH^{n-i}(D_{\sigma}^{*},n-i)$
($\forall$ $i\geq1$, $\sigma\in\Delta(i)$).
\end{defn}
\begin{rem}
Though we have been working over $\CC$, the above constructions and
definitions descend to $K$. Provided one is willing to work over
a suitable algebraic extension of $K$ (or $\bar{\QQ}$), we can discuss
irreducible components of the $D_{\sigma}^{*}$. For $n-i=1$, the
$D_{\sigma}^{*}$ components are points and must have root-of-unity
coordinates $x_{1}^{\sigma}$ if $\phi$ is tempered. (Hence we recover
Villegas's prescription for $n=2$, that the $\phi_{\sigma}$ be cyclotomic
$\forall$ $\sigma\in\Delta(1)$.) For $n-i=2$, the tempered condition
is equivalent to $\{x_{1}^{\sigma},x_{2}^{\sigma}\}$ giving torsion
classes in $K_{2}^{M}$ of the $\bar{\QQ}$-function fields of the
irreducible component curves $C$ of $D_{\sigma}^{*}$, since $ker\{CH^{2}(C,2)\to CH^{2}(\eta_{C},2)\}=\oplus_{p\in C(\bar{\QQ})}CH^{1}(p,2)=0$.
\end{rem}
Henceforth, we assume $\phi$ reflexive. This means in particular
that $\PP_{\Delta}$ is Fano and $X_{t}\in|-\mathsf{{K}}_{\PP_{\Delta}}|$.
A maximal projective triangulation of the (integral) dual polytope
$\Delta^{\circ}$ produces a crepant partial%
\footnote{for $n=4$, $\PDT$ may still have point singularities in $\cup_{\tilde{\sigma}\in\DT(4)}\DD_{\tilde{\sigma}}$.%
} desingularization $\PDT\rTo^{\mu}\PD$ \cite{Ba1}. Affine charts
for $\PDT$ are obtained from monomial generators for the integral
points of the cones dual to the cones on $tr(\Delta^{\circ})$. The
cones on $\Delta^{\circ}$ likewise provide affine charts for $\PD$;
and in both cases the relations between the monomials produce local
equations for $\PP_{(\cdot)}$. The two sets of affine charts are
related by blow-up along coordinate subspaces, and locally $\mu$
is just the proper transform. Let $\tilde{\Delta}$ be the combinatorial
dual (simplicial complex) of $tr(\Delta^{\circ})$. Note that since
the polytope $\DT$ no longer lives in $\RR^{n}$, the $Proj$ mechanism
doesn't make sense for constructing $\PDT$ from $\DT$. However,
the correspondence $\DT\longleftrightarrow\PDT$ is not merely formal:
the face structure of $\DT$ correctly describes the topology of $\tilde{\DD}:=\PDT\m(\CC^{*})^{n}$.
(It seems possible, though we have not checked, that $\tilde{\DD}$
is the Stanley-Reisner variety of $\tilde{\Delta}$.) $\tilde{\DD}$
is a normal-crossing divisor in $\PDT$.

If $\phi$ is regular and $n\leq4$, then the $\mu$-preimage of $X_{t}$
is a smooth CY-$(n-1)$-fold for general $t$; call this $\TXT$.
Denote the discriminant locus $\{t\in\PP^{1}\,|\,\TXT\text{{\, singular}}\}=:\L$,
and $\DD_{\tilde{\sigma}}^{(*)}\cap\TXT=:D_{\ST}^{(*)}$. For $\ST\in\DT(i-k)$
lying over $\sigma\in\Delta(i)$ (i.e., $\mu(\DD_{\ST})=\DD_{\sigma}$),
toric coordinates on $\DD_{\ST}\subset\PDT$ are given by $x_{1}^{\sigma},\ldots,x_{n-i}^{\sigma};\, y_{1}^{\ST},\ldots,y_{k}^{\ST}$.
Here the $\{x_{j}^{\sigma}\}$ are the toric coordinates on $\DD_{\sigma}$;
and the $\{y_{j}^{\ST}\}$ blow-up coordinates ranging freely on $D_{\ST}^{*}\cong D_{\sigma}^{*}\times(\CC^{*})^{k}\subset\TXT$,
which is cut out of $\DD_{\ST}^{(*)}$ by $(0=)\phi_{\tilde{\sigma}}(\underline{x}^{\sigma};\underline{y}^{\sigma}):=\phi_{\sigma}(\underline{x}^{\sigma})$.
Moreover, $\cup_{\ST\in\DT(1)}D_{\ST}=\tilde{D}$ is a NCD on $\TXT$;
this is clear from regularity, dual combinatorics of $\DT\longleftrightarrow tr(\Delta^{\circ})$
and the fact that facets of $tr(\Delta^{\circ})$ are elementary simplices:
e.g., $(n-2)$-faces of $tr(\Delta^{\circ})$ contain $n-1$ vertices
$\Longleftrightarrow$ edges of $\DT$ abut $n-1$ facets $\Longleftrightarrow$
$n-1$ components of $\tilde{D}$ intersect transversely in a point.
But we won't make essential use of this.

We elaborate briefly on the affine charts for $\PDT$. These are in
$1$-$1$ correspondence with vertices $\tilde{v}\in\DT(n)$ or equivalently
with facets $\tilde{f}\in tr(\Delta^{\circ})(1)$ --- i.e. the elementary
$(n-1)$-simplices occurring in the maximal triangulation. Write $\Sigma(\tilde{f})$
for the cone on $\tilde{f}$ and $\{\underline{m}_{i}\}_{i=1}^{n}$
for the integral generators of the $n$ edges of $\Sigma(\tilde{f})^{\circ}$.
(One can also view these as the $n$ edges of $\DT$ emerging from
$\tilde{v}$.) The $\{\underline{m}_{i}\}$ will generate $\Sigma(\tilde{f})^{\circ}\cap\ZZ^{n}$
iff $\tilde{f}$ is regular (see \cite{Ba1} $2.2.6$), in which case
the chart is nonsingular. More precisely, it is $\cong\CC^{n}$ with
coordinates $\{z_{i}:=\underline{x}^{\underline{m}_{i}}\}$, and the
local equation of $\tilde{X}^{\l}$ reads $P(z_{1},\ldots,z_{n})-\l z_{1}\cdots z_{n}=0$
($P$ a polynomial with nonvanishing constant term). For $n=2,3$
$(n-1=1,2$) elementary $\implies$ regular and this is all true;
for $n\geq4$ the implication is false. So in general one will have
$k$ more generators supplementing the $\{\underline{m}_{i}\}$, hence
a singular affine chart with $n+k$ monomial coordinates (and $k$
relations), and a more complicated local equation for $X^{\l}$ (e.g.
see $\S2$). The (integral) exponent vectors of the $k$ new monomials
come from the interior of the cone $\Sigma(\tilde{f})^{\circ}$ (and
of its faces). For $n=4$ the only possible singularity of the chart
is a $\QQ$-factorial terminal singularity at the origin. 

Continue to assume $\phi$ regular and $n\leq4$. For $\ST_{i}\in\DT(i)$
we may define iterated residue maps\[
CH^{n}(\PDT\m\tilde{\DD},n)\to CH^{n-1}(\DD_{\ST_{1}}^{*},n-1)\to\cdots\to CH^{n-i}(\DD_{\ST_{i}}^{*},n-i),\]
given a choice of flag $(\ST_{i}\subsetneq)\,\ST_{i-1}\subsetneq\cdots\subsetneq\ST_{1}$,
$\ST_{j}\in\DT(j)$. The composition is independent of the choice,
and is denoted $Res_{\ST_{i}}^{i}$; a similar construction yields
$Res_{\ST}^{i}:\, CH^{n}(\TXT\m\tilde{D},n)\to CH^{n-i}(D_{\ST}^{*},n-i)$
for $t\notin\L$. If we remove tildes, the $Res_{\sigma}^{i}$ still
make sense; note in particular that all singularities (on $\PD,\, X_{t},\,\DD_{\sigma},\, D_{\sigma}$
for any $\sigma$) are in codimension $\geq2$. For example, if $\sigma'\subsetneq\sigma$
($\sigma'\in\Delta(i+1)$, $\sigma\in\Delta(i)$) with toric coordinates
$x_{1}^{\sigma}=x_{1}^{\sigma'}$, $\ldots$, $x_{n-i-1}^{\sigma}=x_{n-i-1}^{\sigma'}$,
$x_{n-i}^{\sigma}$ on $\DD_{\sigma'}^{*}$, one has a smooth affine
neighborhood $\DD_{\sigma'}^{*}\times\AA_{x_{n-i}^{\sigma}}^{1}\subset\DD_{\sigma}$.
This allows for easy computation of the iterated residues.

Let $\xi:=\left\langle \{x_{1},\ldots,x_{n}\}\right\rangle \in CH^{n}\left((\CC^{*})^{n}=\PDT\m\tilde{\DD}=\PD\m\DD,\, n\right)$
denote the class of the coordinate symbol. For $t\notin\L$ this restricts
to $\xi_{t}\in CH^{n}(X_{t}^{*}=\TXT^{*},\, n)$, either by pulling
back the $\{x_{i}\}$ directly or by invoking contravariant functoriality
of higher Chow groups ($\otimes\QQ$) for arbitrary morphisms between
smooth varieties \cite{L2}.

\begin{lem}
The diagram

\xymatrix{& & & CH^n(\PDT \m \tilde{\DD},n)  \ar [r] ^{Res^i_{\ST}} \ar [d] ^{I^*_t} & CH^{n-i}(\DD_{\ST}^*,n-i) \ar [d] ^{I_{\ST}^*} \\ & & & CH^n(\TXT^*,n) \ar [r] ^{Res^i_{\ST}} & CH^{n-i}(D_{\ST}^*,n-i)}

 ${}$\\
commutes for any $\ST\in\Delta(i)$, as does a similar diagram with
all tildes removed.
\end{lem}
\begin{proof}
With or without tildes, this is based on iterated application ($\ell=0,1,\ldots,i-1$)
of a quasi-isomorphism which may be proved using the moving lemmas
of \cite{L2} and \cite{Bl1}. Writing \[
\DD^{[i]}:=\cup_{\sigma\in\Delta(i)}\DD_{\sigma},\,\,\,\DD^{[0]}:=\PD,\,\,\, D^{[i]}:=X_{t}\cap\DD^{[i]},\]
this is \[
\frac{Z^{n-\ell}(\DD^{[\ell]}\m\DD^{[\ell+2]},\bullet)_{_{D^{[\ell]}\m D^{[\ell+2]}}}}{\iota_{_{*}}\left(Z^{n-\ell-1}(\DD^{[\ell+1]}\m\DD^{[\ell+2]},\bullet)_{_{D^{[\ell+1]}\m D^{[\ell+2]}}}\right)}\rTo^{\simeq}Z^{n-\ell}\left(\DD^{[\ell]}\m\DD^{[\ell+1]},\bullet\right)_{_{D^{[\ell]}\m D^{[\ell+1]}}}.\]
A $\db$-closed element on the r.h.s. can therefore be moved into
good position, extended to $\DD^{[\ell]}\m\DD^{[\ell+2]}$, and differentiated
(to yield a cycle supported on $\DD^{[\ell+1]}\m\DD^{[\ell+2]}$),
compatibly with pullbacks to $X_{t}$.
\end{proof}
The point is to use the lemma to compute the $Res_{\ST\text{{\, or\,}}\sigma}^{i}$
(bottom row) on $\xi_{t}$. For one thing, it is clear that the result
is constant in $t$ and descends to $CH^{n-i}\left((D_{\ST\text{{\, or\,}}\sigma}^{*})_{K},n-i\right)$.
The next result follows easily from the lemma combined with the foregoing
discussion.

\begin{prop}
For $t\notin\L$, $\sigma\in\Delta(i)$, and $\ST\in\DT(i-k)$ lying
over $\sigma$ in the above sense,\[
Res_{\sigma}^{i}\xi_{(t)}=(I_{\sigma}^{*})\left\langle \pm\{x_{1}^{\sigma},\ldots,x_{n-i}^{\sigma}\}\right\rangle \]
\[
Res_{\ST}^{i-k}\xi_{(t)}=(I_{\ST}^{*})\left\langle \pm\{x_{1}^{\sigma},\ldots,x_{n-i}^{\sigma},y_{1}^{\ST},\ldots,y_{k}^{\ST}\}\right\rangle ,\]
where the parenthetical expressions are optional.
\end{prop}
It follows that if all $Res_{\sigma}^{i}\xi_{t}$ are trivial (hence,
if $\phi$ is tempered), then so are all $Res_{\ST}^{i}\xi_{t}$ ---
in particular, all $Res_{\ST}^{1}$'s.

\begin{rem}
(i) The regularity assumption on $\phi$ is not strictly necessary
for these results. For $n=2$, we need only ask that the general $\TXT$
(equivalently, $X_{t}$) be nonsingular; whereas for $n=3$ $A$-$D$-$E$
(rational) singularities are allowed (on $\TXT$) provided they occur
in $\tilde{D}^{[2]}:=\cup_{\ST\in\DT(2)}D_{\ST}$. Note however that
in Proposition 1.4 the formulas for $Res_{\sigma\text{{\, or\,}}\tilde{\sigma}}^{i}\xi_{t}$
(not $\xi$) are multiplied by the multiplicity of (components of)
$D_{\sigma\text{{\, or\,}}\tilde{\sigma}}$ in case these are nonreduced.

(ii) The $Res_{\sigma}^{i}$, $Res_{\ST}^{i-k}$ are trivially $0$
on $CH^{n}(\TXT^{*},n)$ (hence on $\xi_{t}$) for $i=n$ (in particular,
for $\ST$ lying over a point), since $D_{\sigma},D_{\ST}=\emptyset$
in that case.
\end{rem}

\subsection{Completing the coordinate symbol}

Turning our attention to the family, we define ($\l=t^{-1}$)\[
\tilde{\X}:=\{(\l,x)\,|\, x\in\tilde{X}^{\l}\}\subseteq\PP_{\l}^{1}\times\PDT.\]
Recalling that $\tilde{X}_{0}=\tilde{X}^{\infty}=\tilde{\DD}$, set\[
\tilde{\X}_{-}:=\tilde{\X}\m(\{\infty\}\times\tilde{X}^{\infty})\subset\AA_{\l}^{1}\times\PDT,\]
and noting that $\TXM\cap\AA^{1}\times\tilde{\DD}\cong\AA^{1}\times\tilde{D}$,\[
\TXM^{*}:=\TXM\m\AA^{1}\times\tilde{D}=\{(\l,x)\,|\, x\in(\TXL)^{*}\}\subset\AA^{1}\times(\CC^{*})^{n}.\]

\begin{defn}
We say \textbf{$\xi$ ($\in H_{\M}^{n}((\CC^{*})^{n},\QQ(n))$) completes
to a family of motivic cohomology classes}, if $\exists$ $\Xi\in H_{\M}^{n}(\TXM,\QQ(n))$
such that the pullbacks of $\xi,\Xi$ to $H_{\M}^{n}((\TXL)^{*},\QQ(n))$
agree $\forall$ $\l\in\AA^{1}$. That is, in the diagram

\begin{equation} \\
\xymatrix{{\Xi} \in \ar @{|->} [d] & H^n_{\M}(\TXM,\QQ(n)) \ar [d]_{(\iota^{\l})^*} & H^n_{\M}((\CC^*)^n,\QQ(n)) \ar [d]^{(I^{\l})^*} & \ni \xi \ar @{|->} [d] \\ {\Xi^\l} \in & H^n_{\M}(\TXL,\QQ(n)) \ar [r]_{r^{\l}} & H^n_{\M} ((\TXL)^*,\QQ(n)) & \ni \xi^{\l}} \\
\end{equation} ${}$\\
we must have for each $\l$, $r^{\l}(\Xi^{\l})=\xi^{\l}$. (Here $\TXM$,
$\TXL$, and even $(\TXL)^{*}$ may all be singular.)
\end{defn}
To state general conditions under which we can produce such a $\Xi$,
we introduce some more notation (mainly for subsets of $\tilde{D}$).
When $\phi$ is not regular, it has a nonempty irregularity locus\[
\mathcal{{I}}:=\text{{\, union\, over\, all\,\,}}\ST\text{{\, of\, singularities\, or\, nonreduced\, components\, of\,\,}}D_{\ST}^{*}\]
(which is just where $\phi_{\ST}$ vanishes together with all its
partials). Writing $\II^{n}:=\cup_{i}\{x_{i}=1\}\subset\PDT$ (where
$\{x_{i}\}_{i=1}^{n}\subset K(\PDT)^{*}$ extend the $(\CC^{*})^{n}$-coordinates),
set\[
\mathcal{{J}}:=\text{{\, union\, of\, all\,\,}}D_{\ST},\,\ST\in\DT(1),\text{{\, which\, are\, not\, contained\, in\,\,}}\II^{n}\cap\tilde{\DD}.\]
For $n=3$ specifically, where we will allow $A_{1}$-singularities
(ordinary double points) on the general $\TXL$ (but only at $\tilde{D}^{[2]}$),
write $\mathcal{{A}}$ ($\subseteq\mathcal{{I}}$) for the collection
of these,\[
\{\alpha_{1},\ldots,\alpha_{k}\}:=\mathcal{{A}}\cap\mathcal{{J}},\text{{\,\, and}}\mspace{100mu}\]
\[
\{\D_{1},\ldots,\D_{\ell}\}:=\text{{irreducible\, curves\, in\,}}\tilde{D}\mspace{10mu}\]
\[
\mspace{200mu}\text{{\, avoiding\, the\, set\,}}(\mathcal{{A}}\m\mathcal{{A}}\cap\mathcal{{J}})\cup(\mathcal{{I}}\m\mathcal{{A}}).\]
There is a linear map of vector spaces\[
\E:\,\QQ\left\langle \D_{1},\ldots,\D_{\ell}\right\rangle \to\QQ\left\langle \alpha_{1},\ldots,\alpha_{k}\right\rangle \]
obtained by sending generators $[\D_{i}]\mapsto\sum_{\alpha_{j}\in\D_{i}}[\alpha_{j}]$.

\begin{thm}
Let $\phi$ be reflexive and tempered, $n\leq4$. Also assume in case\\
\\
$\underline{n=2}$: the general $X^{\l}$ is nonsingular.\\
\\
$\underline{n=3}$: (a) the general $\TXL$ is nonsingular apart from
$A_{1}$-singularities at

$\mspace{60mu}$points $\A\subseteq\II^{3}\cap\tilde{\DD}^{[2]}$;

$\mspace{40mu}$(b) $\I\subseteq\II^{3}(\cap\tilde{\DD})$, $\I\cap\J\subseteq\A$;
and

$\mspace{40mu}$(c) either

$\mspace{60mu}$(i) $\E$ is surjective, or

$\mspace{60mu}$(ii) $K$ is totally real and the irreducible component
curves of $\tilde{D}$ 

$\mspace{80mu}$are nonsingular and defined over $K$.\\
\\
$\underline{n=4}$: (a) $\phi$ is regular,

$\mspace{40mu}$(b) $K$ is totally real, and

$\mspace{40mu}$(c) each irreducible component of each $D_{\sigma}$,
$\sigma\in\Delta(2)$ resp. $\Delta(3)$, 

$\mspace{60mu}$admits a dominant morphsim defined over $K$ from
$\AA^{1}$ resp. $\AA^{0}$.\\
\\
Then $\xi$ completes to a family of motivic cohomology classes (see
Defn. 1.6).
\end{thm}
\begin{rem}
(i) For ease of application we have stated the additional requirements
for $n=2,4$ in terms of $X^{\l},D$; whereas for $n=3$ they are
phrased in terms of $\TXL,\tilde{D}$. (We are $not$ saying all singularities
must be $A_{1}$'s on $X^{\l}$; just that $A_{1}$'s are all that
remains after passing to $\TXL$.)

(ii) The additional requirements for $n=3$ may be significantly relaxed
if all we want to do is complete $\xi$ to a class in $H_{\M}^{n}(\TXL,\QQ(n))$
for some fixed $\l$. Obviously, taking $\l$ to be very general and
spreading out would then also yield a class in $H_{\M}^{n}(\TXM\times_{\rho,\mathbb{A}^{1}}U,\QQ(n))$
for some \'etale neighborhood $U\overset{\rho}{\to}\mathbb{A}^{1}$
--- i.e. not on the family $\TXM$ but on a finite pullback. Here
are two possibilities:\\
(1) Drop {}``general'' in (a), drop requirement (b), assume (c)(i)
(but only make $\{\D_{i}\}$ avoid $\A\m\A\cap\J$ in the definition
of $\E$). If $\TXL$ is smooth, (c)(i) is empty.\\
(2) Allow $A$-$D$-$E$ singularities (call the set of these $\A'$):
more precisely, $\TXL$ nonsingular except at $\A'\subseteq\II^{3}\cap\DD^{[2]}$;
and each irreducible component of $\J$ contains at most one point
of $\A'$. (We should also note that $\TXL$ is still a {[}singular]
$K3$ surface in this case, and its minimal desingularization is a
smooth $K3$.)

(iii) With the caveat that the following simplification comes at the
expense of important examples, all three additional requirements (for
$n=3$) may be done away with if we assume $\phi$ regular: in fact,
$(a)$, $(b)$, and $(c)(i)$ collapse.

(iv) We make no claim that this result is exhaustive for $n=3$ or
$4$. Indeed, if (for $n=3$) the general $\TXL$ is nonsingular and
$\I\subset(\cup D_{\tilde{\sigma}}^{*})\cap\II^{3}$ consists of $K$-rational
points ($K$ totally real), then (although we may not have $\I\cap\J=\emptyset$)
the conclusion still holds.
\end{rem}
\begin{proof}
Noting that $\TXM^{*}\cong(\CC^{*})^{n}$ and that the resulting map
\[
H_{\M}^{n}(\TXM,\QQ(n))\rTo^{r}H_{\M}^{n}((\CC^{*})^{n},\QQ(n))\]
 completes equation $(1.1)$ to a commutative diagram, it suffices
to construct $\Xi\in r^{-1}(\xi)$.

Before doing so, we briefly sketch how the map $(\iota^{\delta})$
to $H_{\M}^{n}(\tilde{X}^{\delta},\QQ(n))$ can be computed explicitly
in terms of higher Chow cycles, when $\delta\in\L$ ($\implies$ $\tilde{X}^{\delta}$
is singular with desingularization $\widetilde{\tilde{X}^{\delta}}$).
For simplicity, assume $sing(\tilde{X}^{\delta})=:\ms$, $\widetilde{\tilde{X}^{\delta}}\times_{\tilde{X}^{\delta}}\ms=:\ms'$,
and $\TXM$ are smooth: then $H^{-n}$ of \[
\hat{Z}^{n}(\tilde{X}^{\delta},-\bullet):=Cone\left\{ Z^{n}(\widetilde{\tilde{X}^{\delta}},-\bullet)_{_{\ms'}}\oplus Z^{n}(\ms,-\bullet)\rTo_{\text{{pullbacks}}}^{\text{{diff.\, of}}}Z^{n}(\ms',-\bullet)\right\} [-1]\]
computes $H_{\M}^{n}(\tilde{X}^{\delta},\QQ(n))$. (In general, $Z^{n}$
of $\ms,\ms'$ must each be replaced by a $Cone$ complex, also denoted
$\hat{Z}^{n}$.) Assuming $\Xi$ has been produced, and representing
it by a cycle in $Z^{n}(\TXM,n)_{_{\ms\cup\tilde{X}^{\delta}}}$,
a representative of $(\iota^{\delta})^{*}\Xi$ is obtained by pulling
back to $\widetilde{\tilde{X}^{\delta}}$ and $\ms$ (which gives
a triple of the form $(*,*,0)$).

Now, we will first explain the construction of $\Xi$ in case the
total space $\TXM$ (and fixed general $\TXL$) is nonsingular, as
is the case when $\phi$ is regular. (However, we don't assume that
$\tilde{D}$ is a NCD or even that its components are smooth.) In
the (commutative) diagram

\small \xymatrix{{\xi\in} \ar @{|->} [d] & CH^n((\TXM^*)_K,n) \ar [r]^{Res^1_{\ST}\mspace{50mu}} \ar [d]_{(\iota^{\l})^*} & CH^{n-1}((D_{\ST}^*\times \AA^1)_K,n-1) \ar [d] & CH^{n-1}((D_{\ST}^*)_K,n-1) \ar [l] _{\cong} \ar  @{^(->} [ld] \\ {\xi^{\l}\in} & CH^n((\TXL)^*_{\CC},n) \ar [r] ^{Res^1_{\ST}\mspace{30mu}} & CH^{n-1} ((D_{\ST}^*)_{\CC},n-1),}

\normalsize ${}$\\
our hypothesis that $\phi$ is tempered (together with Proposition
1.4) implies $Res_{\ST}^{1}\xi^{\l}=0$, hence that $Res_{\ST}^{1}\xi=0$
$\forall$ $\ST\in\DT(1)$. The local-global spectral sequence\[
E_{1}^{i,-j}(n):=\left\{ \begin{array}{cc}
CH^{n}(\TXM^{*},j)\,[\cong H_{\M}^{n}((\CC^{*})^{n},\QQ(n))] & ,\,\,\, i=0\\
\\\oplus_{\ST\in\DT(i)}CH^{n-i}(D_{\ST}^{*}\times\AA^{1},j-i) & ,\,\,\, i>0\\
\\0 & ,\,\,\, i<0\end{array}\right.\]
with $d_{1}:\, E_{1}^{0,-n}(n)\to E_{1}^{1,-n}(n)$ given by $\oplus_{\ST\in\DT(1)}Res_{\ST}^{1}$,
has \[
E_{\infty}^{0,-n}(n)\cong im\left\{ CH^{n}(\TXM,n)\to CH^{n}(\TXM^{*},n)\right\} \]
\[
\cong\bigcap ker\left\{ d_{i}:\, E_{i}^{0,-n}(n)\to E_{i}^{i,n-i+1}(n)\right\} \]
\[
\cong\left\{ \begin{array}{cc}
ker(d_{1}) & ,\text{{\, for\,}}n=2,3\\
\\ker(d_{1})\cap ker(d_{2}) & ,\text{{\, for\,}}n=4\end{array}\right..\]
(Warning: the $d_{i}$ are not the above $Res^{i}$ for $i>1$; see
\cite{Ke1} for a description.) So for $n=2,3$ we automatically get
the desired class $\Xi\in CH^{n}(\TXM,n)\cong H_{\M}^{n}(\TXM,\QQ(n))$.

For $n=4$, the stated conditions imply that the $\{D_{\ST}^{*}\}_{\ST\in\DT(2)}$
are Zariski-open subsets $U\subseteq\AA_{K}^{1}$ (obtained by omitting
points wsith coordinates $\in K$). Since $CH^{1}(\text{{pt.}},3)$
is zero, $CH^{2}(U,3)\cong CH^{2}(\AA_{K}^{1},3)\cong CH^{2}(Spec(K),3)\cong K_{3}^{ind}(K)=0$
for $K$ totally real (\# field); since $E_{2}^{2,-5}(4)$ is a subquotient
of $\oplus_{\ST\in\DT(2)}CH^{2}((D_{\ST}^{*}\times\AA^{1})_{K},3)$
we are done.

So we have reduced to examining additional complications arising from
the case of $\TXM$ singular insofar as this is allowed by the conditions
of the Theorem. If $n=2$, the singularities occur in $\tilde{D}\times\L$
and are always rational (surface) singularities of type $A_{1}$,
$A_{2}$, or $A_{3}$ (see \cite{BPV} for defintion). The last observation
is verified using the table of $16$ $2$-dimensional reflexive polytopes
in \cite{BSk}. Briefly, a singularity $Q\in sing(\TXM)$ occurs due
to a multiple root $r_{Q}$ of $\phi_{\sigma}(x_{1}^{\sigma})$ for
some $\sigma\in\Delta(1)$. In a neighborhood of $\{(x_{1}^{\sigma}-r_{Q},\, x_{2},\,\lambda-\delta)=(0,0,0)\}=Q$
the equation of $\TXM$ is of the form\[
0\,=\,(x_{1}^{\sigma}-r_{Q})^{k}\Psi_{1}(x_{1}^{\sigma}-r_{Q})\,+\,(x_{2}^{\sigma})^{\ell(>0)}\Psi_{2}(x_{1}^{\sigma}-r_{Q},x_{2})\,-\,(\l-\delta)(x_{1}^{\sigma}-r_{Q})x_{2}^{\sigma}-(\l-\delta)x_{2}^{\sigma},\]
where $\Psi_{1},\,\Psi_{2}$ are holomorphic ($\neq0$ at $Q$) and
$2\leq k\leq4$. (Note $(\l-\delta)x_{2}^{\sigma}$ is quadratic and
nonzero, and is not cancelled out.) At any rate, the canonical desingularization
\cite{BPV} produces $\widetilde{\TXM}\rTo^{b}\TXM$ with $b^{-1}(Q)=$
a chain $\RR_{Q}$ of ($1$, $2$, or $3$) rational curves for each
$Q\in sing(\TXM)$. Writing $\widetilde{\TXM}^{*}:=b^{-1}(\TXM^{*})\cong(\CC^{*})^{2}$,
there are some extra $Res^{1}$'s of $\xi\in CH^{2}(\widetilde{\TXM}^{*},2)$
to deal with, in $CH^{1}(\UU_{\QQ},1)$ for $\UU_{\QQ}\subseteq\RR_{\QQ}$
Zariski open. But this is clearly just (for $Q=\{(r_{Q},\delta)\}\in D_{\ST}\times\L$
as above) $\{r_{Q}\}$, which is necessarily a root of unity (due
to the tempered requirement), hence trivial. So $\xi$ comes from
$\Xi\in CH^{2}(\widetilde{\TXM},2)$. In view of the long-exact sequence
{[}with $\sqcup=\sqcup_{Q\in sing(\TXM)}$]\[
\to H_{\M}^{2}(\TXM,\QQ(2))\to CH^{2}(\widetilde{\TXM},2)\oplus CH^{2}(\sqcup Q,2)\to H_{\M}^{2}(\sqcup\RR_{Q},2)\to\]
and the identfication of $CH^{2}(Q,2)$ and $H_{\M}^{2}(\RR_{Q},\QQ(2))$
(working over $\bar{K}=\bar{\QQ}$) with $K_{2}^{M}(\bar{\QQ})=0$,
$\Xi$ descends to $H_{\M}^{2}(\TXM,\QQ(2))$.

If $n=3$, then we admit fiberwise $A_{1}$-singularities $\alpha$;
since these live in $\tilde{D}^{[2]}$, their location in $\PDT$
is fixed as $\l$ varies. So for each $\alpha\in\A$, $\{\alpha\}\times\AA^{1}\subseteq sing(\TXM)$.
Since these are ordinary double points, a minimal resolution for the
generic fiber is effected merely by blowing up $\PDT$ at each $\alpha$.
(The proper transform $\hat{\X}_{-}\subset Bl_{\A}(\TXM)$ of $\TXM$
is still possibly singular over a discriminant set $=:\L\subset\AA^{1}$.)
We write $\hat{\X}_{-}\rTo^{B}\TXM$ for the resulting morphism, which
has its own {}``exceptional divisors'' $B^{-1}(\alpha\times\AA^{1})$
and proper transforms $\hat{D}\,(\times\AA^{1})$ of $\tilde{D}\,(\times\AA^{1})$. 

Let $\PP_{\alpha}^{2}$ denote the exceptional divisor in $Bl_{\A}(\PDT)$
over $\alpha\in D_{\ST}$, $\ST\in\DT(2)$; and let $X,Y,Z$ be homogeneous
coordinates with $X=0,\, Y=0$ the equations of $\PP_{\alpha}^{2}\cap\hat{\DD}_{\ST_{1}}$,
$\PP_{\alpha}^{2}\cap\hat{\DD}_{\ST_{2}}$ (where $\ST_{1},\ST_{2}$
are the facets of $\DT$ meeting $\ST$). The equation for $B^{-1}(\alpha\times\AA^{1})\subseteq\PP_{\alpha}^{2}\times\AA_{(\l)}^{1}$
must be of the form

\begin{equation} 
f(X,Y,Z)+\l XY = 0 \\
\end{equation}with $f\nequiv0$ of homogeneous degree $2$.

Let $\{p_{i}\}_{i=1}^{4}$ denote the (not necessarily distinct) points
of intersection of $f=0$ and $XY=0$. Stereographic projection, say,
through $p_{1}$ to the $Z=0$ line {}``uniformizes'' the conic
(uniformly in $\l$), so that $B^{-1}(\alpha\times\AA^{1})\cong\PP^{1}\times\AA^{1}=:\PP_{\alpha}^{1}\times\AA^{1}$.
(If (c)(ii) holds, then this can be done over $K$.) Clearly the $\{p_{i}\}$
are the points where $\hat{D}_{\ST_{1}},\hat{D}_{\ST_{2}}$ meet the
conic $(1.2)$. Since they and their images $q_{i}\in\PP_{\alpha}^{1}$
under projection are constant in $\l$, we see that \[
B^{-1}(\alpha\times\AA^{1})\cap(\hat{D}_{\ST_{j}}\times\AA^{1})=\left\{ \begin{array}{cc}
q_{1}\cup q_{2} & \text{{\,\, if\,\,}}j=1\\
\\q_{3}\cup q_{4} & \text{{\,\, if\,\,}}j=2\end{array}\right\} \times\AA^{1}\subseteq\PP_{\alpha}^{1}\times\AA^{1},\]
for $j=1,2$. 

Suppose a component $D_{\I}$ of (say) $D_{\ST_{1}}$ passing through
$\alpha$ belongs to $\I$. Since $\I\subseteq\II^{3}\cap\tilde{\DD}$,
some $x_{i}\equiv1$, and another $x_{j}\equiv0$ or $\infty$ on
$D_{\I}$. Hence $D_{\I}$ is a double line (double in the sense of
the multiplicity of $\TXL\cdot D_{\ST_{1}}$ there); this means that
$p_{1}=p_{2}$ and no other components of $D_{\ST_{1}}$ pass through
$\alpha$. It follows that any component of $\J$ passing through
$\alpha$ belongs to $D_{\ST_{2}}$ and has tangent line (at $\alpha$)
distinct from $T_{\alpha}D_{\I}$ (i.e., $\{p_{1},p_{2}\}$ and $\{p_{3},p_{4}\}$
are disjoint). Since $\I\cap\J\subseteq\A$, this argument makes it
clear that the proper $B$-transforms of $\I\,(\times\AA^{1})$ and
$\J\,(\times\AA^{1})$ do not meet.

Now $\hat{\ms}:=sing(\hat{\X}_{-})\subseteq\hat{\I}\times\L$, hence
does not intersect $\hat{\J}\times\AA^{1}$ (the proper transform
of $\J\times\AA^{1}$). Let $\widetilde{\hat{\X}_{-}}\rTo^{\beta}\hat{\X}_{-}$
be a desingularization (which is an $\cong$ off $sing(\hat{\X}_{-})$),
and write $\Q_{\alpha}:=\beta^{-1}(\PP_{\alpha}^{1}\times\AA_{(\l)}^{1})$,
$\cup_{\alpha\in\A}\Q_{\alpha}=:\Q$. Obviously $\beta^{-1}(\hat{\J}\times\AA^{1})\cong\hat{\J}\times\AA^{1}$,
so we may write $\Q^{-}:=\Q\m(\hat{\J}\times\AA^{1})\cap\Q$; the
$\Q_{\alpha}$ are rational surfaces, and the $\Q_{\alpha}^{-}$ have
rational curves missing. Finally, put $\ms:=sing(\TXM)=B(\hat{\ms})\cup(\A\times\AA^{1})$
and $b:=B\circ\beta:\,\widetilde{\hat{\X}_{-}}\to\TXM$, and note
that $b^{-1}(\ms)=\beta^{-1}(\hat{\ms})\cup\Q$. As above, we want
to use the l.e.s.\[
\to H_{\M}^{3}(\TXM,\QQ(3))\to CH^{3}(\widetilde{\hat{\X}_{-}},3)\oplus H_{\M}^{3}(\ms,\QQ(3))\rTo^{i^{*}-b^{*}}H_{\M}^{3}(b^{-1}(\ms),\QQ(3))\to\]
to obtain a class $\Xi$ in the first term from a pair $(\Xi_{0},0)$
in the middle, with $i^{*}\Xi_{0}=0$.

To construct $\Xi_{0}$, begin with the coordinate symbol $\xi\in Z^{3}((\widetilde{\hat{\X}_{-}}\m b^{-1}(\tilde{D}))\cong(\CC^{*})^{3},\,3)$,
which (as $\I\subseteq\II^{3}$) obviously extends to $\xi\in Z_{_{\db-cl}}^{3}(\widetilde{\hat{\X}_{-}}\m\hat{\J}\times\AA^{1},\,3)_{_{\beta^{-1}(\hat{\ms}\cup\Q^{-})}}$.
(It actually pulls back to $0$ on $\beta^{-1}(\hat{\ms})$ and $\Q^{-}$.)
Clearly the $Res^{1}$'s are all $0$. Combining this with the moving
lemmas of Levine and Bloch, there exist $\Gamma\in Z^{3}(\widetilde{\hat{\X}_{-}}\m\hat{\J}\times\AA^{1},\,4)_{_{\beta^{-1}(\hat{\ms})\cup\Q^{-}}}$
and $\Xi_{0}\in Z_{_{\db-cl}}(\widetilde{\hat{\X}_{-}},\,3)_{_{\beta^{-1}(\hat{\ms})\cup\Q\,[=b^{-1}(\ms)]}}$
such that $\xi+\db\Gamma$ is the restriction of $\Xi_{0}$. The pullback
of $\Xi_{0}$ to $b^{-1}(\ms)$ gives a cocycle in the complex computing
$H_{\M}$, $\hat{Z}^{3}(b^{-1}(\ms),-\bullet):=$\[
Cone\left\{ \hat{Z}^{3}(\beta^{-1}(\hat{\ms}),-\bullet)_{_{\beta^{-1}(\hat{\ms})\cap\Q}}\oplus Z^{3}(\Q,-\bullet)_{_{\beta^{-1}(\hat{\ms})\cap\Q}}\to\hat{Z}^{3}(\beta^{-1}(\hat{\ms})\cap\Q,-\bullet)\right\} [-1].\]
This can be {}``moved'' by a coboundary (in the cone complex) to
essentially an element of $Z_{\db-cl}^{3}(\Q,3)_{_{\beta^{-1}(\hat{\ms})\cap\Q}}$
supported on $\Q\cap\hat{\J}\times\AA^{1}$. Moreover, the components
of $\Q_{\alpha}\cap\hat{\J}\times\AA^{1}$ ($\alpha\in\A$) are pairwise
disjoint $\AA^{1}$'s which are $\rateq$ (as divisors) on $\Q_{\alpha}$
by functions $\hat{f}_{\alpha}\in\bar{\QQ}(\Q_{\alpha})$ restricting
to $1$ on $\Q_{\alpha}\cap\beta^{-1}(\hat{\ms})$. (Pull back to
$\Q_{\alpha}$ $f\in\bar{\QQ}(\PP_{\alpha}^{1})^{*}$ which has $(f)=q_{3}-q_{4}$
and $f(q_{1}=q_{2})=1$, in the only nontrivial situation.) Since
$CH^{2}(\AA^{1},3)\cong CH^{2}(\text{{pt.}},3)$ one can move the
elements of $Z^{3}(\Q_{\alpha},3)$ so as to make them constant along
each of the supporting $\AA^{1}$'s, and then {}``collect'' all
these constant cycles along only one such $\AA^{1}$, by using {[}$\db$-coboundaries
of] cycles (of the form $\mathfrak{{A}}\otimes\hat{f}_{\alpha}\in Z^{3}(\Q_{\alpha},4)$)
restricting to $0$ at $\Q_{\alpha}\cap\beta^{-1}(\hat{\ms})$. The
constant $\AA^{1}$-supported cycles are then killed by adding constant
cycles on the $b^{-1}(\D_{j}\times\AA^{1})\cong\D_{j}\times\AA^{1}$
to $\Xi_{0}$, via $Z^{2}(\D_{j}\times\AA^{1},\,3)\hookrightarrow Z^{3}(\widetilde{\hat{\X}_{-}},3)$.
That we have {}``enough'' $\D_{j}$'s to kill all constant cycles
on the $\Q_{\alpha}$'s is guaranteed (if (c)(i) holds) by surjectivity
of $\E$. Alternatively, if (c)(ii) holds then all of the above is
valid over $K$ (as opposed to $\bar{K}$), and $K$ totally real
$\implies$ the $CH^{2}(\AA_{K}^{1},3)$-classes embedded in the $\Q_{\alpha}$'s
self-annihilate.
\end{proof}

\subsection{Examples of $\phi$ satisfying the Theorem}

Here are specific ways to realize the conditions of the Theorem (in
particular, the tempered condition); $\phi$ is defined over a number
field $K$ as usual.

\begin{cor}
Let $\phi$ be reflexive with cyclotomic edge polynomials and root-of-unity
vertex coefficients. Furthermore for\\
\\
$\underline{n=2}$: assume the general $X_{t}$ is nonsingular.\\
\\
$\underline{n=3}$: \textbf{assume the facets of $\Delta$ have no
interior points}, and that $\phi$ is regular.\\
\\
$\underline{n=4}$: \textbf{assume the facets of $\Delta$ are elementary
$3$-simplices} (all points of $\Delta$ other than $\{\underline{0}\}$
are vertices), with coefficients $\pm1$ only (except at $\{\underline{0}\}$).%
\footnote{There are 151 such reflexive $4$-polytopes, with a maximum of $12$vertices.
\cite{No}%
}\\
\\
Then $\xi$ completes.
\end{cor}
\begin{example}
Take $\phi$ to be {[}an arbitrary constant plus] the characteristic
(Laurent) polynomial of the $vertex$ $set$ of any reflexive polytope
$\Delta$ satisfying the relevant assumption in boldface. This will
be regular in case $n=2,4$, and also for $n=3$ provided none of
the facets are of the form (c) {[}see proof below] with $\frac{a}{2^{m}}$,
$\frac{b}{2^{m}}$ both odd for the same $m\in\ZZ^{\geq0}$.%
\footnote{Out of the 899 reflexive $3$-polytopes with interior-point-free facets,
this leaves us with 239. \cite{No}%
}
\end{example}
\begin{rem}
For $n=3$, we can also allow triangular facets $\sigma$ with interior
points, provided the only monomials appearing (with nonzero coefficients)
in $\phi_{\sigma}$ correspond to the vertices of $\sigma$.%
\footnote{This gets us up to 1071 resp. 358 $3$-polytopes, depending on whether
the special type (c) facets are admitted. \cite{No}%
}
\end{rem}
\begin{proof}
\emph{(of Corollary).} For $n=2$ it suffices to show $\phi$ tempered,
and this is obvious.

For $n=3$, one can easily classify (up to shift and unimodular transformation)
facets $\sigma$ with no interior points. Viewed in a $2$-plane $\RR_{\sigma}$,
they are all convex hulls of $3$ or $4$ points: (a) $\{(0,0),\,(2,0),\,(0,2)\}$,
(b) $\{(0,0),\,(0,1),\,(a,0)\}$, or (c) $\{(0,0),\,(0,1),\,(a,0),\,(b,1)\}$
(with $a,b\in\NN$). In each case $\phi_{\sigma}(x_{1}^{\sigma},x_{2}^{\sigma})=0$
can only yield ($D_{\sigma}^{*}=$) a Zariski open subset of a rational
curve. (Since $\phi$ is regular, $D_{\sigma}$ is also nonsingular.)
For $\sigma'\in\Delta(2)$, $\phi_{\sigma'}$ cyclotomic $\implies$$\{x_{1}^{\sigma'}\}$
gives $0$ in $CH^{1}(D_{\sigma'}^{*},1)$. Hence (for $\sigma\in\Delta(1)$)
$\{x_{1}^{\sigma},x_{2}^{\sigma}\}\in\left\{ ker(\text{{Tame}})\subseteq CH^{2}(D_{\sigma}^{*},2)\right\} =im\left\{ CH^{2}(D_{\sigma},2)\to CH^{2}(D_{\sigma}^{*},2)\right\} .$
But $CH^{2}(\PP_{K}^{1},2)\cong K_{2}^{M}(K)=0$ (in fact, $K_{2}^{M}(\bar{\QQ})=0$),
and so $\phi$ is tempered. The remaining conditions follow from regularity
by Remark 1.8(iii).

For $n=4$, the tempered condition is again clear for edges $\sigma''\in\Delta(3)$,
so fix $\sigma'\subset\sigma$, $\sigma\in\Delta(1)$ and $\sigma'\in\Delta(2)$;
$\sigma$ is a triangle and $\sigma'$ a tetrahedron. Any two edges
of $\sigma'$ (viewed as integral vectors) generate $\RR_{\sigma'}\cap\ZZ^{4}$,
and so one may choose the monomials $x_{1}^{\sigma'},x_{2}^{\sigma'}$
so that $\phi_{\sigma'}=1+x_{1}^{\sigma'}+x_{2}^{\sigma'}$ (ignoring
the $\pm1$ issue). This makes plain the $\AA_{\QQ}^{1}$-uniformizability
of $D_{\sigma'}$ (condition (c) of Thm. 1.7), since $\phi_{\sigma'}=0$
is the equation of $D_{\sigma}^{*}$ (in local toric coordinates);
it is also clear that $\{x_{1}^{\sigma'},x_{2}^{\sigma'}\}\in CH^{2}(D_{\sigma'}^{*},2)$
vanishes. Next, one can choose monomials $x_{1}^{\sigma}(:=x_{1}^{\sigma'})$,
$x_{2}^{\sigma}(:=x_{2}^{\sigma'})$, $x_{3}^{\sigma}$ generating
$\RR_{\sigma}\cap\ZZ^{4}$ such that $\phi_{\sigma}=1+x_{1}^{\sigma}+x_{2}^{\sigma}+(x_{1}^{\sigma})^{a}(x_{2}^{\sigma})^{b}(x_{3}^{\sigma})^{c}$
($a,b\in\ZZ^{\geq0},\, c\in\NN$). We must show that $\{x_{1}^{\sigma},x_{2}^{\sigma},x_{3}^{\sigma}\}$
vanishes in $CH^{3}(D_{\sigma}^{*},3)$, where $D_{\sigma}^{*}\cong\{(x_{1}^{\sigma},x_{2}^{\sigma},x_{3}^{\sigma})\in(\CC^{*})^{3}\,|\,\phi_{\sigma}(\underline{x}^{\sigma})=0\}.$
This requires a short calculation for which we rewrite $x_{i}^{\sigma}=:y_{i}$
and write elements of $CH^{3}(D_{\sigma}^{*},3)$ as symbols --- as
if they were in $K_{3}^{M}(\bar{\QQ}(D_{\sigma}))$. However, we have
explicitly checked that the following relations actually hold over
$D_{\sigma}^{*}$ (for the relevant graph cycles) and not just $\eta_{_{D_{\sigma}^{*}}}$:\[
\left\{ y_{1},\, y_{2},\, y_{3}\right\} \,=\,\frac{1}{c}\left\{ y_{1},\, y_{2},\, y_{1}^{a}y_{2}^{b}y_{3}^{c}\right\} \,=\,\frac{1}{c}\left\{ -\frac{y_{1}}{y_{2}},\,-y_{2},\,-y_{1}^{a}y_{2}^{b}y_{3}^{c}\right\} \]
\[
=\,\frac{1}{c}\left\{ -\frac{y_{1}}{y_{2}},\,-\left(1+\frac{y_{1}}{y_{2}}\right)y_{2},\,-y_{1}^{a}y_{2}^{b}y_{3}^{c}\right\} \,=\,\frac{1}{c}\left\{ -\frac{y_{1}}{y_{2}},\,-(y_{1}+y_{2}),\,-y_{1}^{a}y_{2}^{b}y_{3}^{c}\right\} .\]
Using $1+y_{1}+y_{2}+y_{1}^{a}y_{2}^{b}y_{3}^{c}=0$ yields \[
\frac{1}{c}\left\{ -\frac{y_{1}}{y_{2}},\,-(y_{1}+y_{2}),\,1+(y_{1}+y_{2})\right\} ,\]
 which is zero (again over all of $D_{\sigma}^{*}$). Hence $\phi$
is tempered. Regularity of $\phi$ (i.e., $\Delta$-regularity of
$\phi-\l$ for general $\l$) along the faces is obvious from the
explicit equations for $\phi_{\sigma},\phi_{\sigma'},\phi_{\sigma''}$
(and irregularities in the torus $(\CC^{*})^{4}$ for generic $\l$
are impossible by a simple calculus argument). 
\end{proof}
\begin{example}
For $n=4$, there are examples (where $\xi$ completes) that do not
fall under the aegis of Theorem $1.7$ --- e.g. $\phi=x_{1}^{-1}x_{2}^{-1}x_{3}^{-1}x_{4}^{-1}(1+\sum_{i=1}^{4}x_{i}^{5}),$
which gives the Fermat quintic family in $\PP^{4}$. One must verify
directly that $\left\langle \{\underline{x}\}\right\rangle \in CH^{4}(\tilde{\X}_{-}^{*},4)$
lies in $\ker(d_{1})\cap\ker(d_{2})$, in the local-global spectral
sequence described in the Theorem's proof. This means checking that
the residues of (a representative of) $\left\langle \{\underline{x}\}\right\rangle $
in $\oplus_{\tilde{\sigma}\in\tilde{\Delta}(1)}Z^{3}(D_{\tilde{\sigma}}^{*}\times\AA^{1},3)$
are killed by relations (in $Z^{3}(D_{\tilde{\sigma}}^{*}\times\AA^{1},4)$),
then that differences of residues of these relations in $\oplus_{\tilde{\sigma}\in\tilde{\Delta}(2)}Z^{2}(D_{\tilde{\sigma}}^{*}\times\AA^{1},3)$
are trivialized as well. This is left to the reader. 
\end{example}
\begin{rem}
For $n=2$, one can sometimes avoid going modulo torsion and complete
$\xi$ to a class $\tilde{\Xi}\in H_{\M}^{2}(\widetilde{\TXM},\ZZ(2))$
($\cong CH^{2}(\widetilde{\TXM},2)$ but without our implicit $\otimes\QQ$
convention). Namely, for each edge $\sigma$, let $x_{(1)}^{\sigma}=x_{1}^{a_{\sigma}}x_{2}^{b_{\sigma}}$
(where $(a_{\sigma},b_{\sigma})=1$) generate $\RR_{\sigma}\cap\ZZ^{2}$.
Then it suffices to require (besides smoothness of the general $X^{\l}$)
the edge polynomial $\phi_{\sigma}$ to have only $(-1)$ as root
if $a_{\sigma}$ and $b_{\sigma}$ are both odd, and only $(+1)$
as root otherwise. This follows simply from (integral) computation
of the $Tame$ symbol of $\{x_{1},x_{2}\}$.
\end{rem}
We conclude this section with a discussion of what can be done for
an arbitrary reflexive $3$-polytope $\Delta$ if we are only after
getting a $\Xi^{\lambda}$ for general $\lambda$ (as in Remark 1.8(ii)).
An arbitrary facet $\sigma\in\Delta(1)$ inherits the integral structure
$\ZZ^{3}\cap\RR_{\sigma}$ (and is obviously not in general itself
reflexive).

\begin{fact}
\cite{No} Up to shift and unimodular transformation, there are $344$
possibilities for $\sigma$, and they all satisfy $\ell(\sigma)>2\ell^{*}(\sigma)$.
\end{fact}
Fix an isomorphism $\ZZ^{2}\overset{\cong}{\to}\ZZ^{3}\cap\RR_{\sigma}$,
and denote the corresponding toric coordinates on $\mathbb{D}_{\tilde{\sigma}}^{*}$
by $x_{1}^{\sigma},x_{2}^{\sigma}$. Writing $\ell'(\sigma):=\ell(\sigma)-\ell^{*}(\sigma)-1$,
let $\mathfrak{M}_{\sigma}=\mathfrak{M}_{\sigma}^{*}\cup(\mathfrak{M}_{\sigma}\setminus\mathfrak{M}_{\sigma}^{*})=\{\underline{m}_{i}^{*}\}_{i=1}^{\ell^{*}(\sigma)}\cup\{\underline{m}_{j}'\}_{j=0}^{\ell'(\sigma)}$
be the decomposition of $\sigma\cap\ZZ^{2}$ into interior and edge
points. The ample linear system $|\mathcal{O}_{\mathbb{D}_{\tilde{\sigma}}}(1)|\cong\PP^{\ell(\sigma)-1}$
is parametrized by Laurent polynomials\[
\phi_{\sigma;[\underline{\alpha}:\underline{\beta}]}(\underline{x}^{\sigma})\,:=\,\sum_{i=1}^{\ell^{*}(\sigma)}\alpha_{i}\cdot(\underline{x}^{\sigma})^{\underline{m}_{i}^{*}}+\sum_{j=0}^{\ell'(\sigma)}\beta_{j}\cdot(\underline{x}^{\sigma})^{\underline{m}_{j}^{'}}\,=\, A_{\underline{\alpha}}(\underline{x}^{\sigma})+B_{\underline{\beta}}(\underline{x}^{\sigma}),\]
and consists (generically) of genus-$\ell^{*}(\sigma)$ curves. Let
$\mathcal{V}_{\sigma}^{irr}\subset\PP^{\ell(\sigma)-1}$ be the locus
of ($\phi_{\sigma}$ cutting out) $\ell^{*}(\sigma)$-nodal irreducible
rational curves $C_{\phi_{\sigma}}$ in this system. It seems entirely
reasonable to hope that \begin{equation}
\mathcal{V}^{irr}_{\sigma} \text{ is nonempty for all } \sigma \in \Delta(1)
\\
\end{equation} is satisfied for all reflexive $\Delta\subset\RR^{3}$; this may
be decidable by applying the tropical methods of \cite{Mi}. In fact
one has

\begin{fact}
\cite{Mi,Ty} If $\mathcal{V}_{\sigma}^{irr}\neq\emptyset$, its Zariski
closure $\overline{\mathcal{V}_{\sigma}^{irr}}$ (the so-called \emph{Severi
variety}) is a codimension-$\ell^{*}(\sigma)$ irreducible subvariety
of $\PP^{\ell(\sigma)-1}$.
\end{fact}
Here, then, is our {}``most general'' example for $n=3$:

\begin{prop}
For a reflexive $3$-polytope $\Delta$ satisfying (1.3), there exists
a tempered Laurent polynomial $\phi$ (with Newton polytope $\Delta$)
defining a family of (generically smooth) $K3$ surfaces $\{\tilde{X}_{t}\}$
such that (for general $t$) the toric symbol completes to a $CH^{3}(\tilde{X}_{t},3)$-class
$\Xi_{t}$.
\end{prop}
\begin{proof}
Let $\mathcal{U}\subset\PP^{\ell(\sigma)-1}$ be the complement of
the $\PP^{\ell^{*}(\sigma)-1}$ defined by $\underline{\beta}=\underline{0}$.
Since $\dim(\overline{\mathcal{V}_{\sigma}^{irr}})=\ell(\sigma)-\ell^{*}(\sigma)-1>\ell^{*}(\sigma)-1$
by Facts 1.14-15, $\overline{\mathcal{V}_{\sigma}^{irr}}\cap\mathcal{U}\neq\emptyset$.
Consider the projection $\mathcal{U}\overset{\rho}{\to}\PP^{\ell^{'}(\sigma)}$
induced by $[\underline{\alpha}:\underline{\beta}]\mapsto[\underline{\beta}]$;
we contend that its restriction to $\overline{\mathcal{V}_{\sigma}^{irr}}\cap\mathcal{U}$
is generically an immersion.

Indeed, otherwise a generic $C_{\phi_{\sigma}}\in\mathcal{V}_{\sigma}^{irr}$
deforms while keeping its intersection with the boundary $\mathbb{D}_{\tilde{\sigma}}\setminus(\C^{*})^{2}\,=:\,\mathsf{D}$
fixed. The normal bundle of the composition $f:\,\PP^{1}\cong\widetilde{C_{\phi_{\sigma}}}\twoheadrightarrow C_{\phi_{\sigma}}\hookrightarrow\mathbb{D}_{\tilde{\sigma}}$
is $N_{f}:=f^{*}(\theta_{\mathbb{D}_{\tilde{\sigma}}}^{1})/\theta_{\PP^{1}}^{1}\cong\mathcal{O}_{\PP^{1}}(-2+f^{*}(\mathsf{D}))$.
A deformation of this form would yield a nonzero section of $N_{f}(-f^{*}(\mathsf{D}))\cong\mathcal{O}_{\PP^{1}}(-2)$,
which is impossible.

Since $\dim(\overline{\mathcal{V}_{\sigma}^{irr}})=\ell^{'}(\sigma)$
we conclude that $\rho(\mathcal{V}_{\sigma}^{irr}\cap\mathcal{U})\subset\PP^{\ell^{'}(\sigma)}$
is open, and therefore contains a Zariski-dense subset corresponding
to cyclotomic edge polynomials (with distinct roots on each edge).
So we get countably many $\phi_{\sigma;[\underline{\alpha}:\underline{\beta}]}$
defining irreducible nodal rational curves $C_{\phi_{\sigma}}$ with
regular, cyclotomic edge polynomials; and $\underline{\alpha},\underline{\beta}$
can be taken to lie in $\bar{\QQ}$.

Globalizing this to the $3$-polytope, there is a choice of $\phi(x_{1},x_{2},x_{3})$,
all of whose facet polynomials $\phi_{\sigma}$ are of this form.
Clearly, $\phi$ is tempered if the classes $\{x_{1}^{\sigma},x_{2}^{\sigma}\}\in K_{2}(\bar{\QQ}(\widetilde{C_{\phi_{\sigma}}}))\cong K_{2}(\bar{\QQ}(\PP^{1}))$
vanish. But since the edges of $\phi_{\sigma}$ are cyclotomic, $\{x_{1}^{\sigma},x_{2}^{\sigma}\}\in\ker(Tame)=K_{2}(\bar{\QQ})=\{0\}$.
\end{proof}

\section{\textbf{The fundamental regulator period}}

The $1$-parameter families $\{\TXT\}$ of CY toric hypersurfaces
produced by Theorem $1.7$ have in a neighborhood of $t=0$ a canonical
family of cycles $\tilde{\varphi}_{t}$ vanishing (in $H_{n-1}(\tilde{X}_{0})$)
at $t=0$. (In fact, using \cite{LTY}%
\footnote{or the Clemens-Schmid sequence: SSR replaces $\tilde{X}_{0}$ by a
NCD $'\tilde{X}_{0}$, and \[
H_{n-1}({}'\tilde{X}_{0})(-n+1)\to H^{n-1}({}'\tilde{X}_{0})\to H_{lim}^{n-1}(\tilde{X}_{t})\rTo^{N}H_{lim}^{n-1}(\tilde{X}_{t})\]
is exact (with $\QQ$-coefficients), where $N=\log(T)$ and weights
of $H^{n-1}({}'\tilde{X}_{0})$ {[}resp. $H_{n-1}({}'\tilde{X}_{0})(-n+1)$]
lie in $[0,n-1]$ {[}resp. $[n-1,2n-2]$]. So maximal unipotent monodromy
of $T$ $\Longleftrightarrow$ $N^{n-1}\neq0$ $\Longleftrightarrow$
$\hm{(}\QQ(0),\ker(N))\neq\{0\}$ $\Longleftrightarrow$ $\hm{(}\QQ(0),H^{n-1}({}'\tilde{X}_{0}))\neq\{0\}$
$\Longleftrightarrow$ $H^{0}({}'\tilde{X}_{0}^{[n-2]})\to H^{0}({}'\tilde{X}_{0}^{[n-1]})$
is not surjective (where $'\tilde{X}_{0}^{[i]}:=$desingularization
of $i^{\text{th}}$ coskeleton of $'\tilde{X}_{0}$). The last criterion
follows from the fact that the dual graph of $'\tilde{X}_{0}$ is
$\partial\{tr(\Delta^{\circ})\}$, which is topologically a triangulation
of $S^{n-1}$. %
} one can show that they have maximal unipotent monodromy there, provided
{[}for $n=4$] $\PP_{\tilde{\Delta}}$ is smooth.) What we aim to
do in this section, is to pair $\tilde{\varphi}_{t}$ against the
regulator image\[
AJ(\Xi_{t})\in H^{n-1}(\TXT,\CC/\QQ(n))\cong Hom_{\QQ}\left(H_{n-1}(\TXT,\QQ),\CC/\QQ(n)\right)\]
over a punctured disk $\bar{D}_{|t_{0}|}^{*}(0)$ extending to the
singular fiber (at $t_{0}\in\L$) nearest the one at $t=0$. The resulting
(multivalued) function is called the {}``fundamental regulator period'';
the {}``fundamental period'' is just the period of a canonical holomorphic
form $\tilde{\omega}_{t}\in\Omega^{n-1}(\TXT)$ over $\tilde{\varphi}_{t}$.
The regulator computation has some surprisingly beautiful and easy
corollaries related to differential equations, number theory, and
local mirror symmetry.

For the next two subsections, it will suffice to assume\\
(a) $\phi$ is reflexive with root-of-unity vertex coefficients (denoted
$\zeta$);\\
(b) the generic $\TXT$ has at worst Gorenstein orbifold singularities,%
\footnote{in this case $\L\subset\PP^{1}$ records only the {}``more'' singular
fibers where the local system $R^{n-1}\tilde{\pi}_{*}\QQ$ has monodromy.%
} and these lie in $\tilde{D}$; and\\
(c) $\xi$ completes to $\Xi_{t}\in H_{\M}^{n}(\TXT,\QQ(n))$ as in
Definition $1.6$.\\
So in principle $n$ could be $>4$. The importance of (a) is that
it amounts to a choice of the parameter $t$ normalizing (in fact,
for $n=2$ trivializing) the rational limit mixed Hodge structure
at $0$.

\subsection{The vanishing cycle and fundamental period}

Pick a vertex $\underline{v}\in\Delta(n)$ and a facet (=elementary
$(n-1)$-simplex) $\ST^{\circ}\in tr(\Delta^{\circ})(1)$ contained
in the facet $\sigma_{\underline{v}}^{\circ}\in\Delta^{\circ}(1)$
dual to $\underline{v}$. We may assume

\begin{equation} 
\begin{matrix} \ST^{\circ} \text{ has a flag of } i\text{-faces } (0\leq i \leq n-1) \\ \text{contained in } i\text{-faces of } \sigma_{\underline{v}}^{\circ},\text{ e.g.:} \mspace{50mu} \end{matrix} \\
\end{equation}

$\mspace{150mu}$\includegraphics[scale=0.5]{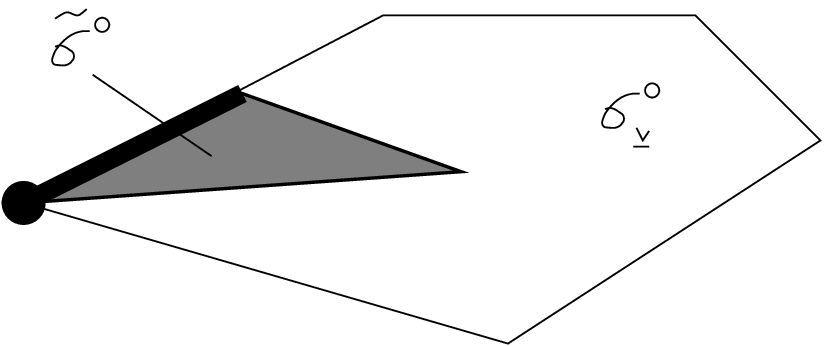}

${}$\\
The simplex $\ST^{\circ}$ is dual to some $\ST\in\DT(n)$ lying over
$\underline{v}$, and produces an affine coordinate chart around $\tilde{v}$
for $\PDT$. Referring to the discussion at the end of $\S1.1$, the
coordinates are monomials in the $\{x_{i}\}$ of two types:$\vspace{2mm}$
\\
(i) those corresponding (essentially) to edges of $\DT$ emerging
from $\tilde{v}$ (or to edges of $\Sigma(\ST^{\circ})^{\circ}$):\[
z_{1},\, z_{2},\,\ldots,\, z_{n}\,;\]
 \\
(ii) those corresponding to interior points of (faces of) $\Sigma(\ST^{\circ})^{\circ}$,
and having powers which are monomials in the $\{z_{j}\}$:\[
(u_{1},\ldots,u_{k_{2}})=:\underline{u}_{2};\,(u_{k_{2}+1},\ldots,u_{k_{3}})=:\underline{u}_{3};\,\ldots;\,(u_{k_{n-1}+1},\ldots,u_{k_{n}})=:\underline{u}_{n}.\]
(We have organized these so that powers of the $u_{k_{m-1}+1},\ldots,u_{k_{m}}$
are expressible in $z_{1},\ldots,z_{m}$.) Writing out the $\underline{u}\longleftrightarrow\underline{z}$
relations gives the local affine (singular) equations for $\PDT$.
The local affine equation for $\tilde{X}^{\l}$ is then a $polynomial$
in $(\underline{z},\,\underline{u})$ obtained from $\l-\phi(\underline{x})$
by dividing out the $\zeta\underline{x}^{\underline{v}}$ term: $0=\Phi_{\underline{v}}(\underline{z},\,\underline{u})=$\[
1+\phi_{1}(z_{1})+\phi_{2}(z_{1},z_{2};\underline{u}_{2})+\cdots+\{\phi_{n}(z_{1},\ldots,z_{n};\underline{u}_{n})-\l\underline{z}^{\underline{\mu}_{1}}\underline{u}_{n}^{\underline{\mu}_{2}}\}\]
where the $k^{\text{{th}}}$ term vanishes if $z_{k}=0$; $(2.1)$
$\implies$ none of the $\phi_{k}\equiv0$. Write $\DD_{\ST_{i}}=$
closure of locus in \{affine chart $\cap\PDT$\} where $z_{n-i+1}=\cdots=z_{n}=0$
(and $\ST_{i}$ for the corrresponding faces $\in\DT(i)$, dual to
the flag $(2.1)$), and\[
\Phi_{\underline{v},\ST_{i}}(z_{1},\ldots,z_{n-i};\underline{u})=:\Phi_{\underline{v}}|_{\DD_{\ST_{i}}}=1+\sum_{k\leq n-i}\phi_{k}.\]
 In particular, note that the divisor $\DD_{\ST_{1}}$ is defined
by $z_{n}=0$.

Define on $\PP_{\Delta}$, $\Omega_{t}\in\Gamma(\hat{\Omega}_{\PP_{\Delta}}^{n}(\log X_{t}))$
by\[
\Omega_{t}:=\frac{\dlog x_{1}\wedge\cdots\wedge\dlog x_{n}}{1-t\phi(\underline{x})}=\l\frac{\bigwedge^{n}\dlog\underline{x}}{\l-\phi(\underline{x})}\,,\]
 and let\[
\omega_{t}:=Res_{X_{t}}(\Omega_{X_{t}})\in\hat{\Omega}^{n-1}(X_{t})\,;\]
these have $\mu^{*}$-pullbacks $\tilde{\Omega}_{t}$, $\tilde{\omega}_{t}(\in\Omega^{n-1}(\TXT))$.
Let $\epsilon>0$ and define the real $n$-torus\[
\tor:=\{|z_{1}|=\cdots=|z_{n}|=\e\}\cap\PDT\in Z_{n}^{top}(\PDT\m\TXT\cup\DDT).\]
For fixed $\e>0$ it is clear (using $\Phi_{\v}$ above) that for
$|\l|\,>$ some fractional power of $\frac{1}{\e}$, i.e. for $|t|<\delta(\e)$
sufficiently small, $\tor$ avoids $\TXT$. One has the {}``membrane''
\[
\Gamma_{\v,\e}:=\{|z_{1}|=\cdots=|z_{n-1}|=\e,\,|z_{n}|\leq\e\}\in C_{n+1}^{top}(\PDT\m\DDT^{-})\]
where $\DDT^{-}:=\bigcup_{\ST\neq\ST_{1}}\DD_{\ST}$; this bounds
on the real $n$-torus:\[
\d\Gamma_{\v,\e}=(-1)^{n-1}\tor.\]
We specify our family of vanishing cycles by demanding that for $|t|<\delta(\e)$\[
-\van_{t}\homeq\TXT\cap\Gamma_{\v,\e}\in Z_{n-1}^{top}(\TXT).\]

Now the exponent vectors $\underline{m}_{i}$ relating $\{z_{i}\}\longleftrightarrow\{x_{j}\}$
($z_{i}=\underline{x}^{\underline{m}_{i}}$) form a rationally invertible
matrix. Hence, $\tor=\{|x_{i}|=\e^{q_{i}}\,(\forall i)\}\subset(\CC^{*})^{n}\subset\PDT$
for some nonzero rational numbers $q_{i}$. For the fundamental period
we have therefore\[
A(t):=\int_{\van_{t}}\tilde{\omega}_{t}\,=\,\int_{\van_{t}}Res_{\TXT}(\tilde{\Omega}_{t})\,=\,\frac{1}{2\pi i}\int_{Tube(\van_{t})=\tor}\tilde{\Omega}_{t}\]
\[
=\,\,\frac{1}{2\pi i}\int_{_{|x_{1}|^{^{1/q_{1}}}=\cdots=|x_{n}|^{^{1/q_{n}}}=\e}}\left(\sum_{m=0}^{\infty}t^{m}\phi(\underline{x})^{m}\right)\bigwedge^{n}\dlog\underline{x}\]
\[
=\,\,(2\pi i)^{n-1}\sum_{m=0}^{\infty}\frac{t^{m}}{(2\pi i)^{n}}\oint\phi(\underline{x})^{m}\bigwedge^{n}\dlog\underline{x}\]

\begin{equation} 
= \, \, (2\pi i)^{n-1}\sum_{m=0}^{\infty} [\phi(\underline{x})^m]_0 t^m , \\ 
\end{equation}${}$\\
where $[\cdot]_{0}$ takes the constant term of a Laurent polynomial.
While we proved this for $|t|<\delta(\e)$ ($\implies$ $|t\phi(\underline{x})|<1$
on $\tor$), the period and the power series extend to $D_{|t_{0}|}^{*}(0)$
and agree there since both functions are analytic.

\subsection{The period of the Milnor regulator current}

Given a symbol $\linebreak$$\left\langle \{f_{1},\ldots,f_{n}\}\right\rangle \in CH^{n}(Y,n)$
as in $\S1.1$ (but with $Y$ $smooth$ quasi-projective of $\dim<n$),
recall from \cite{Ke1,KLM} that $AJ\left\langle \{\underline{f}\}\right\rangle \in H^{n-1}(Y,\CC/\QQ(n))$
is represented by the regulator current

\begin{equation} 
 R_n \{\underline{f}\} \, = \log f_1 \dlog f_2 \wedge \cdots \wedge \dlog f_n \, -\, (2\pi i)\delta_{T_{f_1}} \wedge R_{n-1}\{ f_2,\ldots,f_n\} \, \in \, \D^{n-1}(Y), \\
\end{equation}${}$\\
where \[
T_{f}:=f^{-1}\{\RR^{\leq0}\cup\{\infty\},\text{{\, oriented\, from\,\,\,}}\infty\text{{\, to\,\,\,}}0\}\]
 is the {}``cut'' in $\arg(f)\in(-\pi,\pi)$. ($R_{1}\{f\}$ is
just the $0$-current $\log f$.) Note that in $(2.3)$ we have omitted
the $\QQ(n)$-valued $\delta$-current; modulo this, $R_{n}$ is $d$-closed.

\begin{rem}
(i) There are some good real-position requirements for the above representation
(see \cite{KL,KLM}), which we omit but which are satisfied below.

(ii) If the integral cohomology of $Y$ is torsion-free, as in the
case of an open elliptic curve, we can replace $\QQ(n)$ by $\ZZ(n)$.
\end{rem}
The vanishing cycle $\van_{t}$ extends to a multivalued section of
$\RR^{n-1}\tilde{\pi}_{*}\ZZ$ over $\PP^{1}\m\L$, and \begin{equation}
\Psi(t):=AJ(\Xi_t)(\van_t) \\
\end{equation} yields a multivalued holomorphic function (see Cor. 2.4ff). (It remains
multivalued after going modulo $\QQ(n)$, due to monodromy of $\van_{t}$.)
We want to compute $\Psi(t)$ for $t\in U_{\e}:=\{|t|<\delta(\e)\text{{\, and\,\,}}\arg(t)\in(-\frac{\pi}{4},\frac{\pi}{4})\}$.
Consider the diagrams$\vspace{6mm} $

$ \mspace{80mu} $ \xymatrix{{\xi_t}:= \left< \{ x_1,\ldots,x_n \} \right> & & & {\Xi_t} \ar @{|->} [lll] }

 $\mspace{80mu} $\xymatrix{CH^n(\TXT \m \tilde{D} , n) \ar [d]^{AJ} & H^n_{\M} (\TXT ,\QQ(n)) \ar [l]^{r_t} \ar  [d]^{AJ} \\ H^{n-1}(\TXT \m \tilde{D}, \CC/\QQ(n)) & H^{n-1}(\TXT , \CC/\QQ(n)) \ar [l]^{\jmath^*}, }$ \vspace{6mm} $ 

\small \xymatrix{{\hat{\xi}_t}:= \left< \{ \l - \phi(\underline{x}),x_1,\ldots ,x_n \} \right> \ar @{|->} [rrrr] & & & & {\left( \xi_t , Res^1_{\ST}\hat{\xi}_t \right)} }

 \xymatrix{CH^{n+1}(\PDT \m \DDT \cup \TXT ,n+1) \ar [r]^{Res \mspace{80mu}} \ar [d]^{AJ} & CH^n(\TXT \m \tilde{D},n)\oplus CH^n(\DD_{\ST_1}^* \m D_{\ST_1}^* ,n) \ar [d]^{AJ} \\ H^n(\PDT \m \DDT \cup \TXT , \CC/\QQ(n+1)) \ar [r]^{Res \mspace{100mu}} & H^{n-1}(\TXT \m \tilde{D} ,\CC/\QQ(n))\oplus H^{n-1}(\DD_{\ST_1}^* \m D_{\ST_1}^* ,\CC/\QQ(n)) , }
$ \vspace{6mm} $ 

\xymatrix{[\tor] & & & & {\left( [\Gamma_{\v,\e}\cap \TXT], [\Gamma_{\v,\e} \cap \DD_{\ST_1}] \right) } \ar @{|->} [llll] \ar @{|->} [rrrr] & & & & [\van_t] }
\xymatrix{H_n(\PDT \m \DDT \cap \TXT , \QQ) & H_{n-1}(\TXT \m \tilde{D} ,\QQ)\oplus H_{n-1}(\DD_{\ST_1}^* \m D_{\ST_1}^*, \QQ) \ar [l]_{Tube \mspace{100mu}} \ar [r]^{\mspace{120mu} (\jmath_*,0)} & H_{n-1}(\TXT,\QQ)). } \normalsize
$ \vspace{4mm} $ 

${}$\\
These suggest that\[
\Psi(t)\,=\, AJ(\xi_{t})(\Gamma_{\v,\e}\cap\TXT)\,\,=\]
\[
-\frac{1}{2\pi i}AJ(\hat{\xi}_{t})(\tor)+(-1)^{n}AJ(Res_{\ST_{1}}^{1}\hat{\xi}_{t})(\Gamma_{\v,\e}\cap\DD_{\ST_{1}}),\]
the first term of which we can compute directly using the regulator
formula $(2.3)$; we will show the second zero by an induction argument.

Working on $\PDT\m\TXT\cup\DDT$, we have $\frac{1}{2\pi i}AJ(\hat{\xi}_{t})(\tor)\,=$

\begin{equation} \\
{\frac{1}{2\pi i}} \int_{\tor} R \{ \l - \phi(\underline{x}), x_1, \ldots , x_n \} \\
\end{equation}\[
=\,\frac{1}{2\pi i}\oint_{_{|x_{1}|^{^{1/q_{1}}}=\cdots=|x_{n}|^{^{1/q_{n}}}=\e}}\log(\l-\phi)\bigwedge^{n}\dlog\underline{x}\,,\]
since $t\in U_{\e}$ and $\underline{x}\in\tor$ $\implies$ $|\phi(\underline{x})|\leq\frac{1}{\delta(\e)}<|\l|$
and $\arg(\l)\in(-\frac{\pi}{4},\frac{\pi}{4})$ $\implies$ $x\notin T_{\l-\phi(\underline{x})}$.
Using $\l-\phi=t^{-1}(1-t\phi)$ and $|t\phi|<1$, we see the latter\[
=\,-(2\pi i)^{n-1}\left\{ \log t\,+\,\sum_{m\geq1}\frac{[\phi(\underline{x})^{m}]_{0}t^{m}}{m}\right\} .\]

On the other hand, we can manipulate the regulator current in $(2.5)$
by only \{coboundary on $\PDT\m\TXT\cup\DDT$\}$+$\{$\QQ(n)$-currents\}
to obtain a rational multiple of $R\{\Phi_{\v},z_{1},\ldots,z_{n}\}$.
This is done by using multilinearity and anticommutativity relations
for symbols valid in $CH^{n}(\PDT\m\TXT\cup\DDT)$ and the map of
complexes in \cite{KLM}. The relations are used first to multiply
$\l-\phi$ by $\underline{x}^{-\v}$ (which just gives $\Phi_{\v}(\underline{z};\underline{u})$,
and then to turn $\{x_{1},\ldots,x_{n}\}$ into $q\cdot\{z_{1},\ldots,z_{n}\}$
($0\neq q\in\QQ$). Hence $(2.5)\,=$ \[
\frac{q}{2\pi i}\int_{\tor}R\{\Phi_{\v},z_{1},\ldots,z_{n}\},\]
and enlarging the domain to $\PDT\m\DDT^{-}$ and using $(-1)^{n-1}\tor=\d\Gamma_{\v,\e}$
gives \[
\frac{-q}{2\pi i}\int_{\Gamma_{\v,\e}}d[R\{\Phi_{\v};\underline{z}\}]\]
\[
=\, q\left(\int_{\Gamma_{\v,\e}\cap\TXT}R\{z_{1},\ldots,z_{n}\}\,\pm\,\int_{\Gamma_{\v,\e}\cap\DD_{\ST_{1}}}R\{\Phi_{\v,\ST_{1}},z_{1},\ldots,z_{n-1}\}\}\right)\]
\[
=\,-\int_{\van_{t}}R\{x_{1},\ldots,x_{n}\}\,\pm\, q\int_{\d\Gamma_{\v,\e}^{(1)}}R\{\Phi_{\v,\ST_{1}},z_{1},\ldots,z_{n-1}\}\]
where the switch from $R\{\underline{z}\}$ back to $R\{\underline{x}\}$
(in the first term) is valid on $\TXT^{*}$ and \[
\Gamma_{\v,\e}^{(i)}\,:=\,\left\{ |z_{1}|=\cdots=|z_{n-i-1}|=\e,\,|z_{n-i}|\leq\e,\,|z_{n-i+1}|=\cdots=|z_{n}|=0\right\} \]
\[
\in\, C_{n-i+1}^{top}(\DD_{\ST_{i}}).\]
Of course $\int_{\van_{t}}R\{\underline{x}\}\equiv\Psi(t)$ $\mod\,\QQ(n)$.

Now we may argue inductively:\[
\int_{\d\Gamma_{\v,\e}^{(i)}}R\{\Phi_{\v,\ST_{i}},z_{1},\ldots,z_{n-i}\}\,=\,\pm\int_{\Gamma_{\v,\e}^{(i)}}d[R]\,=\]
\[
2\pi i\left(\pm\int_{\Gamma_{\v,\e}^{(i)}\cap\DD_{\ST_{i+1}}}R\{\Phi_{\v,\ST_{i+1}},z_{1},\ldots,z_{n-i-1}\}\,\pm\,\int_{\Gamma_{\v,\e}^{(i)}\cap D_{\ST_{i}}}R\{z_{1},\ldots,z_{n-i}\}\right).\]
Since $D_{\ST_{i}}$ is defined by vanishing of $\Phi_{\v,\ST_{i}}=1+\phi_{1}+\cdots+\phi_{n-i}$,
which is $\approx1$ on $\Gamma_{\v,\e}^{(i)}$, $\Gamma_{\v,\e}^{(i)}\cap D_{\ST_{i}}=\emptyset$
and this becomes\[
\pm2\pi i\int_{\d\Gamma_{\v,\e}^{(i+1)}}R\{\Phi_{\v,\ST_{i+1}},z_{1},\ldots,z_{n-i-1}\}\]
for $i<n-1$. When $i=n-1$, $\Gamma_{\v,\e}^{(n-1)}\cap\DD_{\ST_{n}(=\v)}$
is just the origin, $\Phi_{\v,\ST_{n}}$ is $1$, and \[
\int_{\Gamma_{\v,\e}^{(n-1)}}R\{\Phi_{\v},\ST_{n}\}\,=\,\log1\,=\,0.\]
We have proved

\begin{thm}
Assuming hypotheses (a)-(c) at the beginning of the section, the fundamental
regulator period for $\Xi_{t}$ is \begin{equation}
\Psi(t)\equiv (2\pi i)^{n-1} \{ \log t + \sum_{m\geq 1} \frac{[\phi^m]_0}{m}t^m\} \,\,\,\, \mod \QQ(n),
\end{equation} for all $t\in U_{\e}$.
\end{thm}
\begin{rem}
(a) For $\TXT$ smooth, $AJ(\Xi_{t})$ is represented (by \cite{KLM})
by the class of a closed $(n-1)$-current $R_{\Xi_{t}}':=R_{\Xi_{t}}+(2\pi i)^{n}\delta_{\d^{-1}T_{\Xi_{t}}}$
(modulo cycles modifying the membrane $\d^{-1}T_{\Xi_{t}}$) in $H^{n-1}(\TXT,\CC)/im\{H_{n-1}(\TXT,\QQ(n))\}$,
and $\Psi(t)\equiv\int_{\van_{t}}[R_{\Xi_{t}}']$. For brevity, we
denote $R_{\Xi_{t}}'=:R_{t}'$. We think of $[R_{t}']$ as a multivalued
section of $\H_{\tilde{\X}/\PP^{1}}^{n-1}:=R^{n-1}\tilde{\pi}_{*}\CC\otimes\O_{\PP^{1}}$
over $\PP^{1}\m\L$.

(b) Theorem $2.2$ is valid $\mod\,\ZZ(2)$ if $n=2$, Remark $1.11$
applies, and vertex coefficients of $\phi$ are all $1$.

(c) The apparent similarity (of the $\sum_{m\geq1}$ in the Theorem)
to the formal group law in \cite{BS} is somewhat deceptive, as their
$\ell(t)$ would correspond to $\sum_{m\geq0}\frac{[\phi^{m}]_{0}}{m+1}t^{m+1}$
in the present notation.
\end{rem}
Now assume henceforth that the general $\TXT$ is nonsingular (or
is a surface with $A_{1}$ singularities). The Gauss-Manin connection
$\nabla$ kills periods hence \linebreak $H^{n-1}(\TXT,\QQ(n))$-ambiguities
in $[R_{t}']$, and $\nabla[R_{t}']\in\Gamma(\PP^{1},\Omega_{\PP^{1}}^{1}\left\langle \log\L\right\rangle \otimes\F^{n-1}\H_{\tilde{\X}/\PP^{1}}^{n-1})$
(see \cite{Ke1}). Writing $\delta_{t}:=t\d_{t}:=t\frac{d}{dt}$,
this implies that \[
\nabla_{\delta_{t}}[R_{t}']=f(t)[\tilde{\omega}_{t}]\]
for $f\in\bar{K}(\PP^{1})^{*}$. To find $f$, we take periods of
both sides:\[
\frac{1}{(2\pi i)^{n-1}}t\frac{d}{dt}\int_{\van_{t}}[R_{t}']\,=\,\frac{f(t)}{(2\pi i)^{n-1}}\int_{\van_{t}}\tilde{\omega}_{t}\,,\]
and for $t\in U_{\e}$ this becomes \[
t\frac{d}{dt}\left\{ \log t\,+\,\sum_{m\geq1}\frac{[\phi^{m}]_{0}}{m}t^{m}\right\} \,=\, f(t)\sum_{m\geq0}[\phi^{m}]_{0}t^{m}.\]
So $f(t)\equiv1$ on $U_{\e}$, hence on $\PP^{1}$. There exists
a Picard-Fuchs operator $D_{PF}=\delta_{t}^{r}+\sum_{k=0}^{r-1}g_{k}(t)\delta_{t}^{k}$
($g_{k}\in\bar{K}(\PP^{1})^{*}$, $r\leq rk(R^{n-1}\tilde{\pi}_{*}\CC)$)
satisfying $D_{PF}A(t)=0$, and $\nabla_{PF}[\tilde{\omega}_{t}]=0$.

\begin{cor}
On $\PP^{1}\m\L$, $\nabla_{\delta_{t}}[R_{t}']=[\tilde{\omega}_{t}]$,
and the periods of $R_{t}'$ (e.g. $\Psi(t)$) satisfy the homogeneous
equation $(D_{PF}\circ\delta_{t})(\cdot)=0.$
\end{cor}
${}$

\begin{cor}
The classes $\Xi_{t}\in H_{\M}^{n}(\TXT,\QQ(n))$ and $\xi_{t}\in CH^{n}(\TXT^{*},n)$
are \linebreak (AJ-)nontrivial for general $t\in\PP^{1}$.
\end{cor}
\begin{proof}
There are several simple ways to see this; the first is that Theorem
$2.2$ $\implies$ $\Psi(t)\to\infty$ as $t\to0$, which obviously
shows \[
0\neq AJ(\xi_{t})\in Hom_{\QQ}(H_{n-1}(\TXT^{*},\QQ),\CC/\QQ(n)).\]
One can also use nonvanishing of the infinitesimal invariant $\nabla[R_{t}']$,
and there is an abstract way to do this which bypasses Corollary $2.4$
(and the Theorem). Recall $\TXM^{*}\cong(\CC^{*})^{n}$, and consider
the diagram $ \vspace{2mm} $ 

\small \xymatrix{ CH^n(\TXM^*,n) \ar [d]^{cl} & H^n_{\M}(\TXM ,n) \ar [l]_{\jmath^*} \ar [rr]^{ \{ AJ_t \}_{t\in \PP^1 \m \L} \mspace{100mu}} \ar [d]^{cl} & & H^0 \left( \PP^1 \m \L \; , \; \H^{n-1}_{\tilde{\X}/\PP^1} / R^{n-1}\tilde{\pi}_* \QQ(n) \right) \ar [d]^{\nabla} \\ F^nH^n(\TXM^*,\CC) & F^nH^n(\TXM ,\CC) \ar [l]_{\jmath^*} \ar @{^(->} [rr] & & H^0 \left( \PP^1 \m \L \; , \; \Omega^1_{\PP^1} \otimes \F^{n-1} \H^{n-1}_{\tilde{\X}/\PP^1} \right) } \normalsize

${}$\\
in which\[
\jmath^{*}(\Omega_{\Xi})=\jmath^{*}(cl(\Xi))=cl\left\langle \{\underline{x}\}\right\rangle =[\bigwedge^{n}\dlog\underline{x}]\neq0.\]
(Note that this implies that $\bigwedge^{n}\dlog\underline{x}$ extends
to a holomorphic form on $\TXM$, namely $\Omega_{\Xi}$.) One could
also base a proof on Corollary $2.9$ below.
\end{proof}
To put the last result in context, we recall the vanishing theorem
of \cite{Ke1} as it applies to the case of CY's. For $X/\CC$ smooth
projective of dimension $n-1$, let\[
K_{n}^{M}(X):=im\{CH^{n}(X,n)\to K_{n}^{M}(\CC(X))\},\]
and\[
\underline{H}^{n-1}(\eta_{X},\CC/\QQ(n)):=im\{H^{n-1}(X,\CC/\QQ(n))\to\mspace{-20mu}\begin{array}[t]{c}
\underrightarrow{\lim}\\
\Tiny{}^{^{\begin{array}{c}
D\subset X\\
codim.\,1\end{array}}}\end{array}\mspace{-20mu}H^{n-1}(X\m D,\CC/\QQ(n))\}\]
\[
\cong Gr_{N}^{0}H^{n-1}(X,\CC/\QQ(n)),\]
where $N^{\bullet}$ is the coniveau filtration. (This is nonzero
for a CY since $[\omega]\notin N^{1}$; for a surface it is $H_{tr}^{2}$.)
Then the $AJ$ map\[
K_{n}^{M}(X)\to\underline{H}^{n-1}(\eta_{X},\CC/\QQ(n))\]
is zero for $X$ a CY arising as a very general complete intersection
in $\PP^{n+r}$ of multidegree $(D_{0},\ldots,D_{r})$, $\sum D_{j}=n+r+1$,
and $n\geq3$ ($X\neq$curve). (Probably a similar result holds with
$\PP^{n+r}$ replaced by another toric Fano variety.) In contrast,
a general member of a $1$-parameter family arising from Theorem $1.7$
is still rather special, $\phi$ having coefficients in a number field
which are further restricted by the tempered requirement. In fact,
since $0\neq[\tilde{\omega}_{t}]=\nabla_{\delta_{t}}[R_{t}']\in\frac{N^{0}}{N^{1}}H^{n-1}(\TXT,\CC/\QQ(n))$
for general $t$ and $\nabla_{\delta_{t}}\N^{1}\H^{n-1}\subseteq\N^{1}\H^{n-1},$
we see that generically $0\neq[R_{t}']\in Gr_{N}^{0}$ $\implies$
$\{\underline{x}\}\in K_{n}^{M}(\TXT)$ is ($AJ$-)nontrivial.

So far, little to nothing has been said regarding the behavior of
$\Psi(t)$ globally or near $t_{1}\in\L\m\{0\}=:\L^{*}$. Fix a base
point $0'\in U_{\epsilon}$, let $\mathfrak{P}$ denote the space
of $C^{\infty}$ paths $P:\,[0,1]\to\PP^{1}\m\{0\}$ satisfying $P(0)=0'$,
$P([0,1))\subset\PP^{1}\m\L$, and write $P([0,1])=:|P|$. Define
a projection $\rho:\,\mathfrak{P}\to\PP^{1}\m\{0\}$ by $\rho(P):=P(1)$,
and let $\Phi_{P}=\cup_{t\in|P|}\van_{t}$ (with $[\van_{\rho(P)}]\in H_{n-1}(\tilde{X}_{\rho(P)},\ZZ)$)
be a {}``topological continuation'' of the vanishing cycle. There
is an obvious equivalence relation on $\mathfrak{P}^{\circ}:=\rho^{-1}(\PP^{1}\m\L)$
--- namely, $P_{1},P_{2}\in\rho^{-1}(t)$ are equivalent iff the restriction
of $R^{n-1}\tilde{\pi}_{*}\ZZ$ to $|P_{1}|\cup|P_{2}|$ is trivial.
Extend this to $t\in\L^{*}$ by requiring only that the union of $(|P_{1}|\cup|P_{2}|)\m\{t\}$
with some subset of $D_{\varepsilon}^{*}(t)$ have trivial monodromy.
Denote the quotient spaces by $\check{\mathfrak{P}}^{\circ}\subset\check{\mathfrak{P}}$,
topologizing the latter in analogy with the extended upper half-plane.
Note that $\L^{*}$ splits into finite and (unipotent and non-unipotent)
infinite monodromy fibers; $\rho^{-1}$ of the former should be thought
of as points interior to $\check{\mathfrak{P}}$, $\rho^{-1}$ of
the latter as cusps.

We want to clarify the following

\subsubsection*{Assertion:}

$\Psi(t)$ lifts to a well-defined, continuous function on $\check{\mathfrak{P}}$
with holomorphic restriction to $\check{\mathfrak{P}}^{\circ}$.\\
\\
To do this, we must finish defining $\Psi(t)$ by observing that $(2.4)$
makes sense (in $\CC/\QQ(n)$) even for $t\in\L^{*}$ once the homology
class $\van_{t}\in H_{n-1}(\TXT,\ZZ)$ is fixed. Since the MHS $H^{n}(\TXT)$
has weights $\leq n$, $\hm{(}\QQ(0),H^{n}(\TXT,\QQ(n)))=\{0\}$ and
$H_{\H}^{n}(\TXT,\QQ(n))\cong\ext(\QQ(0),H^{n-1}(\TXT,\QQ(n)))\cong H^{n-1}(\TXT,\CC/\QQ(n)).$
So $AJ(\Xi_{t})$ is at least defined in the last group (though we
won't say how to compute it until $\S4$), and $(2.4)$ simply pairs
homology and cohomology.

Fix $t\in\PP^{1}\m\{0\}$, $P\in\rho^{-1}(t)$ and $\Phi_{P}$ (hence
$\van_{t}$). By functoriality of KLM currents (moving $\Xi$ if necessary
to lie in $Z^{n}(\TXM,n)_{\TXT}$), $\int_{\van_{t}}R_{\Xi_{t}}=\int_{\van_{t}}R_{\Xi}$
for any $t\in\PP^{1}\m\{0\}$. If we accept (in anticipation of $\S4.1$)
that $AJ(\Xi_{t})(\van_{t})\equiv\int_{\van_{t}}R_{\Xi_{t}}$ even
for $t\in\L^{*}$, then $(2.4)$ gives \[
\Psi(t)=\int_{\van_{t}}R_{\Xi}=\int_{\Phi_{P}}d[R_{\Xi}]+\int_{\van_{0'}}R_{\Xi}\qneq\int_{\Phi_{P}}\Omega_{\Xi}+\Psi(0')\]
 for the continuation of $\Psi$ corresponding to $P$. The Assertion
follows, using $\Omega_{\Xi}\in\Omega^{n}(\TXM)$ and Morera's theorem
for the holomorphicity (which we already know in any case), and {}``smoothing
out'' any $\QQ(n)$-discrepancies.

As for the local behavior of (the multivalued function) $\Psi(t)$
at $t_{1}\in\L^{*}$ on $\PP^{1}$, this must be consistent with the
continuity on $\check{\mathfrak{P}}$. In $q:=t-t_{1}$ we have in
general $\Psi=$ holomorphic plus terms of the form $q^{\beta}(\log^{k}q)H(q)$
where $\beta\in\QQ^{+},\, k\in\{1,\ldots,n-1\}$, and $H$ is holomorphic.
For example, in the unipotent case suppose we have monodromy $T\van_{t}=\van_{t}+\eta_{t}$;
then $\eta_{t}\in\text{im}(T-I)$ implies (by Clemens-Schmid) that
$\eta_{t_{1}}$ is zero in $H_{n-1}(X_{t_{1}},\ZZ)$, hence pairs
to $0$ (mod $\QQ(n)$) with $AJ(\Xi_{t_{1}})$. Moreover, if $\eta_{t}\in\ker(T-I)$
then we simply have $\Psi=\Psi_{0}(q)+q(\log q)\Psi_{1}(q)$ where
$\Psi_{0},\,\Psi_{1}$ are holomorphic (and single-valued). 

Now let $t_{0}$ be the smallest%
\footnote{Of course there might be more than one element of smallest ($\neq0$)
modulus; in this event just choose one.%
} nonzero element of $\L$; i.e. (at least if $\phi$ is regular) $\frac{1}{t_{0}}$
is the critical value of $\phi$ of largest finite modulus. Putting
the above discussion together with Corollary 2.4 yields

\begin{cor}
The $\Psi(t)$ computation in Theorem $2.2$ holds $\forall t\in\bar{D}_{|t_{0}|}^{*}$.
\end{cor}
\begin{proof}
The convergence and continuity of $\sum\frac{[\phi^{m}]_{0}}{m}t^{m}$
at the boundary follows from a bit of Tauberian theory, combined with
the fact that $A(t)=\delta_{t}\Psi(t)$ has at worst a $\log^{n-1}(t-t_{0})$
pole at $t_{0}$. Then one invokes continuity of $\Psi(t)$ itself. 
\end{proof}
We conclude with a number-theoretic application. Various authors \cite{Be1,De2,RV}
have noticed a relation between the logarithmic Mahler measure ${\bf m}$
of a Laurent polynomial $Q(x_{1},\ldots,x_{n})$ and real regulator
periods (or special values of $L$-functions) associated to the variety
$Q=0$. Writing\[
\hat{\TT}^{n}:=\{|x_{1}|=\cdots=|x_{n}|=1\}\subset(\CC^{*})^{n},\]
this is\[
{\bf m}(Q):=\frac{1}{(2\pi i)^{n}}\int_{\hat{\TT}^{n}}\log|Q|\bigwedge^{n}\dlog\underline{x}.\]
the real regulator is just the composition\[
H_{\M}^{n}(\TXT,\QQ(n))\rTo^{AJ}H^{n-1}(\TXT,\CC/\QQ(n))\rOnto^{\pi_{\RR}}H^{n-1}(\TXT,\RR(n-1)),\]
where (on the level of currents) $\pi_{\RR}$ takes $R_{\Xi_{t}}'$
to its {}``$(2\pi i)^{n-1}\cdot$real''-part $r_{\Xi_{t}}\in\mathcal{{D}}_{\RR(n-1)}^{n-1}(\TXT)$.
(The latter is $(2\pi i)^{n}\cdot$Goncharov's current \cite{Go},
up to coboundary.) In the present context the two are related as follows.

\begin{cor}
Under the conditions of Theorem $2.2$,\[
-Re\left(\frac{1}{(2\pi i)^{n-1}}\Psi(t)\right)=\frac{-1}{(2\pi i)^{n-1}}\int_{\van_{t}}[r_{t}]={\bf m}(t^{-1}-\phi)\]
for all $t$ in \[
\mathcal{S}:=\overline{\{\text{{connected\, component\, of\,\,\,}}(\PP^{1}\m\{\frac{_{1}}{^{\phi(\hat{\TT}^{n})}}\})\text{{\, containing\,\,\,}}\{0\}\}}\m\{0\}\,\,\subseteq\,\PP^{1},\]
where the bar denotes analytic closure.
\end{cor}
\begin{proof}
Consider the equation\[
\frac{1}{(2\pi i)^{n-1}}\int_{\van_{t}}[R_{t}']\,=\,\log t\,+\,\sum_{m\geq1}\frac{[\phi^{m}]_{0}}{m}t^{m}\,=\,\frac{-1}{(2\pi i)^{n}}\int_{\hat{\TT}^{n}}\log(t^{-1}-\phi)\bigwedge^{n}\dlog\underline{x},\]
where the first equality holds by Theorem $2.2$ for (say) $t\in U_{\epsilon}$,
and the second for $t(\neq0)$ such that $|t|<|\phi(\underline{x})|^{-1}$
$\forall\underline{x}\in\hat{\TT}^{n}$. (Note that $|\phi|$ is bounded
above on $\hat{\TT}^{n}$.) Now the l.h.s. is analytic multivalued
on $\PP^{1}\m\L$, while the r.h.s. is analytic multivalued as long
as $(0\neq)\, t$ doesn't pass through $\{\frac{1}{\phi(\hat{\TT}^{n})}\}$
(so that log retains a continuous single-valued branch on the image
$t^{-1}-\phi(\hat{\TT}^{n})$). Since they agree on an analytic open
set, they continue to agree on (the covering space of) the obvious
connected component of $\PP^{1}\m\L\cup\{\frac{1}{\phi(\hat{\TT}^{n})}\}$.
Taking real parts of both sides kills multivaluedness,%
\footnote{To see this on the r.h.s., replace $\frac{\bigwedge^{n}\dlog\underline{x}}{(2\pi i)^{n}}$
by $\bigwedge^{n}\text{darg}\underline{x}$; for the l.h.s., one easily
sees that $\van_{t}$ has no monodromy on $\mathcal{S}$ (though $[R_{t}']$
may, which is harmless).%
} and the equality extends to the analytic closure by continuity, erasing
$\L\m\{0\}$ (where $\int_{\van_{t}}[r_{t}]$ is finite).
\end{proof}

\subsection{The higher normal function}

For this subsection, take the family $\tilde{\X}\overset{\tilde{\pi}}{\to}\PP^{1}$
to be as in (the assumptions of) Theorem $1.7$. Given any (possibly
singular) fiber $\tilde{X}_{t\neq0}$, we have $AJ(\Xi_{t})\in H^{n-1}(\TXT,\CC/\QQ(n))$.
If $\R_{t}\in H^{n-1}(\TXT,\CC)$ is any lift of this class, then
since $\tilde{\omega}_{t}=\frac{1}{2\pi i}Res_{\TXT}\tilde{\Omega}_{t}\in H_{\TXT}^{n+1}(\PDT,\CC)\cong H_{n-1}(\TXT,\CC),$
the pairing $\left\langle \R_{t},[\tilde{\omega}_{t}]\right\rangle \in\CC$
makes sense. For $\TXT$ smooth and $\R_{t}=[R_{t}']$ as in Remark
$2.3(a)$, this is just $\int_{\TXT}R_{t}'\wedge\tilde{\omega}_{t}$.

\begin{defn}
The \textbf{higher normal function associated to $\Xi$} is the multivalued
function\[
\nu(t):=\left\langle \R_{t},[\tilde{\omega}_{t}]\right\rangle \]
on $\PP^{1}\m\L$, where $\R_{t}$ is a (multivalued) continuous family
of lifts of $AJ_{\TXT}(\Xi_{t})$.
\end{defn}
This is a highly transcendental function, but applying $D_{PF}$ kills
the ambiguities (which are periods of $\tilde{\omega}$) and produces
$g(t):=D_{PF}\nu(t)\in\bar{K}(\PP^{1})$ (see \cite{dAM2}). Viewed
as an element of $\bar{K}(\PP^{1})/D_{PF}\bar{K}(\PP^{1})$, $g$
is the class of a certain extension of $\mathcal{{D}}$-modules attached
to $\Xi$. Alternatively, it is the inhomogeneous term of the Picard-Fuchs
equation\[
D_{PF}(\cdot)=g\]
satisfied by $\nu$, and its nonvanishing would give another proof
of nontriviality of $\Xi_{t}$: $g\neq0$ $\implies$ $\nu\neq$period
of $\tilde{\omega}$ $\implies$ $\R_{t}\notin H^{n-1}(\TXT,\QQ(n))$
{[}general $t$] $\implies$ general $AJ(\Xi_{t})\nequiv0$. Note
that conversely, if the $\CC$-span of the $\{\nabla_{\delta_{t}}^{i}[\tilde{\omega}_{t}]\}_{i=0}^{r-1}$
is a (complexified) Hodge structure for general $t$, then it is possible
to show (using $\nabla_{\delta_{t}}\R_{t}=[\twt]$ from Corollary
$2.4$) $g\neq0$.

The study of inhomogeneous PF equations for $higher$ normal functions
was initiated by M\"uller-Stach and del Angel \cite{dAM1,dAM2,dAM3}.
Their work focused on families of higher cycles $\eta_{t}\in CH^{p}(X_{t},2p-n)$
($p<n=\dim X+1$), in which case $\int_{X_{t}}R_{\eta_{t}}'\wedge\omega_{t}$
reduces to integration of $\omega_{t}$ over a real membrane. Here
we want to demonstrate that the case $p=n$ is also accessible and
interesting.

The \textbf{Yukawa coupling} is the function $\Y\in K(\PP^{1})$ defined
by\[
\Y(t):=\left\langle [\tilde{\omega}_{t}],\nabla_{\delta_{t}}^{n-1}[\tilde{\omega}_{t}]\right\rangle \]
for $t\notin\L$. ($A_{1}$-singularities for such $t$ are harmless
here, as $[\tilde{\omega}]$ lifts to $H^{n-1}(\widetilde{\TXT})$.)
The next result implies this is the inhomogeneous term in many cases
including that of elliptic curves ($n=2$) and $K3$ surfaces ($n=3$)
with generic Picard rank $19$.

\begin{cor}
If the order of $D_{PF}$ is $(r=)\, n$, i.e. if the $\mathcal{{D}}$-module
generated by $[\tilde{\omega}_{t}]$ has rank $n$, then $g=\Y$.
\end{cor}
\begin{proof}
Compute first\[
\delta_{t}\left\langle \R_{t},[\twt]\right\rangle =\left\langle [\twt],[\twt]\right\rangle +\left\langle \R_{t},[\twt]\right\rangle =\left\langle \R_{t},\nabla_{\delta_{t}}[\twt]\right\rangle ,\]
then inductively\[
\delta_{t}^{j<n}\left\langle \R_{t},[\twt]\right\rangle =\delta_{t}\left\langle \R_{t},\nabla_{\delta_{t}}^{j-1}[\twt]\right\rangle =\left\langle [\twt],\nabla_{\delta_{t}}^{j-1}[\twt]\right\rangle +\left\langle \R_{t},\nabla_{\delta_{t}}^{j}[\twt]\right\rangle .\]
By Hodge type and Griffiths transversality, this\[
=\,\left\langle \R_{t},\nabla_{\delta_{t}}^{j}[\twt]\right\rangle .\]
Hence, with $D_{PF}=\delta_{t}^{n}+\sum_{k=0}^{n-1}g_{k}(t)\delta_{t}^{k}$,
\[
D_{PF}\nu(t)=\Y(t)+\left\langle \R_{t},\nabla_{PF}[\twt]=0\right\rangle =\Y(t).\]

\end{proof}
\begin{rem}
For $r=n=2,3,4$ $\Y(t)$ is computed by an obvious differential equation.
To state it, recall that by \cite{LTY} we have maximal unipotent
monodromy at $t=0$. Hence $g_{j}(t)=tf_{j}(t)$ for $f_{j}$ holomorphic
at $t=0$, and with $q_{2}=1$, $q_{3}=\frac{2}{3}$, $q_{4}=\frac{1}{2}$
we get $\delta_{t}\Y(t)=-q_{n}tf_{n-1}(t)\Y(t)$ $\implies$ $\Y(t)=\kappa\exp\{-q_{n}\int f_{n-1}(t)dt\}.$
From above, $\Y=g$ must be a rational function, and $f_{n-1}(t)=-\frac{M}{q_{n}}\cdot\frac{\Y'(t)}{\Y(t)}$
(for $M\in\ZZ$). (If one has maximal unipotent monodromy also at
$t=\infty$, then $M$ can be determined also.) The value of $\kappa$
requires more precise (e.g. modular) information about the family.
Note that for $n=2$, $n=3$ and $rk(Pic)=19$, or $n=4$ and $h^{3}=4$,
Corollary $2.4$ $\implies$ $g\neq0$ $\implies$ $\kappa\neq0$.
\end{rem}
We prove next an interesting result on the monodromy of (a choice
of branch of) $\nu$. Recall from $\S1.3$ the definitions (for all
$n$) of $\J,\I\subseteq\tilde{D}$ and for $n=3$ set $\mathcal{{D}}:=$normalization
of $\J$ at $\J\cap\A$. From the proof of Theorem $1.7$, $\hat{\X}\rTo^{B}\tilde{\X}$
is the simultaneous resolution of the $A_{1}$-singularities $\A(\times\PP^{1})$,
and $\mathcal{{D}}$ is just the proper transform of $\J$ (along
$\hat{X}_{t}\to\TXT$). Let $\J^{-}$ be the union of the $D_{\ST}$'s
that are not in $\II$ $and$ not of the form $\{x_{i_{1}}+x_{i_{2}}=1,\, x_{i_{3}}^{\pm1}=0\}$.
For all $n$, let $\oTT:=\RR_{x_{1}}^{-}\times\cdots\times\RR_{x_{n}}^{-}\subset(\CC^{*})^{n}$
with analytic closure $\TT^{n}\subset\PDT$; note that its class in
$H_{n}(\PDT,\DDT)$ is Lefschetz dual to that of the $n$ torus $\hat{\TT}^{n}$
in $H_{n}((\CC^{*})^{n})$. Let $\K$ denote the analytic closure
of $\phi(\oTT)$ in $\PP_{\l}^{1}$, with (open) complement $U:=\PP^{1}\m\K\subseteq\AA_{\l}^{1}$,
and set $\tilde{\X}_{U}:=\tilde{\pi}^{-1}(U)\subseteq\TXM$, $\tilde{\X}_{\K}:=\tilde{\pi}^{-1}(\K)\subseteq\tilde{\X}$.
(If $U$ is not connected, replace it by a single connected component,
and augment $\K$ by the other connected components.) Finally, let
$X:=\tilde{X}^{\l_{0}}$ be a very general fiber (with $\l_{0}\in U$).

\begin{prop}
(a) Let $\TXM$ be one of the families from Theorem $1.7$ with nonsingular
general fiber and assume $\ker\{H_{n-2}(\J)\to H_{n-2}(X)\}=0.$ Then
there exists a single-valued family of cohomology classes $\R^{\l}\in H^{n-1}(\tilde{X}^{\l},\CC)$
lifting $AJ(\Xi^{\l})$ for $\l\in U$. (This includes singular fibers
\emph{{[}}$=U\cap\L$\emph{]} unless $n=2$ and $\J\cap\I$ is nonempty.)

(b) For $n=3$ and $\A$ nonempty (the case excepted above), $H_{1}(\J^{-}\m\J^{-}\cap\A)=0$
$\implies$ conclusion of (a) holds as stated. If we assume instead
$H_{1}(\mathcal{{D}})=0$, then the conclusion only holds with $\tilde{X}^{\l}$
replaced by $\hat{X}^{\l}$ (and $\R^{\l}$ lifts $AJ(\Xi_{0}^{\l})\in H^{n-1}(\hat{X}^{\l},\CC/\QQ(n))$).
\end{prop}
\begin{rem}
(i) For $n=2$, the assumption of (a) says $\J$ is one point; for
$n=3$ it says $H_{1}(\J)=0$: $\J$ is a configuration of rational
curves whose associated graph has no loop.

(ii) The continuation of $\R^{\l}$ around a loop not in $U$ may
no longer be single-valued over $U$.

(iii) A relaxation of the hypotheses (e.g. allowing singularities
in the general fiber, $\phi$ not regular) may be necessary to produce
examples for $n=4$.
\end{rem}
\begin{proof}
We do this under the assumption that the total space $\tilde{\X}$
is nonsingular. (While such examples come out of Theorem $1.7$, we
don't know if any of these survive the extra requirements for this
Proposition; nevertheless, the main ideas are contained in our {}``artificial''
proof, and the more general situation is treated with cone complexes
as in Theorem $1.7$'s proof.) Write $\overline{Z}^{p}(\cdot,n)$
for $\d_{\B}$-closed higher Chow precycles.

In the proof of Theorem $1.7$ we started by {}``completing'' $\xi=\{\underline{x}\}\in\overline{Z}^{n}(\TXM\m\J\times\AA^{1},n)$
to $\Xi\in\overline{Z}^{n}(\TXM,n)$ restricting to $\xi+\d_{\B}\gamma$
(on $\TXM\m\J\times\AA^{1}$); since $\xi\in\overline{Z}_{\RR}^{n}(\TXM\m\J\times\AA^{1},n)_{X\m\J(\times\{x_{0}\})}$,
we may arrange to have \[
\Xi\in\overline{Z}_{\RR}^{n}(\TXM,n)_{X}\,,\,\,\,\,\,\gamma\in Z_{\RR}^{n}(\TXM\m\J\times\AA^{1},n+1)_{X\m\J}\,,\]
 the first pulling back to $\Xi^{\l_{0}}\in\overline{Z}_{\RR}^{n}(X,n)$.
We take the analytic closure of the $\d$-closed Borel-Moore $C^{\infty}$
chain $T_{\xi}$ on $\TXM\m\J\times\AA^{1}$ to get $\overline{T_{\xi}}\in Z_{n}^{top}(\tilde{\X},\tilde{X}_{0}\cup\J\times\PP^{1})$.
Since $(\tilde{\X}_{U}\m\J\times U)\cap\TT^{n}=\emptyset$ by construction,
we see that $\overline{T_{\xi}}$ maps to $0$ in $Z_{n}^{top}(\tilde{\X},\tilde{\X}_{\K}\cup\J\times\PP^{1})$.
Clearly $\overline{T_{\Xi}}\in Z_{n}^{top}(\tilde{\X},\tilde{X}_{0})$
maps to $\overline{T_{\xi}}+\d\overline{T_{\gamma}}$ in $Z_{n}^{top}(\tilde{\X},\tilde{X}_{0}\cup\J\times\PP^{1}),$
hence to $\d\overline{T_{\gamma}}$ in $Z_{n}^{top}(\tilde{\X},\tilde{\X}_{\K}\cup\J\times\PP^{1})$;
and so in $Z_{n}^{top}(\tilde{\X},\tilde{\X}_{\K})$, $\overline{T_{\Xi}}$
is homologous to a cycle $\tau\in Z_{n}^{top}(\J\times(\PP^{1},\K))\cong Z_{n}^{top}(\J\times(\overline{U},\d\overline{U}))$
(where $\d\overline{U}:=\overline{U}\m U$). (The latter may be put
in good position with respect to $X$, since $T_{\Xi}$ is.)

Now $0=F^{n}H^{n}(X,\CC)\cap H^{n}(X,\QQ(n))$ $\implies$ $0\homeq T_{\Xi^{\l_{0}}}=T_{\Xi}\cap X$
(on $X$) $\implies$ $\tau\cap X\homeq0$ (on $X$). Moreover, $H_{n}(\J\times(\overline{U},\d\overline{U}))=H_{n-2}(\J)\otimes H_{2}(\overline{U},\d\overline{U})\cong H_{n-2}(\J)$
since $U$ connected $\implies$$H_{2}(\overline{U},\d\overline{U})=\QQ$,
$\K$ connected $\implies$ $U$ simply connected $\implies$ $H_{1}(\overline{U},\d\overline{U})=0$,
and obviously $H_{0}(\overline{U},\d\overline{U})=0.$ Hence, $\ker\{H_{n-2}(\J)\to H_{n-2}(X)\}=0$
$\implies$ $\tau\homeq0$ $\implies$ $\exists$ $\Gamma\in Z_{n+1}^{top}(\tilde{\X},\tilde{\X}_{\K})$
with $\d\Gamma=T_{\Xi}$ (mod $\tilde{\X}_{\K}$), and we define $R_{\Xi}':=R_{\Xi}+(2\pi i)^{n}\delta_{\Gamma}\in\mathcal{{D}}^{n-1}(\tilde{\X}_{U})$.
One has $d[R_{\Xi}']=\Omega_{\Xi}\in F^{n}\mathcal{{D}}^{n}(\tilde{\X}_{U})$.

This $\Omega_{\Xi}$, being a $d$-closed $(n,0)$-current, is in
fact $C^{\infty}$ (i.e. holomorphic) by standard regularity results.
On $\tilde{\X}_{U}$ it is cohomologous to $0$, hence $d\eta$ there
for some $C^{\infty}$ $(n-1)$-form $\eta$. Hence $R_{\Xi}'-\eta$
is closed and $\exists$ $(n-2)$-current $\kappa$ such that $R_{\Xi}'-\eta+d[\kappa]$
is $C^{\infty}$ (in the same class); obviously $R_{\Xi}'+d[\kappa]$
is also $C^{\infty}$ (but not closed), and so pulls back to every
fiber to give a continuous family of (closed $C^{\infty}$ forms $\implies$)
classes in $\{H^{n-1}(\tilde{X}^{\l},\CC)\}_{\l\in U}$ (including
singular fibers).

Next pick any $\l_{1}\in U$, put $X_{1}:=\tilde{X}^{\l_{1}}$; we
must show $[\iota_{X_{1}}^{*}(R_{\Xi}'+d[\kappa])]$ lifts $AJ(\iota_{X_{1}}^{*}\Xi_{1})\in H^{n-1}(X,\CC/\QQ(n))$
for some {}``move'' $\Xi_{1}$ of $\Xi$. Namely, use $\M\in Z_{\RR}^{n}(\TXM,n+1)$
to get $\Xi_{1}:=\Xi+\d_{\B}\M\in\overline{Z}_{\RR}^{n}(\TXM,n)_{X_{1}}$,
and $\mu\in C_{n+2}^{top}(\tilde{\X},\tilde{\X}_{\K})$ to move $\Gamma$
to $\Gamma_{1}:=\Gamma-T_{\M}+\d\mu\in C_{n+1}^{top}(\tilde{\X},\tilde{\X}_{\K})_{X_{1}}$.
Note that $\d\Gamma_{1}=\d\Gamma-\d T_{\M}=T_{\Xi}-\d T_{\M}=T_{\Xi_{1}}$,
so that $R_{\Xi_{1}}':=R_{\Xi_{1}}+(2\pi i)^{n}\delta_{\Gamma_{1}}$
has $d[R_{\Xi_{1}}']=\Omega_{\Xi_{1}}=\Omega_{\Xi}$. Moreover, the
$d$-closed pullback $\iota_{X_{1}}^{*}R_{\Xi_{1}}'=R_{\iota_{X_{1}}^{*}\Xi_{1}}+(2\pi i)^{n}\delta_{\d^{-1}(\iota_{X_{1}}^{*}T_{\Xi_{1}})}$
so its class lifts $AJ(\iota_{X_{1}}^{*}\Xi_{1}).$ Now we compare
the two things pulled back, $\iota_{X_{1}}^{*}$ of $\R_{\Xi_{1}}'$
and $\R_{\Xi}'+d[\kappa]$:\[
R_{\Xi_{1}}'\,\,=\,\, R_{\Xi}+d\left[\frac{R_{\M}}{2\pi i}\right]+(2\pi i)^{n}\delta_{T_{\M}+\Gamma_{1}}\]
\[
=\,\, R_{\Xi}+d\left[\frac{R_{\M}}{2\pi i}+(2\pi i)^{n}\delta_{\mu}\right]+(2\pi i)^{n}\delta_{\Gamma}\]
\[
=\,\, R_{\Xi}'+d[=:S\,]\,,\]
hence $R_{\Xi_{1}}'-R_{\Xi}'-d[\kappa]=d[S-\kappa]$. If $S-\kappa$
does not pull back to $X_{1}$, it is replaceable by something that
does (since the l.h.s. does). 
\end{proof}
Stiller \cite{St} studied monodromy of solutions to inhomogeneous
equations, in the case where the corresponding homogeneous equation
$D_{PF}(\cdot)=0$ is solved by the period functions associated to
an elliptic modular surface. It would be interesting to compare his
formula (\cite{St}, Thm. 10) with the following for $n=2$.

\begin{cor}
In the situation of Proposition $2.11$((a) or (b)), the inhomogeneous
equation $D_{PF}(\cdot)=g$ admits a solution single-valued in $U$
(i.e. also finite at $U\cap\L$, except possibly when $n=2$ and $\J\cap\I\neq\emptyset$).
\end{cor}
Of course, this is most interesting in case $ord(D_{PF})=n$ and Corollary
$2.9$ also applies.

As an application of higher normal functions and Corollary $2.4$,
we consider the problem of producing $linearly$ $independent$ families
of higher Chow cycles over $\CP:=\PP_{t}^{1}\m\T$, where $\T\ni\{0\}$
is a collection of points. Since the idea will be to produce independent
topological invariants $[\Omega]\in F^{n}H^{n}(\tilde{\X}_{\CP},\CC)\cap H^{n}(\tilde{\X}_{\CP},\QQ(n))$
($\tilde{\X}_{\CP}:=\tilde{\pi}^{-1}(\CP)$), larger $\T$ is better.
In fact, $\T=\{(t=)\,0\}$ won't do, as $F^{n}H^{n}(\TXM,\CC)\cong F^{n}H^{n}((\CC^{*})^{n},\CC)\cong\CC\left\langle \Omega_{\Xi}=\bigwedge^{n}\dlog\underline{x}\right\rangle $
has rank $1$.

Suppose we have a rational map (defined $/\bar{\QQ}$) of families
satisfying the conditions of Theorem $1.7$: 

$ \mspace{200mu} $ \xymatrix{{\tilde{\X}_{\CP}} \ar @{-->} [r]^{\mathfrak{A}} \ar [d]_{\tilde{\pi}} & {}'\TXM \ar [d]_{{}'\tilde{\pi}} \\ {\CP} \ar [r]^{\alpha \mspace{40mu}} & {\PP^1 \m \{ 0 \}}.}\\
That is, we have Zariski open $\V_{\CP}\subseteq\tilde{\X}_{\CP}$,
hence some blow-up $\mathfrak{{Y}}_{\CP}\rOnto^{\mathfrak{{B}}}\tilde{\X}_{\CP}$,
mapping to ${}'\TXM$ over $\alpha$. Write $\mathfrak{{A}}_{t}:\,\TXT-->{}'\tilde{X}_{\alpha(t)}$,
$u_{i}:=\mathfrak{{A}}^{*}({}'x_{i})\in\bar{\QQ}(\tilde{\X}_{\CP})^{*}$.
If $\mathfrak{{A}}$ is the restriction of a rational map $\PDT\times\PP^{1}-->\PP_{{}'\DT}\times\PP^{1}$
given by $(x_{1},\ldots,x_{n};t)\mapsto(f_{1}(\underline{x};t),\ldots,f_{n}(\underline{x};t);\alpha(t))=({}'x_{1},\ldots,{}'x_{n};{}'t),$
then $u_{i}=f_{i}(\underline{x};t)$.

By pulling $'\Xi$ back to $\mathfrak{Y}_{\CP}$ and pushing forward
along $\mathfrak{{B}}$ we obtain \[
\Theta\,:=\,\mathfrak{A}^{*}({}'\Xi)\,=\,\text{{completion\, of\,\,}}\{\underline{u}\}\,\in\, CH^{n}(\tilde{\X}_{\CP},n).\]
Clearly $\Omega_{\Theta}=\mathfrak{A}^{*}(\Omega_{{}'\Xi})$, and
this is a holomorphic form; since the fibers of $\tilde{\pi}$ are
CY, $[\Omega_{\Theta}]=[(\tilde{\pi}^{*}G)\Omega_{\Xi}]$ for some
$G\in\bar{\QQ}(\PP^{1})^{*}$. On the fibers we have $\mathfrak{A}_{t}^{*}[{}'\tilde{\omega}_{\alpha(t)}]=G(t)[\twt]$,
and $\mathfrak{A}_{t}^{*}({}'\R_{\alpha(t)})=:\mathcal{{S}}_{t}$
lifting $AJ(\Theta_{t})$. Corollary $2.4$ for ${}'\Xi$ says $\nabla_{\delta_{\alpha(t)}}{}'\R_{\alpha(t)}=[{}'\tilde{\omega}_{\alpha(t)}],$
and applying $\mathfrak{A}^{*}$ gives $\nabla_{\delta_{\alpha(t)}}\mathcal{{S}}_{t}=G(t)[\twt]$,
or\[
\nabla_{\delta_{t}}\mathcal{{S}}_{t}=\frac{t\alpha'(t)}{\alpha(t)}G(t)[\twt].\]
Comparing this with $\nabla_{\delta_{t}}\R_{t}=[\twt]$ (and noting
that $\nabla_{\delta_{t}}$ removes the ambiguities in the lifts of
$AJ$ of $\Theta_{t}$, $\Xi_{t}$), we obtain:

\begin{cor}
If $\frac{t\alpha'}{\alpha}G$ is not a rational constant, then the
families of classes $\Theta_{t},\,\Xi_{t}\in CH^{n}(\TXT,n)$ are
(AJ-)independent.
\end{cor}
There are examples where $\alpha(t)=\pm\frac{1}{t}$ and $G(t)=t$
for $n=2$ and $3$, see \cite{Ke2}.

We can also compare the higher normal functions $\nu(t):=\left\langle \R_{t},[\twt]\right\rangle ,$
$\e(t):=\left\langle \mathcal{{S}}_{t},[\twt]\right\rangle $. If
$0\neq g:=D_{PF}\nu$, and $\frac{t\alpha'}{\alpha}G$ is not a rational
constant, then from\[
D_{PF}\e=\frac{t\alpha'}{\alpha}Gg\]
one may deduce independence of the families of Milnor $K$-theory
classes $\{\underline{x}\},\{\underline{u}\}\in K_{n}^{M}(\CC(\TXT))$
for $n=2,3$.

In the event that $\alpha$ is of infinite order (rather than e.g.
an involution like $t\mapsto\pm\frac{1}{t}$), iteratively applying
the above construction (for $\alpha$, $\alpha\circ\alpha$, $\alpha\circ\alpha\circ\alpha$,
etc. which of course requires shrinking $\CP$ at each stage) would
give explicit countable generation for $CH^{n}(\text{{generic fiber}},n).$
However it seems likely (already for $n=2$, by comparing with the
proof of infinite generation in \cite{Co}, sec. 7) that this is not
possible without allowing $\alpha$ to be $algebraic$ and replacing
the Zariski neighborhood $\CP$ with an \'etale one; the relevant
(geometric) generic fiber is then defined over $\overline{\QQ(\PP^{1})}$
(rather than $\bar{\QQ}(\PP^{1})$).

\section{\textbf{An application to local mirror symmetry}}

For any reflexive polytope $\Delta\subset\RR^{n}$ ($n=2,3,4$), the
total space of $\mathsf{K}_{\PP_{\Delta^{\circ}}}$ may be viewed
as a noncompact CY $(n+1)$-fold. If we let $F\in\CC[x_{1}^{\pm1},\ldots,x_{n}^{\pm1}]$
range over Laurent polynomials with $Conv(\mathfrak{M}_{F})=\Delta$,
then the family\[
Y_{F}:=\{F(\underline{x})+u^{2}+v^{2}=0\}\subset(\CC^{*})^{n}\times\CC^{2}\]
 of $(n+1)$-folds is the mirror dual of $\mathsf{K}_{\PP_{\Delta^{\circ}}}$.
These are CY, since the holomorphic form\[
\eta_{F}:=2i\cdot Res_{Y_{F}}\left(\frac{\bigwedge^{n}\dlog\underline{x}\wedge du\wedge dv}{F+u^{2}+v^{2}}\right)\in\Omega^{n+1}(Y_{F})\]
yields a nonvanishing global section of the canonical bundle (i.e.
$\mathsf{K}_{Y_{F}}$). Its periods may be interpreted in terms of
regulator periods on the $X_{F}^{*}:=\{F(\underline{x})=0\}\subset(\CC^{*})^{n}$.
We work out this story in $\S3.1$ and use it to compute the mirror
map for $n=2$ in $\S3.3$. Only in $\S3.4$ (and the end of $\S3.1$)
do we once again require $F$ to be tempered, in order to link up
with $\S\S1,2,4$ and study asymptotic growth of local Gromov-Witten
numbers for $\mathsf{K}_{\PP_{\Delta^{\circ}}}$.

\subsection{Periods of an open CY $3$-fold}

Let $X_{F}\subset\PP_{\Delta}$ be the Zariski closure of $X_{F}^{*}$,
with crepant resolution $\tilde{X}_{F}\subset\PP_{\tilde{\Delta}}$;
denote the inclusion $J:\, X_{F}^{*}\rInto\tilde{X}_{F}$. We assume
$F$ is $\Delta$-regular, so that $\tilde{X}_{F}$ is smooth and
the $D_{\tilde{\sigma}}$ reduced ($\forall i\geq1$, $\tilde{\sigma}\in\tilde{\Delta}(i)$).
Write $\{\underline{x}\}:=\{x_{1},\ldots,x_{n}\}\in CH^{n}((\CC^{*})^{n},n)$
and $\xi_{F}:=I^{*}\{\underline{x}\}\in CH^{n}(X_{F}^{*},n)$ for
its restriction to $X_{F}^{*}\rInto^{I}(\CC^{*})^{n}$. We use a somewhat
nonstandard definition\[
H_{n-1}^{tr}(\tilde{X}_{F}):=\text{{im}}\{H_{n-1}(X_{F}^{*},\QQ)\rTo^{J_{*}}H_{n-1}(\tilde{X}_{F},\QQ)\}\]
for the {}``transcendental part'' of homology; clearly this is everything
for $n=2$ and contains the orthogonal complement of $Pic(\tilde{X}_{F})$
for $n=3.$ Also define\[
\K_{n-1}(X_{F}^{*}):=\ker\{H_{n-1}(X_{F}^{*},\QQ)\rTo^{I_{*}}H_{n-1}((\CC^{*})^{n},\QQ)\}.\]

\begin{lem}
$\K_{n-1}(X_{F}^{*})$ surjects onto $H_{n-1}^{tr}(\tilde{X}_{F})$;
that is, every class $\Gamma$ in \linebreak $H_{n-1}^{tr}(\tilde{X}_{F},\QQ)$
has a representative $\gamma\in Z_{n-1}^{top}(X_{F}^{*};\QQ)$ that
bounds in $(\CC^{*})^{n}$.
\end{lem}
\begin{proof}
Choose an edge $\sigma_{1}\in\Delta(n-1)$ and vertex $\underline{\nu}\in\Delta(n)$
on $\sigma_{1}$. (More precisely, we take $\tilde{\sigma}_{1}\in\tilde{\Delta}(n-1)$
and $\underline{\tilde{\nu}}\in\tilde{\Delta}(n)$ sitting {}``over''
these.) Repeat the construction of $\S2.1$ so that $\Phi_{\underline{\nu}}=0$
locally describes $\tilde{X}_{F}$ and $1+\phi_{1}(z_{1})$ gives
(up to a constant) the edge polynomial of $\sigma_{1}$. Fix a root
$r$($\in\CC^{*}$) of this, define in $Z_{n-1}^{top}(X_{F}^{*};\ZZ)$\[
\delta_{\sigma_{1}}:=\{\Phi_{\underline{\nu}}=0\}\cap\{|z_{2}|=\cdots=|z_{n}|=\epsilon\}\cap\{|z_{1}-r|\text{{\,"small"}}\}\]
and notice $\delta_{\sigma_{1}}\homeq0$ on $\tilde{X}_{F}$. Write
$z_{1}(=x_{1}^{\sigma_{1}})=:\underline{x}^{\underline{m}(\sigma_{1})}$.

Define projections and inclusions

$ \mspace{75mu} $\xymatrix{(\CC^*)^n \ar @{->>} [r]^{\pi_i \mspace{150mu}}  & {\left\{ \underline{x}\in(\CC^*)^n \, | \, x_i=1 \right\} } \cong (\CC^*)^{n-1} \ar @{^(->} [r]^{\mspace{150mu} \iota_i} & (\CC^*)^n \\ & {\left\{ x_i=1, \, |x_j|=1 \, \forall j\neq i \right\} } =: \hat{\TT}^{n-1}_i . \ar @{^(->} [u] }
\\
We can orient everything so that $\pi_{i}{}_{_{*}}(I(\delta_{\sigma_{1}}))\homeq m_{i}(\sigma_{1})\hat{\TT}_{i}^{n-1}$;
hence $I(\delta_{\sigma_{1}})\equiv\sum_{i=1}^{n}m_{i}(\sigma_{1})\iota_{i}{}_{_{*}}(\hat{\TT}_{i}^{n-1}).$
Now the $\{\underline{m}(\sigma_{1})\}$ (taken over all such edges)
generate $\QQ^{n}$; hence the $\{I(\delta_{\sigma_{1}})\}$ generate
$H_{n-1}((\CC^{*})^{n-1},\QQ)$.

Given $\Gamma\in H_{n-1}^{tr}(\tilde{X}_{F})$, let $\gamma^{0}$
be a representative in $Z_{n-1}^{top}(X_{F}^{*})$. We may choose
an appropriate sum $\delta$ of $\delta_{\sigma_{1}}$'s with $I(\gamma^{0})\homeq I(\delta)$;
clearly $\delta\homeq0$ on $\tilde{X}_{F}$, and so taking $\gamma:=\gamma^{0}-\delta$
we are done.
\end{proof}
\begin{rem}
When $|\gamma^{0}|\subseteq X_{F}^{*}\cap\{\RR^{n}\text{{\, or\,}}(i\RR)^{n}\}$,
$I(\gamma^{0})$ bounds on $(\CC^{*})^{n}$ without modification by
a $\delta$. {[}\noun{Proof}: For any cycle $\sZ$ on $(\CC^{*})^{n}$,
$Box^{n}(\sZ):=\sZ+\sum_{k=1}^{n}(-1)^{k}\sum_{|I|=k}(\iota_{I}\circ\pi_{I})_{_{*}}\sZ\homeq0$;
since $H_{n-1}((\CC^{*})^{j<n-1})=0$, it follows that $I(\gamma^{0})-\sum_{i=1}^{n}(\iota_{i}\circ\pi_{i})_{_{*}}I(\gamma^{0})$
bounds (in $(\CC^{*})^{n}$). But if $\gamma^{0}$ has real support
then each $(\pi_{i})_{_{*}}I(\gamma^{0})$ {}``cancels itself out'',
being of the same real dimension as the $real$ part of the target
($=$disjoint union of copies of $(\RR^{+})^{n-1}$).] This is essentially
used for the real, nonvanishing cycle $L_{0}$ (for real $t$ near
$0$) in Appendix A of \cite{Ho}. However, the procedure (employed
there) of {}``bounding'' the vanishing cycles $\{K_{j}\}$ with
$noncompact$ membranes is unnecessary in view of Lemma 1, and also
incorrect in homology.
\end{rem}
\begin{lem}
If $\gamma\in Z_{n-1}^{top}(X_{F}^{*};\ZZ)$ has $I(\gamma)=\partial\mu$,
for $\mu\in C_{n}^{top}((\CC^{*})^{n};\ZZ)$, then\[
\int_{\gamma}R(\xi_{F})\,\equiv\,\int_{\mu}\wedge^{n}\dlog\underline{x}\mspace{20mu}\mod\,\ZZ(n).\]

\end{lem}
\begin{proof}
On $(\CC^{*})^{n}$, $\bigwedge^{n}\dlog\underline{x}=d[R\{\underline{x}\}]\pm(2\pi i)^{n}\delta_{T_{\underline{x}}}$,
and so \[
\int_{\mu}\wedge^{n}\dlog\underline{x}\,\equiv\,\int_{\mu}d[R\{\underline{x}\}]\,=\,\int_{\partial\mu}R\{\underline{x}\}\,=\,\int_{\gamma}I^{*}R\{\underline{x}\}.\]

\end{proof}
We want to construct cycles in $Z_{n+1}^{top}(Y_{F})$ over which
to integrate $\eta_{F}$. Considering $Y_{F}$ as a fiber bundle over
$(\CC^{*})^{n}$, we have (for $n=2$) the picture displayed in Figure
3.1.%
\begin{figure}
\caption{\protect\includegraphics[scale=0.8]{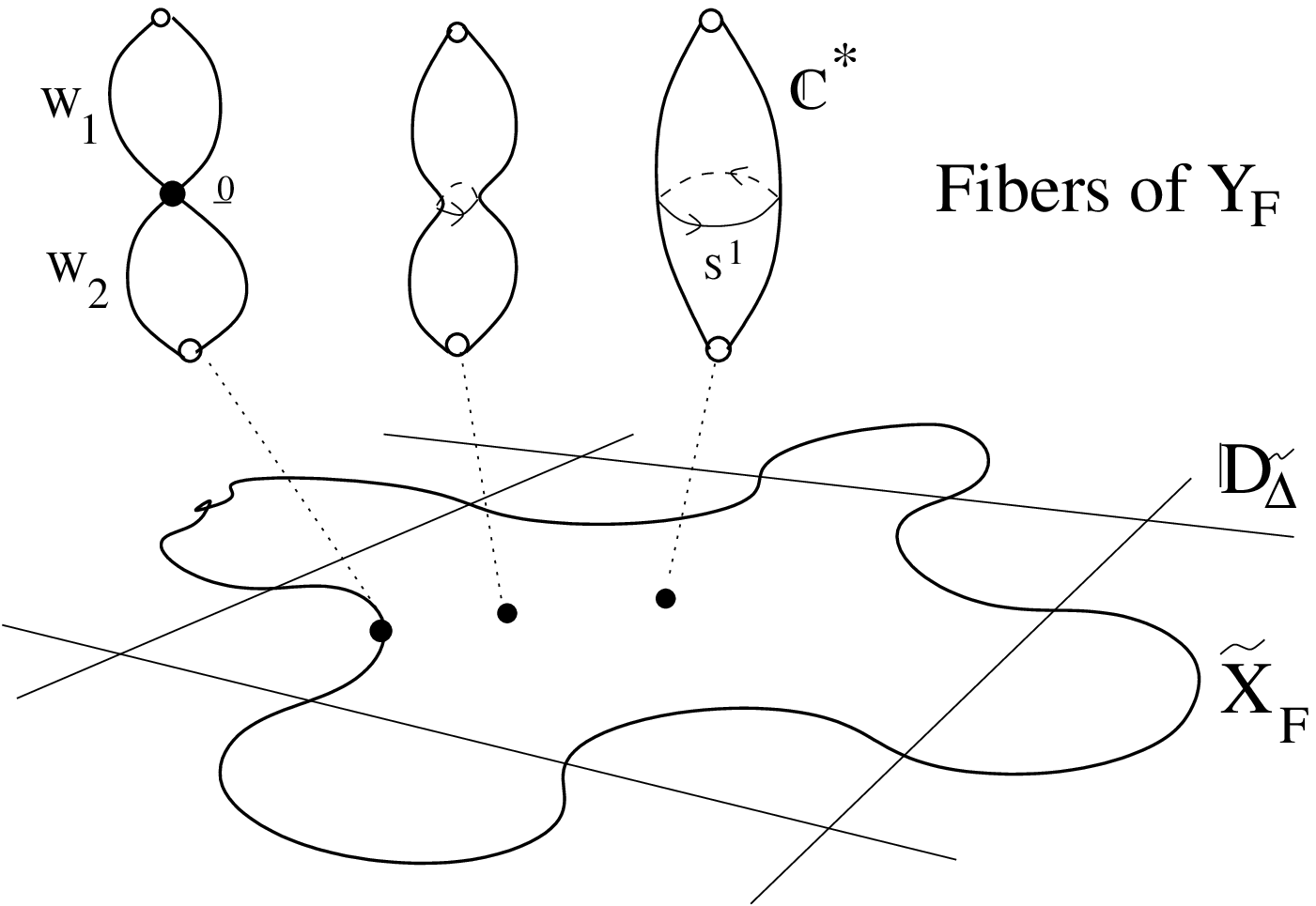}}

\end{figure}
 In a topological sense, we may view $Y$ as the disjoint union of
an $S^{1}$-bundle over $(\CC^{*})^{n}$ with a copy of $X_{F}^{*}$.
More precisely, if $P:\, Y_{F}\rOnto(\CC^{*})^{n}$ sends $(\underline{x},u,v)\mapsto\underline{x}$,
then\[
\underline{x}\in(\CC^{*})^{n}\m X_{F}^{*}\,\,\,\implies\,\,\, P^{-1}(\underline{x})\cong\CC^{*}\,\,\,\text{(homotopic to }S^{1}\text{)}\]
\[
\underline{x}\in X_{F}^{*}\,\,\,\implies\,\,\, P^{-1}(\underline{x})\cong\{u^{2}+v^{2}=0\}=:W=W_{1}\cup W_{2}\]
where $W_{i}\cong\AA_{\CC}^{1}$. In fact, $Y_{F}\supset X_{F}^{*}\times W$
and we can write $W=W_{1}\amalg W_{2}^{*}$ ($W_{2}^{*}:=W_{2}\m\{(0,0)\}$);
the complement $Y_{F}\m(X_{F}^{*}\times W_{1})$ is then homotopic
to $(\CC^{*})^{n}\times S^{1}$.

Consider the long-exact sequence \begin{equation} \xymatrix{ \\ H_n(Y_F \m X_F^* \times W_1 ) \ar [r]_{\cong} \ar [u] & H_{n-1}((\CC^*)^n) \oplus H_n((\CC^*)^n) \\ H_{n-1}(X_F^* \times W_1) \ar [r]_{\cong} \ar [u]^{tube} & H_{n-1}(X_F^*) \ar [u]^{(I_*,0)} \\ H_{n+1}(Y_F) \ar [u]^{\cap} \\ H_{n+1}(Y_F \m X_F^* \times W_1) \ar [r]^{\text{"forget }S^1\text{"}}_{\cong} \ar [u] & H_n((\CC^*)^n) \\ H_n(X_F^* \times W_1) \ar [u]^{tube} \ar [r]_{\cong} & H_n(X_F^*) . \ar [u]^{I_* =0} \\ \ar [u] }  \end{equation}The
bottom $I_{*}$ is $0$ because the dual map $[F^{n}]H^{n}((\CC^{*})^{n})\to H^{n}(X_{F}^{*})$
must be, as $\dim(X_{F}^{*})=n-1$ $\implies$ $F^{n}H^{n}(X_{F}^{*})=\{0\}$.
The $H_{n-1}(X_{F}^{*})\to H_{n}((\CC^{*})^{n})$ is essentially the
composition of $Tube:\, H_{n-1}(X_{F}^{*})\to H_{n}((\CC^{*})^{n}\m X_{F}^{*})$
with $H_{n}((\CC^{*})^{n}\m X_{F}^{*})\to H_{n}((\CC^{*})^{n})$;
it is $0$ for a similar reason.

Using any $\hat{\TT}_{\underline{\nu},\epsilon}^{n}\in Z_{n}^{top}((\CC^{*})^{n}\m X_{F}^{*})$
(see $\S2.1$) and the topological {}``$S^{1}$-bundle'' structure
of $Y_{F}\,\m(X_{F}^{*}\times W_{1})$, gives a cycle $\hat{\TT}_{Y}^{n+1}\in Z_{n+1}^{top}(Y_{F})$.
Now (3.1) becomes the short-exact sequence\[
\QQ\left\langle \hat{\TT}_{Y}^{n+1}\right\rangle \to H_{n+1}(Y_{F})\to\K_{n-1}(X_{F}^{*}).\]
To construct explicitly an isomorphism\[
M:\,\K_{n-1}(X_{F}^{*})\to H_{n+1}(Y_{F})\left/\QQ\left\langle \hat{\TT}_{Y}^{n+1}\right\rangle \right.,\]
let $\gamma,\,\mu$ be as in Lemma 2 ($\QQ$-coefficients). The cycle
(representing) $M(\gamma)$ will have support in $P^{-1}(|\mu|)$,
with $S^{1}$-fibers over $Int|\mu|$ and point fibers over $|\partial\mu|=|\gamma|$.
More precisely, $M(\gamma)\cap P^{-1}(\underline{x})$ (for $\underline{x}\in|\mu|$)
is given by\[
V\in[-\sqrt{|F(\underline{x})|},\,\sqrt{|F(\underline{x})|}]\,,\,\,\,\, v=e^{\frac{i}{2}\arg(-F(\underline{x}))}V\,,\,\,\,\, u=\pm\sqrt{-(v^{2}+F(\underline{x}))}.\]
Note that $\QQ\left\langle \hat{\TT}_{Y}^{n+1}\right\rangle $ absorbs
the ambiguity arising from the choice of $\mu$.

\begin{lem}
For $\gamma$, $\mu$ as in Lemma 2, \[
\int_{M(\gamma)}\eta_{F}\,\,=\,\,2\pi i\,\int_{\mu}\wedge^{n}\dlog\underline{x}\,.\]
Moreover, $\int_{\hat{\TT}_{Y}^{n+1}}\eta_{F}=(2\pi i)^{n+1}$.
\end{lem}
\begin{proof}
Writing $u':=u+iv$, $v':=u-iv$, we have (away from $v'=0$) $\eta_{F}=Res_{Y_{F}}\left(\frac{\bigwedge^{n}\dlog\underline{x}\,\,\wedge\, du'\wedge dv'}{F(\underline{x})+u'v'}\right)=\bigwedge^{n}\dlog\underline{x}\,\wedge\dlog\, u'$.
The result is now immediate (by integrating {}``first'' over the
$S^{1}$ fibers of $M(\gamma)$).
\end{proof}
Lemmas $3.3$ and $3.4$ imply the following

\begin{prop}
The periods of $\eta_{F}$ are precisely the $\CC\to\CC/\QQ(n+1)$
lifts of the $2\pi i\int_{\gamma}R(\xi_{F})$ for $\gamma\in\K_{n-1}(X_{F}^{*})$,
including the lifts $(2\pi i)^{n+1}\QQ$ of $0$.
\end{prop}
If we now assume $F=\hat{F}$ is tempered,%
\footnote{plus additional assumptions for $n=4$ (cf. Theorem 1.7) %
} then $\xi_{\hat{F}}$ comes from some $\Xi_{\hat{F}}\in CH^{n}(\tilde{X}_{\hat{F}},n)$,
and so $R(\xi_{\hat{F}})$ has no residues to separate periods over
$\gamma_{1}$, $\gamma_{2}$($\in\K_{n-1}$) with $J_{*}\gamma_{1}=J_{*}\gamma_{2}$.
Therefore (using Lemma 3.1), we get

\begin{cor}
The periods of $\eta_{\hat{F}}$ may be expressed in terms of the
regulator periods of {}``transcendental cycles'': $\int_{(\cdot)}\eta_{\hat{F}}$
is the composition \SMALL \[
\frac{H_{n+1}(Y_{\hat{F}})}{M(\ker(J_{*})\cap\K_{n-1})+\QQ\left\langle \hat{\TT}_{Y}^{n+1}\right\rangle }\rTo_{\cong}^{M^{-1}}\frac{\K_{n-1}(X_{\hat{F}}^{*})}{\ker(J_{*})\cap\K_{n-1}}\rTo_{\cong}^{\text{Lemma 3.1}}H_{n-1}^{tr}(\tilde{X}_{\hat{F}})\rTo^{2\pi i\int_{(\cdot)}R(\Xi_{\hat{F}})}\CC/\QQ(n+1).\]

\end{cor}
\normalsize In particular, if we put ourselves in a $1$-parameter
family setting $\hat{F}=1-t\phi(\underline{x})$ for $\phi$ as in
$\S2$, then Corollaries 2.4 and 3.6 beget

\begin{cor}
The $\D$-submodule of $\H_{\tilde{X}_{t}}^{n-1}$ generated by $[\tilde{\omega}_{t}]$
is a quotient of the submodule of $\H_{Y_{t}}^{n+1}$ generated by
$[\eta_{t}]$, via\[
\nabla_{PF}^{(Y,\eta)}=\nabla_{PF}^{(\tilde{X},\tilde{\omega})}\circ\nabla_{\delta_{t}}.\]

\end{cor}
\begin{rem}
If $\tilde{\varphi}_{0}$ is a vanishing cycle (as in $\S2$), with
$\K_{n-1}\ni\varphi_{0}\rMapsto^{J_{*}}\tilde{\varphi}_{0}$, then
by Theorem 2.2 and Corollary 3.6\[
\frac{\int_{M(-\varphi_{0})}\eta_{t}}{\int_{\hat{\TT}_{Y}^{n+1}}\eta_{t}}=\frac{-2\pi i\int_{\tilde{\varphi}_{0}}R(\Xi_{t})}{(2\pi i)^{n+1}}=\frac{\Psi(t)}{-(2\pi i)^{n}}\sim\frac{\log t}{2\pi i}\]
as $t\to0$. So this period ratio is custom-made for defining a mirror
map.
\end{rem}

\subsection{The canonical bundle as a CY toric variety}

We specialize to the case $n=2$ for the remainder of the section.
Let $\Delta\subset\RR^{2}$ be a reflexive polytope with vertices
$\underline{\nu}^{(1)}$, $\ldots$, $\underline{\nu}^{(r+2)}$ numbered
counter-clockwise. Together with $\underline{\nu}^{(0)}=\{\underline{0}\}$,
these are the {}``relevant integral points'' of $\Delta$ (any interior
points of edges are excluded). We have a (partial) triangulation $tr(\Delta)$
using the segments $\mathfrak{s}^{(k)}=[\underline{\nu}^{(0)},\underline{\nu}^{(k)}]$,
and write $\underline{\nu}^{(i,j)}:=\underline{\nu}^{(j)}-\underline{\nu}^{(i)}$.

A fan $\Sigma_{\Delta}$ is obtained by taking cones on $tr(\Delta)\times\{1\}\,\subset\,\RR^{3}$.
The generators of $\Sigma_{\Delta}(1)$ are $\{\underline{\hat{\nu}}^{(0)},\ldots,\underline{\hat{\nu}}^{(r+2)}\}$
where $\underline{\hat{\nu}}^{(k)}=(\underline{\nu}^{(k)},1).$ The
associated toric variety $Y^{\circ}$ is the total space of $\mathsf{K}_{\PP_{\Delta^{\circ}}}\rTo^{\rho}\PP_{\Delta^{\circ}}$.
The line bundle $\mathsf{K}_{Y^{\circ}}$ is trivialized by a {[}global
nonvanishing] {}``tautological section'', making $Y^{\circ}$ an
open CY $3$-fold. If edges of $\Delta$ have interior integral points
$\underline{u}^{(\ell)}$ then $Y^{\circ}$ is singular (but normal).
When we refer to the {}``singular case'' resp. {}``smooth case''
below, this is what is meant.

The curves $C_{i}^{\circ}\subset Y^{\circ}$ dual to subfans $\Sigma_{\mathfrak{s}^{(i)}}$
are in $1$-$1$ correspondence with edges of $\Delta^{\circ}$, and
are supported on the {}``$0$-section'' $D_{0}^{\circ}\cong\PP_{\Delta^{\circ}}\subset Y^{\circ}$.
The $[C_{i}^{\circ}]$ generate $H_{2}(Y^{\circ},\ZZ)$, and the \textbf{Mori
cone} (of effective curves) in $H_{2}(Y^{\circ},\RR)$ is just obtained
by taking $\RR^{\geq0}$-linear combinations of them. We assume henceforth
that the Mori cone with $this$ integral structure is smooth (cf.
\cite{CK}, p. 32; this implies simplicial). A simple example where
both $Y^{\circ}$ and Mori are smooth is shown in Figure 3.2.%
\begin{figure}
\caption{\protect\includegraphics[scale=0.7]{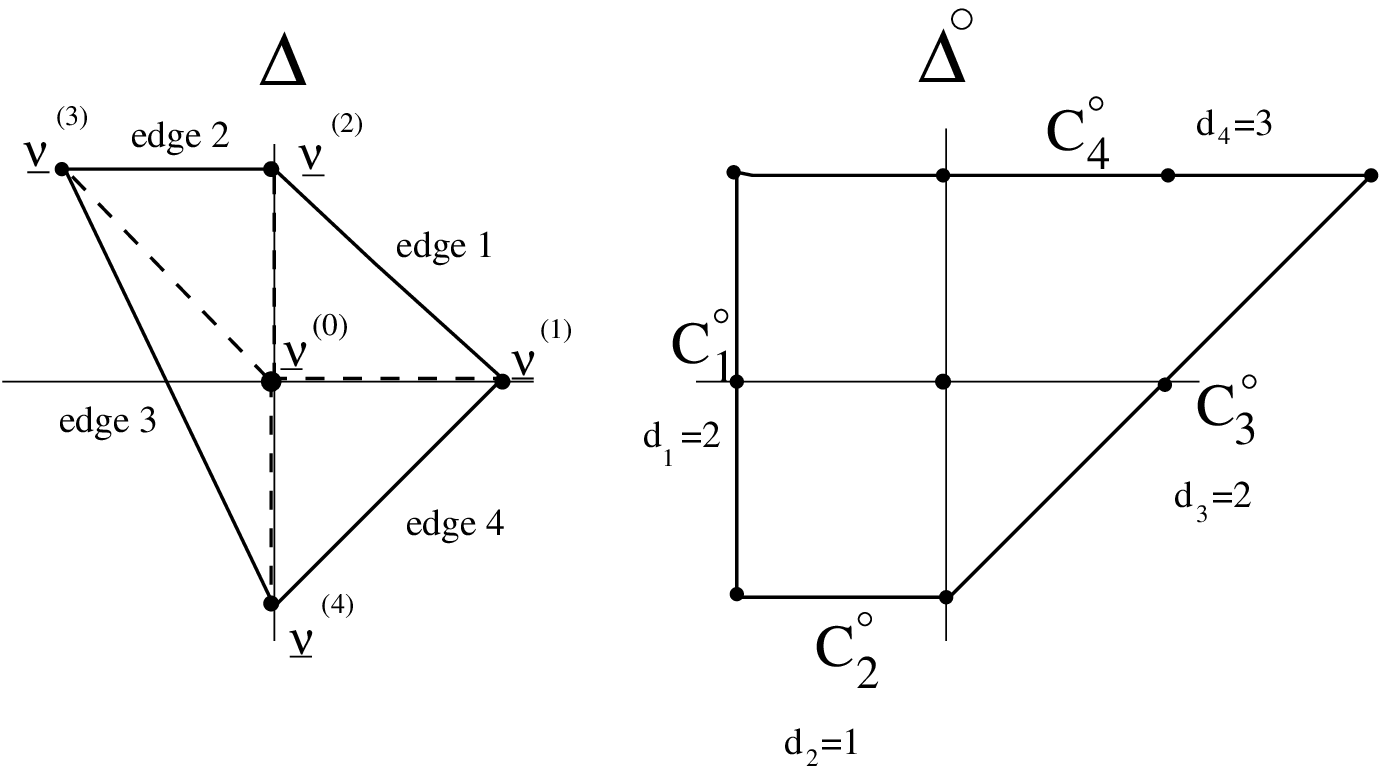}}

\end{figure}

The divisors $D_{i}^{\circ}$ dual to subfans $\Sigma_{\underline{\nu}^{(i)}}$,
$i=0,\ldots,r+2$, generate $H^{2}(Y^{\circ},\QQ)$. If $\PP_{\Delta^{\circ}}$
(and $Y^{\circ}$) are smooth then the $D_{i}^{\circ}=\rho^{-1}(C_{i}^{\circ})$.
Otherwise, using the $\underline{u}^{(\ell)}$ to refine $\Sigma_{\Delta}$
yields the crepant resolution $Y^{\circ}\lTo^{\hat{p}}\tilde{Y}^{\circ}$
over $\PP_{\Delta^{\circ}}\lTo^{p}\PP_{\tilde{\Delta}^{\circ}}$.
Denote the exceptional divisors $E_{\ell}^{\circ}$ (for $p$) and
$\hat{E}_{\ell}^{\circ}:=\tilde{\rho}^{-1}(E_{\ell}^{\circ})$ (for
$\hat{p}$); we have $H^{2}(Y^{\circ},\QQ)\cong\ker\{H^{2}(\tilde{Y}^{\circ})\to H^{2}(\cup\hat{E}_{\ell}^{\circ})\}$.
Writing $\tilde{C}_{i}^{\circ}=p^{*}C_{i}^{\circ}$ for the proper
transforms, the $D_{i}$ are then represented by cycles on $\tilde{Y}^{\circ}$
of the form $\tilde{D}_{j}^{\circ}:=\tilde{\rho}^{-1}(\tilde{C}_{i}^{\circ})+\sum_{\ell}\beta_{\ell}^{i}\hat{E}_{\ell}^{\circ}$
for $\beta_{\ell}^{i}\in\QQ$ satisfying $(\tilde{C}_{i}^{\circ}+\sum\beta_{\ell}^{i}E_{\ell}^{\circ})\cdot E_{k}^{\circ}=0$
$\forall i,k$. 

Intersections $M_{ij}:=\left\langle C_{i}^{\circ},D_{j}^{\circ}\right\rangle $
under the pairing $H^{2}(Y^{\circ})\times H_{2}(Y^{\circ})\to\QQ$
are then computed by $\tilde{C}_{i}^{\circ}\cdot\tilde{D}_{j}^{\circ}$.
These need not be integers (see \cite{Fu}) but the matrix $[M_{ij}]_{i,j\geq1}$
is symmetric. The \textbf{K\"ahler cone} is the dual of Mori in $H^{2}(Y^{\circ},\RR)$
under this pairing; it is represented by divisors $\{D=\sum\alpha_{j}D_{j}\,|\,\left\langle C_{i},D\right\rangle \geq0\,(\forall i)\}.$

In general we have in $H^{2}(Y^{\circ})$\[
D_{0}^{\circ}\,\equiv\,-\sum_{i\geq1}D_{i}^{\circ}\,\equiv\,\rho^{-1}(\mathsf{K}_{\PP_{\Delta^{\circ}}})\,\equiv\,-\rho^{-1}(X^{\circ})\]
where $X^{\circ}$ is any anticanonical (elliptic curve) hypersurface
in good position with respect to $\DD_{\Delta^{\circ}}$. Writing
$d_{i}-1\,:=\,$number of interior points of the edge of $\Delta^{\circ}$
dual to $\underline{\nu}^{(i)}$, we have ($i\geq1$)\[
-\left\langle C_{i}^{\circ},D_{0}^{\circ}\right\rangle _{Y^{\circ}}\,=\,\left\langle C_{i}^{\circ},X^{\circ}\right\rangle _{\PP_{\Delta^{\circ}}}\,=\, d_{i}.\]
Put $e_{i}-1\,:=\,$number of interior points on the edge {}``next''
(in the counterclockwise direction) to $\underline{\nu}^{(i)}$. We
are in the singular case iff some $e_{i}>1$.

We are interested in a very explicit (and standard) presentation of
the Mori cone: first, we write down generators for the integral relations
on the $\underline{\hat{\nu}}^{(i)}$ as follows. For any $k\in\{1,\ldots,r+2\}$,
let%
\footnote{$\ell_{k-1}^{(k)}$ is replaced by $\ell_{r+2}^{(k)}$ for $k=1$.%
} $\ell_{k-1}^{(k)}\underline{\hat{\nu}}^{(k-1)}+\ell_{k+1}^{(k)}\underline{\hat{\nu}}^{(k+1)}$
be the minimal $\ZZ^{+}$-linear combination lying in the line containing
$\mathfrak{s}^{(k)}$, and then choose $\ell_{k}^{(k)}\in\ZZ$, $\ell_{0}^{(k)}\in\ZZ^{\leq0}$
such that \begin{equation} \ell^{(k)}_0 \underline{\hat{\nu}}^{(0)} + \ell^{(k)}_{k-1} \underline{\hat{\nu}}^{(k-1)} + \ell^{(k)}_k \underline{\hat{\nu}}^{(k)} + \ell^{(k)}_{k+1} \underline{\hat{\nu}}^{(k+1)} = 0. \end{equation} 

\begin{rem}
One can show that these take the form\small \[
\ell_{0}^{(k)}=\frac{-e_{k}e_{k-1}d_{k}}{e_{(k,k-1)}}\,,\,\,\,\ell_{k-1}^{(k)}=\frac{e_{k}}{e_{(k,k-1)}}\,,\,\,\,\ell_{k}^{(k)}=\frac{e_{k}e_{k-1}d_{k}-e_{k}-e_{k-1}}{e_{(k,k-1)}}\,,\,\,\,\ell_{k+1}^{(k)}=\frac{e_{k-1}}{e_{(k,k-1)}},\]
\normalsize where $e_{(k,k-1)}:=\gcd(e_{k},e_{k-1})$.
\end{rem}
This procedure determines a vector $\underline{\ell}^{(k)}\in\ZZ^{
}$ with \[
d_{k}\ell_{j}^{(k)}\,=\,-\ell_{0}^{(k)}M_{kj}\,=\,\left\langle -\ell_{0}^{(k)}C_{k}^{\circ},D_{j}^{\circ}\right\rangle .\]
(In the smooth case, $d_{k}=-\ell_{0}^{(k)}$.) That is, the relations
vectors $\underline{\ell}^{(i)}$ are essentially the rows of $M$
with denominators cleared; write $L$ for the new matrix.

The Mori cone can be represented by the $\RR^{\geq0}$-span $\M\subset\RR^{r+3}$
of rows of $L$; by our above assumption (on Mori), $\M$ is simplicial.
However, the integral structures may not be the same in the {}``singular
case'', so $\M$ may not be smooth. More concretely, write $\MM:=\{\RR\text{{-span of }}\underline{\ell}^{(i)}\}\subset\RR^{r+3}$,
with integral lattice $\MM_{\ZZ}=\MM\cap\ZZ^{r+3}$, and $\M_{\ZZ}=\M\cap\MM_{\ZZ}$.
Then the affine toric variety\[
U_{\Delta}:=Spec\left\{ \CC[\underline{a}^{\underline{m}}\,|\,\underline{m}\in\M_{\ZZ}]\right\} \]
is just $\AA^{r}$ in the smooth case but can be singular in the singular
case.

Using the fact that $\M$ is simplicial, take the $\{\underline{\ell}^{(i_{k})}\}_{k=1}^{r}$
which cannot be written as $\RR^{\geq0}$-linear combinations of the
other $\{\underline{\ell}^{(j)}\}$. (In the singular case, if any
$\underline{\ell}^{(i)}$ are the same, we choose the one for which
the {}``dual'' $d_{i}$ is minimized.) Note that $\M$ is smooth
iff $\ZZ^{\geq0}\left\langle \{\underline{\ell}^{(i_{k})}\}\right\rangle $
is all of $\M_{\ZZ}$. Next, let $\alpha_{k}^{i}\in\QQ$ be such that
$J_{m}^{\circ}:=\sum_{j=1}^{r+2}\alpha_{m}^{j}D_{j}^{\circ}$ satisfy\[
\left\langle C_{i_{k}}^{\circ},J_{m}^{\circ}\right\rangle _{Y^{\circ}}\,\,\left(=\,\sum_{j=1}^{r+2}\alpha_{m}^{j}\left|\frac{d_{i_{k}}}{\ell_{0}^{(i_{k})}}\right|\ell_{j}^{(i_{k})}\right)\,=\,\delta_{km}.\]
(That is, if we omit a couple of rows from $L$, the $\{\alpha_{m}^{j}\}$
give linear combinations of the columns that yield $\hat{e}_{m}\in\RR^{r}$.)
These $\{J_{m}^{\circ}\}$ then generate the K\"ahler cone. We have
$\sum d_{i_{k}}J_{k}^{\circ}\equiv-D_{0}^{\circ}$ since $\sum_{k}d_{i_{k}}\left\langle C_{i_{j}}^{\circ},J_{k}^{\circ}\right\rangle =d_{i_{j}}=-\left\langle C_{i_{j}}^{\circ},D_{0}^{\circ}\right\rangle .$

\begin{rem}
The $\{\alpha_{m}^{j}\}$ are nonnegative, since the K\"ahler cone
lies in the effective divisor cone, see \cite{CK}. It follows that
$\M_{\ZZ}\supseteq\MM\cap(\ZZ^{\geq0})^{r+3}$.
\end{rem}
Now we use this construction to identify the complex structure moduli
we will use, for the anticanonical hypersurface $X_{\underline{a}}$
given by the Zariski closure of\[
F_{\underline{a}}(\underline{x}):=\sum_{i=0}^{r+2}a_{i}\underline{x}^{\underline{\nu}^{(i)}}=0\]
in $\PP_{\Delta}$. The coordinate patch in simplified polynomial
moduli space $\overline{\M}_{simp}$ (cf. \cite{CK}) on which it
is natural to work is just $U_{\Delta}$, with coordinates \[
t_{k}:=\underline{a}^{\underline{\ell}^{(i_{k})}}\,,\,\,\,\,\,\,\, k=1,\ldots,r.\]
In the singular case, to parametrize $U_{\Delta}$ one really needs
all $r+2$ of the $\underline{a}^{\underline{\ell}^{(i)}}=:s_{i}$
together with their relations, but the functions we consider will
be defined in terms of the $\{t_{k}\}$. Moreover, the inclusion of
$\M_{\ZZ}$ into the true Mori integral lattice (generated by the
$\{C_{i_{k}}^{\circ}\}$) defines a smooth finite cover $\AA^{r}\cong\tilde{U}_{\Delta}\to U_{\Delta}$
with coordinates $\{\tilde{t}_{k}\}$ satisfying $(\tilde{t}_{k})^{\mu_{k}}=t_{k}$,
for $\mu_{k}:=\frac{|\ell_{0}^{(i_{k})}|}{d_{i_{k}}}=\frac{e_{i_{k}}e_{i_{k}-1}}{e_{(i_{k},i_{k}-1)}}$.
This is where we really want to work.

\subsection{Construction of the mirror map via regulator periods}

The family $Y_{\underline{a}}:=\{u^{2}+v^{2}+F_{\underline{a}}(\underline{x})=0\}\subset(\CC^{*})^{2}\times\CC^{2}$
treated (in greater generality) above, with holomorphic form $\eta_{\underline{a}}$,
is considered to be the mirror of $\mathsf{K}_{\PP_{\Delta^{\circ}}}$.
This is in part because its periods satisfy the relevant GKZ equations
$\mathcal{D}_{k}(\cdot)=0$.%
\footnote{For a more thorough conceptual treatment of local mirror symmetry,
the reader is encouraged to consult \cite{CKYZ}, \cite{dOFS}, \cite{Ho}.%
} The $\mathcal{D}_{k}$ are essentially the push-forwards, under the
map $(\CC^{*})^{r+3}\to(\CC^{*})^{r}$ given by $\underline{a}\mapsto\underline{t}$,
of\[
\tilde{\mathcal{D}}_{k}=\prod_{\{j\,|\,\ell_{j}^{(i_{k})}>0\}}\partial_{a_{j}}^{|\ell_{j}^{(i_{k})}|}-\prod_{\{j\,|\,\ell_{j}^{(i_{k})}<0\}}\partial_{a_{j}}^{|\ell_{j}^{(i_{k})}|}.\]
 In view of Proposition 3.5, we will work instead with regulator periods
on $X_{\underline{a}}^{*}$ to construct the (inverse of the) mirror
map. This will be a map from complex structure parameters $\underline{\tilde{t}}$
to complexified K\"ahler parameters \begin{equation} \xymatrix{ {} \tilde{U}_{\Delta} \; \supset \; \tilde{\mathcal{P}}  \ar @{~>} [rr] & & \left\{ \ZZ \left< \left\{J_k^{\circ}\right\}_{k=1}^r \right> \subset H^{1,1}(Y^{\circ},\QQ) \right\}  \otimes_{\ZZ} (\CC/\ZZ) , } \end{equation}
where $\tilde{\mathcal{{P}}}\to\mathcal{{P}}\to D_{\varepsilon}^{*}(0)^{\times r}$
are small punctured polycylinders centered at $\underline{0}$ in
$\tilde{U}_{\Delta}\to U_{\Delta}\to\AA^{r}$.

We will follow the method of $\S\S2.1-2$ for computing these periods,
taking $\underline{\nu}:=\underline{\nu}^{(j)}$ and $z_{1}:=\underline{x}^{e_{j}^{-1}\underline{\nu}^{(j,j+1)}}$
(see beginning of $\S2.1$). The local affine equation of $\tilde{X}_{\underline{a}}$
is then given by\[
(f_{\underline{a}}(\underline{z})+a_{0})z_{1}z_{2}\,=\, a_{j}+a_{j+1}z_{1}^{e_{j}}+\phi_{2}(z_{1},z_{2})+a_{0}z_{1}z_{2}\,=\,0,\]
where $\phi_{2}(z_{1},0)=0$. Assuming $0<|a_{i}|\ll|a_{0}|$ ($\forall i$)
{[}hence $0<|t_{k}|\ll1$ ($\forall k$)], consider the family of
cycles\[
\hat{\varphi}_{0}^{(j)}:=\{|z_{1}|=\epsilon\,,\,|z_{2}|\leq\epsilon\}\cap\tilde{X}_{\underline{a}}\subset X_{\underline{a}}^{*}.\]
This may be thought of as a vanishing cycle being pinched to the {}``point
at vertex $\underline{\nu}^{(j)}$'' as $a_{j}\to0$.

As in $\S2.2$ we set (working integrally)\[
\xi_{\underline{a}}:=\{x_{1},x_{2}\}\equiv\{(-1)^{\sigma_{j}}z_{1},(-1)^{\sigma_{j-1}}z_{2}\}\in CH^{2}(X_{\underline{a}}^{*},2)\]
where $\sigma_{j}:=\left|\frac{\nu_{1}^{(j,j+1)}\nu_{2}^{(j,j+1)}}{e_{j}^{2}}\right|$
gives essentially the sign from Remark $1.13$.

In $CH^{3}((\CC^{*})^{2}\m X_{\underline{a}}^{*},3)$ we define \[
\hat{\xi}_{\underline{a}}:=\left\{ a_{0}+f_{\underline{a}}(\underline{z}),\,(-1)^{\sigma_{j}}z_{1},\,(-1)^{\sigma_{j-1}}z_{2}\right\} \]
\[
\equiv\left\{ (-1)^{\sigma_{j}+\sigma_{j-1}}(a_{j}+a_{j+1}z_{1}^{e_{j}}+\mathcal{{O}}(z_{2})),\,(-1)^{\sigma_{j}}z_{1},\,(-1)^{\sigma_{j-1}}z_{2}\right\} .\]
This has residue $\xi_{\underline{a}}$ along $\tilde{X}_{\underline{a}}$,
so that $\frac{1}{2\pi i}AJ(\xi_{\underline{a}})(\hat{\varphi}_{0}^{(j)})=$\[
\frac{1}{(2\pi i)^{2}}AJ(\hat{\xi}_{\underline{a}})(|z_{1}|=|z_{2}|=\epsilon)-\frac{1}{2\pi i}AJ(Res_{\{z_{2}=0\}}^{1}\hat{\xi}_{\underline{a}})(|z_{1}|=\epsilon)\,=\]
\[
\int_{_{|z_{1}|=|z_{2}|=\epsilon}}\log(a_{0}+f_{\underline{a}}(\underline{z}))\frac{\dlog(z_{1})}{2\pi i}\wedge\frac{\dlog(z_{2})}{2\pi i}\mspace{100mu}\]
\[
\mspace{100mu}-\,\int_{_{|z_{1}|=\epsilon}}\log((-1)^{\sigma_{j}+\sigma_{j-1}}(a_{j}+a_{j+1}z_{1}^{e_{j}}))\frac{\dlog(z_{1})}{2\pi i}\,=\]
\[
\log(a_{0})-\sum_{k\geq1}\frac{1}{k}\left[\left(-\frac{1}{a_{0}}f_{\underline{a}}(\underline{z})\right)^{k}\right]_{0}-\log((-1)^{\sigma_{j}+\sigma_{j-1}}a_{j})\,=\]
\begin{equation} -\log \left( (-1)^{\sigma_j + \sigma_{j-1}} \frac{a_j}{a_0} \right) - H(\underline{a}). \end{equation} 
Here $[\cdot]_{0}$ takes the terms constant in $z_{1},z_{2}$. Now
in the smooth case (essentially following pp. 160-1 \cite{CK})\[
H(\underline{a})=\sum_{m\geq1}\frac{1}{m}\sum_{\ell_{1},\ldots,\ell_{r+2}}\frac{(\sum\ell_{j})!}{\prod(\ell_{j}!)}\cdot\frac{\prod a_{i}^{\ell_{i}}}{(-a_{i})^{\sum\ell_{i}}}\]
\[
=\sum_{m\geq1}\frac{1}{m}\sum_{n_{1},\ldots,n_{r}}\frac{(\sum n_{k}|\ell_{0}^{(i_{k})}|)!}{\prod_{j}(\sum n_{k}\ell_{j}^{(i_{k})})!}\cdot\prod_{k}((-1)^{\ell_{0}^{(i_{k})}}t_{k})^{n_{k}}.\]
The first big $\sum$ is over non-negative integers $\{\ell_{j}\}$
satisfying $\sum\ell_{j}=m$, $\sum\ell_{j}\underline{\nu}^{(j)}=0$;
the second is over integers $\{n_{k}\}$ with $\sum n_{k}\underline{\ell}^{(i_{k})}\in\ZZ\times(\ZZ^{\geq0})^{r+2}$
and $\sum n_{k}|\ell_{0}^{(i_{k})}|=m$. By Remark 3.10 we can take
these $n_{k}\geq0$, and so $H$ is holomorphic (and well-defined)
in a neighborhood of $\underline{0}$ in $U_{\Delta}$. In the singular
case, we replace $\sum_{n_{1},\ldots,n_{r}}$ by a sum over $\MM\cap(\ZZ^{\geq0})^{r+3}$
(which involves non-redundant choices of $\{n_{i}\}_{i=1}^{r+2}$)
and use all the $\ell^{(i)}$ and $s_{i}$ (not just the $\ell^{(i_{k})}$
and $t_{k}$). The resulting $H$ is defined on $U_{\Delta}$ and
pulls back to a holomorphic function on $\tilde{U}_{\Delta}$. Henceforth
it will be written $H(\underline{s})$.

Clearly the {}``$\log$''-term of (3.4) makes no sense on $U_{\Delta}$
or even $\tilde{U}_{\Delta}$; this reflects the fact that $\xi_{\underline{a}}$
is not invariant under the action of the torus $(\CC^{*})^{2}$. But
the periods of $R\{x_{1},x_{2}\}$ over cycles in\[
\K(X_{\underline{a}}^{*}):=\ker\{H_{1}(X_{\underline{a}}^{*},\ZZ)\to H_{1}((\CC^{*})^{2},\ZZ)\}\]
$are$ torus-invariant, and $r$ \textbf{distinguished vanishing cycles}
in $\K(X_{\underline{a}}^{*})$ are given by\[
\varphi_{0}^{[k]}:=-\sum_{j=1}^{r+2}\ell_{j}^{(i_{k})}\hat{\varphi}_{0}^{(j)}\,,\,\,\,\,\,\,\,\, k=1,\ldots,r\,.\]
The map $H_{1}(X_{\underline{a}}^{*})\rOnto H_{1}(\tilde{X}_{\underline{a}})$
induced by inclusion sends $\varphi_{0}^{[k]}$ to $\ell_{0}^{(i_{k})}$
times a primitive vanishing cycle $\tilde{\varphi}_{0}$. If $\varphi_{1}\in\K(X_{\underline{a}}^{*})$
is a lift of a complimentary generator $-\tilde{\varphi}_{1}$, then
$AJ(\xi_{\underline{a}})(\varphi_{1})$ and the $AJ(\xi_{\underline{a}})(\varphi_{0}^{[k]})$
form a $\QQ$-basis for the periods (modulo $\QQ(2)$) of $AJ(\xi_{\underline{a}})=[R\{x,y\}]$
over cycles in $\K(X_{\underline{a}}^{*})$. One should view the $\varphi_{0}^{[k]}$
as differing by loops around points of $D\subset\tilde{X}_{\underline{a}}$,
hence the $AJ(\xi_{\underline{a}})(\varphi_{0}^{[k]})$ as differing
by residues. 

Now we slightly change our notation to bring it in line with \cite{Ho}.
Write (multivalued) functions of $\underline{t}$\[
\tilde{w}^{(0)}:=(2\pi i)^{3}=\int_{\hat{\TT}_{Y}^{3}}\eta_{\underline{a}}\,,\]
\[
\tilde{w}_{k}^{(1)}:=2\pi i\, AJ(\xi_{\underline{a}})(\varphi_{0}^{[k]})=\int_{M(\varphi_{0}^{[k]})}\eta_{\underline{a}}\,,\]
\[
\tilde{w}^{(2)}:=2\pi i\, AJ(\xi_{\underline{a}})(\varphi_{1})=\int_{M(\varphi_{1})}\eta_{\underline{a}}\,,\]
and normalize these by setting $w_{\cdot}^{(\cdot)}:=\tilde{w}_{\cdot}^{(\cdot)}/\tilde{w}^{(0)}.$

\begin{thm}
The $w_{k}^{(1)}$ are well-defined $\CC/\ZZ$-valued functions on
$\mathcal{{P}}$, given by $\frac{1}{2\pi i}$ times\[
\log\left((-1)^{\epsilon_{k}}t_{k}\right)+|\ell_{0}^{(i_{k})}|H(\underline{s}),\]
where $\epsilon_{k}:=\sum_{j=1}^{r+2}(\sigma_{j}+\sigma_{j-1})\ell_{j}^{(i_{k})}$.
\end{thm}
\begin{defn}
The (inverse) mirror map (3.3) is given by\[
(\tilde{t}_{1},\ldots,\tilde{t}_{r})\longmapsto\sum_{k=1}^{r}J_{k}^{\circ}\otimes W_{k}^{(1)}(\underline{\tilde{t}}),\]
where $W_{k}^{(1)}(\underline{\tilde{t}}):=\frac{1}{\mu_{k}}w_{k}^{(1)}(\underline{s}(\underline{\tilde{t}}))$.
\end{defn}
\begin{rem}
(i) \cite{Ho} considers the (conjectural!) map \[
mir:\, K^{c}(Y^{\circ})\to H_{3}(Y,\ZZ)\]
arising from Kontsevich's homological mirror symmetry conjecture,
and proposes that one should have $\hat{\TT}_{Y}^{3}=mir(\mathcal{{O}}_{pt.})$,
$\frac{1}{\mu_{k}}M(\varphi_{0}^{[k]})=mir(\mathcal{{O}}_{C_{i_{k}}^{\circ}}(-J_{k}^{\circ}))$,
$M(\varphi_{1})=mir(\mathcal{{O}}_{D_{0}})$.

(ii) Set $\delta_{T}:=\sum_{j=1}^{r}|\ell_{0}^{(i_{j})}|\delta_{t_{j}}$.
The $W_{k}^{(1)}$ are logarithmic integrals of periods of $\omega_{\underline{a}}:=Res\left(\frac{\frac{dx_{1}}{x_{1}}\wedge\frac{dx_{2}}{x_{2}}}{F_{\underline{a}}(x_{1},x_{2})}\right)$
in the (limited) sense that \[
\delta_{T}W_{k}^{(1)}=\frac{d_{i_{k}}}{(2\pi i)^{2}}\int_{\tilde{\varphi}_{0}}\omega_{\underline{a}}\]
 for each $k$. We also write (after \cite{CKYZ}) $\partial_{S}:=\sum_{k=1}^{r}d_{i_{k}}\partial_{W_{k}^{(1)}}$.
\end{rem}

\subsection{Growth of local Gromov-Witten invariants}

Define (on $\tilde{\mathcal{{P}}}$) the Gromov-Witten prepotential\[
\F_{loc}(\underline{W}^{(1)})\,:=\,\frac{1}{2}\sum_{j,\ell}\left\langle J_{j}^{\circ}|_{\PP_{\Delta^{\circ}}},J_{\ell}^{\circ}\right\rangle W_{j}^{(1)}W_{\ell}^{(1)}\,\,\,\,+\,\,\,\{\begin{array}{c}
_{\text{{lower-order}}}\\
^{\text{{terms}}}\end{array}\}(\underline{W}^{(1)})\]
\[
\mspace{100mu}-\,\sum_{k_{1},\ldots,k_{r}}\left(\sum_{j=1}^{r}d_{i_{j}}k_{j}\right)N_{k_{1},\ldots,k_{r}}Q_{1}^{k_{1}}\cdots Q_{r}^{k_{r}},\]
where $Q_{j}:=\exp(2\pi i\, W_{j}^{(1)})$ and $N_{\underline{k}}$
is the genus zero (local) G-W invariant {}``counting rational curves
in {[}the total space of] $\mathsf{K}_{\PP_{\Delta^{\circ}}}$''
of homology class $\sum k_{j}[C_{i_{j}}^{\circ}]\in H_{2}(Y^{\circ},\ZZ)$.
(See \cite{Li} $\S6.1$ for a precise definition.) \cite{CKYZ} originally
obtained (essentially) this expression by writing a compact CY $3$-fold
$\mathfrak{X}$ (with prepotential $\F$) as a torically described
elliptic fibration over $\PP_{\Delta^{\circ}}$, and taking the limit
of {[}a suitable partial of] $\F$ under degeneration of the fiber.
Morally, the resulting (local) $N_{\underline{k}}$ were supposed
to measure the contribution of the zero-section $\PP_{\Delta^{\circ}}$
to G-W invariants of $\mathfrak{X}$.

Here then is the fundamental local mirror symmetry prediction:

\begin{conjecture}
\emph{(\cite{CKYZ},\cite{Ho})} For a suitable choice of $\varphi_{1}$,
\begin{equation} \F_{loc}(\underline{W}^{(1)}) = w^{(2)}(\underline{\tilde{t}}) \end{equation} 
under the mirror map.
\end{conjecture}
To summarize: the first regulator period yields the mirror map; the
second gives the prepotential.

We will now pull (3.5) back to a {}``diagonal slice'' of $\tilde{\mathcal{P}}$
where residual differences between the $w_{k}^{(1)}$ vanish. Write
\begin{equation} \phi:=\sum_{j=1}^{r+2} \alpha_j \underline{x}^{\underline{\nu}^{(j)}} \; , \; \; F_{\phi,t}(\underline{x}):=1-t\phi(\underline{x}) ; \end{equation} 
this gives $a_{0}=1$, $a_{j}=t\alpha_{j}$,\[
t_{k}=(-1)^{\ell_{0}^{(i_{k})}}\left(\prod_{j=1}^{r+2}\alpha_{j}^{\ell_{j}^{(i_{k})}}\right)t^{|\ell_{0}^{(i_{k})}|}.\]
If we further set \begin{equation} \alpha_j := (-1)^{\sigma_j + \sigma_{j-1} + 1} , \end{equation} 
then $t_{k}(t)=(-1)^{\epsilon_{k}}t^{|\ell_{0}^{(i_{k})}|}$, and
the {}``slice'' is given by $\tilde{t}_{k}(t):=\zeta_{k}t^{d_{i_{k}}}$
($\zeta_{k}=$ some root of unity with $\mu_{k}^{\text{{\,\, th}}}$
power $(-1)^{\epsilon_{k}}$; the choice won't affect calculations).
The pullback of $W_{k}^{(1)}(\underline{\tilde{t}})$ under $t\mapsto\underline{\tilde{t}}(t)$
is then simply\[
W_{k}^{(1)}(t)=\frac{d_{i_{k}}}{2\pi i}\left\{ \log t\,+\, H(t)\right\} =:d_{i_{k}}w^{(1)}(t),\]
where $H(t)\,(:=H(\underline{s}(t)))$ can frequently be easier to
determine than $H(\underline{s})$.

So the map of families $\{F_{\phi,t}(\underline{x})=0\}\to\{F_{\underline{a}}(\underline{x})=0\}/(\CC^{*})^{3}$
induces a {}``diagonal'' embedding $\mathfrak{D}:\, w^{(1)}\mapsto(d_{i_{1}}w^{(1)},\ldots,d_{i_{r}}w^{(1)})$
of K\"ahler moduli. Clearly $\mathfrak{D}^{*}\circ\partial_{S}=\partial_{w^{(1)}}\circ\mathfrak{D}^{*}$,
and by (3.5)\[
\mathfrak{D}^{*}\F_{loc}(\underline{W}^{(1)})=w^{(2)}(t(w^{(1)}));\]
it follows that \begin{equation} \mathfrak{D}^* \partial_S^2 \F_{loc}(\underline{W}^{(1)}) = \left( \frac{d}{dw^{(1)}} \right)^2 w^{(2)}(t(w^{(1)})). \end{equation} 

For the l.h.s. of (3.8),\[
\partial_{S}^{2}\F_{loc}\,=\,\sum_{j,\ell}d_{i_{j}}d_{i_{\ell}}\left\langle J_{j}^{\circ}|_{\PP_{\Delta^{\circ}}},J_{\ell}^{\circ}\right\rangle _{Y^{\circ}}-(2\pi i)^{2}\sum_{k_{1},\ldots,k_{r}}\left(\sum_{j=1}^{r}d_{i_{j}}k_{j}\right)^{3}N_{k_{1},\ldots,k_{r}}Q_{1}^{k_{1}}\cdots Q_{r}^{k_{r}}\]
\[
=\,\left\langle -\mathsf{K}_{\PP_{\Delta^{\circ}}},-\mathsf{K}_{\PP_{\Delta^{\circ}}}\right\rangle _{\PP_{\Delta^{\circ}}}-(2\pi i)^{2}\sum_{D\geq1}D^{3}\sum_{\{\underline{k}\,|\,\sum d_{i_{j}}k_{j}=D\}}N_{\underline{k}}\underline{Q}^{\underline{k}}.\]
Thinking of $\underline{k}$ as the homology class $\sum k_{j}[C_{i_{j}}^{\circ}]\in H_{2}(Y^{\circ})=H_{2}(\PP_{\Delta^{\circ}})$,
we have $\left\langle \underline{k},X^{\circ}\right\rangle _{\PP_{\Delta^{\circ}}}=\sum k_{j}d_{i_{j}}$;
hence applying $\mathfrak{D}^{*}$ yields\[
\sum_{i=1}^{r+2}d_{i}\,\,-\,\,(2\pi i)^{2}\sum_{D\geq1}D^{3}\left(\sum_{\{\underline{k}\,|\,\left\langle \underline{k},X^{\circ}\right\rangle =D\}}N_{\underline{k}}\right)Q^{D}\]
where $Q=\exp(2\pi i\, w^{(1)})$. Note that the constant term just
records the number of components $\N_{0}$ of the singular fiber of
the diagonal family at $t=0$ (after a minimal desingularization of
the total space). We also rechristen the sum in parentheses $N_{D}^{\left\langle X^{\circ}\right\rangle }$.
It would be very interesting to have an interpretation of these numbers
in terms of $X^{\circ}$ alone,%
\footnote{To venture out on a limb, can one suitably define a class in $K_{2}$
of (the nerve of) the Fukaya category (of $X^{\circ}$), which completes
$X^{\circ}$ to a datum {}``mirror'' to the family $\{X_{t}\}$
together with $\{\xi_{t}\in K_{2}(X_{t})\}$? Is there then a {}``regulator''
of this class which pairs with $\mathcal{O}_{D_{0}}|_{X^{\circ}}$
(recall $M(\varphi_{1})$'s conjectural mirror is $\mathcal{O}_{D_{0}}$)
to yield the prepotential $\mathcal{F}_{loc}$?%
} since the mirror map is defined only in terms of $X$ (not $Y$).

For the r.h.s. of (3.8), write $\pi^{(1)}$ and $\pi^{(2)}$ for the
$\frac{\text{{periods}}}{(2\pi i)^{2}}$ of $\omega_{t}:=Res_{X_{\phi,t}}\left(\frac{\bigwedge\dlog\underline{x}}{F_{\phi,t}}\right)$;
then $\delta_{t}w^{(\ell)}(t)=\pi^{(\ell)}(t)$ ($\ell=1,2$). So
we have \[
\frac{d}{dw^{(1)}}w^{(2)}=\frac{\delta_{t}w^{(2)}}{\delta_{t}w^{(1)}}=\frac{\pi^{(2)}}{\pi^{(1)}},\]
and applying one more $\frac{d}{dw^{(1)}}$ yields\[
\frac{\delta_{t}\left(\frac{\pi^{(2)}}{\pi^{(1)}}\right)}{\delta_{t}w^{(1)}}=\frac{\pi^{(1)}\delta_{t}\pi^{(2)}-\pi^{(2)}\delta_{t}\pi^{(1)}}{(\pi^{(1)})^{3}}.\]
Writing this in terms of functions from $\S\S2.1,\,2.3$ for the diagonal
family $\tilde{X}_{\phi,t}$ (and dividing l.h.s. and r.h.s. by $(2\pi i)^{2}$),
we have the following equality of a G-W generating function and Yukawa
coupling \begin{equation} \frac{\N_0}{(2\pi i)^2} - \sum_{D\geq 1} D^3 N_D^{\left< X^{\circ} \right>} Q^D = \frac{\Y(t)}{(A(t))^3} \end{equation} 
under the mirror map. The latter is just the local analytic isomorphism
$t\mapsto Q(t)$ {[}$Q(0)=0$], extending at least to $\overline{D_{|t_{0}|}}$.
(Recall $\tilde{X}_{\phi,t_{0}}$ is the singular fiber nearest $t=0$
in the punctured diagonal family.) The r.h.s. of (3.9) blows up at
$t_{0}$ since $\Y(t)\sim\frac{1}{t-t_{0}}$ and $A(t)\sim\log(t-t_{0})$
(up to constants) for $t\ \to t_{0}$. Hence the l.h.s. series has
radius of convergence%
\footnote{If there is more than one $t_{0}$ of minimal modulus, one should
of course pick the one that minimizes $|Q(t_{0})|$; but in every
case we have tested, symmetry ensures that this is independent of
the choice.%
} $|Q(t_{0})|=\exp\{\Re(2\pi i\, w^{(1)}(t_{0}))\}=\exp\{\frac{1}{2\pi}\Im(\Psi(t_{0}))\}$
where $\Psi(t)=(2\pi i)^{2}w^{(1)}(t).$

\begin{thm}
Let $\Delta$ be a reflexive polytope $\subseteq\RR^{2}$ such that
the Mori cone of $Y^{\circ}:=\mathsf{K}_{\PP_{\Delta^{\circ}}}$ is
smooth, determine $\phi(\underline{x})$ by (3.6), (3.7), and let
$\Psi(t)$ and $|t_{0}|$ be as in Corollary 2.6. Assume Conjecture
3.14. Then the local Gromov-Witten invariants of $Y^{\circ}$ have
exponential growth-rate \begin{equation} \limsup_{D\to \infty} |N_D^{\left< X^{\circ} \right> }|^{\frac{1}{D}} = e^{\frac{-1}{2\pi} \Im (\Psi (t_0))} . \end{equation} 
\end{thm}
\begin{rem}
(i) In $\S4$ we will describe a procedure for computing the {}``regulator
period'' $\Psi(t_{0})$ on a singular elliptic fiber of Kodaira type
$I_{n}$. This identifies with the image of an indecomposable $K_{3}$
class under the composition\[
K_{3}^{ind}(\bar{\QQ})\cong H_{\M,hom}^{2}(\tilde{X}_{t_{0}}/\bar{\QQ}\,,\,\QQ(2))\rTo^{AJ^{2,2}}H^{1}(\tilde{X}_{t_{0}},\CC/\QQ(2))\cong\CC/\QQ(2),\]
which (after taking the imaginary part) coincides (up to a factor
of 2) with the Borel regulator. This explains the occurrence of Dirichlet
$L$-functions in results of \cite{MOY} related to (3.10). (We will
be more precise about the field of definition in $\S4$.)

(ii) Equation (3.9) gives, for $t=0$, the correct value $\Y(0)=2\pi i\,\N_{0}$.
\end{rem}
Finally, we want to explain how {}``reasonable'' assumptions on
the $\{N_{D}^{\left\langle X^{\circ}\right\rangle }\}$ lead to a
more precise characterization of their growth. (The argument is similar
to that in \cite{CdOGP} but more rigorous.) Let $d:=\gcd\{d_{i}\,|\, i=1,\ldots,r+2\}$,
put $\tilde{\Psi}(t)=d\cdot\{\Psi(t)-\Re(\Psi(t_{0}))\}$, and define
{}``normalized'' quantities\[
\tilde{N}_{D}:=-d^{3}N_{d\cdot D}^{\left\langle X^{\circ}\right\rangle }e^{-i\frac{d\cdot D}{2\pi}\Re(\Psi(t_{0}))}\,\,,\,\,\,\,\,\,\tilde{Q}:=\exp\{\frac{-i}{2\pi}\tilde{\Psi}(t)\}.\]
Reindexing, (3.9) becomes\[
-\frac{\N_{0}}{4\pi^{2}}+\sum_{D\geq1}D^{3}\tilde{N}_{D}\tilde{Q}^{D}\,=\,\frac{\Y(t)}{A^{3}(t)}\,;\]
and we assume$\vspace{2mm}$\\
(a) the $\tilde{N}_{D}$ are uniformly positive (or negative) for
sufficiently large $D$.$\vspace{2mm}$\\
Next define $n_{D}\,(>0)$ by \[
\tilde{N}_{D}=\pm e^{\frac{-D}{2\pi i}\tilde{\Psi}(t_{0})}D^{-3}n_{D},\]
and assume that$\vspace{2mm}$\\
(b) $\lim_{D\to\infty}n_{D}\log^{2}D$ exists (in the extended reals
$\RR^{\geq0}\cup\{\infty\}$),$\vspace{2mm}$\\
i.e. that the $\tilde{N}_{D}$ {}``do not oscillate too much'' in
the limit.

Now asymptotically as $t\to t_{0}$ (keeping $t-t_{0}\in\RR$ and
$|t|<|t_{0}|$), $\pi^{(1)}\sim-m\pi^{(2)}(t_{0})\frac{\log|t-t_{0}|}{2\pi i}$
(where $m\in\ZZ^{+}$ is essentially the number of components of $X_{t_{0}}$);
logarithmically integrating this, we have $x:=\frac{1}{2\pi i}(\tilde{\Psi}(t_{0})-\tilde{\Psi}(t))=\frac{d}{2\pi i}(\Psi(t_{0})-\Psi(t))\sim d\cdot m\cdot\pi^{(2)}(t_{0})\left(\frac{t}{t_{0}}-1\right)\log|t-t_{0}|.$
This implies the r.h.s. of $\vspace{2mm}$\\
(c) $\pm\sum_{D\geq1}n_{D}e^{-Dx}=\frac{\Y(t)}{A^{3}(t)}+\frac{\N_{0}}{4\pi^{2}}$$\vspace{2mm}$\\
is asymptotic to $\frac{d}{mx\log^{2}(t-t_{0})}\sim\frac{d}{mx\log^{2}x}$,
where we can replace $t\to t_{0}$ by $x\to0^{+}$.

We need a result from Laplace Tauberian theory.

\begin{lem}
Given a sequence $\{n_{k}\}$ of real numbers satisfying\\
\emph{(a')} $n_{k}$ positive (or at least $n_{k}\geq-\frac{C}{\log^{2}k}$
for some $C>0$)\\
\emph{(b')} $\lim_{n\to\infty}n_{k}\log^{2}k$ exists (finite or infinite)\\
\emph{(c')} $\sum_{k=0}^{\infty}n_{k}e^{-kx}\sim\frac{1}{x\log^{2}x}$
as $x\to0^{+}$.\\
(Here \emph{(a')} is the {}``Tauberian'' hypothesis.) Then $n_{k}\sim\frac{1}{\log^{2}k}$
as $k\to\infty$. That is, $n_{k}\log^{2}k\to1$.
\end{lem}
\begin{proof}
For $m_{k}:=\left\{ \begin{array}{cc}
1, & k=0\\
0, & k=1\\
\frac{1}{\log^{2}k}, & k\geq2\end{array}\right.$, it is an exercise in elementary analysis to prove $\sum_{k=0}^{\infty}m_{k}e^{-kx}\sim\frac{1}{x\log^{2}x}$
($x\to0^{+}$), e.g. in the form $\linebreak$$\lim_{y\to\infty}\sum_{k=2}^{\infty}\frac{1}{y}\left(\frac{\log^{2}y}{\log^{2}k}-1\right)e^{-\frac{k}{y}}=0$.
Now let $N(k)$, $M(k)$ be the respective $k^{th}$ partial sums
of $n_{k}$, $m_{k}$, viewed as functions on $\RR^{\geq0}$. Hypothesis
(c') obviously implies $\int_{0}^{\infty}e^{-kx}dN(k)\sim\int_{0}^{\infty}e^{-kx}dM(k)$
(for $x\to0^{+}$) and then (using (a')) \cite{Fr} gives $N(k)\sim M(k)$
for $k\to\infty$. Hypothesis (b') says $\lim_{k\to\infty}\frac{n_{k}}{m_{k}}$
exists (finite or $+\infty$), in which case it must equal $\lim_{k\to\infty}\frac{N(k)}{M(k)}$,
which is $1$.
\end{proof}
In our situation this yields $n_{D}\sim\frac{d}{m\log^{2}D}$, hence
the following result:

\begin{cor}
Under assumptions \emph{(a)} and \emph{(b) above (and the conditions
of Theorem 3.15), the {}``normalized'' G-W invariants have asymptotic
behavior\[
\tilde{N}_{D}\sim\pm\frac{d}{m}\frac{\exp\{-\frac{D\cdot\tilde{\Psi}(t_{0})}{2\pi i}\}}{D^{3}\log^{2}D}\]
for $D\to\infty$.}
\end{cor}
\begin{rem*}
It seems likely that one could use a Fourier Tauberian argument to
eliminate the assumptions.
\end{rem*}

\section{\textbf{First examples: limits of regulator periods}}

A well-traveled road in dealing with computations for 1-parameter
families of varieties is to attempt to recognize {}``modularity''
in some suitable sense. For example, this approach was employed in
\cite{Do1,Do2} to describe mirror maps and Picard-Fuchs equations
for families of CY's. Here (in $\S8$) we use it, for the families
(and higher cycles) produced by Theorem 1.7, to compute the cycle-class,
higher normal function, and regulator periods --- especially their
limiting values at cusps. The central purpose of this section, in
contrast, is to illustrate a procedure inspired by \cite{Bl2} for
computing these {}``special values'' of $\Psi(t)$ (at singular
fibers), that does not rely on modularity. This leads to a formula
(Prop. 4.4) for essentially the $\tilde{\Psi}(t_{0})$ of Thm. 3.15/Cor.
3.18, which we apply to some key examples in $\S4.3$. Throughout
this section $\TXM$ is as in Theorem 1.7 (so that $\Xi$ and $\Psi$
have the established meaning).

\subsection{$AJ$ map for singular fibers}

Fixing $\alpha\in\L^{*}$, write $\tilde{X}_{\alpha}=:Y=\cup Y_{i}$
with $Y_{i}$ irreducible,%
\footnote{We do not require that $\tilde{\pi}^{-1}(\alpha)=\sum m_{i}Y_{i}$
to be reduced, here or in the $Y=$ NCD case.%
} $\tilde{\varphi}_{\alpha}=\sum\varphi_{i}$ for $\varphi_{i}\in C_{n-1}^{\text{top}}(Y_{i})$;
and assume $\Xi\in Z_{_{\d_{\B}-\text{cl.}}}^{n}(\TXM,n)_{Y}$ so
that the $\Xi_{i}:=\Xi\cdot Y$ are defined. Our first goal is to
verify the claim from $\S2.2$ that \begin{equation}
AJ(\Xi_{\alpha})(\tilde{\varphi}_{\alpha}) = \int_{\tilde{\varphi}_{\alpha}}R_{\Xi} = \sum_i \int_{\varphi_i}R_{\Xi_i} ; \\
\end{equation} to this end we review briefly the computation of $AJ(\Xi_{\alpha})$
from $\S8$ of \cite{KL}. The (somewhat technical) general conditions
under which it (hence $(4.1)$) is valid are described in loc. cit.
following Prop. 8.17, and allow for all singular curves, as well as
any local-normal-crossing or nodal singularities.

Here we shall focus on the case $Y=$ NCD, writing $Y_{I}:=\cap_{i\in I}Y_{i}$,
$Y^{[j]}:=\amalg_{|I|=j+1}Y_{I}$, and $Y^{I}$ for the collection
$\{Y_{J}\cap Y_{I}\}_{J\cap I=\emptyset}$ of subsets of $Y_{I}$.
This {}``hyper-resolution'' of $Y$ gives rise to 4th quadrant double-complexes
\[
\left.\begin{array}{c}
Z_{Y}^{\ell,m}(n):=Z^{n}(Y^{[\ell]},-m)_{\#}\\
\,\,\,\,\,\,\,\,\,\,\,\,\,\,\,\,\,\,\,\,\,\,\,\,\,\,\,\,\,\,\,\,:=\oplus_{_{|I|=\ell+1}}Z_{\RR}^{n}(Y_{I},-m)_{_{Y^{I}}}\\
\d_{\B}:\, Z_{Y}^{\ell,m}(n)\to Z_{Y}^{\ell,m+1}(n)\\
\mathfrak{I}:\, Z_{Y}^{\ell,m}(n)\to Z_{Y}^{\ell+1,m}(n)\end{array}\,\,\,\,\right|\,\,\,\,\begin{array}{c}
C_{\ell,m}^{Y}(n):=C_{2n+m-1}^{\text{top}}(Y^{[\ell]};\QQ)\\
\,\,\,\,\,\,\,\,\,\text{(piecewise }C^{\infty}\text{ chains)}\\
\d_{\text{top}}:C_{\ell,m}^{Y}(n)\to C_{\ell,m-1}^{Y}(n)\\
Gy:\, C_{\ell,m}^{Y}(n)\to C_{\ell-1,m}^{Y}(n)\end{array}\]
where $\mathfrak{I}$ (resp. $Gy$) is the alternating sum (cf. loc.
cit. for signs) of pullbacks (resp. pushforwards). These have associated
simple complexes/total differentials/(co)homology\[
\left.\begin{array}{c}
Z_{Y}^{\bb}(n):=\mathbf{s}^{\bb}Z_{Y}^{\bb,\bb}(n)\\
\udb:=\db\pm\mathfrak{I}\\
H^{*}(Z_{Y}^{\bb}(n))\cong H_{\M}^{2n+*}(Y,\QQ(n))\end{array}\,\,\,\,\right|\,\,\,\,\begin{array}{c}
C_{\bb}^{Y}(n):=\mathbf{s}_{\bb}C_{\bb,\bb}^{Y}(n)\\
\udt:=\d_{\text{top}}\pm Gy\\
H_{*}(C_{\bb}^{Y}(n))\cong H_{2n+*-1}(Y)\end{array}.\]
 The KLM currents ($\mathfrak{Z}\,\,\mapsto\,\, T_{\mathfrak{Z}},\Omega_{\mathfrak{Z}},R_{\mathfrak{Z}}$)
give a map of complexes (described in full in loc. cit.) inducing
an Abel-Jacobi map from $H_{\M}^{2n+*}(Y,\QQ(n))$ to\[
H_{\D}^{2n+*}(Y,\QQ(n))\aeq\ext(\QQ(0),H^{2n+*-1}(Y,\QQ(n)))\beq H^{2n+*-1}(Y,\CC/\QQ(n)).\]
For $*=-n$ in particular, this is \begin{equation}
AJ^{n,n}_Y: H^n_{\M}(Y,\QQ(n))\to Hom(H_{n-1}(Y,\QQ),\CC/\QQ(n)). \\
\end{equation} To compute this for $\dim(Y)=n-1$, let \begin{equation}
\begin{matrix}
\mathfrak{Z} = \sum_{\ell} \{\mathfrak{Z}^{[\ell]}\in Z^n_{\RR}(Y^{[\ell]},n+\ell) \} \in \{ \ker(\udb)\subset Z^{-n}_Y(n) \} \\
\gamma = \sum_{\ell} \{ \gamma^{[\ell]}\in C^{\text{top}}_{n-\ell-1}(Y^{[\ell]};\QQ) \} \in \{ \ker(\udt) \subset C_{-n}^Y(n) \} , 
\end{matrix} \\
\end{equation} with each $\gamma^{[\ell]}$ (resp. $\mathfrak{Z}^{[\ell]}$) decomposing
into $\{\gamma_{I}\}_{|I|=\ell+1}$ (resp. $\{\mathfrak{Z}_{I}\}_{|I|=\ell+1}$).
Then \begin{equation}
AJ^{n,n}_Y(\mathfrak{Z})(\gamma) \equiv \sum_{\ell\geq 0} \int_{\gamma^{[\ell]}} R_{\mathfrak{Z}^{[\ell]}} = \sum_{\ell\geq 0} \sum_{|I|=\ell+1}\int_{\gamma_I}R_{\mathfrak{Z}_I} \\
\end{equation} gives a well-defined pairing $H^{-n}(Z_{Y}^{\bb}(n))\times H_{-n}(C_{\bb}^{Y}(n))\to\CC/\QQ(n)$.
Now consider the map\[
I_{Y}^{*}:\, Z_{_{\RR,\db-\text{cl.}}}^{n}(\TXM,n)_{Y}\to\{\ker(\udb)\subset Z_{Y}^{-n}(n)\}\]
given by restricting to the irreducible components of $Y$. That is,
if $\mathfrak{Z}=I_{Y}^{*}\Xi$ then $\mathfrak{Z}^{[0]}$ is the
collection $\{\iota_{Y_{i}}^{*}\Xi\}$ while $\mathfrak{Z}^{[\ell]}=0$
for $\ell>0$. Let $\gamma$ be the $\udt$-cycle corresponding to
$\tilde{\varphi}_{\alpha}$: i.e. $\gamma^{[0]}=\{\varphi_{i}\}$,
while the $\gamma^{[\ell]}(\neq0)$ comprise iterated boundaries of
the $\varphi_{i}$. Then\[
AJ(\Xi_{\alpha})(\tilde{\varphi}_{\alpha})=AJ(\mathfrak{Z})(\gamma)\eqe\sum_{i}\int_{\varphi_{i}}R_{\iota_{Y_{i}}^{*}\Xi=\Xi_{i}}\]
confirms $(4.1)$.

Continuing to assume $Y$ a (connected) NCD of dimension $n-1$, we
want to say something about the $value$ of $(4.1)$ in $\CC/\QQ(n)$.
Place the {}``weight'' filtration\[
W_{\beta}H_{\M}^{2n+*}(Y,\QQ(n)):=\text{im}\{H^{*}(\mathbf{s}^{\bb}Z_{Y}^{(\bb\geq-n-\beta),\bb}(n))\to H^{*}(Z_{Y}^{\bb}(n))\}\]
on motivic cohomology, and note that $W_{-2n+1}H_{\M}^{n}(Y,\QQ(n))$
consists of those classes representable by $\udb$-cocycles supported
on points $p_{I}:=Y_{I}$, $|I|=n$. (For simplicity we assume these
are each one point.) This is compatible with the weight filtration
on the generalized Jacobians in the sense that $AJ_{Y}^{n,r}$ is
{}``filtered'' by maps\[
W_{\bb}H_{\M}^{2n-r}(Y,\QQ(n))\eeq\ext(\QQ(0),W_{\bb-1}H^{2n-r-1}(Y,\QQ(n))).\]
In particular the target of $W_{-2n+1}AJ_{Y}^{n,n}$ is $\ext(\QQ(0),\QQ(n)^{\oplus b_{Y}})\cong(\CC/\QQ(n))^{\oplus b_{Y}}$,
where $b_{Y}:=\text{rk}\{\text{coker}(H^{0}(Y^{[n-2]})\to H^{0}(Y^{[n-1]}))\}$.
For $Y=\tilde{X}_{\alpha}$ a degenerate CY, $b_{Y}=0$ or $1$: $b_{Y}=1$
implies maximal quasi-unipotent monodromy about $\alpha$; and in
the unipotent case, maximal monodromy $\implies$ $b_{Y}=1$.

We need to be more precise about the field of definition: recall that
$\TXM$ is defined over a number field $K$; it may be that $\alpha\notin K$,
and that to {}``separate components'' of $Y$ requires an algebraic
field extension larger than $K(\alpha)$.

\begin{defn}
$\mathsf{L}/K(\alpha)$ is a \emph{splitting field} for the NCD $Y$
iff all the components $Y_{I}$ of the hyper-resolution are defined
over $\mathsf{L}$. Furthermore, $Y$ is \emph{simple} iff all $Y_{I}$
are rational.
\end{defn}
With such a choice of $\mathsf{L}$, and assuming $b_{Y}=1$, we have
\begin{equation} 
W_{-2n+1}H^n_{\M}(Y/\mathsf{L},\QQ(n))\cong CH^n(Spec(\mathsf{L}),2n-1)\cong K^{\text{alg}}_{2n-1}(\mathsf{L})\otimes \QQ . \\
\end{equation} Let $\gamma,\mathfrak{Z}$ be as in $(4.3)$ with $\gamma^{[n-1]}=\{q_{I}[p_{I}]\}_{|I|=n}$
($q_{I}\in\QQ$) and $[\mathfrak{Z}]\in(4.5)$. Then $\mathfrak{Z}\equiv\{\mathfrak{W}_{I}\}_{|I|=n}$
modulo $\udb$-coboundary, and \begin{equation}
AJ^{n,n}_Y(\sZ)(\gamma)=AJ^{2n-1,n}_{Spec(\mathsf{L})}(\sum \pm q_I \sW_I) \in \CC/\QQ(n), \\
\end{equation} where in light of $(4.5)$ $AJ_{Spec(\mathsf{L})}^{2n-1,n}$ should
be thought of essentially as the Borel regulator. The key result,
which the computations below will reflect (but not use), is

\begin{prop}
Let $n=2$ or $3$, $Y=\tilde{X}_{\alpha}$ be a simple NCD with abelian
splitting field extension $\mathsf{L}/\QQ$, and if $n=3$ assume
$\mathsf{L}$ totally real. Then $H_{\M}^{n}(Y/\mathsf{L},\QQ(n))=W_{-2n+1}H_{\M}^{n}(Y/\mathsf{L},\QQ(n))$,
and $\Psi(\alpha)$ is a sum of Dirichlet $L$-series $L(\chi,n)$
with algebraic coefficients.
\end{prop}
\begin{rem*}
For $\mathsf{L}$ non-abelian one might hope to relate the collection
of values of $\Psi$ at (some) points of $\L^{*}$ to Artin $L$-series
corresponding to a representation of $Gal(\mathsf{L}/\QQ)$.
\end{rem*}
\begin{proof}
In order to {}``move'' an arbitrary $\udb$-cocycle (in $Z_{Y}^{-n}(n)$)
into $Z_{Y}^{n-1,-2n+1}(n)$, we need only know that (for $n=2$)
$CH^{2}(Y_{i},2)=\{0\}$ ($\forall i$) and (for $n=3$) $CH^{3}(Y_{i},3)$
and $CH^{3}(Y_{ij},4)$ are $0$ ($\forall i,j$). This follows from
vanishing of $CH^{p}(\PP_{\mathsf{L}}^{1},n)\cong_{n.c.}CH^{p}(\mathsf{L},n)\oplus CH^{p-1}(\mathsf{L},n)$
and (for $S:=Bl_{\{p_{1},\ldots,p_{N}\}}(\PP^{2})$) \[
CH^{p}(S_{\mathsf{L}},n)\cong CH^{p}(\mathsf{L},n)\oplus CH^{p-1}(\mathsf{L},n)^{\oplus(N+1)}\oplus CH^{p-2}(\mathsf{L},n).\]
Now since $\Xi$ is (like $\mathcal{X}$) defined over $K$, its pullback
to {[}the components of] $Y$ is defined over $\mathsf{L}$. The last
statement (of the Prop.) then follows from Beilinson's fundamental
result \cite{B1,Ne1} on higher regulators of a cyclotomic field ($\supset\mathsf{L}$),
together with $(4.5)$ and $(4.6)$.
\end{proof}
For actually computing $(4.1)$ we shall take a different approach,
for which one may drop the assumption that $Y$ is a NCD. Using the
fact that $\Xi$ and $\xi$ differ by a $\db$-coboundary on $\TXM^{*}$,
$\int_{\van_{t}}R_{\Xi}\equiv\int_{\van_{t}}R_{\xi}$ (mod $\QQ(n)$)
provided $\van_{t}$ does not meet $\tilde{D}$. For $t=\alpha$ this
yields \begin{equation}
\Psi(\alpha) \qneq \sum_i \int_{\varphi_i}R\{x_1|_{Y_i},\ldots ,x_n|_{Y_i}\}. \\
\end{equation} In the event that $(t=)\alpha=t_{0}$ (at the boundary of convergence
of $(2.6)$), using Corollary 2.6 gives \begin{equation}
\log(t_0)+\sum_{k\geq 1}\frac{[\phi^k]_0}{k}t^k \qoeq \frac{1}{(2\pi i)^{n-1}} \sum_i \int_{\varphi_i} R\{\underline{x}\}|_{Y_i} ; \\
\end{equation} in particular, if $t_{0}\in\RR^{+}$ and $K\subset\RR$ then the
$\text{l.h.s.}=\Re(\text{r.h.s.})$.

These formulas are of greatest practical use --- i.e. the r.h.s. of
(4.7-8) is directly computable --- when the $\{Y_{I}\}$ are rational
(and explicitly parametrized). This is automatic for $n=2$, but unfortunately
(at least for $(4.8)$) doesn't tend to occur at $t_{0}$ for $n=3$
--- in all the examples we have analyzed (see e.g. $\S\S4.4,8.5$),
the $K3$ acquires a node there.

We conclude with a general result which best captures the sense in
which {}``singular'' $AJ_{\tilde{X}_{\alpha}}(\Xi_{\alpha})$ is
a $limit$ of {}``smooth'' \{$AJ_{\TXT}(\Xi_{t})$\}. Let $\X\mapg\ms$
be a proper, dominant morphism of smooth varieties with $\dim(\ms)=1$
and unique%
\footnote{Since $\ms$ isn't required to be  complete, this can be arranged
by omitting other singular fibers.%
} singular fiber $X_{0}$. Assume $X_{0}$ is a $reduced$ NCD so that
the local degeneration (over a disk with coordinate $s$)

\xymatrix{& & & & \X_{\Delta}^*  \ar @{^(->} [r] \ar [d]^f & \X_{\Delta} \ar [d]^{\bar{f}} & X_0 \ar @{_(->} [l]_{\iota_{X_0}} \ar [d] & \cup Y_i \ar @{=} [l] \\ & & & &  {}\Delta^* \ar @{^(->} [r]^\jmath & \Delta & \{0\} \ar @{_(->} [l] }${}$\\
is semistable; and let $\Xi^{*}\in CH^{p}(\X\m X_{0},r)$. Define
the local system $\HH_{\QQ}:=R^{2p-r-1}f_{*}\QQ(p)$, cohomology sheaves
$\H:=R^{2p-r-1}f_{*}\CC\otimes\mathcal{O}_{\Delta^{*}}$ with holomorphic
Hodge subsheaves $\F^{m}$, and Jacobian sheaf (via the s.e.s.) \begin{equation}
\HH_{\QQ} \hookrightarrow \frac{\H}{\F^p} \twoheadrightarrow \J^{p,r}. \\
\end{equation} Then $\Xi^{*}$ gives rise to the higher normal function \[
\nu_{\Xi^{*}}(s):=AJ_{X_{s}}(\Xi_{s})\in\Gamma(\Delta^{*},\J^{p,r}),\]
 where $\Xi_{s}:=\iota_{X_{s}}^{*}(\Xi^{*})$. Writing $T\in Aut(\HH_{\QQ})$
for the (unipotent) monodromy operator (with $N:=\log T$), consider
the Clemens-Schmid exact sequence of MHS\[
\cdots\to H^{2p-r-1}(X_{0})\maprho H_{lim}^{2p-r-1}(X_{s})\mapNN H_{lim}^{2p-r-1}(X_{s})(-1)\to\cdots\]
and the canonically extended sheaves $\H_{e},\,\F_{e}^{p}$, and \begin{equation}
\jmath_* \HH_{\QQ}  \hookrightarrow \frac{\H_e}{\F^p_e} \twoheadrightarrow \J^{p,r}_e \\
\end{equation} over $\Delta$. Set%
\footnote{$(\jmath_{*}\HH_{\QQ})_{0}$ is the stalk of the local system at $0$
(i.e., invariant cycles), while $\H_{e,0}$ and $\F_{e,0}^{p}$ are
the fibers (over $0$) of the corresponding holomorphic vector bundles.%
} \[
J_{lim}^{p,r}(X_{s}):=\frac{\H_{e,0}}{(\jmath_{*}\HH_{\QQ})_{0}+\F_{e,0}^{p}}\cong\ext(\QQ(0),H_{lim}^{2p-r-1}(X_{s},\QQ(p)))\]
and $J^{p,r}(X_{0}):=\ext(\QQ(0),H^{2p-r-1}(X_{0},\QQ(p)))$; then
$\rho$ induces \[
J(\rho):\, J^{p,r}(X_{0})\to J_{lim}^{p,r}(X_{s}).\]
Note that any section $\nu\in\Gamma(\Delta,\J_{e}^{p,r})$ has a well-defined
{}``value'' $\nu(0)\in J_{lim}^{p,r}(X_{s})$.

\begin{prop}
Suppose $Res_{X_{0}}(\Xi^{*})\in CH^{p-1}(X_{0},r-1)\,(\cong H_{\M,X_{0}}^{2p-r+1}(\X,\QQ(p)))$
is zero. Then $\nu_{\Xi^{*}}$ lifts uniquely to a section $\nu\in\Gamma(\Delta,\J_{e}^{p,r})$,
and we define $\lim_{s\to0}\nu_{\Xi^{*}}(s):=\nu(0)\in J_{lim}^{p,r}(X_{s})$.
Furthermore, if $\Xi\in CH^{p}(\X,r)$ restricts to $\Xi^{*}$ then\[
\lim_{s\to0}\nu_{\Xi^{*}}(s)=J(\rho)(AJ_{X_{0}}(\iota_{X_{0}}^{*}\Xi)).\]

\end{prop}
\begin{proof}
(Sketch.) The existence of $\Xi$ follows from Bloch's moving lemma
\cite{Bl1}, and we can put it into good position relative to $X_{0}$.
Since \[
\iota_{X_{0}}^{*}(cl(\Xi))\in\hm{(}\QQ(0),H^{2p-r}(X_{0},\QQ(p)))=\{0\},\]
and $X_{0}$ is a deformation retract of $\X_{\Delta}$, the restriction
of $cl(\Xi)=[\Omega_{\Xi}]=(2\pi i)^{p}[T_{\Xi}]$ to $\X_{\Delta}$
(hence to $\X_{\Delta}^{*}$) is trivial.%
\footnote{After this step, remaining details are similar to those in \cite{GGK1}
$\S3$.%
} So the image of $\nu_{\Xi^{*}}$ in $H^{1}(\Delta^{*},\HH_{\QQ})$
vanishes, and its lift to $\Gamma(\Delta^{*},\frac{\H}{\F^{p}})$
is actually computed by fiberwise integration of the completed regulator
current $R_{(\Xi|_{\X_{\Delta}})}'':=R_{\Xi}|_{\X_{\Delta}}-d^{-1}(\Omega_{\Xi}|_{\X_{\Delta}})+(2\pi i)^{p}\delta_{\d^{-1}(T_{\Xi}|_{\X_{\Delta}})}$
against sections of $\bar{f}_{*}F^{n-p}A_{\X/\ms}^{2(n-p)+r-1}(\log X_{0})$
($n=\dim\X$). As $s\to0$ these integrals do not blow up, so the
lift extends to $\tilde{\nu}\in\gamma(\Delta,\frac{\H_{e}}{\F_{e}^{p}})$;
this has image $\nu\in\Gamma(\Delta,\J_{e}^{p,r})$. (In fact, at
$s=0$ they compute $AJ_{X_{0}}(\iota_{X_{0}}^{*}\Xi)$ by generalizing
the argument used to prove $(4.1)$ above.) The uniqueness of $\nu$
is a simple argument using the long-exact cohomology sequences of
$(4.9),(4.10)$.
\end{proof}

\subsection{Formula for $AJ$ on a N\'eron $N$-gon}

Returning to the setting of Theorem 1.7, we will now compute the r.h.s.
of $(4.7)$ for Kodaira type $I_{N}$ degenerations of elliptic curves.
Specialize to the case $n=2$, $\tilde{X}_{\alpha}=Y=\cup_{i=1}^{N}Y_{i}$
with each $Y_{i}\cong\PP^{1}$, $Y_{i_{0}i_{1}}$ nonempty iff $i_{0}-i_{1}\equiv\pm1$
mod $N$, and $Y^{[2]}\cap\tilde{D}=\emptyset$. Let $z_{i}:\, Y_{i}\mapf\PP^{1}$
be such that $z_{i}(Y_{i,i-1})=\infty$, $z_{i}(Y_{i,i+1})=0$, and
$\tilde{\varphi}_{\alpha}=\varepsilon_{\alpha}\cdot\sum_{i=1}^{N}T_{z_{i}}$
(for some $\varepsilon\in\ZZ$). Then restrictions of toric coordinates
$x_{1}|_{Y_{i}},x_{2}|_{Y_{i}}$ will be written\[
f_{i}(z_{i})=A_{i}\prod_{j}(1-\frac{\alpha_{ij}}{z_{i}})^{d_{ij}}\,\,\,,\,\,\,\,\,\,\,\, g_{i}(z_{i})=B_{i}\prod_{k}(1-\frac{z_{i}}{\beta_{ik}})^{e_{ik}}\]
(with no $\alpha_{ij}$ or $\beta_{ik}$ $0$ or $\infty$); note
that $\sum_{j}d_{ij}=\sum_{k}e_{ik}=0$ ($\forall i$) and \[
(f_{i}(0),g_{i}(0))=(A_{i}\prod_{i}\alpha_{ij}^{d_{ij}},\, B_{i})\,\,\,,\,\,\,\,\,\,\,\,(f_{i}(\infty),g_{i}(\infty))=(A_{i},\, B_{i}\prod_{k}\beta_{ik}^{-e_{ik}}).\]

Since $Y$ is a singular fiber in a family of elliptic curves produced
via a $tempered$ Laurent polynomial, $Tame_{\xi}\{f_{i},g_{i}\}$
is torsion for every $\xi\in|(f_{i})|\cup|(g_{i})|$. We do $not$
require that $|(f_{i})|\cap|(g_{i})|=\emptyset$, so for sums over
both $j$ and $k$ the notation $\sum_{j,k}'$ means to omit terms
for which $\alpha_{ij}=\beta_{ik}$. In particular, we set\[
\N_{f_{i},g_{i}}:=\sum_{j,k}{}^{'}d_{ij}e_{ik}[\frac{\alpha_{ij}}{\beta_{ik}}]\,\in\,\ZZ[\PP^{1}\m\{0,\infty\}]\]
and $\N_{\alpha}:=\sum_{i}\N_{f_{i},g_{i}}$. Another important notational
point is that $\log z$ is regarded as a $0$-current with branch
cut along $T_{z}$, so that (with $\dlog z:=\frac{dz}{z}$) $\delta_{T_{z}}=\frac{1}{2\pi i}(\dlog z-d[\log z])$;
also $d[\frac{\dlog z}{2\pi i}]=\delta_{\{0\}}-\delta_{\{\infty\}}.$
While this approach {}``keeps track of branches of log'', a nasty
side effect is that $\log a-\log b\neq\log\frac{a}{b}$; although
the discrepancy lies in $\ZZ(1)$ this becomes significant when multiplied
by another function.

Now recalling that\[
R\{f,g\}:=\log f\dlog g-2\pi i(\log g)\delta_{T_{f}},\]
one easily checks that (in $\mathcal{D}^{1}(Y_{i}\m|(f_{i})|\cup|(g_{i})|)$)\[
R\{f_{i},g_{i}\}\,\equiv\,\sum_{j,k}{}^{'}d_{ij}e_{ik}R\{1-\frac{\alpha_{ij}}{z_{i}},1-\frac{z_{i}}{\beta_{ik}}\}\,+\, R\{f_{i},B_{i}\}+R\{A_{i},g_{i}\}\]
where the equivalence is generated by $d\{0\text{-currents which are }0\text{ at }z=0,\infty\}$
and $\delta_{\{\ZZ(2)[\frac{1}{2}]-chains\}}.$ This gives the r.h.s.
of $(4.7)$ (for now omitting $\varepsilon_{\alpha}$)\[
\sum_{i,j,k}{}^{'}d_{ij}e_{ik}\int_{T_{z_{i}}}R\{1-\frac{\alpha_{ij}}{z_{i}},1-\frac{z_{i}}{\beta_{ik}}\}\,-\,2\pi i\sum_{i}\log B_{i}\int_{T_{z_{i}}}\delta_{T_{f_{i}}}\,+\,\sum_{i}\log A_{i}\int_{T_{z_{i}}}\dlog g_{i}.\]
Rewriting $\int_{T_{z_{i}}}(\cdot)$ as $\frac{1}{2\pi i}\int_{\PP^{1}}(\frac{dz_{i}}{z_{i}}-d[\log z_{i}])\wedge(\cdot)=\frac{-1}{2\pi i}\int_{\PP^{1}}(\cdot)\wedge\frac{dz_{i}}{z_{i}}+\frac{1}{2\pi i}\int_{\PP^{1}}(\log z_{i})d(\cdot)$
yields \begin{equation}
\begin{matrix} 
\sum_{i,j,k}'  d_{ij}e_{ik} \left( \int_{T_{1-\frac{\alpha_{ij}}{z_i}}} \log (1-\frac{z_i}{\beta_{ik}}) \frac{dz_i}{z_i}  + \int_{\PP^1}\frac{\log z_i}{2\pi i}d \left[R\{ 1-\frac{\alpha_{ij}}{z_i},1-\frac{z_i}{\beta_{ik}} \} \right] \right) \\ 
+\frac{1}{2\pi i}\sum_i \log B_i \int_{\PP^1} \{(\log f_i)d[\frac{dz_i}{z_i}]-(\log z_i)d[\frac{df_i}{f_i}] \} + \frac{1}{2\pi i}\sum_i \log A_i \int_{\PP^1}(\log z_i)d[\frac{dg_i}{g_i}].
\end{matrix} \\
\end{equation}

The directed line segments (for distinct $a,b\in\CC^{*}$)\[
T_{1-\frac{a}{z}}=e^{i\arg a}[0,|a|]\,\,\,,\,\,\,\,\,\,\,\, T_{1-\frac{z}{b}}=e^{i\arg b}[-\infty,|b|]\]
 in $\PP^{1}$ do not intersect unless $\arg a\equiv\arg b$ (mod
$2\pi\ZZ$) and $|b|<|a|$, in which case a global perturbation as
in $\S9$ of \cite{Ke1} may be deployed to kill the intersection.
Since in general\[
d[R\{f,g\}]=2\pi i(\log f|_{(g)}-\log g|_{(f)})-(2\pi i)^{2}\delta_{T_{f}\cdot T_{g}},\]
$(4.11)$ becomes ($\Psi(\alpha)\qteq$) \begin{equation}
\begin{matrix}
-\sum_{i,j,k}' d_{ij}e_{ik} \{Li_2(\frac{\alpha_{ij}}{\beta_{ik}}) + (\log \alpha_{ij} - \log \beta_{ik} ) \log (1-\frac{\alpha_{ij}}{\beta_{ik}}) \} \\ \\
+ \sum_i \log g_i(0) (\log f_i(0) - \log f_i(\infty) ) \\ \\
- \sum_i \log B_i \sum_j d_j \log \alpha_{ij} + \sum_i \log A_i \sum_k e_{ik} \log \beta_{ik}.
\end{matrix} \\
\end{equation} This is the best we can do without further information.

Next, suppose that we know $\Psi(\alpha)$ is pure imaginary (up to
$\QQ(2)$), or just want its imaginary part. Taking $\Im\{(4.12)\}$
gives \begin{equation}
\begin{matrix}
-\sum_{i,j,k}' d_{ij}e_{ik}\{\Im Li_2(\frac {\alpha_{ij}}{\beta_{ik}})+\log|\frac{\alpha_{ij}}{\beta_{ik}}|\arg(1-\frac{\alpha_{ij}}{\beta_{ik}}) \} \\ \\
+\sum_i \log|g_i(0)|\left( \arg f_i(0)-\arg f_i(\infty) \right) + \sum_i \arg(g_i(0))\log|\frac{f_i(0)}{f_i(\infty)} | \\ \\
-\sum_i \arg (g_i(0)) \log|\frac{f_i(0)}{f_i(\infty)}| + \sum_i \arg(f_i(\infty)) \log|\frac{g_i(0)}{g_i(\infty)}| \\ \\
-\sum_i \log|B_i| \sum_j d_{ij} \arg \alpha_{ij} + \sum_i \log|A_i| \sum_k e_{ik} \arg \beta_{ik} \\ \\ 
-\sum_i \sum_j d_{ij} \arg \alpha_{ij} \log \left|\prod_k ' (1-\frac{\alpha_{ij}}{\beta_{ik}})^{e_{ik}} \right| + \sum_i\sum_k e_{ik} \arg \beta_{ik} \log \left| \prod_j ' (1-\frac{\alpha_{ij}}{\beta_{ik}})^{d_{ij}} \right| ,
\end{matrix} \\
\end{equation} where the $\prod_{k}',\,\prod_{j}'$ mean to omit terms which are
$0$. The last $4$ terms of $(4.13)$ may be rerranged to give\[
\sum_{i}\sum_{\xi\in\CC^{*}}\arg(\xi)\log\left|\frac{\left\{ A_{i}\prod_{j}'(1-\frac{\alpha_{ij}}{\xi})^{d_{ij}}\right\} ^{\nu_{\xi}(g_{i})}}{\left\{ B_{i}\prod_{k}'(1-\frac{\xi}{\beta_{ik}})^{e_{ik}}\right\} ^{\nu_{\xi}(f_{i})}}\right|=\]
\[
\sum_{i}\sum_{\xi\in\CC^{*}}\arg(\xi)\log\left|Tame_{\xi}\{f_{i},g_{i}\}\right|=0.\]
The 2nd and 3rd rows of $(4.13)$, after obvious cancellations, yield
the collapsing sum\[
\sum_{i}\left\{ \log|g_{i}(0)|\arg f_{i}(0)-\log|g_{i}(\infty)|\arg f_{i}(\infty)\right\} =0.\]
This leaves us with the first row, which is just\[
-\sum_{i,j,k}{}^{'}d_{ij}e_{ik}D_{2}(\frac{\alpha_{ij}}{\beta_{ik}})=:-D_{2}(\N_{\alpha}),\]
 where $D_{2}(z):=\Im(Li_{2}(z))+\log|z|\arg(1-z)$ is the (real,
single-valued) Bloch-Wigner function. Summarizing this discussion
and combining with $(4.8)$ gives immediately

\begin{prop}
For a family of elliptic curves as in Theorem 1.7 ($n=2$), with $\tilde{X}_{\alpha}$
a N\'eron $N$-gon,%
\footnote{this includes $N=1,2$%
} $\Psi(\alpha)\qteq\varepsilon_{\alpha}\cdot(4.12)$ with $\Im(\Psi(\alpha))=-\varepsilon_{\alpha}D_{2}(\N_{\alpha})$.
In particular if $\alpha=t_{0}$, and $K(t_{0})\subset\RR$, we have
\begin{equation}
\log \left| \frac{1}{t_0} \right| -\sum_{k\geq 1}\frac{[\phi^k]_0}{k}t_0^k = \frac{\varepsilon_{\alpha}}{2\pi}D_2(\N_{t_0}), \\
\end{equation} plus or minus $\pi i$ if $t_{0}<0$.
\end{prop}
If the family $\TXM$ or a $t\mapsto t^{\kappa}$ quotient thereof
has just three singular fibers, then the l.h.s. of $(4.12)$ is a
special value of a {}``hypergeometric integral'' or Meijer $G$-function,
and such identities seem to go back essentially to Ramanujan. In addition,
the Meijer $G$-functions studied in \cite{MOY} for the $E_{6},\, E_{7},\, E_{8}$
cases below are nothing but $\frac{1}{2\pi i}$ times the regulator
period $\Psi(t^{\kappa})$.

We should emphasize that $(4.14)$ (as derived above) is a $motivic$
identity which directly reflects the limit $AJ$ result Prop. 4.3.

\subsection{Examples $\mathbf{D_{5},\, E_{6},\, E_{7},\, E_{8}}$}

We turn now to 4 {}``mirror pairs'' of elliptic curve families with
common fundamental periods. The Laurent polynomials $\phi_{I},\phi_{II}$
in the first column of the table below have dual Newton polytopes
and are of the type considered in Example 1.10. The corresponding
$\tilde{\X}_{I},\tilde{\X}_{II}$ are smooth and the second column
lists their Kodaira fiber types over $t=0$, $t\in\L\cap\CC*$, and
$t=\infty$ (in that order). These 2 families share a common degree-$\kappa$
quotient (over simply $t\mapsto t^{\kappa}$ for each $\tilde{\X}_{II}$),
whose singular fibers (after a minimal desingularization of the total
space) are listed next. This is followed by the Dynkin diagram type
of the dual graph of the singular fiber over $t^{\kappa}=\infty$
(in the quotient), which we use to {}``identify'' each example.
The vanishing-cycle periods about $t=0$ (being pullbacks from the
quotient families) take the form $A_{I}(t)=A_{II}(t)=\sum_{m\geq0}a_{m}t^{\kappa m}$,
and so $\Psi_{I}(t)=\Psi_{II}(t)=2\pi i(\log t+\sum_{m\geq1}\frac{a_{m}}{\kappa m}t^{\kappa m})$.
Finally, if we take $\phi=\phi_{II}$ in $\S3$, then the $\{N_{D}^{\left\langle X^{\circ}\right\rangle }\}$
are local Gromov-Witten invariants of the $Y_{II}^{\circ}$ indicated
and these will have exponential growth rate $\exp(-\Re(\frac{\Psi_{II}(t_{0})}{2\pi i}))$
by $(3.10)$.\\
\\
\begin{tabular}{|c|c|c|c|c|c|c|c|}
\hline 
$\begin{array}{c}
\phi_{I}\\
\phi_{II}\end{array}$ & \small fibers of $\tilde{\X}$\normalsize  & $\kappa$ & $\begin{array}{c}
\text{fibers of}\\
\widetilde{\tilde{X}/\ZZ_{\kappa}}\end{array}$ & \tiny $\begin{array}{c}
\text{type}\\
\text{at }\infty\end{array}$\normalsize & $a_{m}$ & $t_{0}$ & $Y_{II}^{\circ}$\tabularnewline
\hline
\hline 
$\begin{array}{c}
(x+\frac{1}{x})(y+\frac{1}{y})\\
x+\frac{1}{x}+y+\frac{1}{y}\end{array}$ & $\begin{array}{c}
I_{4},2I_{1},I_{4}\\
I_{8},2I_{1},I_{2}\end{array}$ & 2 & $I_{4},I_{1},I_{1}^{*}$ & $D_{5}$ & ${2m \choose m}^{2}$ & $\frac{1}{4}$ & $K_{_{\PP^{1}\times\PP^{1}}}$\tabularnewline
\hline 
$\begin{array}{c}
\frac{x^{2}}{y}+\frac{y^{2}}{x}+\frac{1}{xy}\\
x+y+\frac{1}{xy}\end{array}$ & $\begin{array}{c}
I_{3},3I_{3},I_{0}\\
I_{9},3I_{1},I_{0}\end{array}$ & 3 & $I_{3},I_{1},IV^{*}$ & $E_{6}$ & ${3m \choose m,m,m}$ & $\frac{1}{3}$ & $K_{\PP^{2}}$\tabularnewline
\hline 
$\begin{array}{c}
\frac{x}{y}+\frac{y^{3}}{x}+\frac{1}{xy}\\
x+y+\frac{1}{x^{2}y}\end{array}$ & $\begin{array}{c}
I_{4},4I_{2},I_{0}\\
I_{8},4I_{1},I_{0}\end{array}$ & 4 & $I_{2},I_{1},III^{*}$ & $E_{7}$ & ${4m \choose 2m,m,m}$ & $\frac{1}{2\sqrt{2}}$ & $K_{_{\PP(1,1,2)}}$\tabularnewline
\hline 
$\begin{array}{c}
\frac{x}{y}+\frac{y^{2}}{x}+\frac{1}{xy}\\
x+y+\frac{1}{x^{3}y^{2}}\end{array}$ & $\begin{array}{c}
I_{6},6I_{1},I_{0}\\
I_{6},6I_{1},I_{0}\end{array}$ & 6 & $I_{1},I_{1},II^{*}$ & $E_{8}$ & ${6m \choose 3m,2m,m}$ & $\frac{1}{4^{\frac{1}{3}}\sqrt{3}}$ & $K_{_{\PP(1,2,3)}}$\tabularnewline
\hline
\end{tabular}${}$\\
\\
\\
Obviously we may use $either$ $\tilde{\X}_{I}$ or $\tilde{\X}_{II}$
to compute $\Psi_{II}(t_{0})$($=\Psi_{I}(t_{0})$), and for $E_{6},\, E_{7},\, E_{8}$
we will use $\tilde{\X}_{I}$. For $D_{5}$, we use instead the family
$\tilde{\X}$ produced by $\phi:=\frac{(x-1)^{2}(y-1)^{2}}{xy}$,
with $t_{0}=\frac{1}{16}$ and $A(t)=\sum_{m\geq0}{2m \choose m}^{2}t^{m}$
(hence $\Psi_{II}(t)=\frac{1}{2}\Psi(t^{2})$); in fact, its minimal
desingularization $is$ the quotient family.

What we now do in each case is find an explicit parametrization of
(each component of) $\tilde{X}_{t_{0}}$ via $\{f_{i},g_{i}\}$, then
compute $\N:=\N_{t_{0}}$ and $D_{2}(\N)$. First, to record some
notation: we shall consider $L$-functions $L(\chi,s):=\sum_{k\geq1}\frac{\chi(k)}{k^{s}}$
of primitive Dirichlet characters\[
\chi_{-3}(\cdot)=0,\,1,\,-1,\,\ldots\,\,\,\,\,(\text{mod}\,\,\,\,3)\]
\[
\chi_{-4}(\cdot)=0,\,1,\,0,\,-1,\,\ldots\,\,\,\,\,(\text{mod}\,\,\,\,4)\]
\[
\chi_{+i,5}(\cdot)=0,\,1,\, i,\,-i,\,-1,\,\ldots\,\,\,\,\,(\text{mod}\,\,\,\,5)\]
\[
\chi_{-i,5}(\cdot)=0,\,1,\,-i,\, i,\,-1,\,\ldots\,\,\,\,\,(\text{mod}\,\,\,\,5)\]
\[
\chi_{-8}(\cdot)=0,\,1,\,0,\,1,\,0,\,-1,\,0,\,-1,\,\ldots\,\,\,\,\,(\text{mod}\,\,\,\,8)\]
at $s=2$. An easy way to get such values is by taking Bloch-Wigner
of roots of unity: e.g. for $\zeta_{a}=e^{\frac{2\pi i}{a}}$,\[
D_{2}(\zeta_{a})=\Im(Li_{2}(\zeta_{a}))+0=\sum_{k\geq1}\frac{\Im(\zeta_{a}^{k})}{k^{2}}.\]
To simplify $D_{2}(\N)$ to terms of this form, we manipulate $\N$
in a quotient of the pre-Bloch group $\B_{2}(\CC)$. Namely, work
in $\ZZ[\PP_{\CC}^{1}\m\{0,1,\infty\}]$ modulo (the subgroup generated
by) relations: $[\xi]+[\frac{1}{\xi}]$; $[1-\xi]+[\xi]$; $[\xi]+[\bar{\xi}]$;
and $\sum_{i=1}^{5}[\xi_{i}]$ where (with subscripts mod $5$) $\xi_{i}=1-\xi_{i+1}\xi_{i-1}$
($\forall i$), pictured as in Figure 4.1.%
\begin{figure}
\caption{\protect\includegraphics[bb=0bp 0bp 124bp 114bp,scale=0.7]{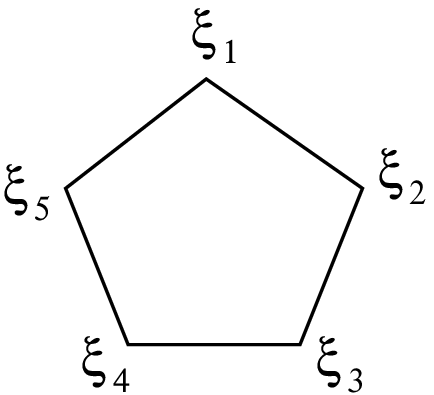}}

\end{figure}
 (These are all well-known relations on $D_{2}$, see \cite{Bl2}.)\\
\\
$\underline{\mathbf{D_{5}}}:$ In $\PP^{1}\times\PP^{1}$, $1-\frac{1}{16}\frac{(x-1)^{2}(y-1)^{2}}{xy}=0$
is an $I_{1}$ normalized by\[
f(z)=-\frac{(1+\frac{1}{z})^{2}}{(1-\frac{1}{z})^{2}}\,\,\,,\,\,\,\,\,\,\,\, g(z)=-\frac{(1+\frac{z}{i})^{2}}{(1-\frac{z}{i})^{2}}.\]
Hence $\N=8[-i]-8[i]\equiv-16[i]$, and \[
D_{2}(\N)=-16D_{2}(i)=-16L(\chi_{-4},2).\]
(So in fact the correct $D_{2}(\N_{t_{0}})$ to use for $\phi_{II}$
is $-8L(\chi_{-4},2)$.)\\
\\
$\underline{\mathbf{E}_{\mathbf{6}}}:$ In $\PP^{2}$, $0=1-\frac{1}{3}\frac{x^{3}+y^{3}+1}{xy}=\frac{-1}{3xy}(1+x+y)(1+\zeta_{3}x+\bar{\zeta_{3}}y)(1+\bar{\zeta_{3}}x+\zeta_{3}y)$
is normalized by\[
f_{1}(z_{1})=\bar{\zeta_{3}}\frac{(1-\frac{\zeta_{3}}{z_{1}})}{(1-\frac{1}{z_{1}})}\,\,\,,\,\,\,\,\,\,\,\, g_{1}(z_{1})=\frac{(1-\frac{z_{1}}{\bar{\zeta_{3}}})}{(1-z_{1})},\]
\[
f_{2}(z_{2})=\frac{(1-\frac{\zeta_{3}}{z_{2}})}{(1-\frac{1}{z_{2}})}\,\,\,,\,\,\,\,\,\,\,\, g_{2}(z_{2})=\bar{\zeta_{3}}\frac{(1-\frac{z_{2}}{\bar{\zeta_{3}}})}{(1-z_{2})},\]
\[
f_{3}(z_{3})=\zeta_{3}\frac{(1-\frac{\zeta_{3}}{z_{3}})}{(1-\frac{1}{z_{3}})}\,\,\,,\,\,\,\,\,\,\,\, g_{3}(z_{3})=\zeta_{3}\frac{(1-\frac{z_{3}}{\bar{\zeta_{3}}})}{(1-z_{3})},\]
so that $\N=3[\bar{\zeta_{3}}]-6[\zeta_{3}]\equiv-9[\zeta_{3}]$ and
\[
D_{2}(\N)=-9D_{2}(\zeta_{3})=-\frac{9\sqrt{3}}{2}L(\chi_{-3},2).\]
\\
\\
$\underline{\mathbf{E_{7}}}:$ In $\PP(1,1,2)$, $0=1-\frac{1}{2\sqrt{2}}\frac{x^{2}+y^{4}+1}{xy}=\frac{-1}{2\sqrt{2}xy}(x+iy^{2}-\sqrt{2}y-i)(x-iy^{2}-\sqrt{2}y+i)$
is normalized by\[
f_{1}(z_{1})=-\sqrt{2}\frac{(1-\frac{\gamma}{z_{1}})(1-\frac{\delta}{z_{1}})}{(1+\frac{1}{z_{1}})^{2}}\,\,\,,\,\,\,\,\,\,\,\, g_{1}(z_{1})=\frac{1-z_{1}}{1+z_{1}},\]
\[
f_{2}(z_{2})=\sqrt{2}\frac{(1-\frac{\gamma}{z_{2}})(1-\frac{\delta}{z_{2}})}{(1+\frac{1}{z_{2}})^{2}}\,\,\,,\,\,\,\,\,\,\,\, g_{2}(z_{2})=\frac{1-z_{2}}{1+z_{2}},\]
where $\gamma:=i(\sqrt{2}-1)$, $\delta:=i(\sqrt{2}+1)$ (and $\gamma\delta=-1$).
We read off\[
\N=2[\gamma]+2[\delta]-2[-\gamma]-2[-\delta]-2[-1]=4[\gamma]+4[\delta]\]
using $\bar{\gamma}=-\gamma$, $\bar{\delta}=-\delta$. Now using
the three $5$-term relations pictured in Figure 4.2,%
\begin{figure}

\caption{\protect\includegraphics[scale=0.7]{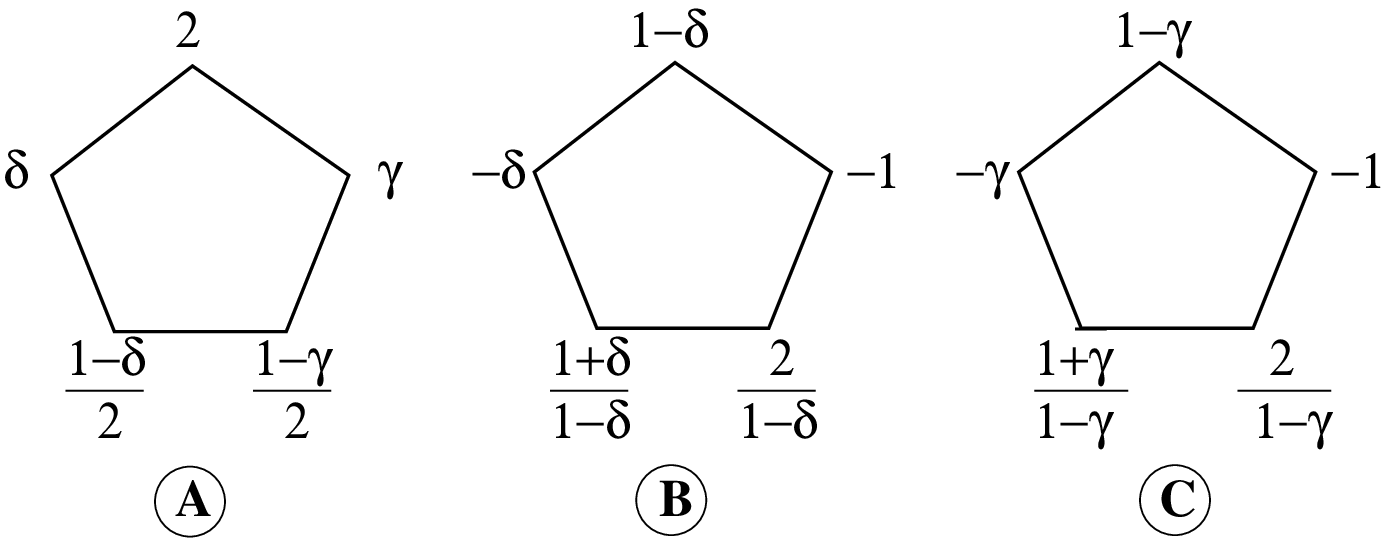}}

\end{figure}
 together with $\frac{1+\gamma}{1-\gamma}=\zeta_{8}$, $\frac{1+\delta}{1-\delta}=\zeta_{8}^{3}$,
we have\[
[\gamma]+[\delta]\,\aequiv\,2([\gamma]+[\delta])\,+\,[\frac{1-\delta}{2}]\,+\,[\frac{1-\gamma}{2}]\]
\[
\equiv\,-[-\gamma]-[1-\gamma]-[-\delta]-[1-\delta]-[\frac{2}{1-\gamma}]-[\frac{2}{1-\delta}]\]
\[
\bcequiv\,[\zeta_{8}]+[\zeta_{8}^{3}].\]
Hence \[
D_{2}(\N)=4D_{2}(\zeta_{8})+4D_{2}(\zeta_{8}^{3})=-2i\sum_{k\geq1}k^{-2}\{\zeta_{8}^{k}+\zeta_{8}^{3k}-\zeta_{8}^{5k}-\zeta_{8}^{7k}\}\]
\[
=4\sqrt{2}L(\chi_{-8},2).\]
\\
\\
$\underline{\mathbf{E_{8}}}:$ In $\PP(1,2,3)$, $1-\frac{x^{2}+y^{3}+1}{4^{\frac{1}{3}}3^{\frac{1}{2}}xy}=0$
is an $I_{1}$ whose normalization takes the form\[
f(z)=\sqrt{3}\frac{\prod_{j=1}^{3}(1-\frac{\alpha_{j}}{z})}{(1-\frac{1}{z})^{3}}\,\,\,,\,\,\,\,\,\,\,\, g(z)=\sqrt[3]{2}\frac{\prod_{k=1}^{2}(1-\frac{z}{\beta_{k}})}{(1-z)^{2}},\]
where $\prod\alpha_{j}=\prod\beta_{k}=1$, $g(\alpha_{j})=-\zeta_{3}^{j}$
and $f(\beta_{k})=(-1)^{k}i$.

\begin{conjecture*}
$\sum_{i,j}[\frac{\alpha_{j}}{\beta_{k}}]-3\sum_{k}[\frac{1}{\beta_{k}}]-2\sum_{j}[\alpha_{j}]\equiv\frac{20}{3}[i].$
\end{conjecture*}
If this is true then $D_{2}(\N)=\frac{20}{3}L(\chi_{-4},2)$.\\

In each of these $4$ cases, $\varepsilon_{t_{0}}=-1$ and multiplying
$(4.14)$ by $\kappa$ yields\begin{equation}
\log \left| \frac{1}{t_0^{\kappa}} \right| - \sum_{m\geq 1}\frac{a_m}{m}(t_0^{\kappa})^m = \frac{-\kappa}{2\pi} D_2(\N_{t_0}) ; \
\end{equation}or on an individual basis (writing $G:=\sum_{k\geq0}\frac{(-1)^{k}}{(2k+1)^{2}}$
for Catalan's constant)\[
D_{5}:\,\,\,\,\,\log16-\sum_{m\geq1}\frac{{2m \choose m}^{2}}{m(16)^{m}}=\frac{8}{\pi}G,\]
\[
E_{6}:\,\,\,\,\,\log27-\sum_{m\geq1}\frac{(3m)!}{m(m!)^{3}(27)^{m}}=\frac{27\sqrt{3}}{4\pi}L(\chi_{-3},2),\]
\[
E_{7}:\,\,\,\,\,\log64-\sum_{m\geq1}\frac{(4m)!}{m(2m)!(m!)^{2}(64)^{m}}=\frac{8\sqrt{2}}{\pi}L(\chi_{-8},2),\]
\[
E_{8}:\,\,\,\,\,\log432-\sum_{m\geq1}\frac{(6m)!}{m(3m)!(2m)!m!(432)^{m}}\maybe\frac{20}{\pi}G.\]
Of these identities, $D_{5}$ and $E_{6}$ were known to \cite{RV},
while $E_{7}$ and $E_{8}$ were conjectured on the basis of numerical
experiment in \cite{Bo,MOY}. The latter two examples (modulo the
$E_{8}$ Conjecture) make the strongest case for the method of Prop.
4.4; they ae not amenable to the approach in $\S8.4$ since $\tilde{\X}_{I},\,\tilde{\X}_{II},\,\tilde{\X}_{II}/\ZZ_{\kappa}$
all fail to be modular in the sense required there.

The 4 cases in this section correspond to fundamental examples in
the local mirror symmetry literature. The instanton numbers that appear
in \cite{MOY} Table 1, \cite{S} Ex. 1-4, and \cite{CKYZ} Table
7 ({}``rational'') have the same exponential growth rates as our
$\{N_{\kappa D}^{\left\langle X^{\circ}\right\rangle }\}$, namely
$\exp\{\text{r.h.s. of }(4.15)\}$. The {}``$\kappa D$'' (instead
of $D$) appears due to a discrepancy in indexing of cohomology classes.

\subsection{Other examples}

We begin with an elliptic curve family for which $\Psi(t_{0})$ involves
more than one Dirichlet character: the universal curve with a marked
$5$-torsion point, or {}``$A_{5}$'' family. This arises via minimal
desingularization of the $\tilde{\X}$ obtained from\[
\phi=\frac{(1-x)(1-y)(1-x-y)}{xy},\]
and is birational to the family considered by \cite{Bk} in relation
to irrationality of $\zeta(2)$. This has\[
A(t)=\sum_{m\geq0}\left(\sum_{\ell=0}^{m}{m \choose \ell}^{2}{m+\ell \choose \ell}\right)t^{m}\,\,\,,\,\,\,\,\,\,\,\, t_{0}=\frac{-11\pm5\sqrt{5}}{2},\]
with singular fibers $I_{5},\, I_{1},\, I_{1},\, I_{5}$; $X_{t_{0}}=\overline{\{1-t\phi=0\}}$
is normalized by \[
f(z)=\gamma\frac{(1-\frac{1}{z})^{2}}{(1-\frac{\zeta_{5}^{2}}{z})(1-\frac{\zeta_{5}^{3}}{z})}\,\,\,,\,\,\,\,\,\,\,\, g(z)=\gamma\frac{(1-\frac{z}{\zeta_{5}})^{2}}{(1-\frac{z}{\zeta_{5}^{4}})(1-\frac{z}{\zeta_{5}^{3}})},\]
where $\gamma=-\frac{1+\sqrt{5}}{2}=2\Re(\zeta_{5}^{2})=\bar{\zeta_{5}}^{2}(\bar{\zeta_{5}}+1)=\zeta_{5}^{2}(\zeta_{5}+1).$
This gives $\N=-4[\zeta_{5}]-4[\zeta_{5}^{2}]+[\zeta_{5}^{3}]+6[\zeta_{5}^{4}]\equiv-10[\zeta_{5}]-5[\zeta_{5}^{2}].$
Writing $\delta_{\pm}:=\sqrt{\frac{5\pm\sqrt{5}}{8}}$ ($\delta_{+}=\Im(\zeta_{5}),\,\delta_{-}=\Im(\zeta_{5}^{2})$)
and $\lambda_{0}=\frac{11+5\sqrt{5}}{2}$, we compute \[
D_{2}(\N)\,=\,-5\{(1+\frac{i}{2})\delta_{+}+(\frac{1}{2}-i)\delta_{-}\}L(\chi_{+i,5},2)\,-\,5\{(1-\frac{i}{2})\delta_{+}+(\frac{1}{2}+i)\delta_{-}\}L(\chi_{-i,5},2)\]
and \[
\log\lambda_{0}-\sum_{m\geq1}\frac{\sum_{\ell=0}^{m}{m \choose \ell}^{2}{m+\ell \choose \ell}}{m\lambda_{0}^{m}}=-\frac{D_{2}(\N)}{2\pi}\,(\in\RR^{+}).\]

Turning to $n=3$, consider the irregular (but reflexive and tempered)
Laurent polynomial\[
\phi=(1-\frac{1}{x})(1-\frac{1}{y})(1-\frac{1}{z})(1-x-y+xy-xyz).\]
 This gives rise to the ({}``Ap\'ery'') family $\tilde{\X}$ of
singular $K3$'s related to irrationality of $\zeta(3)$ from the
Introduction. The general fiber has $7$ $A_{1}$ (node) singularities
and Theorem $1.7$ applies (with $K=\QQ$), producing $\Xi\in H_{\M}^{3}(\TXM,\QQ(3))$.
The degenerations occur over $\L=\{0,t_{0},\frac{1}{t_{0}},\infty\}$
where $t_{0}=(\sqrt{2}-1)^{4}$; $X_{t_{0}}$ and $X_{\frac{1}{t_{0}}}$
just have extra nodes ($\implies$ order $2$ monodromy), while $X_{0}$
and $X_{\infty}$ are unions of rational surfaces (and the corresponding
monodromies maximally unipotent). One can therefore use $(4.7)$ (but
with a different choice $\varphi_{\infty}'$ of topological $2$-cycle)
to directly compute $AJ(\Xi_{\infty})(\varphi_{\infty}')=-2\zeta(3)$.
This is done in \cite{KL} (Example $8.21$) and is behind the assertion
about $V(0)$ in the Introduction.

Now for the $n=2$ families $A_{5},\, D_{5},\, E_{6}$, we can take
advantage of their modularity to obtain an alternate computation of
$\lim_{t\to t_{0}}\Psi(t)$; this is carried out for $D_{5}$ in $\S8.1$
Example 1. Similarly, by identifying the Ap\'ery $K3$ family as
modular (and $\Xi$ essentially as an Eisenstein symbol), one can
compute that (one continuation of) $\Psi(\infty)=-48\zeta(3)$, see
$\S8.5$ Ex. 1 and $\S8.1$ Ex. 2. More interestingly, we can even
use (7.17-18) to compute $\Psi(t_{0})$, which is not amenable to
$(4.8)$ (due to the nodal degeneration). Since the fixed point $\tau_{0}=\frac{i}{\sqrt{6}}\in\HH$
of $\left(\begin{array}{cc}
0 & \frac{-1}{\sqrt{6}}\\
\sqrt{6} & 0\end{array}\right)$ corresponds to $t_{0}$, we have (with ${}'\widehat{\varphi_{\textbf{f},+6}}$
as in (8.5))\[
\Psi\left((\sqrt{2}-1)^{4}\right)\qtheq(2\pi i)^{3}\frac{i}{\sqrt{6}}\mathsf{H}_{[i\infty]}^{[2]}({}'\varphi_{\textbf{f},+6})+\frac{1}{2\pi i}\sum_{n}{}^{'}\lim_{M\to\infty}\sum_{m=-M}^{M}{}^{'}\frac{{}'\widehat{\varphi_{\textbf{f},+6}}(m,n)}{m(m\frac{i}{\sqrt{6}}+n)^{3}},\]
or dividing by $-4\pi^{2}$,\[
4\log(\sqrt{2}-1)+\sum_{k\geq1}\frac{(\sqrt{2}-1)^{4k}}{k}\left\{ \sum_{j=0}^{k}{k \choose j}^{2}{k+j \choose j}^{2}\right\} =\]
\[
4\sqrt{6}\pi-\frac{\sqrt{6}}{8\pi^{3}}\sum_{n\in\ZZ\m\{0\}}\sum_{m\geq1}{}'\widehat{\varphi_{\textbf{f},+6}}(m,n)\frac{(\frac{m^{2}}{18}-n^{2})}{(\frac{m^{2}}{6}+n^{2})^{3}}.\]
 Presumably something more can be said about the r.h.s. but we haven't
attempted this.

\section{\textbf{The classically modular analogue: Beilinson's Eisenstein
symbol}}

The next three sections run parallel to what was done for the toric
symbols in $\S\S1-2$: here we will construct the basic higher cycles,
and in $\S\S6,7$ compute the cycle-class and evaluate the fiberwise
$AJ$ map on them (and consider some variations on the basic cycles).
Starting from an $(\ell+1)$-tuple of functions on an elliptic curve
with divisors supported on $N$-torsion (or the $(\ell+1)$ divisors
themselves, or even just their Pontryagin product), the goal is essentially
to construct a family of $CH^{\ell+1}(\cdot,\ell+1)$-cycles on the
$\ell^{\text{th}}$ fiber product of the universal elliptic curve
with marked $N$-torsion over $\Gamma(N)\diagdown\sH$. The idea comes
from work of Bloch for $\ell=2$ \cite{Bl2,Bl3}, and first appeared
in the generality considered here (but for infinite level) in \cite{B2}.
Interesting aspects of the story include the relationship between
the {}``vertical'' choice of divisors and the {}``horizontal''
values of the resulting global cycle's residues over the cusps; and
the role played by modular forms and especially Eisenstein series.
Much of the material in this section (and $\S6.1$) is expository,
but is set up to better enable the $AJ$ computations (and for potentially
easier reading) than the presentations in the existing literature,
amongst which we have found \cite{B2,DS,Sc,De1} to be especially
helpful.

\subsection{Motivation via the Beilinson-Hodge Conjecture}

For a quasi-projective variety $V$ defined over $\bar{\QQ}$, this
Conjecture predicts that the cycle-class map\[
cl_{V}^{p,r}:\, CH^{p}(V,r)\to\hm{(}\QQ(0),H^{2p-r}(V_{\CC}^{an},\QQ(p)))\]
should surject, i.e. that there {}``exist enough cycles''. In the
context below (with $p=r=\ell+1$), it translates to the statement
that every Eisenstein series is, in a precise sense, the fundamental
class of an {}``Eisenstein cycle'' (or {}``symbol''). This case
will be proved in $\S6.1$ when we compute the classes of the symbols
constructed in $\S5.3$. In a sense our motivation is backwards since
the Eisenstein material was originally a major piece of evidence leading
to the Conjecture.

\subsubsection{Construction of Kuga modular varieties}

$\ZZ^{2\ell}$ acts on $\sH\times\CC^{\ell}$ ($\sH=$upper half-plane)
by \[
\left((m_{1},n_{1}),\ldots,(m_{\ell},n_{\ell})\right)\cdot(\tau;z_{1},\ldots,z_{\ell}):=(\tau;\, z_{1}+m_{1}\tau+n_{1},\ldots,z_{\ell}+m_{\ell}\tau+n_{\ell})\]
and we quotient \[
\ZZ^{2\ell}\diagdown\sH\times\CC^{\ell}\,\,=:\,\,\E^{[\ell]}\mapg\sH.\]
Recall $\Gamma(N):=\ker\{SL_{2}(\ZZ)\to SL_{2}(\ZZ/N\ZZ)\}=\left\{ \left.\amat\right|\,\tiny\begin{array}{c}
ad-bc=1\\
a\equiv1\equiv d\,\,(N)\\
b\equiv0\equiv c\,\,(N)\end{array}\normalsize\right\} $ and take $\Gamma\subset SL_{2}(\ZZ)$ s.t. $\{-\text{id}\}\notin\Gamma$
and $\Gamma\supset\Gamma(N)$ for some $N\geq3$ (such a $\Gamma$
is a $congruence$ $subgroup$ of $SL_{2}(\ZZ)$).

Now $\gamma=\amat\in\Gamma$ acts on $\sH^{*}:=\sH\cup\PP^{1}(\QQ)$
by $\gamma(\tau)=\frac{a\tau+b}{c\tau+d}$, and we define $modular$
$curves$\[
\overline{Y}_{\Gamma}:=\,\Gamma\diagdown\sH^{*}\,\supset\,\Gamma\diagdown\sH\,=:Y_{\Gamma}\]
with the $cusps$ as complement:\tiny\[
\kappa_{\Gamma}:=\overline{Y}_{\Gamma}\m Y_{\Gamma}=\frac{\left\{ \frac{r}{s}\in\PP^{1}(\QQ)\,\left|\begin{array}{c}
\exists p,q\in\nicefrac{\ZZ}{N\ZZ}\text{ s.t.}\\
pr+qs\equiv1\mod N\end{array}\right.\right\} }{\Gamma}=\frac{\left\{ (-s,r)\in(\nicefrac{\ZZ}{N\ZZ})^{2}\left|\,|\left\langle (-s,r)\right\rangle |=N\right.\right\} }{\left\langle \begin{array}{c}
(-s,r)\sim\gamma.(-s,r)=(-cr-ds,ar+bs)\\
(-s,r)\sim(s,-r)\end{array}\right\rangle }.\]
\normalsize One has also the $elliptic$ $points$\[
\varepsilon_{\Gamma}:=\left(\left.\begin{array}{c}
\underbrace{\{\tau\in\sH\,|\,\exists\gamma\in\Gamma\text{ s.t. }\gamma(\tau)=\tau\}}\\
=:\widetilde{\varepsilon}_{\Gamma}\end{array}\right/\Gamma\right)\subset Y_{\Gamma}.\]
Now let $\Gamma$ act on $\E^{[\ell]}\m\pi^{-1}(\widetilde{\varepsilon}_{\Gamma})$
by \[
\gamma.(\tau;[z_{1},\ldots,z_{\ell}]_{\tau}):=\left(\gamma(\tau);\left[\frac{z_{1}}{c\tau+d},\ldots,\frac{z_{\ell}}{c\tau+d}\right]_{\gamma(\tau)}\right);\]
the quotient is denoted $\E_{\gamma}^{[\ell]}\rTo^{\pi_{\Gamma}}Y_{\Gamma}\m\varepsilon_{\Gamma}$
and Shokurov's smooth compatification \cite{So} is $\overline{\E}_{\Gamma}^{[\ell]}\rTo^{\overline{\pi}_{\Gamma}}\overline{Y}_{\Gamma}$
(we just need its existence).

\subsubsection{Monodromy on $\E_{\Gamma}^{[\ell]}$}

To understand monodromy about $\varepsilon_{\Gamma}\cup\kappa_{\Gamma}$,
first take $\ell=1$ and let $\alpha$ resp. $\beta$ be the families
of $1$-cycles $[0,1]$ resp. $[0,\tau]$ on fibers $E_{\tau}$ of
$\E^{[1]}\to\mathfrak{H}$. Each $\gamma\in\Gamma$ should be thought
of as a composition of monodromy transformations with action \[
\alpha\mapsto a\alpha+c\beta\,\,\,,\,\,\,\,\,\,\,\,\beta\mapsto b\alpha+d\beta.\]
If $\gamma$ $fixes$ $\frac{r}{s}\in\PP^{1}(\QQ)$ (resp. $\tau_{0}\in\sH$)
then it corresponds to going around (some number of times) $[\frac{r}{s}]\in\kappa_{\Gamma}$
(resp. $[\tau_{0}]\in\varepsilon_{\Gamma}$). The $\varepsilon_{\Gamma}$
are just the finite monodromy points (order$=3$;%
\footnote{monodromy is locally of the form $\bmat$ in an appropriate basis
(Kodaira type $IV^{*}$).%
} if we hadn't required $-\text{id}\notin\Gamma$ then they could have
order $2$ or $4$); if $\Gamma=\Gamma(N)$ then $\varepsilon_{\Gamma}=\emptyset$.

To put all the cusps on an equal footing with regard to monodromy
matrices, given $\frac{r}{s}\in\PP^{1}(\QQ)$ pick $p,q\in\ZZ$ such
that $pr+qs=1$ and define a {}``local monodromy group'' \[
M_{\Gamma}\left(\left[\frac{r}{s}\right]\right):=\cmat\text{Stab}_{_{\Gamma}}\left(\frac{r}{s}\right)\dmat,\]
which is generated by $\emat$ (or $\fmat$) for some $m|N$ (resp.
$m|\frac{N}{2}$). For $\overline{\E}^{[1]}$, this yields a fiber
of type $I_{m}$ (resp. $I_{m}^{*}$) in Kodaira's classification;
we subdivide $\kappa_{\Gamma}=:\kappa_{\Gamma}^{I}\cup\kappa_{\Gamma}^{I^{*}}$. 

For $\ell\geq1$, one has an isomorphism of VHS\[
\H_{\nicefrac{\E^{[\ell]}}{Y}}^{\ell}\cong\oplus_{0\leq a\leq\left\lfloor \frac{\ell}{2}\right\rfloor }\left(\H_{\nicefrac{\E^{[1]}}{Y}}^{1}(-a)^{\otimes(\ell-2a)}\right)^{\oplus{\ell \choose \ell-2a,a,a}}\]
so that monodromy about type $I$ cusps is (maximally) unipotent for
all $\ell$, while that about type $I^{*}$ cusps is only unipotent
for $\ell$ even (by considering $\ell^{\text{th}}$ symmetric powers
of $\fmat$).

\subsubsection{MHS on the singular fibers of $\overline{\E}_{\Gamma}^{[\ell]}$}

We will use the notation $E_{\Gamma,y}^{[\ell]}$($\cong E_{\tau}^{[\ell]}$
for some $\tau\in\sH$) for smooth fibers and $\hat{E}_{\Gamma,y_{0}}^{[\ell]}$
for singular fibers, which are NCD's in the Shokurov compactification.
(Note: $\hat{E}_{\Gamma,y_{0}}^{[\ell]}$ does not count multiple
fiber-components with multiplicity.)

\paragraph*{$\underline{\text{(A) Elliptic points}}$}

($y_{0}\in\varepsilon_{\Gamma}$) Take a degree-$3$ cover $\widetilde{\overline{Y}}_{\Gamma}\rTo^{\mu}\overline{Y}_{\Gamma}$
with ramification index $3$ at $\tilde{y_{0}}\mapsto y_{0}$, and
let $\widetilde{\overline{\E}}_{\Gamma}^{[\ell]}$ be a smooth resolution
of $\overline{\E}_{\Gamma}^{[\ell]}\times_{\mu}\tilde{Y_{\Gamma}}$.
This maps to ${}'\widetilde{\overline{\E}}_{\Gamma}^{[\ell]}$ where

(a) ${}'\widetilde{\overline{\E}}_{\Gamma}^{[\ell]}\m{}'\tilde{E}_{\Gamma,\tilde{y_{0}}}^{[\ell]}=\tilde{\overline{\E}}_{\Gamma}^{[\ell]}\m\tilde{E}_{\Gamma,\tilde{y_{0}}}^{[\ell]}$
(here $\tilde{E}_{\Gamma,\tilde{y_{0}}}^{[\ell]}$ is possibly singular)

(b) ${}'\tilde{E}_{\Gamma,\tilde{y_{0}}}^{[\ell]}$ is the $\ell^{\text{th}}$
self-product of a $smooth$ elliptic curve ($\tau=e^{\frac{2\pi i}{3}}$
or $e^{\frac{2\pi i}{6}}$), \\
yielding a diagram\\
\xymatrix{& & & & '\widetilde{\overline{\E}}^{[\ell]}_{\Gamma} \ar [rd]_{'\tilde{\pi}} & \widetilde{\overline{\E}}^{[\ell]}_{\Gamma} \ar [r]_{\M} \ar [d]^{\tilde{\pi}} \ar [l]^{p} & \overline{\E}^{[\ell]}_{\Gamma} \ar [d]^{\bar{\pi}} \\ & & & & & \widetilde{\overline{Y}}_{\Gamma} \ar [r]_{\mu} & \overline{Y}_{\Gamma}}\\
\\
Now (a)$+$(b) $\implies$ $H^{\ell+1}(\widetilde{\overline{\E}}_{\Gamma}^{[\ell]}\m\tilde{E}_{\Gamma,y_{0}}^{[\ell]})=W_{\ell+2}H^{\ell+1}(\widetilde{\overline{\E}}_{\Gamma}^{[\ell]}\m\tilde{E}_{\Gamma,\tilde{y_{0}}}^{[\ell]}),$
while $\frac{1}{3}\M_{*}\M^{*}$ is the identity on $H^{\ell+1}(\overline{\E}_{\Gamma}^{[\ell]}\m E_{\Gamma,y_{0}}^{[\ell]})$.
By the localization sequence \[
\to H^{\ell+1}(\overline{\E}_{\Gamma}^{[\ell]}\m\hat{E}_{\Gamma,y_{0}}^{[\ell]})\to H_{\ell}(\hat{E}_{\Gamma,y_{0}}^{[\ell]})(-(\ell+1))\to H^{\ell+2}(\overline{\E}_{\Gamma}^{[\ell]})\to,\]
$H_{\ell}(\hat{E}_{\Gamma,y_{0}}^{[\ell]})$ is a pure HS of weight
$-\ell$.

\paragraph*{$\underline{\text{(B) Non-unipotent cusps}}$}

($y_{0}\in\kappa_{\Gamma}^{I^{*}}$, $\ell$ odd) Even in the quasi-unipotent/non-semistable
degeneration setting, if the total space is smooth (with NCD central
fiber) the Wang sequence, relative homology sequence, and deformation
retract business goes through, yielding a long-exact sequence \begin{equation}
\to H_{\ell+2}(\hat{E}^{[\ell]}_{\Gamma,y_0}(-(\ell+1)) \rTo^{\xi}  H^{\ell}(\hat{E}^{[\ell]}_{\Gamma,y_0} \to H^{\ell}(E^{[\ell]}_{\Gamma,y}) \rTo^{T-I} H^{\ell}(E^{[\ell]}_{\Gamma,y})\to  ; \\
\end{equation} here $\xi$ is a morphism of MHS (as $\iota_{y_{0}}^{*}\circ(\iota_{y_{0}})_{*}$,
it is motivic). For the monodromy matrix, taking $\ell^{\text{th}}$
symmetric power of $\fmat$ for $\ell\geq1$ odd gives $T=\gmat$;
hence $T-I$ has maximal rank and $\xi$ is surjective. Since $H_{\ell+2}(\hat{E})(-(\ell+1))$
has weights $\geq\ell$ and $H^{\ell}(\hat{E})$ weights $\leq\ell$,
we find again that $H^{\ell}(\hat{E})$ (hence $H_{\ell}(\hat{E}_{\Gamma,y_{0}}^{[\ell]})$)
is a $pure$ HS.

\paragraph*{$\underline{\text{(C) Unipotent cusps}}$}

($y_{0}\in\kappa_{\Gamma}^{I^{*}}$ and $\ell$ even; $y_{0}\in\kappa_{\Gamma}^{I}$)
Start with $\ell=1$: taking $y=[i\infty]$ as our prototypical such
cusp and assuming an $I_{m}$ degeneration there, the choice of local
parameter $q^{\frac{1}{m}}=:\tilde{q}:=\exp(\frac{2\pi i}{m}\tau)=\exp(\frac{2\pi i}{m}\frac{\int_{\beta}dz}{\int_{\alpha}dz})$
splits the LMHS:\[
H_{\lim_{\tilde{q}\to0}}^{1}(E_{\Gamma,\tilde{q}})\cong\QQ(0)\oplus\QQ(-1).\]
Similarly, $H_{\lim}^{\ell}(E_{\Gamma,\tilde{q}}^{[\ell]})$ is a
$\oplus$ of copies of $\QQ(0)$ thru $\QQ(-\ell)$ --- in particular
$one$ copy of $\QQ(0)$. (Think of this as a consequence of the fact
that the periods are all powers of $m\log\tilde{q}$; the $\QQ(0)$
corresponds to $\alpha^{\times\ell}$ with period $1$.) $(5.1)$
becomes the Clemens-Schmid sequence\[
\to H_{\ell+2}(\hat{E}_{\Gamma,y_{0}}^{[\ell]})(-(\ell+1))\rTo^{\xi}H^{\ell}(\hat{E}_{\Gamma,y_{0}}^{[\ell]})\to H_{\lim}^{\ell}(E_{\Gamma,y}^{[\ell]})\rTo^{N}H_{\lim}^{\ell}(E_{\Gamma,y}^{[\ell]})\to\]
(where $N=\log(T)$ now makes sense); since $N$ is of type $(-1,-1)$
it kills $\QQ(0)$. By the same reasoning as above, $\text{im}(\xi)$
has pure weight $\ell$; so $H^{\ell}(\hat{E}_{\Gamma,y_{0}}^{[\ell]})$
is completely split into $\QQ(-j)$'s (independent of the choice of
parameter), in particular $H^{\ell}(\hat{E}_{\Gamma,y_{0}}^{[\ell]})\cong\QQ(0)\oplus\H$
where $W_{0}\H=\{0\}$.

\paragraph*{$\underline{\text{Conclusion}}$:}

$\hm{(}\QQ(0),H_{\ell}(\hat{E}_{\Gamma,y_{0}}^{[\ell]}))$ is $\{0\}$
in cases (A) and (B) (or for a smooth fiber), and one copy of $\QQ(0)$
for case (C).

\subsubsection{Residues and Beilinson-Hodge}

Let $\mathfrak{p}\subset Y_{\Gamma}\m\varepsilon_{\Gamma}$ be a finite
point set, and consider open subsets of $\begin{array}{c}
\E_{\Gamma}^{[\ell]}\\
\bigcap\\
\overline{\E}_{\Gamma}^{[\ell]}\end{array}$\[
\begin{array}{ccccc}
(\E_{\Gamma}^{[\ell]})^{\circ} & \rTo^{\pi_{\Gamma}^{\circ}} & Y_{\Gamma}^{\circ} & :=Y_{\Gamma}\m\varepsilon_{\Gamma}\cup\mathfrak{p} & =\overline{Y}_{\Gamma}\m\mathfrak{P}\\
\bigcap &  & \bigcap\\
(\overline{\E}_{\Gamma}^{[\ell]})^{\circ} & \rTo^{\overline{\pi}_{\Gamma}^{\circ}} & \overline{Y}_{\Gamma}^{\circ} & :=\overline{Y}_{\Gamma}\m\kappa_{\Gamma}^{[\ell]}\end{array}\]
where $\mathfrak{P}:=\kappa_{\Gamma}^{I}\cup\kappa_{\Gamma}^{I^{*}}\cup\varepsilon_{\Gamma}\cup\mathfrak{p}$,
and $\kappa_{\Gamma}^{[\ell]}:=\left\{ \begin{array}{c}
\kappa_{\Gamma},\,\,\ell\text{ odd}\\
\kappa_{\Gamma}^{I},\,\,\ell\text{ even}\end{array}\right.$ consists of the unipotent cusps. Applying $\hm{(}\QQ(0),\text{---}\otimes\QQ(\ell+1))$
to the {}``localization sequence''\[
0\to\text{coker}\left\{ \jmath_{\ell+1}^{*}:H^{\ell+1}(\overline{\E}_{\gamma}^{[\ell]})\to H^{\ell+1}((\E_{\Gamma}^{[\ell]})^{\circ})\right\} \rTo^{\oplus\frac{Res_{y_{0}}}{(2\pi i)^{\ell}}}\oplus_{y_{0}\in\mathfrak{P}}H_{\ell}(\qhat{}_{\Gamma,y_{0}}^{[\ell]})(-(\ell+1))\]
\[
\rTo^{(2\pi i)^{\ell+1}\left(\oplus(\imath_{y_{0}})_{*}\right)}\ker\left\{ \jmath_{\ell+2}^{*}:\, H^{\ell+2}(\overline{\E}_{\Gamma}^{[\ell]})\to H^{\ell+2}((\E_{\Gamma}^{[\ell]})^{\circ})\right\} \to0\]
gives \[
\hm{(}\QQ(0),\text{coker}(\jmath_{\ell+1}^{*})\otimes\QQ(\ell+1))\cong\oplus_{y_{0}\in\mathfrak{P}}\hm{(}\QQ(0),H_{\ell}(\qhat{}_{\Gamma,y_{0}}^{[\ell]}))\]
\[
\byt\oplus_{y_{0}\in\kappa_{\Gamma}^{[\ell]}}\QQ(0),\]
since $\ker(\jmath_{\ell+2}^{*})$ has pure weight $\ell+2$ (and
$\ell\geq1$). Using \[
0\to\text{im}(\jmath_{\ell+1}^{*})\to H^{\ell+1}((\E_{\Gamma}^{[\ell]})^{\circ})\to\text{coker}(\jmath_{\ell+1}^{*})\to0,\]
 we then clearly have $\hm{(}\QQ(0),H^{\ell+1}((\overline{E}_{\gamma}^{[\ell]})^{\circ},\QQ(\ell+1)))\subset$\[
\hm{\left(\QQ(0),H^{\ell+1}\left((\E_{\Gamma}^{[\ell]})^{\circ},\QQ(\ell+1)\right)\right)}\rInto^{\oplus\frac{Res}{(2\pi i)^{\ell}}}\oplus_{[\frac{r}{s}]\in\kappa_{\Gamma}^{[\ell]}}\QQ.\]

\begin{claim}
The composition\\
 \xymatrix{ 
CH^{\ell+1} \left( ( \overline{E}^{[\ell]}_{\Gamma} )^{\circ}, \ell+1 \right) \ar @{-->} [rrd] \ar [d]^{[\cdot]}
\\ 
\hm \left( \QQ(0), H^{\ell+1} \left( (\overline{E}^{[\ell]}_{\Gamma})^{\circ},\QQ(\ell+1) \right) \right) \ar [rr]_{\mspace{150mu} \oplus \frac{Res}{(2\pi i)^{\ell}}} 
& & 
\oplus_{[\frac{r}{s}]\in\kappa_{\Gamma}^{[\ell]}} \QQ  
}\\
\\
is surjective.
\end{claim}
If this is true, then we have clearly proved that for any $\mathfrak{P}$
as just described \[
CH^{\ell+1}\left((\E_{\Gamma}^{[\ell]})^{\circ},\ell+1\right)\twoheadrightarrow\hm{\left(\QQ(0),H^{\ell+1}\left((\E_{\Gamma}^{[\ell]})^{\circ},\QQ(\ell+1)\right)\right)},\]
which is the relevant special case of the \textbf{Beilinson-Hodge}
conjecture.

\subsubsection{Holomorphic forms of top degree}

Clearly on $\E^{[\ell]}(\to\mathfrak{H})$ these are of the form\[
\Omega_{F}^{\ell+1}:=F(\tau)dz_{1}\wedge\cdots\wedge dz_{\ell}\wedge d\tau,\]
for $F$ holomorphic ($F\in\mathcal{O}(\mathfrak{H})$). For this
to descend to $\E_{\Gamma}^{[\ell]}$ (recalling from $\S5.1.1$ the
action of $\gamma\in\Gamma\subset SL_{2}(\ZZ)$ on $\E^{[\ell]}\m\pi^{-1}(\widetilde{\varepsilon}_{\Gamma})$),
we must have\[
\Omega_{F}^{\ell+1}=\gamma^{*}\Omega_{F}^{\ell+1}=F(\gamma(\tau))\frac{dz_{1}}{c\tau+d}\wedge\cdots\wedge\frac{dz_{\ell}}{c\tau+d}\wedge\frac{\slt d\tau}{(c\tau+d)^{2}},\]
which is equivalent to \begin{equation}
F(\tau) = \frac{F(\gamma(\tau))}{(c\tau+d)^{\ell+2}} =: F|^{\ell+2}_{\gamma}  (\tau) \, \, (\forall \gamma\in\Gamma). \\
\end{equation}

\begin{defn}
(i) $F\in\mathcal{O}(\mathfrak{H})$ and $(5.2)$ holds$\Longleftrightarrow$
$F(\tau)$ an \textbf{automorphic form} of weight $\ell+2$ with respect
to $\Gamma$.\\
(ii) $\lim_{\tau\to i\infty}F(\tau)=:\mathfrak{R}_{[i\infty]}(F)<\infty$
$\Longleftrightarrow$ $F(\tau)$ \textbf{bounded} at $i\infty$.\\
(iii) $\mathfrak{R}_{[i\infty]}(F)=0$ $\Longleftrightarrow$ $F(\tau)$
\textbf{cusp} at $i\infty$.

Now assuming $F$ automorphic of weight $\ell+2$ (w.r.t. some $\Gamma$):\\
(iv) $F$ is \textbf{cusp (resp. bounded) at }$\mathbf{[\frac{r}{s}]}$
$\Longleftrightarrow$ $F|_{\dmat}^{\ell+2}$ cusp (resp. bounded)
at $i\infty$, where $p,q$ are chosen so that the matrix $\in SL_{2}(\ZZ)$;
and\\
(v) $F$ \textbf{cusp (resp. modular) form} of weight $\ell+2$ (w.r.t.
$\Gamma$) $\Longleftrightarrow$ $F$ cusp (resp. bounded) at every
cusp($\in\kappa_{\Gamma}$).\\

\end{defn}
\begin{rem*}
Unconventionally, a $meromorphic$ $modular$ $form$ will mean the
same thing as $modular$ $form$ except that poles at cusps $\kappa_{\Gamma}$
and elliptic points $\widetilde{\varepsilon}_{\Gamma}$ are permitted.
(For each cusp $[\frac{r}{s}]$, this means $\tau^{-K}F|_{\dmat}^{\ell+2}$
is bounded at $i\infty$ for some $K\in\ZZ^{+}$.) We write $A_{\ell+2}(\Gamma)$
(resp. $S_{\ell+2}(\Gamma)$, $M_{\ell+2}(\Gamma)$, $\check{M}_{\ell+2}(\Gamma)$)
for automorphic (resp. cusp, modular, mero. modular) forms.
\end{rem*}
\begin{example}
Let $F\in A_{\ell+2}(\Gamma)$. If the cusp $[i\infty]\in\kappa_{\Gamma}$
is type $I_{m}$ then $\emat\in\Gamma$, so that $F(\tau+m)=F(\tau)$;
if type $I_{m}^{*}$ then $\fmat\in\Gamma$, ensuring $F(\tau+m)=(-1)^{\ell+2}F(\tau)$.
Either way, $\tilde{q}:=q^{\frac{1}{m}}$ (see $\S5.1.3(C)$) gives
a local coordinate on $\overline{Y}_{\Gamma}$ at $[i\infty]$. In
the unipotent case, we conclude that $F$ has a Laurent expansion
$F(\tau)=\sum_{k\in\ZZ}a_{k}\tilde{q}^{k}$; in the non-unipotent
($I_{m}^{*}$ and $\ell$ odd) case we get instead $F(\tau)=\sum_{k\in\ZZ\text{ odd}}a_{k}\tilde{q}^{\frac{k}{2}}$
($\Omega_{F}$ still gives a well-defined holomorphic $form$ on the
quotient $\E_{\Gamma}$). Evidently, the {}``bounded'' condition
says in both cases that $a_{k}=0$ for $k<0$ (and {}``cusp'' forms
have no constant term); so in the non-unipotent case, bounded $\implies$
cusp.
\end{example}
Shokurov (\cite{So}) proved the following:

\begin{prop}
(i) $\Omega^{\ell+1}(\overline{\E}_{\Gamma}^{[\ell]}\m\pi^{-1}(\kappa_{\Gamma}))=\{\Omega_{F}\,|\, F\in A_{\ell+2}(\Gamma)\}$,
i.e. such $\Omega_{F}$ extend holomorphically across the singular
fibers over elliptic points;

(ii) $\Omega^{\ell+1}(\overline{\E}_{\Gamma}^{[\ell]})\left\langle \log(\overline{\pi}_{\Gamma}^{-1}(\kappa_{\Gamma}))\right\rangle =\Omega^{\ell+1}(\overline{\E}_{\Gamma}^{[\ell]})\left\langle \log(\overline{\pi}_{\Gamma}^{-1}(\kappa_{\Gamma}^{[\ell]})\right\rangle =\{\Omega_{F}\,|\, F\in M_{\ell+2}(\Gamma)\}$;
and

(iii) $\Omega^{\ell+1}(\overline{\E}_{\Gamma}^{[\ell]})=\{\Omega_{F}\,|\, F\in S_{\ell+2}(\Gamma)\}$.
\end{prop}
This gives the dictionary between automorphic forms and holomorphic
forms that we will need. To start relating modular forms to Beilinson-Hodge,
make the following 

\begin{defn}
Given $F\in M_{\ell+2}(\Gamma)$ and $[\frac{r}{s}]\in\kappa_{\Gamma}^{[\ell]}$,
take any $\dmat\in SL_{2}(\ZZ)$ and set \[
\mathfrak{R}_{[\frac{r}{s}]}(F):=\lim_{\tau\to i\infty}F|_{\dmat}^{\ell+2}(\tau)=\lim_{\tau\to i\infty}\frac{F(\frac{r\tau-q}{s\tau+p})}{(s\tau+p)^{\ell+2}}\in\CC.\]

\end{defn}
This gives an interpretation of \textbf{residues}, in the sense that
the following diagram commutes:\\
\xymatrix{
\hm \left( \QQ(0) , H^{\ell+1} \left( ( \overline{\E}^{[\ell]}_{\Gamma} )^{\circ} , \QQ(\ell+1) \right) \right) 
\ar @{^(->} [d] \ar @/_7pc/[dd]_{=:\theta_{\ell+2}}
\ar @{_(->} [rrr]^{ \mspace{100mu} \oplus \frac{Res_{[\frac{r}{s}]}}{(2\pi i)^{\ell}}=:Res } & & & \oplus_{[\frac{r}{s}]\in \kappa_{\Gamma}^{[\ell]}} \QQ \ar @{^(->} [dd]   
\\ 
\Omega^{\ell+1}\left( \overline{\E}^{[\ell]}_{\Gamma} \right) \left\langle \log(\overline{\pi}^{-1}(\kappa_{\Gamma}^{[\ell]})) \right\rangle
\ar @{-->} [rrrd] 
\\ 
M_{\ell+2}(\Gamma) \ar [u]_{\cong} \ar [rrr]_{\oplus \mathfrak{R}_{[\frac{r}{s}]}=:\mathfrak{R}} 
& & & 
\oplus_{[\frac{r}{s}]\in \kappa_{\Gamma}^{[\ell]}} \CC
}\\
\\
where the vertical isomorphism sends $F\mapsto(2\pi i)^{\ell+1}\Omega_{F}$.

\begin{defn}
$M_{\ell+2}^{\QQ}(\Gamma):=\text{im}(\Theta_{\ell+2})=$modular forms
corresponding to holomorphic forms with log poles (at cuspidal fibers)
\emph{and rational periods}.
\end{defn}
By pure thought we have

\begin{prop}
(i) $\mathfrak{R}$ is surjective;

(ii) $\mathfrak{R}|_{M_{\ell+2}^{\QQ}\otimes\CC}$ is injective; and

(iii) $(M_{\ell+2}^{\QQ}\otimes\CC)\oplus S_{\ell+2}\hookrightarrow M_{\ell+2}(\Gamma)$.
\end{prop}
\begin{proof}
Since $\ker(\mathfrak{R})=S_{\ell+2}(\Gamma)$, the kernel of the
dotted arrow is actually $\Omega^{\ell+1}(\overline{E}_{\Gamma}^{[\ell]})$.
This arrow must $surject$, since the $\oplus\QQ$'s (hance $\oplus\CC$'s)
correspond to weight $2\ell+2>\ell+2$ in $\oplus H_{\ell}(\hat{E}_{\Gamma,[\frac{r}{s}]}^{[\ell]})(-(\ell+1))$
(hence cannot be absorbed by the next term in the localization sequence);
(i) follows. Injectivity of $Res$ $\implies$(ii), which $\implies$(iii).
\end{proof}
Now if $Claim$ $5.1$ holds, we have also $M_{\ell+2}^{\QQ}(\Gamma)\twoheadrightarrow\oplus\QQ$,
hence $M^{\QQ}\otimes\CC\twoheadrightarrow\oplus\CC$ (hence $\cong$),
which would imply \begin{equation}
M_{\ell+2}(\Gamma)= \left( M^{\QQ}_{\ell+2}(\Gamma) \otimes \CC \right) \oplus S_{\ell+2}(\Gamma). \\
\end{equation}

\subsubsection{Reduction to \emph{(}$\Gamma=$\emph{)}$\Gamma(N)$}

Assume $SL_{2}(\ZZ)\supset\Gamma\supset\Gamma(N)$. Since $\Gamma(N)\trianglelefteq SL_{2}(\ZZ)$,
$\Gamma(N)\trianglelefteq\Gamma$ and the coset representatives $\{\gamma_{i}\}_{i=1}^{[\Gamma:\Gamma(N)]}$
act on the sheets of the branched cover $\overline{Y}_{\Gamma(N)}\rTo^{\overline{\rho}}\overline{Y}_{\Gamma}$,
and also on\\
\xymatrix{& & & \E^{[\ell]}_{\Gamma(N)} \m \pi^{-1}_{\Gamma(N)}(\overline{\rho}^{-1}(\varepsilon_{\Gamma})) \ar [r]^{\mspace{100mu}\mathcal{P}^{[\ell]}_{\nicefrac{\Gamma(N)}{\Gamma}}} \ar [d]^{\pi_{\Gamma(N)}} & \E_{\Gamma}^{[\ell]} \ar [d]^{\pi_{\Gamma}} \\ & & & Y_{\Gamma(N)}\m \overline{\rho}^{-1}(\varepsilon_{\Gamma}) \ar [r]^{\mspace{50mu} \rho_{\nicefrac{\Gamma(N)}{\Gamma}}}  & Y_{\Gamma}\m \varepsilon_{\Gamma} .
}\\
\\
We can interpret the action of this $\mathcal{P}$ on holomorphic
forms (and eventually, algebraic cycles) in terms of modular forms
and residues:\\
\xymatrix{
\Omega^{\ell+1} \left( \overline{\E}^{[\ell]}_{\Gamma(N)} \right) \left\langle \log \overline{\pi}^{-1}_{\Gamma(N)} (\overline{\rho}^{-1}(\varepsilon_{\Gamma})\cup \kappa_{\Gamma(N)}) \right\rangle \ar @<1ex> [r]^{\mspace{80mu} \mathcal{P}_*}  & \Omega^{\ell+1}\left( \overline{\E}^{[\ell]}_{\Gamma}\right) \left\langle \log \overline{\pi}^{-1}_{\Gamma} (\kappa_{\Gamma}) \right\rangle \ar @<1ex> [l]^{\mspace{80mu} \mathcal{P}^*}
\\
M_{\ell+2}(\Gamma(N)) \ar @<1ex> [r]^{F(\tau)\mapsto \sum_i F|_{\gamma_i}^{\ell+2}(\tau)} \ar [u]_{\cong} \ar [d]^{\mathfrak{R}} &  M_{\ell+2}(\Gamma) \ar [u]_{\cong} \ar @<1ex> [l]^{F(\tau)\leftarrow F(\tau)} \ar [d]^{\mathfrak{R}}
\\
\Upsilon_2(N):=\oplus_{\kappa_{_{\Gamma(N)}}^{[\ell]}}\CC \mspace{80mu} \ar @<1ex> [r]^{\text{trace}\,\,\mathsf{T}^{[\ell]}_{\nicefrac{\Gamma(N)}{\Gamma}}}_{\text{(of }\CC\text{-valued functions on cusps)}} & \mspace{80mu} \oplus_{\kappa_{\Gamma}^{[\ell]}}\CC =: \Upsilon_2(\Gamma) \ar @<1ex> [l]^{\text{pull-back }\mathsf{P}^{[\ell]}_{\nicefrac{\Gamma(N)}{\Gamma}}}
}\\
\\
More precisely (for the {}``trace''): given $[\frac{r_{0}}{s_{0}}]\in\kappa_{\Gamma}^{[\ell]}$,
the image of an element $\{\beta:\,\kappa_{\Gamma(N)}^{[\ell]}\to\CC\}\in\Upsilon_{2}(N)$
takes value $(\mathsf{T}_{*}\beta)([\frac{r_{0}}{s_{0}}])=\sum_{[\frac{r}{s}]\in\overline{\rho}^{-1}([\frac{r_{0}}{s_{0}}])}\text{ord}_{[\frac{r}{s}]}(\overline{\rho})\cdot\beta([\frac{r}{s}])$.
This map is surjective since unipotent cusps cover unipotent cusps;
though when $\ell$ is odd, unipotent ($[\frac{r}{s}]\in\kappa_{\Gamma(N)}^{[\ell]}$)
can map to non-unipotent ($[\frac{r}{s}]\in\kappa_{\Gamma}^{I^{*}}$),
in which case the value is lost.

The main point is that \[
\text{Claim 5.1 (hence Beilinson-Hodge) for }\Gamma(N)\,\,\implies\,\,\text{Claim 5.1 for }\Gamma,\]
since the trace surjects and one can use $\mathcal{P}_{*}$ on higher
Chow cycles, to push them from $\E_{\Gamma(N)}^{[\ell]}$ to $\E_{\Gamma}^{[\ell]}$.
We write $\overline{Y}_{\Gamma(N)}=:\overline{Y}(N)$, $\kappa_{\Gamma(N)}=:\kappa(N)$,
etc. for simplicity.

\emph{Why do we want to do make this reduction?} $Y(N)$ is the moduli
space of elliptic curves with {}``completely marked $N$-torsion''
(in particular, 2 marked generators), so $\E(N)$($:=\E_{\Gamma(N)}$)
has $N^{2}$ $N$-torsion sections --- ideal for building relative
higher Chow cycles (from functions with divisors supported on that
$N$-torsion). Also, all cusps are (unipotent) of type $I_{N}$.%
\footnote{One reason why we exclude $N=2$ is that this is $false$ --- there
are two cusps of type $I_{2}$ and one of type $I_{2}^{*}$.%
} The downside is that $\overline{Y}(N)$ has genus zero only for $N=(2,)3,4,5$. 

For the cusps, writing $\mathfrak{G}(N)$ for the set of subgroups
of $(\nicefrac{\ZZ}{N\ZZ})^{2}$ isomorphic to $\nicefrac{\ZZ}{N\ZZ}$,
we have $\kappa^{[\ell]}(N)=$\[
\kappa(N)=\frac{\left\{ \left.(-s,r)\in(\nicefrac{\ZZ}{N\ZZ})^{2}\,\right|\,|\left\langle (-s,r)\right\rangle |=N\right\} }{\left\langle (-s,r)\sim(s,-r)\right\rangle }=\bigcup_{G\in\mathfrak{G}(N)}\left.G^{*}\right/\left\langle \pm1\right\rangle \]
\[
\cong PSL_{2}(\ZZ/N\ZZ)\left/\left\langle \left(\begin{array}{cc}
1 & *\\
0 & 1\end{array}\right)\right\rangle \right.\,;\]
since each $G\in\mathfrak{G}(N)$ has $|G^{*}|=\phi_{\text{euler}}(N)$,\[
|\kappa(N)|=\frac{\phi_{\text{euler}}(N)}{2}\cdot|\mathfrak{G}(N)|=\frac{N}{2}\prod_{p|N}\left(1-\frac{1}{p}\right)\cdot N\prod_{p|N}\left(1+\frac{1}{p}\right)=\frac{N^{2}}{2}\prod_{p|N}\left(1-\frac{1}{p^{2}}\right).\]
Now given a field $K\subseteq\CC$ set%
\footnote{notationally, we drop $m=1$ or $K=\CC$%
} \[
\Phi_{m}^{K}(N):=\left\{ K\text{-valued functions on }(\ZZ/N\ZZ)^{m}\right\} \]
\[
\Phi_{m}^{K}(N)_{\circ}:=\left\{ \varphi\in\Phi_{m}^{K}(N)\,|\,\varphi(\bar{0},\ldots,\bar{0})=0)\right\} \]
\[
\Phi_{m}^{K}(N)^{\circ}:=\ker\left\{ \text{augmentation map:}\,\Phi_{m}^{K}(N)\to K\right\} .\]
Ultimately, $\Phi_{2}^{K}(N)^{\circ}$ will be divisors ($\otimes\QQ$)
of degree $0$ on $N$-torsion. 

Choose once and for all a representative $(-s,r)$ for each cusp $\sigma\in\kappa(N)$
(s.t. $\sigma=[\frac{r}{s}]$) and a matrix $\cmat\in SL_{2}(\ZZ)$.
Writing \[
\begin{array}{cc}
\pi_{[\frac{r}{s}]}:(\ZZ/N\ZZ)^{2}\twoheadrightarrow\ZZ/N\ZZ\,\,\,,\,\,\,\,\,\,\,\, & \iota_{[\frac{r}{s}]}:\ZZ/N\ZZ\hookrightarrow(\ZZ/N\ZZ)^{2},\\
(m,n)=a(p,q)+b(-s,r)\mapsto a\,\,\,,\,\,\,\,\,\,\,\, & a\mapsto a(-s,r),\end{array}\]
one has \[
(\pi_{[\frac{r}{s}]})_{*}:\Phi_{2}(N)^{(\circ)}\rOnto^{\text{trace}}\Phi(N)^{(\circ)}\,\,\,,\,\,\,\,\,\,\,\,(\iota_{[\frac{r}{s}]})^{*}:\Phi_{2}(N)_{(\circ)}\rOnto^{\text{pullback}}\Phi(N)_{(\circ)},\]
etc.

\subsection{Divisors with $N$-torsion support}

Here we collect together related material on finite Fourier transforms,
$L$-functions, and meromorphic functions on $\E(N)$ with divisors
supported on the $N$-torsion sections. The technical {}``$(p,q)$-vertical''
subsection will be used in $\S7$ to compute the $AJ$ map.

\subsubsection{Some Fourier theory}

We define Fourier transforms \[
\widehat{\,\,}:\,\Phi(N)^{(\circ)}\rTo^{\cong}\Phi(N)_{(\circ)}\]
\[
\varphi(a)\mapsto\widehat{\varphi}(k):=\sum_{a\in\nicefrac{\ZZ}{N\ZZ}}\varphi(a)e^{-\frac{2\pi i}{N}ka}\]
\[
\widehat{\,\,}:\Phi_{2}(N)^{(\circ)}\rTo^{\cong}\Phi_{2}(N)_{(\circ)}\]
\[
\varphi(m,n)\mapsto\widehat{\varphi}(\mu,\eta):=\sum_{(m,n)\in(\nicefrac{\ZZ}{N\ZZ})^{2}}\varphi(m,n)e^{\frac{2\pi i}{N}(\mu n-\eta m)}.\]
One can show (easily) that for $\varphi_{0}\in\Phi(N)$, $\varphi\in\Phi_{2}(N)$
\begin{equation}
\frac{1}{N} \cdot \widehat{\pi^*_{[\frac{r}{s}]}\varphi_0} = (\iota_{[\frac{r}{s}]})_* \widehat{\varphi_0} ,\\
\end{equation} \begin{equation} 
\widehat{(\pi_{[\frac{r}{s}]})_* \varphi} = \iota_{[\frac{r}{s}]}^* \widehat{\varphi} , \\
\end{equation} and also $(\pi_{[\frac{r}{s}]}^{*}\widehat{\varphi_{0}})(\cdot)=\widehat{(\iota_{[\frac{r}{s}]})_{*}\varphi_{0}}(-\cdot)$.%
\footnote{Note that for $N$ prime, one has (dividing by $\frac{\phi_{\text{euler}}(N)}{2}=$number
of cusps {}``in'' each $\ZZ/N\ZZ$ subgroup) $\widehat{\varphi}=\frac{2}{\phi_{\text{euler}}(N)}\sum_{\sigma\in\kappa(N)}(\iota_{[\frac{r}{s}]})_{*}(\iota_{[\frac{r}{s}]})^{*}\widehat{\varphi}$
$\implies$ $\varphi=\frac{2}{N\cdot\phi_{\text{euler}}(N)}\sum_{\sigma}(\pi_{[\frac{r}{s}]})^{*}(\pi_{[\frac{r}{s}]})_{*}\varphi$
for $\varphi\in\Phi_{2}(N)^{\circ}$ but this doesn't hold for $N$
$not$ prime.%
} Finally, if $\mu_{a}:\ZZ/N\ZZ\rTo^{\cong}\ZZ/N\ZZ$ is multiplication
(mod $N$) by $a\in(\ZZ/N\ZZ)^{*}$, one has \begin{equation}
\widehat{\mu_a^*\varphi_0} = \mu_{a^{-1}}^*\widehat{\varphi_0} . \\
\end{equation}

One wonders why undergraduates don't learn these discrete Fourier
transforms in linear algebra (or at least before the continuous/$L^{2}$/$L^{1}$
theory), considering that future mathematicians might use them in
number theory and engineers in MATLAB. Moreover, together with Bernoulli
numbers and polynomials, they have a very attractive application to
computing series yielding rational multiples of powers of $\pi$.
Recall that the Bernoulli numbers \[
B_{0}=1,\, B_{1}=-\frac{1}{2},\, B_{2}=\frac{1}{6},\, B_{3}=0,\, B_{4}=\frac{-1}{30},\, B_{5}=0,\,\text{etc.}\]
satisfy $\sum_{k=0}^{\infty}B_{k}\frac{t^{k}}{k!}=\frac{te^{t}}{e^{t}-1}.$
If we define Bernoulli polynomials \[
B_{k}(x):=\sum_{j=0}^{k}{k \choose j}B_{j}x^{k-j}\]
(e.g., $B_{3}(x)=x^{3}-\frac{3}{2}x^{2}+\frac{1}{2}x$, $B_{4}(x)=x^{4}-2x^{3}+x^{2}-\frac{1}{30}$)
then they consequently satisfy $\sum_{k=0}^{\infty}B_{k}(x)\frac{t^{k}}{k!}=\frac{te^{t(1+x)}}{e^{t}-1}.$
One also has (for $k\geq2$) $B_{k}=\left\{ \begin{array}{c}
\frac{-k!}{(2\pi i)^{k}}2\zeta(k),\,\, k\text{ even}\\
\,\,\,\,\,\,\,\,0,\,\,\,\,\,\,\,\,\,\,\, k\text{ odd}\end{array}\right.$ and correspondingly $B_{k}(x)=\frac{(-1)^{k-1}k!}{(2\pi i)^{k}}\sum_{m\in\ZZ}'\frac{e^{-2\pi imx}}{m^{k}}.$

For us the key calculation is: given $\varphi\in\Phi(N)$ (and $\ell\geq1$),\[
\sum_{a=0}^{N-1}\varphi(a)B_{\ell+2}(\frac{a}{N})\,\,=\,\,\frac{(-1)^{\ell+1}(\ell+2)!}{(2\pi i)^{\ell+2}}\sum_{a=0}^{N-1}\varphi(a)\sum_{m\in\ZZ}{}^{'}\frac{e^{-2\pi im\frac{a}{N}}}{m^{\ell+2}}\]
\[
=\,\,\frac{(-1)^{\ell+1}(\ell+2)!}{(2\pi i)^{\ell+2}}\sum_{m\in\ZZ}{}^{'}\frac{1}{m^{\ell+2}}\begin{array}[t]{c}
\underbrace{\sum_{a=0}^{N-1}\varphi(a)e^{-\frac{2\pi i}{N}ma}}\\
\widehat{\varphi}(m)\end{array}\]
\[
=\,\,\frac{(-1)^{\ell+1}(\ell+2)!}{(2\pi i)^{\ell+2}}\tilde{L}(\widehat{\varphi},\ell+2),\]
where $\tilde{L}(\widehat{\varphi},\ell+2):=\sum_{m\in\ZZ}'\frac{\widehat{\varphi}(m)}{m^{\ell+2}}$
(thinking of $\widehat{\varphi}$ as an $N$-periodic function on
$\ZZ$). Note that, by this calculation, if $\varphi\in\Phi^{\QQ}(N)$
then $regardless$ of rationality of $\widehat{\varphi}$, $\tilde{L}(\widehat{\varphi},\ell+2)$
is always in $\QQ(\ell+2)$.

\begin{example}
(for the undergraduates) $N=4$, $\varphi=0,1,0,-1;\ldots$ $\FT$
$\widehat{\varphi}=0,2i,0,-2i;\ldots$. Say we want to compute $1-\frac{1}{3^{3}}+\frac{1}{5^{3}}-\frac{1}{7^{3}}+\cdots=\sum_{M\geq0}\frac{(-1)^{M}}{(2M+1)^{3}}.$
This is $\frac{1}{2}\cdot\frac{1}{(-2i)}\cdot\sum_{m\in\ZZ}{}'\frac{\widehat{\varphi}(m)}{m^{3}}=\frac{-1}{4i}\cdot\frac{(2\pi i)^{3}}{(-1)^{2}3!}\sum_{a=0}^{3}\varphi(a)B_{3}(\frac{a}{4})=\frac{-8\pi^{3}i}{-4i\cdot6}(B_{3}(\frac{1}{4})-B_{3}(\frac{3}{4}))=\frac{\pi^{3}}{3}\left(\frac{3}{64}-\left(\frac{-3}{64}\right)\right)=\frac{\pi^{3}}{32}.$
Much more complicated rational numbers (than $\frac{1}{32}$) usually
arise.
\end{example}

\subsubsection{The horospherical map}

Now we establish the central number-theoretic Lemma 5.9 which will
ultimately translate to {}``surjectivity of residues of higher Chow
cycle classes onto the cusps,'' hence Beilinson-Hodge. Define for
$\sigma\in\kappa(N)$, $\QQ\subseteq K\subseteq\CC$\[
\mathsf{H}_{\sigma}^{[\ell]}:\Phi_{2}^{K}(N)^{\circ}\to K\mspace{200mu}\]
\[
\mspace{200mu}\varphi\mapsto\frac{(-1)^{\ell}(\ell+1)}{(\ell+2)!}\sum_{a=0}^{N-1}\left((\pi_{\sigma})_{*}\varphi\right)(a)\cdot B_{\ell+2}\left(\frac{a}{N}\right).\]
If the following is true for $K=\CC$ then it holds for any $K$:

\begin{lem}
$\left(\oplus_{\sigma\in\kappa(N)}\mathsf{H}_{\sigma}^{[\ell]}\right):\,\Phi_{2}^{K}(N)^{\circ}\to\Upsilon_{2}^{K}(N)$
is surjective.
\end{lem}
\begin{proof}
Let ($\Phi(N)^{\circ}\supset$) \[
\Upsilon^{[\ell]}(N):=\left\{ \text{functions on }(\nicefrac{\ZZ}{N\ZZ})^{*}\text{ satisfying }f(-y)=(-1)^{\ell}f(y)\right\} \]
\[
\cong\left\{ \text{functions on those cusps }(-s,r)\text{ "contained" in any one }G\in\mathfrak{G}(N)\right\} .\]
Writing \[
\mathsf{L}^{[\ell]}:\,\Phi(N)_{\circ}\to\CC\mspace{100mu}\]
\[
\mspace{100mu}\xi\mapsto-\frac{\ell+1}{(2\pi i)^{\ell+2}}\tilde{L}(\xi,\ell+2),\]
by results of $\S5.2.1$ we have\[
\oplus_{\sigma}\mathsf{H}_{\sigma}^{[\ell]}\,=\,\oplus_{\sigma}\mathsf{L}^{[\ell]}\circ\widehat{\,\,}\circ(\pi_{\sigma})_{*}\,=\,\oplus_{\sigma}\mathsf{L}^{[\ell]}\circ\iota_{\sigma}^{*}\circ\widehat{\,\,}\,=\,\oplus_{_{G\in\mathfrak{G}(N)}}\oplus_{_{a\in(\nicefrac{\ZZ}{N\ZZ})^{*}}}\mathsf{L}^{[\ell]}\circ\mu_{a}^{*}\circ\iota_{\sigma_{G}}^{*}\circ\widehat{\,\,},\]
for $\sigma_{G}$ some choice of generator $(-s,r)$ for each $G\subset(\ZZ/N\ZZ)^{2}$.
Obviously \[
(\oplus_{_{G\in\mathfrak{G}(N)}}\Upsilon^{[\ell]}(N))\subseteq\text{image}\{\oplus_{_{G\in\mathfrak{G}(N)}}\iota_{\sigma_{G}}^{*}:\,\Phi_{2}(N)_{\circ}\to\oplus_{_{\mathfrak{G}(N)}}\Phi(N)_{\circ}\},\]
and $\widehat{\,\,}:\,\Phi_{2}(N)^{\circ}\to\Phi_{2}(N)_{\circ}$
is also obviously surjective; so it will suffice to check the following$\vspace{2mm}$\\
$\underline{Sublemma}:$ $\left(\oplus_{_{a\in(\nicefrac{\ZZ}{N\ZZ})^{*}}}\mathsf{L}^{[\ell]}\circ\mu_{a}^{*}\right)\left|_{_{\Upsilon(N)}}\right.:\,\Upsilon^{[\ell]}(N)\left(\subseteq\Phi(N)_{\circ}\right)\mapf\Upsilon^{[\ell]}(N).$$\vspace{2mm}$\\
$\underline{Pf.}:$ Working over $\CC$, $\Upsilon^{[\ell]}(N)$ is
spanned (depending on $\ell$) by even or odd Dirichlet characters
(mod $N$) $\{\chi_{i}\}_{i=1}^{\frac{1}{2}\phi_{\text{euler}}(N)}$.
These satisfy (by definition) $(\mu_{a}^{*}\chi)(b)=\chi(a)\cdot\chi(b)$.
So $\left(\mathsf{L}^{[\ell]}\circ\mu_{a}^{*}\right)(\chi_{i})=\chi_{i}(a)\cdot\mathsf{L}^{[\ell]}(\chi_{i})$,
and by \cite[sec. VII.2]{Ne1} $\tilde{L}(\chi_{i},\ell+2)\neq0$.
We may therefore $divide$ $\frac{\chi_{i}(\cdot)}{\mathsf{L}^{[\ell]}(\chi_{i})}=:\tilde{\chi}_{i}(\cdot)$,
so that $\left(\mathsf{L}^{[\ell]}\circ\mu_{a}^{*}\right)(\tilde{\chi}_{i})=\chi_{i}(a)$.
Thus each $\chi_{i}$ appears in the image (in $\Upsilon^{[\ell]}(N)$)
of this map, and since they span $\Upsilon^{[\ell]}(N)$ we are done.
\end{proof}
We can be more explicit and produce a {}``rational basis for the
surjection'' of lemma 5.9 (onto $\Upsilon^{[\ell]}(N)$).

\begin{prop}
There exists a unique {}``fundamental vector'' $\varphi_{N}^{[\ell]}\in\Phi^{\QQ}(N)^{\circ}$
satisfying $\mathsf{H}_{\sigma'}^{[\ell]}\left(\frac{1}{N}\pi_{\sigma}^{*}(\varphi_{N}^{[\ell]})\right)=\delta_{\sigma\sigma'}\,\,(\forall\sigma,\sigma'\in\kappa(N))$
.
\end{prop}
\begin{proof}
The proof of the sublemma implies the existence of $\varphi\in\Phi(N)^{\circ}$
with $(i)$ $\mathsf{L}^{[\ell]}(\widehat{\varphi})=1$, $(ii)$$\mathsf{L}^{[\ell]}(\widehat{\mu_{a}^{*}\varphi})=0$
$\forall a\in(\ZZ/N\ZZ)^{*}\m\{\pm1\}$, and $(iii)$ $\widehat{\varphi}(n)=0$
$\forall n$ not relatively prime to $N$. If we ask that $(iv)$
$\varphi(-a)=(-1)^{\ell}\varphi(a)$ ($\forall a$), then $\varphi$
is uniquely determined. Conditions $(i)$-$(iii)$ translate to (somewhat
redundantly expressed) $\QQ$-linear conditions on $\varphi$:

$(i')$ $1=\frac{(-1)^{\ell}(\ell+1)}{(\ell+2)!}\sum_{c=0}^{N-1}\varphi(c)B_{\ell+2}\left(\frac{c}{N}\right)$

$(ii')$ $0=\sum_{c=0}^{N-1}\varphi(ac)B_{\ell+2}\left(\frac{c}{N}\right)$
($\forall a\nequiv\pm1$ $(N)$ with $\gcd(a,N)=1$)

$(iii')$ $0=\sum_{b=0}^{r-1}\varphi(a+b\frac{N}{r})$ ($\forall a=0,\ldots,\frac{N}{r}-1$)
for each $r(\neq1,N$) dividing $N$.\\
Then $\mathsf{H}_{\sigma'}^{[\ell]}(\frac{1}{N}\pi_{\sigma}^{*}\varphi)=\mathsf{L}^{[\ell]}(\frac{1}{N}\widehat{\pi_{\sigma'_{*}}\pi_{\sigma}^{*}\varphi})=\mathsf{L}^{[\ell]}(\iota_{\sigma'}^{*}\iota_{\sigma_{*}}\widehat{\varphi}),$
which is $0$ if $\sigma'$ {}``belongs to a different subgroup''
than $\sigma$ (using condition $(iii)$ if $N$ is not prime); otherwise
it becomes $\mathsf{L}^{[\ell]}(\mu_{a^{-1}}^{*}\widehat{\varphi})$
($=0$ if $\sigma'\nequiv\sigma${[}$\leftrightarrow a\nequiv\pm1$],
by $(ii)$; or $=1$ by $(i)$).
\end{proof}
\begin{example}
Here are a few of the fundamental vectors for $\ell=1,2$ (where we
list the values $\varphi(0),\ldots,\varphi(N-1)$)\[
\varphi_{3}^{[1]}=0,-\frac{81}{2},\frac{81}{2}\,;\,\,\,\,\,\,\,\,\,\varphi_{4}^{[1]}=0,-32,0,32\,;\,\,\,\,\,\,\,\,\,\varphi_{5}^{[1]}=0,-25,-\frac{25}{2},\frac{25}{2},25\,;\]
\[
\varphi_{3}^{[2]}=-162,81,81\,;\,\,\,\,\,\,\,\,\,\varphi_{6}^{[2]}=-\frac{432}{5},-\frac{216}{5},\frac{216}{5},\frac{432}{5},\frac{216}{5},-\frac{216}{5}.\]

\end{example}

\subsubsection{Pontryagin products}

Consider the map\[
\left(\Phi_{2}^{\QQ}(N)^{\circ}\right)^{\otimes\ell+1}\rTo^{*^{^{\ell+1}}}\Phi_{2}^{\QQ}(N)^{\circ}\mspace{300mu}\]
\[
\varphi_{1}\otimes\cdots\otimes\varphi_{\ell+1}\mapsto(\varphi_{1}*\cdots*\varphi_{\ell+1})(m,n):=\sum_{{\{m_{i},n_{i}\}\in(\nicefrac{\ZZ}{N\ZZ})^{2\ell+2}\atop \sum(m_{i},n_{i})\emn(m,n)}}\prod_{i=1}^{\ell+1}\varphi_{i}(m_{i},n_{i})\]
which becomes Pontryagin product when $\Phi_{2}(N)^{\circ}$ is interpreted
as divisors on $N$-torsion.

\begin{lem}
(i) $*^{^{\ell+1}}$ is surjective;

(ii) $\widehat{\varphi_{1}*\cdots*\varphi_{\ell+1}}=\prod_{i=1}^{\ell+1}\widehat{\varphi_{i}}$.
\end{lem}
\begin{proof}
$(ii)$ is a trivial computation.

For $(i)$ write $\beta_{N}(m,n):=\left\{ \begin{array}{c}
\frac{N^{2}-1}{N^{2}}\,\,\,\,(m,n)\equiv(0,0)\\
\,\,\,\frac{1}{N^{2}}\,\,\,\,\text{otherwise}\,\,\,\,\,\end{array}\right.,$ and let $\varphi\in\Phi_{2}^{\QQ}(N)^{\circ}.$ Then $\varphi*\begin{array}[t]{c}
\underbrace{\beta_{N}*\cdots*\beta_{N}}\\
\ell\text{ times}\end{array}=\varphi$.
\end{proof}

\subsubsection{Decomposition into $(p,q)$-verticals}

For $(p,q)\in(\ZZ/N\ZZ)^{2}$ such that $\left\langle (p,q)\right\rangle \cong\ZZ/N\ZZ$,
define in $\Phi_{2}^{\QQ}(N)^{\circ}$ a subgroup of {}``$\mathbf{(p,q)}$\textbf{-vertical}-degree-$0$''
functions\[
\Phi_{2}^{\QQ}(N)_{(p,q)}^{\circ}:=\left\{ \varphi\in\Phi_{2}^{\QQ}(N)\,|\,\sum_{a\in\nicefrac{\ZZ}{N\ZZ}}\varphi(a(p,q)+(m,n))=0\,\,\forall(m,n)\in(\ZZ/N\ZZ)^{2}\right\} .\]
Inside this we have the $set$\[
\mathfrak{S}(N)_{(p,q)}:=\small\left\{ \text{translates of the function }\varphi_{(p,q)}(m,n):=\left\{ \begin{array}{c}
-2,\,\,(m,n)\emn(0,0)\,\,\,\\
\,1,\,\,(m,n)\emn\pm(p,q)\\
\,0,\,\,\text{otherwise}\,\,\,\,\,\,\,\,\,\,\,\,\,\,\,\,\end{array}\right.\right\} .\normalsize\]
The next result says that $\varphi\in\Phi_{2}^{\QQ}(N)^{\circ}$ can
be written as a sum of Pontryagin products where each term contains
only functions from $\mathfrak{S}(N)_{(p,q)}$ for some $(p,q)$.$\vspace{2mm}$\\
\textbf{Decomposition Lemma:} \emph{(i) the map \[
\QQ[\mathfrak{S}(N)_{(p,q)}^{\times(\ell+1)}]\to\Phi_{2}^{\QQ}(N)_{(p,q)}^{\circ}\]
\[
\sum a_{j}[(\varphi_{1}^{(j)},\ldots,\varphi_{\ell+1}^{(j)})]\mapsto\sum a_{j}\varphi_{1}^{(j)}*\cdots*\varphi_{\ell+1}^{(j)}\]
is surjective ($\ell\geq0$);}

\emph{(ii) If $\sigma\in\kappa(N)$ corresponds to $\cmat\in SL_{2}(\ZZ)$
(see the end of $\S5.1.6$), then \[
\Phi_{2}^{\QQ}(N)_{(p,q)}^{\circ}\supset\pi_{\sigma}^{*}\Phi^{\QQ}(N)^{\circ};\]
}

\emph{(iii) $\oplus_{_{G\in\mathfrak{G}(N)}}\pi_{\sigma_{G}}^{*}\Phi^{\QQ}(N)^{\circ}\twoheadrightarrow\Phi_{2}^{\QQ}(N)^{\circ}$
($\sigma_{G}$ as in the proof of Lemma 5.9).}

\begin{proof}
$(i)$ First note $\otimes^{\ell+1}\Phi_{2}^{\QQ}(N)_{(p,q)}^{\circ}\rOnto^{*^{^{\ell+1}}}\Phi_{2}^{\QQ}(N)_{(p,q)}^{\circ}$
using \[
\beta_{N}^{(p,q)}(m,n):=\left\{ \begin{array}{c}
\frac{N-1}{N},\,\,(m,n)\emn(0,0)\\
\frac{1}{N},\,\,\,\,\,(m,n)\in\left\langle (p,q)\right\rangle \setminus\{(0,0)\}\\
0,\,\,\,\,\,\,\text{otherwise}\,\,\,\,\,\,\,\,\,\end{array}\right.\]
 in place of $\beta_{N}$ above; so it suffices to prove case $\ell=0$.
Put $\varphi_{(p,q)}^{\{k\}}(m,n):=\varphi_{(p,q)}((m,n)-k(p,q))$
and $\Delta_{(p,q)}(m,n):=\left\{ \begin{array}{c}
\,\,1,\,\,(m,n)\emn(p,q)\\
-1,\,\,(m,n)\emn(0,0)\\
\,\,0,\,\,\text{otherwise}\,\,\,\,\,\,\,\end{array}\right..$ Translates of $\Delta_{(p,q)}$ clearly generate $\Phi_{2}^{\QQ}(N)_{(p,q)}^{\circ}$,
and $\sum_{k=1}^{N}\frac{k}{N}\varphi_{(p,q)}^{\{k\}}=\Delta_{(p,q)}.$

$(ii)$ obvious.

$(iii)$ Follows from \[
\widehat{\Phi_{2}(N)^{\circ}}=\Phi_{2}(N)_{\circ}=\sum_{_{G\in\mathfrak{G}(N)}}(\iota_{\sigma_{G}})_{*}(\Phi(N)_{\circ})=\sum_{_{G\in\mathfrak{G}(N)}}\widehat{(\pi_{\sigma_{G}})^{*}(\Phi(N)^{\circ})}.\]

\end{proof}

\subsubsection{Functions with divisors supported on $N$-torsion}

Writing $\E(N)$, $\E$ for $\E^{[1]}(N)$, $\E^{[1]}$ we have \\
\xymatrix{
& & & & \E \ar @{->>} [r]_{\mathcal{P}_N} \ar [d]^{\pi} & \E(N) \ar @{^(->} [r] \ar [d]^{\pi(N)} & \overline{\E}(N) \ar [d]^{\overline{\pi}(N)} 
\\
& & & & \mathfrak{H} \ar @{->>} [r]_{\rho_N} & Y(N) \ar @{^(->} [r] & \overline{Y}(N).
}\\
\\
Let $U(N)\incl\E(N)$ be the complement of the $N^{2}$ $N$-torsion
sections; there is a {}``relative divisor'' map%
\footnote{note that $\mathcal{O}^{*}(U(N))\subset\CC(\overline{\E}(N))^{*}$%
} \[
\mathcal{O}^{*}(U(N))\rTo^{\div}\Phi_{2}(N)^{\circ}\]
\[
f\mapsto\varphi_{f}\]
(which ignores divisor components supported on the singular fibers
over cusps $\{\hat{E}_{y_{0}}(N)\,|\, y_{0}\in\kappa(N)\}$). Now
assume $p,q$ have been chosen as in the beginning of $\S5.2.4$.
Taking any $r,s$ such that $\gamma:=\cmat\in SL_{2}(\ZZ)$ , define
$\mathfrak{F}(N)_{\gamma}:=$\small\[
\left\{ f\in\mathcal{O}^{*}(U(N))\,\left|\begin{array}{c}
\mathcal{P}_{N}^{*}f\text{ has "}(p,q)\text{-vertical" }T_{\mathcal{P}_{N}^{*}f}\text{ over the hyperbolic geodesic}\\
(\tau\in)\,\mathcal{A}_{\gamma}:=\{\frac{ibr-q}{ibs+p}\,|\, b\in\RR^{+}\}\subset\mathfrak{H}\text{ connecting }[\frac{r}{s}]\text{ and }[\frac{-q}{p}],\\
\text{in the sense that its support in }E_{\tau}\text{ lies in one connected component}\\
\text{of }W_{\tau}^{(p,q)}(N):=\{\xi(p\tau+q)+\frac{m\tau+n}{N}\,|\, m,n\in\nicefrac{\ZZ}{N\ZZ},\,\xi\in\nicefrac{\CC}{\ZZ}\}\end{array}\right.\right\} .\]
\normalsize 

\begin{lem}
(i) $\div$ is surjective.

(ii) $\div(\mathfrak{F}(N)_{\gamma})\supset\mathfrak{S}(N)_{(p,q)}$.
\end{lem}
\begin{rem}
(a) Together with the Decomposition Lemma, $(ii)$ ensures that we
can actually compute with the KLM formula (because we are able to
work with functions with known $T_{f}$ on $\pi^{-1}$ of the arc
$\A_{\gamma}$).

(b) It is obvious that the definition of $\mathfrak{F}(N)_{\gamma}$
only depends on the coset of $\gamma$ in $SL_{2}(\ZZ)/\Gamma(N)$,
but we won't need this.
\end{rem}
\begin{proof}
$(i)$ Working on $\E$, we will construct a meromorphic function
$f\in\text{im}(\mathcal{P}_{N}^{*})$ with divisor $\sum_{(m,n)\in(\nicefrac{\ZZ}{N\ZZ})^{2}}a_{m,n}\left[\frac{m\tau+n}{N}\right]$
for any given $\{a_{m,n}\}_{_{(m,n)\in(\nicefrac{\ZZ}{N\ZZ})^{2}}}$
satisfying $\sum a_{m,n}m\emn0\emn\sum a_{m,n}n$ and $\sum a_{m,n}=0$.
In fact, we can choose $\{\tilde{a}_{m,n}\}_{(m,n)\in\ZZ^{2}}$ (all
but finitely many zero) {}``lifting'' $\{a_{m,n}\}$ such that $\sum\tilde{a}_{m,n}m=0=\sum\tilde{a}_{m,n}n$
exactly; this leads (following \cite[p. 8.8]{Bl3}) to the construction
of a function $f_{0}$ on $\mathfrak{H}\times\CC$ descending to $\E$:
\begin{equation}
f_0(\tau,z):=\prod_{k\in\ZZ}\prod_{(m,n)\in\ZZ^2} \left( 1-e^{2\pi i(k\tau+z-\frac{m\tau+n}{N})} \right)^{a_{m,n}}. \\
\end{equation} Factoring $f_{0}$ if necessary, we may assume that some $(m_{0},n_{0})\in(\ZZ/N\ZZ)^{2}$
has $a_{m_{0},n_{0}}=0$; then \begin{equation}
f(\tau,z):=\frac{f_0(\tau,z)}{f_0 \left( \tau, \frac{m_0\tau+n_0}{N} \right)} \text{ descends to }\E(N). \\
\end{equation} 

$(ii)$ We will use the proof of $(i)$ to construct $f\in\mathfrak{F}(N)_{\hmat}$
with $\varphi_{f}(m,n)=\left\{ \begin{array}{c}
-1,\,\,(m,n)\emn(\pm1,0)\\
\,\,2,\,\,(m,n)\emn(0,0)\,\,\,\,\\
\,\,0,\,\,\text{otherwise}\,\,\,\,\,\,\,\,\,\,\,\,\,\,\,\,\,\end{array}\right.;$ then the idea is simply to translate and pull back (using the action
of $\cmat\in SL_{2}(\ZZ)$ on $\E(N)$ induced from that on $\E$)
this $f$.

Taking $\tilde{a}_{0,0}=2$, $\tilde{a}_{1,0}=\tilde{a}_{-1,0}=-1$
(all other $\tilde{a}_{m,n}=0$) in $(5.7)$, one easily computes
that (with $\tau=iy\in i\RR^{+}$) $f_{0}(iy,iY)\in\RR^{\leq0}$ for
$Y\in(\frac{-y}{N},\frac{y}{N})$. So on each $E_{\tau=iy}$, $|T_{f_{0}}|\supset\{z=iY\,|\, Y\in[\frac{-y}{N},\frac{y}{N}]\},$
while $f_{0}$ is of \emph{degree} \emph{$2$}; it follows that $T_{f_{0}}$
is just the sum of two directed line segments, from $\pm\frac{\tau}{N}(=\pm\frac{iy}{N})$
to $0$. In $(5.8)$, we take $(m_{0},n_{0})=(0,1)$, and check that
$T_{f}=T_{f_{0}}$ over $\tau=iy$ ($y\in\RR^{+}$), or equivalently
that $f(iy,\frac{1}{N})\in\R^{+}$. To do this, observe that $\overline{f_{0}(iy,\overline{z})}$
is (a) holomorphic and has (b) the same divisor as $f_{0}(iy,z)$
and (c) the same leading coefficient of power series expansion at
$z=0$ ($f_{0}=Cz^{2}+\cdots$, where $[0\neq]\, C\in\RR^{+}$ since
$T_{f_{0}}$ is vertical). Thus $\overline{f_{0}(\overline{z})}=f_{0}(z)$,
which $\implies$ $\overline{f_{0}(\frac{1}{N})}=f_{0}(\frac{1}{N})\,(\in\RR)$.
Since $\frac{1}{N}\notin T_{f_{0}}$, $f_{0}(\frac{1}{N})\in\RR^{+}$.
\end{proof}
Now we can obtain meromorphic functions on $\E^{[\ell]}(N)$ by noting
that $\E^{[\ell]}(N)=\times_{Y(N)}^{\ell}\E(N)$, $\E^{[\ell]}=\times_{\mathfrak{H}}^{\ell}\E$,
and (by abuse of notation) writing the projections to these factors
\tiny \xymatrix{\E^{[\ell]}(N) \ar [rr]_{z_i} \ar [rd] & & \E(N) \ar [ld] \\ & Y(N)} \normalsize 
so that $f(z_{i})$ denotes $z_{i}^{*}f$, etc.

\subsection{Construction of the Eisenstein symbols}

\subsubsection{Eisenstein series}

Since the cycle-class computation ($\S6.1$) will show that these
series actually yield modular forms, we won't bother proving this
directly. Note that for the double sums $\sum_{m,n}'$ means to omit
$(m,n)=(0,0)$. 

For $N\geq3$ and $\ell\in\ZZ^{+}$ define \[
\EE_{\ell+2}(\Gamma(N)):=\left\{ F\in\mathcal{O}(\mathfrak{H})\,\left|\, F\text{ of form }\sum_{(m,n)\in\ZZ^{2}}{}^{'}\frac{\psi(m,n)}{(m\tau+n)^{\ell+2}}\text{ for }\psi\in\Phi_{2}(N)\right.\right\} .\]
(The series is necessarily convergent.)

\begin{lem}
The map \[
\Phi_{2}(N)^{\circ}\rTo^{\mathsf{E}^{[\ell]}}\EE_{\ell+2}(\Gamma(N))\]
defined by \[
\varphi\mapsto E_{\varphi}^{[\ell]}(\tau):=\frac{-(\ell+1)}{(2\pi i)^{\ell+2}}\sum_{(m,n)\in\ZZ^{2}}{}^{'}\frac{\widehat{\varphi}(m,n)}{(m\tau+n)^{\ell+2}}\]
is surjective.
\end{lem}
\begin{proof}
Let $\psi_{0}:=\left\{ \begin{array}{c}
N^{\ell+2}-1,\,\,(m,n)\emn(0,0)\\
\,\,-1\,\,\,\,\,\,\,\,\,\,,\,\,\,\,\text{otherwise}\,\,\,\,\,\,\,\,\,\end{array}\right.;$ then $\sum'\frac{\psi_{0}(m,n)}{(m\tau+n)^{\ell+2}}$ is obviously
$0$. This implies that we may assume $\psi\in\Phi_{2}(N)_{\circ}$
($\implies\psi=\widehat{\varphi},$$\varphi\in\Phi_{2}(N)^{\circ}$)
in the definition.
\end{proof}
Put $\EE_{\ell+2}^{\QQ}(\Gamma(N)):=\mathsf{E}^{[\ell]}\left(\Phi_{2}^{\QQ}(N)^{\circ}\right).$
(Clearly $\EE_{\ell+2}=\EE_{\ell+2}^{\QQ}\otimes\CC$.)

\begin{lem}
For $E_{\varphi}^{[\ell]}\in\EE_{\ell+2}^{\QQ}(\Gamma(N))$, $\lim_{\tau\to i\infty}E_{\varphi}^{[\ell]}(\tau)=\mathsf{H}_{[i\infty]}^{[\ell]}(\varphi)\,\,(\in\QQ).$
\end{lem}
\begin{proof}
$\lim_{\tau\to i\infty}\sum_{m,n}'\frac{\widehat{\varphi}(m,n)}{(m\tau+n)^{\ell+2}}=\sum_{n}'\frac{\widehat{\varphi}(0,n)}{n^{\ell+2}}=\sum_{n}'\frac{(\iota_{[i\infty]}^{*}\widehat{\varphi})(n)}{n^{\ell+2}}=\sum_{n}'\frac{\widehat{\pi_{[i\infty]_{*}}\varphi}(n)}{n^{\ell+2}}$\[
=\tilde{L}(\widehat{\pi_{[i\infty]_{*}}\varphi},\ell+2)=\frac{-(2\pi i)^{\ell+2}}{\ell+1}\mathsf{H}_{[i\infty]}^{[\ell]}(\varphi),\]
by $\S\S5.2.1-2$.
\end{proof}

\subsubsection{Group actions}

Writing $\mathfrak{S}_{\ell}$ for the symmetric group, let $\mathcal{G}:=\mathfrak{S}_{\ell}\ltimes(\ZZ/2\ZZ)^{\ell}$
act on $\mathfrak{H}\times\CC^{\ell}$ by\[
(c,\underline{\epsilon})(\tau;\, z_{1},\ldots,z_{\ell}):=(\tau;\,(-1)^{\epsilon_{1}}z_{c(1)},\ldots,(-1)^{\epsilon_{\ell}}z_{c(\ell)})\,;\]
this descends to $\E^{[\ell]}$ and $\E^{[\ell]}(N)$. Fixing $N$,
let $\Lambda^{\ell}:=(\ZZ/N\ZZ)^{2\ell}$ act on $\E^{[\ell]}$ via
translations\[
\text{tr}_{\underline{\lambda}}(\tau;\, z_{1},\ldots,z_{\ell}):=\left(\tau;\, z_{1}+\frac{\lambda_{1}\tau+\lambda_{2}}{N},\ldots,z_{\ell}+\frac{\lambda_{2\ell-1}\tau+\lambda_{2\ell}}{N}\right);\]
this descends to $\E^{[\ell]}(N).$

\subsubsection{Inclusions and open subsets of $\E^{[\ell]}(N)$}

$\begin{array}[t]{cccc}
\supset & \bar{U}^{[\ell]}(N) & \supset & \tilde{U}^{[\ell]}(N)\\
 & \cup &  & \cup\\
 & U^{[\ell]}(N) & \supset & \hat{U}^{[\ell]}(N)\end{array}$ (to be defined). Writing {}``FP'' for fixed points, set\[
\hat{W}_{N}^{[\ell]}:=\bigcup_{\underline{\lambda}\in\Lambda^{\ell}}\text{tr}_{\underline{\lambda}}\left\{ \cup_{(c,\underline{\epsilon})\in\mathcal{G}}\text{FP}((c,\underline{\epsilon}))\right\} \subset\E^{[\ell]}\,\,,\,\,\,\,\,\,\,\,\,\hat{W}^{[\ell]}(N):=\mathcal{P}_{N}(W_{N}^{[\ell]}),\]
\[
\hat{U}^{[\ell]}(N):=\E^{[\ell]}(N)\m\hat{W}^{[\ell]}(N).\]
Next, generalize $U(N)$ in two different ways:\[
U^{[\ell]}(N):=\times_{Y(N)}^{\ell}U(N)\,\,,\,\,\,\,\,\,\,\,\bar{U}^{[\ell]}(N):=\E^{[\ell]}(N)\m\{N^{2\ell}\text{ }N\text{-torsion sections}\}.\]
The inclusion $\mathfrak{H}\times\CC^{\ell}\hookrightarrow\mathfrak{H}\times\CC^{\ell+1}$
given by\[
(z_{1},\ldots,z_{\ell})\mapsto(-z_{1},z_{1}-z_{2},\ldots,z_{\ell-1}-z_{\ell},z_{\ell})=:(u_{1},\ldots,u_{\ell+1})\]
descends to define maps $\iota:\,\E^{[\ell]}\hookrightarrow\E^{[\ell+1]}$
and \[
\iota(N):\,\E^{[\ell]}(N)\hookrightarrow\E^{[\ell+1]}(N).\]
Finally, put \[
\tilde{U}^{[\ell]}(N):=\iota(N)^{-1}\left(U^{[\ell+1]}(N)\right).\]
To summarize,\[
\left.\begin{array}{c}
\bar{U}^{[\ell]}(N)\\
\hat{U}^{[\ell]}(N)\\
\tilde{U}^{[\ell]}(N)\\
U^{[\ell]}(N)\end{array}\right\} \begin{array}{c}
\text{means the}\\
\text{"complement of}\\
\text{translates of"}\end{array}\left\{ \begin{array}{c}
z_{1}=\cdots=z_{\ell}=0\,\,\,\,\,\,\,\,\,\,\,\,\,\,\,\,\,\,\,\,\,\,\,\,\,\,\,\,\,\,\,\,\,\,\\
\text{all }z_{i}=\pm z_{j},\, z_{i}=0\,\,\,\,\,\,\,\,\,\,\,\,\,\,\,\,\,\,\,\,\,\,\,\,\,\,\,\,\,\,\,\,\\
z_{1}=0,\, z_{1}=z_{2},\ldots,\, z_{\ell-1}=z_{\ell},\, z_{\ell}=0\\
z_{1}=0,\, z_{2}=0,\ldots,\, z_{\ell}=0\,\,\,\,\,\,\,\,\,\,\,\,\,\,\,\,\,\,\,\,\,\,\,\end{array}\right.\]
and makes sense in $\E^{[\ell]}$ or $\E^{[\ell]}(N)$ (where in $\E^{[\ell]}$
these open sets are denoted instead $\bar{U}_{N}^{[\ell]},\,\hat{U}_{N}^{[\ell]},$
etc.). Denote the $U$-complements (i.e. the translates of the sets
on the r.h.s.) by $\bar{W}$, $\hat{W}$, etc.

\subsubsection{Completion of symbols }

Write $\QQ[\mathcal{O}^{*}(U(N))]$ for the $\otimes\QQ$ free-abelian
group on the set of elements of $\mathcal{O}^{*}(U(N))$, and recall
$\square:=\PP^{1}\m\{1\}$. To each monomial $\textbf{f}:=f_{1}\otimes\cdots\otimes f_{\ell+1}\in\otimes^{\ell+1}\QQ[\mathcal{O}^{*}(U(N))]$
we associate the graph cycle $\{\textbf{f}\}:=$\small  \[
\{f_{1}(u_{1}),\ldots,f_{\ell+1}(u_{\ell+1})\}:=\left\{ (\tau;\,\underline{u};\, f_{1}(\tau,u_{1}),\ldots,f_{\ell+1}(\tau,u_{\ell+1}))\,\left|\,(\tau,\underline{u})\in U^{[\ell+1]}(N)\right.\right\} \]
\normalsize \[
\subset U^{[\ell+1]}(N)\times\square^{\ell+1}.\]
Its pullback by $\iota(N)$ should be thought of as the symbol \begin{equation}
\iota^*\{\textbf{f}\}:=\{ f_1(-z_1),f_2(z_1-z_2),\ldots ,f_{\ell}(z_{\ell-1}-z_{\ell}),f_{\ell+1}(z_{\ell})\}, \\
\end{equation} which is evidently in good position (i.e. yields a higher Chow precycle)
over all of $\bar{U}^{[\ell]}(N)$. To kill $\db$ of this symbol
in $\hat{W}^{[\ell]}(N)$, we flip it about components of $\hat{W}^{[\ell]}(N)$
and subtract the result: writing $\tilde{\G}:=\G\ltimes\Lambda^{\ell}$,
define \[
\tilde{\G}^{*}:=\frac{1}{\ell!2^{\ell}N^{2\ell}}\left\{ \sum_{(c,\underline{\epsilon},\underline{\lambda})\in\tilde{\G}}(-1)^{sgn(\sigma)+\sum\epsilon_{i}}(c,\underline{\epsilon})^{*}(\text{tr}_{\underline{\lambda}})^{*}\right\} ,\]
 and $\tilde{\G}_{0}^{*}$ if signs are $removed$. (There is also
$\G^{*}$, defined by forgetting the $\frac{1}{N^{2\ell}}\sum_{\underline{\lambda}}(\text{tr}_{\underline{\lambda}})^{*}$
part.)

Now consider the diagram\\
\xymatrix{
& & 
\otimes^{\ell+1}\QQ \left[\mathcal{O}^*(U(N))\right] 
\ar [d]^{\left\{\cdot\right\}}
\ar @/^5pc/ [rddd]^{\tiny \begin{matrix} \tilde{\G}^*\circ \iota(N)^* \circ \{\cdot\}  \text{, followed} \\ \text{by Zariski closure} \end{matrix} \normalsize} 
\\ 
& & Z^{\ell+1}_{_{\db\text{-cl.}}}\left(U^{[\ell+1]}(N),\ell+1\right)_{_{\{\sum u_i=0\}}} 
\ar [d]^{\iota(N)^*} 
\\ 
& & Z^{\ell+1}_{_{\db\text{-cl.}}}\left(\tilde{U}^{[\ell]}(N),\ell+1\right)
\ar [d]^{\tilde{\G}^*}
\\
& & 
\left[ Z^{\ell+1}_{_{\db\text{-cl.}}}\left(\hat{U}^{[\ell]}(N),\ell+1\right)\right]^{\tilde{\G}} 
&
\left[ Z_{_{\db\text{-cl.}}}^{\ell+1} \left( \bar{U}^{[\ell]}(N),\ell+1 \right) \right]^{\tilde{\G}} \ar [l]_{\text{restriction}}
}\\
\\
in which we denote the images of $\textbf{f}$ as follows:\\
\xymatrix{& & & & \textbf{f} \ar @{|->} [d] \ar @/^4pc/ @{|->} [rrrrddd] 
\\
& & & & \left\{ \textbf{f} \right\} \ar @{|->} [d]
\\
& & & & \iota_{_{(N)}}^* \left\{ \textbf{f} \right\} \ar @{|->} [d]
\\
& & & & Z_{\textbf{f}}
& & & & \overline{Z_{\textbf{f}}}.
\ar @{|->} [llll]
}\\
\\
Unless%
\footnote{This condition just says $0\notin|(f_{1})|*\cdots*|(f_{\ell+1})|$.%
} $\alpha_{1}+\cdots+\alpha_{\ell+1}\neq0$ $\forall\{\alpha_{1},\ldots,\alpha_{\ell+1}\}\in|(f_{1})|\times\cdots\times|(f_{\ell+1})|$,
extending the cycle over the $N$-torsion sections to $\E^{[\ell]}(N)$
requires a {}``move'' (by adding a $\db$-coboundary); such a move
always exists, as\[
\left[CH^{\ell+1}(\E^{[\ell]}(N),\ell+1)\right]^{\tilde{\G}}\rTo_{\text{restriction}}^{\cong}\left[CH^{\ell+1}(\hat{U}^{[\ell]}(N),\ell+1)\right]^{\tilde{\G}},\]
but of course this eliminates well-definedness on the level of $precycles$
(but not cycle-class) for the resulting\[
\mathfrak{Z}_{\textbf{f}}\in Z_{_{\db\text{-cl.}}}^{\ell+1}(\E^{[\ell]}(N),\ell+1).\]

\begin{prop}
We have a well-defined map of precycles\[
\otimes^{\ell+1}\QQ[\mathcal{O}^{*}(U(N))]\longrightarrow[Z_{_{\db\text{-cl.}}}^{\ell+1}(\bar{U}^{[\ell]}(N),\ell+1)]^{\tilde{\G}}\]
\[
\mathcal{\textbf{\emph{f}}}\longmapsto\overline{Z_{\textbf{\emph{f}}}}.\]
Going modulo relations, this induces a well-defined map\\
\\
\xymatrix{& & \mathcal{O}^*(U(N))^{\otimes \ell+1} \ar [r] \ar[rd] & \left[ CH^{\ell+1}(\bar{U}^{[\ell]}(N),\ell+1)\right]^{\tilde{\G}} 
\\ & & \emph{\textbf{f}} \ar @{|->} [rd] & \left[ CH^{\ell+1}(\E^{[\ell]}(N),\ell+1)\right]^{\tilde{\G}} \ar [u]_{\cong} \\ & & & \left\langle \mathfrak{Z}_{\emph{\textbf{f}}} \right\rangle
.}
\end{prop}

\section{\textbf{Fundamental class computations}}

\subsection{Cycle-class of the Eisenstein symbol}

\subsubsection{More Fourier theory}

Now we introduce {}``fiberwise Fourier series'' for \\
\xymatrix{& & & & & \E \ar [d]_{\pi} \\ & & & & &  \mathfrak{H} \ar @/_1pc/ [u]_{\se:=\text{zero section.}}}
\\
\\
Writing coordinates $(\tau,u=x+y\tau)$ on $\E$, and $\nu:=\bar{\tau}-\tau$,
we note that $du$ is only well-defined in $\Omega^{1}(\nicefrac{\E}{\uhp})$,
whereas \[
\widetilde{du}:=du-\frac{\bar{u}-u}{\nu}d\tau\in A^{1,0}(\E)\,\,\,\,\,\,\,\,\,\text{[resp. }\widetilde{d\bar{u}}:=\overline{\widetilde{du}}\in A^{0,1}(\E)\text{]}\]
make sense on $\E$.

Let $\Gamma:=\Gamma(\uhp,R^{1}\pi_{*}\ZZ)\cong\ZZ\left\langle [\alpha],[\beta]\right\rangle $,
so that $\gamma=m[\beta]+n[\alpha]=${}``$(m,n)$''$\in\Gamma$
has period $\omega(\gamma):=\pi_{*}(du\cdot\delta_{\gamma})=m\tau+n$
against $du$; and write \[
\overline{\chi_{\gamma}}(u):=\exp(2\pi i(mx-ny))\,\,\,,\,\,\,\,\,\,\,\,\, d\overline{\chi_{\gamma}}(u)=\frac{2\pi i}{\nu}\{\overline{\omega(\gamma)}du-\omega(\gamma)d\overline{u}\}\overline{\chi_{\gamma}}.\]
Associate to a current $\K\in\D^{M}(\E)$ {}``Fourier coefficients''\[
\hat{\K}(\gamma):=\left\{ \begin{array}{cc}
\pi_{*}(\K\cdot\overline{\chi_{\gamma}})\in\D^{M-2}(\uhp), & M\geq2\\
\nu^{-1}\pi_{*}(\K\cdot\overline{\chi_{\gamma}}\widetilde{du}\wedge\widetilde{d\bar{u}})\in\D^{M}(\uhp),\,\, & M<2\end{array}\right.\]
for each $\gamma\in\Gamma$. (Note: $\nu^{-1}du\wedge d\bar{u}=dx\wedge dy$.)

\begin{lem}
(i) If $\K\in A^{M}(\E)$ ($M<2$) then \[
\se^{*}\K=\sum_{\gamma\in\Gamma}\hat{\K}(\gamma),\]
and the r.h.s. is absolutely convergent.

(ii) Recalling the notation of $\S5.2.5$,%
\footnote{Warning: in this section we are no longer using $\gamma$ to denote
$\cmat\in SL_{2}(\ZZ).$%
} if $\K\in\D^{0}(E_{\tau})$ is a smooth function on the complement
of%
\footnote{Obviously the singularities are $L^{1}$-integrable since $\K$ is
a current.%
} $W_{\tau}^{(p,q)}(N)\m\{\text{connected component of }u=0\}$, then
\[
\K(0)=\se^{*}\K=\sum_{k\in\ZZ}\sum_{j\in\ZZ}^{P.V.}\hat{\K}(jp-ks,jq+kr)\]
where $\sum_{j\in\ZZ}^{P.V.}:=\lim_{J\to\infty}\sum_{j=-J}^{J}$ (or
alternatively, add $\pm j$ terms then sum $j\geq0$).
\end{lem}
\begin{proof}
\emph{(i)} is just the statement {}``$\K(0)=$\{inverse FT evaluated
at $0$\}$=\sum$\{Fourier coefficients\}'' for smooth functions.

\emph{(ii)} Say $(p,q)=(1,0)$, $M=0$. Then (working on some $E_{\tau}$)
put $G_{k}(x):=\int_{0}^{1}\K(x,y)e^{-2\pi iny}dy\,\in\D^{0}(\nicefrac{\CC}{\ZZ})$;
this restricts to a $smooth$ $function$ on the complement of $\{\frac{1}{N},\frac{2}{N},\ldots,\frac{N-1}{N}\}.$
By \cite[Cor. 41.4]{WM} $G_{k}(0)=\sum_{j\in\ZZ}^{P.V.}\widehat{G_{k}}(j)=\sum_{j\in\ZZ}^{P.V.}\int_{0}^{1}G_{k}(x)e^{2\pi ijx}dx\rEq^{\text{Fubini}}\sum_{j\in\ZZ}^{P.V.}\int\int_{E_{\tau}}\K(x,y)\overline{\chi_{(j,k)}}dx\wedge dy=\sum_{j\in\ZZ}^{P.V.}\hat{\K}(j,k).$
But the $\{G_{k}(0)\}$ are the Fourier coefficients of the smooth
function $\K(0,y)$ $\implies$ $\K(0,0)=\sum G_{k}(0)$.
\end{proof}
\begin{lem}
If $F\in\D^{0}(\E)$, $\frac{\d F}{\d\bar{u}}\in\D^{0}(\E)$ is defined
and $\widehat{\frac{\d F}{\d\bar{u}}}(\gamma)=\frac{2\pi i\omega(\gamma)}{\nu}\hat{F}(\gamma).$
\end{lem}
${}$

\begin{lem}
Let $f\in\mathcal{O}^{*}(U_{N})$, and write $\widehat{\varphi_{f}}(\gamma):=\widehat{\varphi_{f}}(m,n).$

(i) $\widehat{\delta_{(f)}}(\gamma)=\widehat{\varphi_{f}}(\gamma)$;

(ii) $\widehat{\log f}(\gamma)=\frac{\int_{T_{f}}\overline{\chi_{\gamma}}du}{\omega(\gamma)}$
for $\gamma\neq(0,0)$, while $\widehat{\log f}(0)=0$ if $f\in\mathfrak{F}(N)_{\cmat}$;

(iii) $\text{dlog}f=\alpha_{f}\widetilde{du}+\beta_{f}d\tau$ $\implies$$\widehat{\alpha_{f}}(0)=\widehat{\beta_{f}}(0)=0$
while for $\gamma\neq(0,0)$,\[
\widehat{\alpha_{f}}(\gamma)=\frac{-\widehat{\varphi_{f}}(\gamma)}{\omega(\gamma)}\,\,\text{ and }\,\,\widehat{\beta_{f}}(\gamma)=\frac{\widehat{\phi_{f}}(\gamma)}{2\pi i(\omega(\gamma))^{2}}.\]

\end{lem}
\begin{proof}
Lemma 6.2 and 6.3(i), (iii) (which uses 6.2) are essentially done
in \cite{B2}. For (ii) (and to get a feel for how the others go),\tiny 
\[
\widehat{\log f}(\gamma)=\nu^{-1}\pi_{*}(\log f\,\overline{\chi_{\gamma}}du\wedge d\bar{u})=(2\pi i)^{-1}\omega(\gamma)^{-1}\pi_{*}\left(\log f\,\frac{2\pi i}{\nu}\{\overline{\omega(\gamma)}du-\omega(\gamma)d\bar{u}\}\overline{\chi_{\gamma}}\wedge du\right)\]
\[
=(2\pi i)^{-1}\omega(\gamma)^{-1}\pi_{*}(\log f\, d\overline{\chi_{\gamma}}\wedge du)=(2\pi i)^{-1}\omega(\gamma)^{-1}\{-\pi_{*}(\overline{\chi_{\gamma}}d[\log f]\wedge du)+\pi_{*}(\overline{\chi_{\gamma}}\begin{array}[t]{c}
\underbrace{\frac{df}{f}\wedge du}\\
0\end{array})\}\]
\small \[
=\omega(\gamma)^{-1}\pi_{*}(\overline{\chi_{\gamma}}\delta_{T_{f}}\wedge du)=\frac{\int_{T_{f}}\overline{\chi_{\gamma}}du}{\omega(\gamma)},\]
\normalsize where at the end we have used $d[\log f]=\frac{df}{f}-2\pi i\delta_{T_{f}}$.
As for $\widehat{\log f}(0)$, we have $\widehat{\log|f|}(0)=\nu^{-1}\pi_{*}(\log|f|\, du\wedge d\bar{u})=0$
since $\log|f|du\wedge d\bar{u}=\text{dlog}|f|\wedge d\bar{u}=d[\log|f|d\bar{u}]$.
Now, using our prototype (from the proof of Lemma 5.13(ii)) for $f\in\mathfrak{F}(N)_{\hmat}$
with $\overline{f(\bar{z})}=f(z)$, one finds ($\tau\in i\RR^{+}$)
that $\pi_{*}(\arg f\, du\wedge d\bar{u})=\pi_{*}(\arg\bar{f}\, du\wedge d\bar{u})=\pi_{*}(-\arg f\, du\wedge d\bar{u})$.
(A similar argument works in general.)
\end{proof}
\begin{lem}
Let $f\in\mathfrak{F}(N)_{\cmat}$, $\gamma=(m,n)$. Then over $(\tau\in)\,\A_{\cmat}\subset\uhp$,
\[
\int_{T_{f}}\overline{\chi_{\gamma}}d\left\{ \begin{array}{c}
u\\
\bar{u}\end{array}\right\} =\frac{p\left\{ \begin{array}{c}
\tau\\
\bar{\tau}\end{array}\right\} +q}{2\pi i(mq-np)}\widehat{\varphi_{f}}(\gamma)\]
if $mq-np\neq0$; otherwise the l.h.s. is $0$.
\end{lem}
\begin{proof}
Represent $T_{f}$ as a sum of straight paths of the following type,
assuming $(f)=\sum_{K=0}^{N-1}a_{K}\left[K\frac{p\tau+q}{N}+L\frac{-s\tau+r}{N}\right]$
($L\in\{0,\ldots,N-1\}$ fixed). For the paths, write \[
\mathsf{P}:\,[0,1]\hookrightarrow E_{\tau}\]
\[
t\mapsto L\frac{-s\tau+r}{N}+t(p\tau+q);\]
then\[
T_{f}\,=\,\sum_{K}a_{K}\left\{ \frac{N-K}{N}\cdot\mathsf{P}\left(\left[0,\nicefrac{K}{N}\right]\right)-\frac{K}{N}\cdot\mathsf{P}\left(\left[\nicefrac{K}{N},1\right]\right)\right\} +b\cdot\mathsf{P}([0,1])\,=:\,\tilde{T}_{f}+S_{f},\]
where $b\in\QQ$. We have\small  \[
\mathsf{P}^{*}(\overline{\chi_{\gamma}}du)\,=\, e^{2\pi i\left\{ m\left(\frac{Lr}{N}+qt\right)-n\left(\frac{-Ls}{N}+pt\right)\right\} }(p\tau+q)dt\,=\, e^{2\pi it(mq-np)}e^{\frac{2\pi iL}{N}(mr+ns)}(p\tau+q)dt.\]
\normalsize 

Now $\frac{1}{p\tau+q}\int_{S_{f}}\overline{\chi_{\gamma}}du=b\cdot e^{\frac{2\pi iL}{N}(mr+ns)}\int_{0}^{1}e^{2\pi it(mq-np)}dt$
is obviously $0$ if $mq-np\neq0$; but if $mq-np=0$ then\small 
\[
\frac{e^{\frac{-2\pi iL}{N}(mr+ns)}}{p\tau+q}\int_{T_{f}}\overline{\chi_{\gamma}}du\,=\,\int_{T_{f}}du\,=\,\frac{1}{2\pi i}\int(\frac{df}{f}-d[\log f])\wedge du\,=\,\frac{1}{2\pi i}\int(\log f)d[du]=0.\]
\normalsize For $mq-np\neq0$ we have\tiny \[
\frac{1}{p\tau+q}\int_{\tilde{T}_{f}}\overline{\chi_{\gamma}}du\,=\, e^{\frac{2\pi iL}{N}(mr+ns)}\sum_{K}a_{K}\left\{ \frac{N-K}{N}\int_{0}^{\frac{K}{N}}e^{2\pi it(mq-np)}dt-\frac{K}{N}\int_{\frac{K}{N}}^{1}e^{2\pi it(mq-np)}dt\right\} \]
\[
=\frac{e^{2\pi i\frac{L}{N}(mr+ns)}}{2\pi i(mq-np)}\left(\sum_{K}a_{K}e^{2\pi i\frac{K}{N}(mq-np)}-\sum_{K}a_{K}\right)=\frac{1}{2\pi i(mq-np)}\sum_{K}a_{K}\overline{\chi_{\gamma}}\left(K\frac{p\tau+q}{N}+L\frac{-s\tau+r}{N}\right)\]
\small  \[
=\frac{\widehat{\varphi_{f}}(\gamma)}{2\pi i(mq-np)}\]
\normalsize (where we have used that $\sum a_{K}=0$).
\end{proof}
\begin{rem}
Lemma 6.3(iii) can be read $\int_{E_{\tau}}\overline{\chi_{\gamma}}\text{dlog}f\wedge d\bar{u}=\frac{-\nu\widehat{\varphi_{f}}(\gamma)}{\omega(\gamma)}.$
\end{rem}

\subsubsection{Main computation; proof of Beilinson-Hodge}

We now use the Fourier {}``technology'' to compute\[
CH^{\ell+1}(\E^{[\ell]}(N),\ell+1)\rTo^{[\cdot]}\hm{\left(\QQ(0),H^{\ell+1}\left(\E^{[\ell]}(N),\QQ(\ell+1)\right)\right)}\]
for\[
\mathfrak{Z}_{\textbf{f}}\longmapsto\Omega_{\mathfrak{Z}_{\textbf{f}}}\in\Omega^{\ell+1}(\overline{\E}^{[\ell]}(N))\left\langle \log\pi^{-1}(\kappa(N))\right\rangle .\]
By $\S5.1.5$, $\mathcal{P}_{N}^{*}\Omega_{\mathfrak{Z}_{\textbf{f}}}=(2\pi i)^{\ell+1}\Omega_{F_{\textbf{f}}}=(2\pi i)^{\ell+1}F_{\textbf{f}}(\tau)dz_{1}\wedge\cdots\wedge dz_{\ell}\wedge d\tau$
for some $F_{\textbf{f}}(\tau)\in M_{\ell+2}^{\QQ}(\Gamma(N)),$ and
it is \emph{this modular form} we must identify. Consider $\Omega_{\iota(N)^{*}\{\textbf{f}\}}\in\Omega^{\ell+1}(\overline{\E}^{[\ell]}(N))\left\langle \log\left(\overline{\tilde{W}^{[\ell]}(N)}\cup\pi^{-1}(\kappa(N))\right)\right\rangle $,
which pulls back by $\tilde{\G}^{*}$ to $\Omega_{\overline{Z_{\textbf{f}}}}$.
The latter is $not$ affected by moving $\overline{Z_{\textbf{f}}}$
into good position over $\bar{W}^{[\ell]}(N)$ and completing it to
$\mathfrak{Z}_{\textbf{f}}$; so $\Omega_{\mathfrak{Z}_{\textbf{f}}}=\tilde{\G}^{*}\Omega_{\iota(N)^{*}\{\textbf{f}\}}=\tilde{\G}^{*}\iota(N)^{*}\dlog f_{1}(u_{1})\wedge\cdots\wedge\dlog f_{\ell+1}(u_{\ell+1}).$

Write $\mathfrak{A}_{\{\textbf{f}\}}:=(-1)^{\ell}\Omega_{\mathcal{P}_{N}^{*}\{\textbf{f}\}}\wedge\widetilde{d\bar{u}_{1}}\wedge\cdots\wedge\widetilde{d\bar{u}_{\ell}}\in A^{\ell+1,\ell}(\E^{[\ell+1]})\left\langle \log W_{N}^{[\ell+1]}\right\rangle $
and $\iota^{*}\mathfrak{A}_{\{\textbf{f}\}}=\mathcal{P}_{N}^{*}\Omega_{\iota(N)^{*}\{\textbf{f}\}}\wedge\widetilde{d\bar{z}_{1}}\wedge\cdots\wedge\widetilde{d\bar{z}_{\ell}}\in A^{\ell+1,\ell}(\E^{[\ell]})\left\langle \log\tilde{W}_{N}^{[\ell]}\right\rangle \subset\D^{\ell+1,\ell}(\E^{[\ell]}).$
Using the diagram \begin{equation}
\xymatrix{\E^{[\ell]} \ar @{^(->} [rr]^{\iota} \ar [rrd]_{\pi^{[\ell]}} & & \E^{[\ell+1]} \ar @{->>} [rr]^{P} \ar [d]_{\pi^{[\ell+1]}\vspace{2mm}} & & \E \ar[lld]_{\pi} \\ & & \uhp \ar @/_1pc/ [rru]_{\mathfrak{e}} 
} \\
\end{equation} where $P(\tau;[u_{1},\ldots,u_{\ell+1}]_{\tau}):=(\tau;[u_{1}+\cdots+u_{\ell+1}]_{\tau})$,
we compute $\pi_{*}^{[\ell]}(\iota^{*}\mathfrak{A}_{\{\textbf{f}\}})$
in two different ways.

For the first,\[
\pi_{*}^{[\ell]}(\iota^{*}\mathfrak{A}_{\{\textbf{f}\}})=\pi_{*}^{[\ell]}(\tilde{\G}_{0}\iota^{*}\mathfrak{A}_{\{\textbf{f}\}})=\pi_{*}^{[\ell]}\{\tilde{\G}^{*}(\mathcal{P}_{N}^{*}\Omega_{\iota(N)^{*}\{\textbf{f}\}})\wedge\widetilde{d\bar{z}_{1}}\wedge\cdots\wedge\widetilde{d\bar{z}_{\ell}}\}\]
\[
=\pi_{*}^{[\ell]}\left\{ (2\pi i)^{\ell+1}F_{\textbf{f}}(\tau)dz_{1}\wedge\cdots\wedge dz_{\ell}\wedge d\tau\wedge\widetilde{d\bar{z}_{1}}\wedge\cdots\wedge\widetilde{d\bar{z}_{\ell}}\right\} \]
\[
=(-1)^{{\ell+1 \choose 2}}(2\pi i)^{\ell+1}\nu^{\ell}F_{\textbf{f}}(\tau)d\tau\in A^{1,0}(\uhp).\]

For the second, \[
\pi_{*}^{[\ell]}(\iota^{*}\mathfrak{A}_{\{\textbf{f}\}})=\mathfrak{e}^{*}P_{*}\mathfrak{A}_{\{\textbf{f}\}}\rEq^{\text{Lemma }6.1(i)}\sum_{\gamma\in\Gamma}\widehat{P_{*}\mathfrak{A}_{\{\textbf{f}\}}}(\gamma)\]
\[
=\nu^{-1}\sum_{\gamma\in\Gamma}\pi_{*}(\overline{\chi_{\gamma}}P_{*}\mathfrak{A}_{\{\textbf{f}\}}\wedge\widetilde{du}\wedge\widetilde{d\bar{u}})\]
\[
=\nu^{-1}\sum_{\gamma\in\Gamma}\pi_{*}^{[\ell+1]}\left((P^{*}\overline{\chi_{\gamma}})\mathfrak{A}_{\{\textbf{f}\}}\wedge(\widetilde{du_{1}}+\cdots+\widetilde{du_{_{\ell+1}}})\wedge P^{*}\widetilde{d\bar{u}}\right).\]
Writing $\dlog(\mathcal{P}_{N}^{*}f_{i}(u_{i}))=\alpha_{i}\widetilde{du_{i}}+\beta_{i}d\tau$,
this\small \[
=(-1)^{{\ell+2 \choose 2}}\nu^{-1}\sum_{\gamma\in\Gamma}\sum_{i=1}^{\ell+1}\pi_{*}^{[\ell+1]}\left\{ \left(\prod_{k=1}^{\ell+1}\overline{\chi_{\gamma}}(u_{k})\right)\beta_{i}\prod_{j\neq i}\alpha_{j}\,\widetilde{du_{1}}\wedge\widetilde{d\bar{u}_{1}}\wedge\cdots\wedge\widetilde{du_{_{\ell+1}}}\wedge\widetilde{d\bar{u}_{_{\ell+1}}}\wedge d\tau\right\} \]
\normalsize \[
=(-1)^{{\ell+2 \choose 2}}\nu^{\ell}\sum_{\gamma\in\Gamma}\sum_{i=1}^{\ell+1}\widehat{\beta_{i}}(\gamma)\prod_{j\neq i}\widehat{\alpha_{i}}(\gamma)d\tau=\frac{(-1)^{\ell}(\ell+1)}{2\pi i}(-1)^{{\ell+2 \choose 2}}\nu^{\ell}\sum_{\gamma\in\Gamma}{}^{'}\frac{\prod_{i=1}^{\ell+1}\widehat{\varphi_{f_{i}}}(\gamma)}{(\omega(\gamma))^{\ell+2}}.\]
So defining $\varphi_{\textbf{f}}:=\varphi_{f_{1}}*\cdots*\varphi_{f_{\ell+1}}\in\Phi_{2}^{\QQ}(N)^{\circ}$
(and linearly extending this to sums of {}``monomials'' $f_{1}\otimes\cdots\otimes f_{\ell+1}$),
we have proved

\begin{thm}
$F_{\textbf{\emph{f}}}(\tau)=\frac{-(\ell+1)}{(2\pi i)^{\ell+2}}\sum_{m,n\in\ZZ^{2}}\frac{\widehat{\varphi_{\textbf{\emph{f}}}}(m,n)}{(m\tau+n)^{\ell+2}}=E_{\varphi_{\textbf{\emph{f}}}}^{[\ell]}(\tau)$.
\end{thm}
Together with Lemma 5.12(i), the Decomposition Lemma (i), and Lemma
5.15, this immediately yields

\begin{cor}
$\EE_{\ell+2}(\Gamma(N))\subset M_{\ell+2}(\Gamma(N))$, $\EE_{\ell+2}^{\QQ}(\Gamma(N))\subset M_{\ell+2}^{\QQ}(\Gamma(N)).$
\end{cor}
(In particular, the map $\mathcal{O}^{*}(U(N))^{\otimes\ell+1}\to\EE_{\ell+2}^{\QQ}(\Gamma(N))$
defined by $f_{1}\otimes\cdots\otimes f_{\ell+1}\mapsto E_{\varphi_{\textbf{f}}}^{[\ell]}(\tau)$
is surjective.)

What is striking here is how simple cycles (once they are constructed)
make it to prove statements about related objects: in this case, that
\emph{Eisenstein series are modular forms}; in the same spirit we
can \emph{identify their {}``values'' at cusps}, and show that \emph{they
yield all holomorphic forms with log poles and $\QQ$-periods}.

\begin{cor}
For $\sigma\in\kappa(N)$,\[
\frac{1}{(2\pi i)^{\ell}}Res_{\sigma}(\Omega_{\mathfrak{Z}_{\textbf{\emph{f}}}})=\mathfrak{R}_{\sigma}(F_{\textbf{\emph{f}}})=\mathsf{H}_{\sigma}^{[\ell]}(\varphi_{\textbf{\emph{f}}})=\frac{-(\ell+1)}{(2\pi i)^{\ell+2}}\tilde{L}(\widehat{(\pi_{\sigma})_{_{*}}\varphi_{\textbf{\emph{f}}}},\ell+2).\]

\end{cor}
\begin{proof}
The outer equalities are just $\S5.1.5$ and $\S5.2.2$, respectively
($\forall\sigma$). For $\sigma=[i\infty]$, $\mathfrak{R}_{[i\infty]}(F_{\textbf{f}}):=\lim_{\tau\to i\infty}F_{\textbf{f}}(\tau)=\lim_{\tau\to i\infty}E_{\varphi_{\textbf{f}}}^{[\ell]}(\tau)=\mathsf{H}_{[i\infty]}^{[\ell]}(\varphi_{\textbf{f}})$
by $\S5.3.1$.

Now $SL_{2}(\ZZ)$ acts compatibly on the diagram\\
\xymatrix{\bar{W}^{[\ell+1]}_N \ar [rd]^P \ar @{^(->} [rr] & & \E^{[\ell+1]} \ar [rd]^P \ar [rr] \ar '[d][dd] & & \E^{[\ell+1]}(N) \ar [rd]^P \ar '[d][dd] \\
& \bar{W}_N \ar @{^(->} [rr] & & \E \ar [rr] \ar [ld] & & \E(N) \ar [ld] \\
& & \uhp \ar [rr] & & Y(N) \ar @{^(->} [r] & \overline{Y}(N) = Y(N)\cup \kappa(N)
}\\
\\
since $\Gamma(N)\trianglelefteq SL_{2}(\ZZ).$ In particular, the
action on connected components of $\bar{W}_{N}$ (the union of $N$-torsion
sections) induces an action (by pullback) on $\Phi_{2}^{\QQ}(N)^{\circ}$
compatible with Pontryagin $*$ and pullbacks of functions $\in\mathcal{O}^{*}(U_{N})$,
etc. Explicitly, $M_{\sigma}:=\dmat$ sends: (in $\kappa(N)$) $[i\infty]\mapsto[\frac{r}{s}]=:\sigma$,
(in $\uhp$) $\tau\mapsto\frac{r\tau-q}{s\tau+p}=:\tau_{0}$, (in
$\bar{W}_{N}$) $m\frac{\tau}{N}+n\frac{1}{N}\mapsto\frac{1}{N}\frac{m\tau+n}{s\tau+p}=(mp-ns)\frac{\tau_{0}}{N}+(mq+nr)\frac{1}{N}=:\mu\frac{\tau_{0}}{N}+\eta\frac{1}{N}$,
and (in $\Phi_{2}^{\QQ}(N)^{\circ}$, by pullback) $\varphi_{\textbf{f}}(\mu,\eta)\mapsto\left(\dmat^{*}\varphi_{\textbf{f}}\right)(m,n):=\varphi_{\textbf{f}}(mp-ns,mq+nr).$
So\[
\left(\pi_{[i\infty]_{*}}\dmat^{*}\varphi_{\textbf{f}}\right)(m)=\sum_{n\in\nicefrac{\ZZ}{N\ZZ}}\varphi_{\textbf{f}}(mp-ns,mq+nr)=\sum_{n}\varphi_{\textbf{f}}(m(p,q)+n(-s,r))\]
\[
=\left(\pi_{[\frac{r}{s}]_{*}}\varphi_{\textbf{f}}\right)(m),\]
and\[
\frac{Res_{\sigma}}{(2\pi i)^{\ell}}\left(\Omega_{\mathfrak{Z}_{\textbf{f}}}\right)=\frac{Res_{[i\infty]}}{(2\pi i)^{\ell}}\left(M_{\sigma}^{*}\Omega_{\mathfrak{Z}_{\textbf{f}}}\right)=\frac{Res_{[i\infty]}}{(2\pi i)^{\ell}}\left(\Omega_{\mathfrak{Z}_{M_{\sigma}^{*}\textbf{f}}}\right)=\frac{-(\ell+1)}{(2\pi i)^{\ell+2}}\tilde{L}\left(\widehat{\pi_{[i\infty]_{*}}\varphi_{M_{\sigma}^{*}\textbf{f}}},\,\ell+2\right)\]
\[
=\frac{-(\ell+1)}{(2\pi i)^{\ell+2}}\tilde{L}\left(\widehat{\pi_{\sigma_{*}}\varphi_{\textbf{f}}},\,\ell+2\right).\]

\end{proof}
\begin{cor}
(i) Claim 5.1 holds for $\Gamma(N)$ ($\implies$Beilinson-Hodge for
$\E^{[\ell]}(N)$).

(ii) $\EE_{\ell+2}^{\QQ}(\Gamma(N))=M_{\ell+2}^{\QQ}(\Gamma(N))\cong\hm{\left(\QQ(0),H^{\ell+1}(\E^{[\ell]}(N),\QQ(\ell+1)\right)},$
with dimension $|\kappa(N)|$.

(iii) $M_{\ell+2}(\Gamma(N))=\EE_{\ell+2}(\Gamma(N))\oplus S_{\ell+2}(\Gamma(N))$.
\end{cor}
\begin{rem*}
Note that $\dim_{\CC}\EE=\dim_{\QQ}\EE^{\QQ}=\dim_{\QQ}M^{\QQ}\leq\dim_{\CC}M$
in general.
\end{rem*}
\begin{proof}
is basically contained in the diagram\\
\tiny \xymatrix{ & & \mathcal{O}^*(U(N))^{\otimes \ell+1} \ar @{->>} [lld]_{\otimes^{\ell+1}\div} \ar [rd]^{\mathfrak{Z}} \\ \left( \Phi^{\QQ}_2(N)^{\circ} \right)^{\otimes \ell + 1} \ar @{->>} [d]^{*^{^{\ell+1}}} & & & CH^{\ell+1}(\E^{[\ell]}(N),\ell+1) \ar [d]^{[\cdot ]} \\  \Phi^{\QQ}_2(N)^{\circ} \ar @{->>} [rrdd]_{\mathsf{H}=\oplus_{_{\sigma\in\kappa(N)}}\mathsf{H}^{[\ell]}_{\sigma}} \ar @{->>} [r] & \EE^{\QQ}_{\ell+2}(\Gamma(N)) \ar @{^(->} [r] \ar [rdd]^{\oplus \mathfrak{R}_{\sigma}|_{\EE^{\QQ}}} & M^{\QQ}_{\ell+2}(\Gamma(N)) \ar @<1ex> [r]^{[(2\pi i)^{\ell+1}\Omega_{(\cdot)}] \mspace{100mu}}_{\cong \mspace{100mu}} \ar [dd]^{\oplus \mathfrak{R}_{\sigma}|_{M^{\QQ}}} & \hm \left( \QQ(0), H^{\ell+1}(\E^{[\ell]}(N),\QQ(\ell+1)) \right) \ar @{^(->} [ldd]^{Res = \oplus \frac{Res_{\sigma}}{(2\pi i)^{\ell}}} \ar @<1ex> [l]^{\theta_{\ell+2} \mspace{100mu}} \\ \\ & & \Upsilon^{\QQ}_2(N)} \normalsize \\
\\
(The arrows around the outer left surject by $\S\S5.2.5,\,5.2.3,\,5.2.2$
(resp.), as does the map to $\EE_{\ell+2}^{\QQ}$ by $\S5.3.1$; the
map $from$ $\EE_{\ell+2}^{\QQ}$ injects by Cor. 6.7 and $Res$ by
$\S5.1.4$. The upper pentagon commutes by Theorem 6.6, and the lower
triangles by Cor 6.8.) We can track $\textbf{f}:=f_{1}\otimes\cdots\otimes f_{\ell+1}$
though the diagram:\\
\xymatrix{& & & \textbf{f} \ar @{|->} [ld] \ar @{|->} [rd]  \\ & & \varphi_{f_1} \otimes \cdots \otimes \varphi_{f_{\ell+1}} \ar @{|->} [d] & & \left\langle \mathfrak{Z}_{\textbf{f}} \right\rangle \ar @{|->} [d] \\ & & \varphi_{\textbf{f}} \ar @{|->} [r] \ar @{|-->} [rd] & F_{\textbf{f}} \ar @{|->} @<1ex> [r] \ar @{|->} [d] & [\Omega_{\mathfrak{Z}_{\textbf{f}}}] \ar @{|-->} [ld] \ar @{|->} @<1ex> [l] \\ & & & \mathfrak{R}(F_{\textbf{f}}) }\\
\\
To see (i), note the composition $\mathsf{H}\circ*^{\ell+1}\circ\otimes^{\ell+1}\div$
surjective $\implies$$Res\circ[\cdot]\circ\mathfrak{Z}$ surjective
$\implies$ $Res\circ[\cdot]$ surjective (=Claim 5.1) ($\implies$
$[\cdot]$ surjective (=Beilinson-Hodge)).

For (ii), $Res\circ[\cdot]\circ\mathfrak{Z}$ surjective $\implies$
$[\cdot]\circ\mathfrak{Z}$ surjective (and $Res$ $\cong$) $\implies$
$\theta_{\ell+2}\circ[\cdot]\circ\mathfrak{Z}$ surjective $\implies$
$\EE^{\QQ}\subseteq M^{\QQ}$ is equality. Finally, $\dim\Upsilon_{2}^{\QQ}(N)=|\kappa(N)|$.

Now (iii) follows from eqn. (5.3).
\end{proof}
\begin{rem}
Corollary $6.9$ holds for arbitrary congruence subgroups $\Gamma$
(between $\Gamma(N)$ and $SL_{2}(\ZZ)$), given an appropriate definition
of Eisenstein series for $\Gamma$. This is (referring to $\S5.1.6$)\[
\EE_{\ell+2}^{\QQ}(\Gamma):=\mathfrak{R}^{-1}\left(\mathsf{P}_{\nicefrac{\Gamma(N)}{\Gamma}}^{[\ell]}(\Upsilon_{2}^{\QQ}(\Gamma))\right)\cap\EE_{\ell+2}^{\QQ}(\Gamma(N)),\]
the important point being that these are generated by $\varphi\in\Phi_{2}(N)^{\circ}$
satisfying \emph{$\mathsf{H}_{[\frac{r}{s}]}^{[\ell]}(\varphi)=\mathsf{H}_{[\frac{r'}{s'}]}^{[\ell]}(\varphi)$
whenever $[\frac{r}{s}]$, $[\frac{r'}{s'}]$ $\in\kappa(N)$ map
to the same cusp in $\kappa(\Gamma)$.} We'll look at this condition
further below (in $\S6.2.1$).

Also, a version of the above construction can be made to work for
$\PP\Gamma(2)$ (by choosing an $\cong$ subgroup of $SL_{2}(\ZZ)$)
if $\ell$ is even, but we have omitted this.
\end{rem}

\subsubsection{Additional calculations for the cycle-class}

The results of $\S6.1.2$ lead naturally to a basis for $\EE_{\ell+2}^{\QQ}(\Gamma(N))$
whose elements correspond to holomorphic $(\ell+1)$-forms with $\QQ(\ell+1)$
periods and \emph{log poles along the fiber over exactly one cusp
$\sigma$. }(In some sense this is the most explicit confirmation
of Beilinson-Hodge.)

Writing \[
\Gamma(N)_{i\infty}:=Stab(i\infty\in\uhp^{*})=\left\{ \left(\begin{array}{cc}
1 & aN\\
0 & 1\end{array}\right)\right\} \subset\Gamma(N)\]
\[
PSL_{2}(\ZZ)_{i\infty}:=Stab(i\infty\in\uhp^{*})=\left\{ \pm\left(\begin{array}{cc}
1 & a\\
0 & 1\end{array}\right)\right\} \subset PSL_{2}(\ZZ),\]
we have a short-exact sequence\[
\Gamma(N)_{i\infty}\diagdown\Gamma(N)\,\,\longrightarrow\,\, PSL_{2}(\ZZ)_{i\infty}\diagdown PSL_{2}(\ZZ)\,\,\longrightarrow\,\,\begin{array}[t]{c}
\underbrace{\left\langle \imat\right\rangle \diagdown PSL_{2}(\ZZ/N\ZZ)}\\
\cong\kappa(N)\end{array}.\]
Hence \[
E_{\varphi}^{[\ell]}(\tau)=\frac{-(\ell+1)}{(2\pi i)^{\ell+2}}\sum_{(m,n)\in\ZZ^{2}}{}^{'}\frac{\widehat{\varphi}(m,n)}{(m\tau+n)^{\ell+2}}\]
\[
=\frac{-(\ell+1)}{(2\pi i)^{\ell+2}}\mspace{-50mu}\sum_{\tiny\begin{array}[t]{c}
\pm(m_{0},n_{0})\in\ZZ^{2}/\pm\\
\text{rel. prime}\\
\mspace{150mu}\updownarrow\gamma=\jmat\\
\sum_{\gamma\in\frac{PSL_{2}(\ZZ)}{PSL_{2}(\ZZ)_{i\infty}}}\end{array}}\mspace{-40mu}\sum_{\mathfrak{z}\in\ZZ}{}^{'}\frac{\widehat{\varphi}(\mathfrak{z}m_{0},\mathfrak{z}n_{0})}{(\mathfrak{z}m_{0}\tau+\mathfrak{z}n_{0})^{\ell+2}}\]
\[
=\frac{-(\ell+1)}{(2\pi i)^{\ell+2}}\sum_{\tiny\begin{array}[t]{c}
\sigma\in\kappa(N)\\
\parallel\mspace{60mu}\\
{}[\frac{r}{s}]\mspace{60mu}\end{array}}\mspace{-50mu}\sum_{\tiny\begin{array}[t]{c}
\gamma'\in\frac{\Gamma(N)}{\Gamma(N)_{i\infty}}\cdot\cmat\\
\mspace{150mu}\updownarrow\gamma'=\jmat\\
\sum_{\Tiny\begin{array}[t]{c}
(m_{0},n_{0})\text{ rel. prime,}\\
\emn(-s,r)\end{array}}\end{array}}\mspace{-40mu}\sum_{\mathfrak{z}\in\ZZ}{}^{'}\frac{\widehat{\varphi}(\mathfrak{z}m_{0},\mathfrak{z}n_{0})}{\mathfrak{z}^{\ell+2}(m_{0}\tau+n_{0})^{\ell+2}}.\]
Now since (in the sum) $(m_{0},n_{0})\emn(-s,r)$, $\widehat{\varphi}(\mathfrak{z}m_{0},\mathfrak{z}n_{0})=\widehat{\varphi}(-\mathfrak{z}s,\mathfrak{z}r)=(\iota_{[\frac{r}{s}]}^{*}\widehat{\varphi})(\mathfrak{z})=\widehat{\pi_{[\frac{r}{s}]_{*}}\varphi}(\mathfrak{z)}$
and the above \[
=\sum_{\sigma\in\kappa(N)}\left[\frac{-(\ell+1)}{(2\pi i)^{\ell+2}}\sum_{\mathfrak{z}\in\ZZ}{}^{'}\frac{\widehat{\pi_{[\frac{r}{s}]_{*}}\varphi}(\mathfrak{z})}{\mathfrak{z}^{\ell+2}}\right]\sum_{\tiny\begin{array}[t]{c}
(m_{0},n_{0})\emn(-s,r)\\
\gcd(m_{0},n_{0})=1\end{array}}\frac{1}{(m_{0}\tau+n_{0})^{\ell+2}}\]
\[
=:\sum_{\sigma\in\kappa(N)}\mathsf{H}_{\sigma}^{[\ell]}(\varphi)\tilde{E}_{\sigma}^{[\ell]}(\tau),\]
where the $\sum_{\mathfrak{z}}=\tilde{L}(\widehat{\pi_{[\frac{r}{s}]_{*}}\varphi},\ell+2)$
and $\mathsf{H}_{\sigma}^{[\ell]}(\varphi)$ ($\sigma=[\frac{r}{s}]$)
is the entire bracketed quantity.

\begin{prop}
(i) We have, for $\sigma=[\frac{r}{s}]$, \small \[
\tilde{E}_{\sigma}^{[\ell]}(\tau)=\sum_{\tiny\begin{array}[t]{c}
(m_{0},n_{0})\in\ZZ^{2}\\
\text{rel. prime,}\\
\emn(-s,r)\end{array}}\mspace{-20mu}\frac{1}{(m_{0}\tau+n_{0})^{\ell+2}}=\mspace{-30mu}\sum_{\tiny\begin{array}[t]{c}
(\alpha',\beta')\in\ZZ^{2}\\
\gcd(r+N\alpha',s+N\beta')=1\end{array}}\mspace{-50mu}\frac{1}{(r+N\alpha'-(s+N\beta')\tau)^{\ell+2}}\]
\[
=\mspace{-30mu}\sum_{\tiny\begin{array}[t]{c}
(\alpha,\beta)\in\ZZ^{2}\\
\gcd(1+N\alpha,N\beta)=1\end{array}}\mspace{-40mu}\frac{1}{[(1+\alpha N)(r-s\tau)+\beta N(q+p\tau)]^{\ell+2}}\]
\normalsize In particular,\[
\tilde{E}_{[i\infty]}^{[\ell]}(\tau)=\sum_{\tiny\begin{array}[t]{c}
(\alpha,\beta)\in\ZZ^{2}\\
\gcd(1+N\alpha,N\beta)=1\end{array}}\mspace{-20mu}\frac{1}{(1+N\alpha-N\beta\tau)^{\ell+2}}.\]

(ii) The $\{\tilde{E}_{\sigma}^{[\ell]}(\tau)\}_{\sigma\in\kappa(N)}$
give a basis for the $\EE_{\ell+2}^{\QQ}(\Gamma(N))$, satisfying
$\mathfrak{R}_{\sigma'}(\tilde{E}_{\sigma}^{[\ell]})=\delta_{\sigma\sigma'}.$

(iii) Given $\textbf{\emph{f}}\in\mathcal{O}^{*}(U(N))^{\otimes\ell+1}$,\[
F_{\textbf{\emph{f}}}(\tau)=\sum_{\sigma\in\kappa(N)}\mathsf{H}_{\sigma}^{[\ell]}(\varphi_{\textbf{\emph{f}}})\tilde{E}_{\sigma}^{[\ell]}(\tau).\]

\end{prop}
\begin{proof}
for (ii), pick for each $\sigma$ a $\varphi\in\Phi_{2}^{\QQ}(N)^{\circ}$
so that $\mathsf{H}_{\sigma'}^{[\ell]}(\varphi)=\delta_{\sigma\sigma'}$,
and plug into the computation above. The remainder is clear.
\end{proof}
Next, we have a $q$-series expansion at $[i\infty]$ for the usual
Eisenstein series associated to a {}``divisor on $N$-torsion''
$\varphi\in\Phi_{2}^{\QQ}(N)^{\circ}$: write $q_{0}:=e^{\frac{2\pi i\tau}{N}}=${}``$q^{\frac{1}{N}}$'',
$\xi_{N}(a):=e^{\frac{2\pi ia}{N}}$, ${}^{\ell}\widehat{\varphi}(m,n):=\widehat{\varphi}(m,n)+(-1)^{\ell}\widehat{\varphi}(-m,-n).$

\begin{prop}
$E_{\varphi}^{[\ell]}(\tau)=$\[
\mathsf{H}_{[i\infty]}^{[\ell]}(\varphi)+\frac{(-1)^{\ell+1}}{N^{\ell+2}\ell!}\sum_{M\geq1}q_{0}^{M}\left\{ \sum_{r|M}r^{\ell+1}\left(\sum_{n_{0}\in\nicefrac{\ZZ}{N\ZZ}}\xi_{N}(n_{0}r)\cdot{}^{\ell}\widehat{\varphi}(\frac{M}{r},n_{0})\right)\right\} .\]

\end{prop}
\begin{proof}
essentially in \cite{Gu} for $\ell$ even (also see \cite{Mi}),
but can be derived from scratch using ideas in \cite{Si} (will be
done below for $q$-series of regulator periods).
\end{proof}
Since $q_{0}$ is the local coordinate at $[i\infty]\in\overline{Y}(N)$,
this yields a $power$-$series$ $expansion$ for $F_{\textbf{f}}$
there. We have not tried to directly compute $q$-expansions for the
$\tilde{E}_{\sigma}^{[\ell]}$, but one can plug $\varphi:=\frac{1}{N}\pi_{\sigma}^{*}\varphi_{N}^{[\ell]}$
into $E_{\varphi}^{[\ell]}$ to have the same effect (see Prop. 5.10).
We are particularly interested in the case $\sigma=[i\infty]$. First,
a simplification of Prop. 6.12:

\begin{cor}
For $\varphi_{0}\in\Phi^{\QQ}(N)^{\circ}$, $\varphi:=\frac{1}{N}\pi_{[i\infty]}^{*}\varphi_{0}$,
we have\[
E_{\varphi}^{[\ell]}(\tau)=\frac{(-1)^{\ell}}{\ell!(\ell+2)}\sum_{a=0}^{N}\varphi_{0}(a)B_{\ell+2}(\frac{a}{N})+\frac{(-1)^{\ell+1}}{N^{\ell+1}\ell!}\sum_{\mu\geq1}q_{0}^{N\mu}\left\{ \sum_{r|\mu}r^{\ell+1}\cdot{}^{\ell}\varphi_{0}(r)\right\} ,\]
where ${}^{\ell}\varphi_{0}(a)=\varphi_{0}(a)+(-1)^{\ell}\varphi_{0}(-a).$
\end{cor}
\begin{proof}
${}^{\ell}\widehat{\varphi}={}^{\ell}\widehat{(\frac{1}{N}\pi_{[i\infty]}^{*}\varphi_{0})}=\iota_{[i\infty]_{*}}{}^{\ell}\widehat{\varphi_{0}}$
$\implies$$\sum_{n_{0}}\xi_{N}(n_{0}r)\cdot{}^{\ell}\widehat{\varphi}(\frac{M}{r},n_{0})=0$
if $N\nmid\frac{M}{r}$; otherwise $=\sum_{n_{0}}\xi_{N}(n_{0}r)\cdot{}^{\ell}\widehat{\varphi_{0}(n_{0})}=N\cdot{}^{\ell}\varphi_{0}(r)$.
Put $M=\mu N$.
\end{proof}
Now take $\varphi_{0}$ to be the {}``fundamental vector'' $\varphi_{N}^{[\ell]}$;
then\[
E_{\varphi}^{[\ell]}(\tau)=1+\frac{2(-1)^{\ell+1}}{N^{\ell+1}\ell!}\sum_{\mu\geq1}q_{0}^{N\mu}\left\{ \sum_{r|\mu}r^{\ell+1}\varphi_{N}^{[\ell]}(r)\right\} \]
has $\mathfrak{R}_{\sigma}(E_{\varphi}^{[\ell]})=\delta_{\sigma,[i\infty]}$.

\begin{example}
If $\ell=1$ and $N=3$, from Example 5.11 we get\[
1-9\sum_{\mu\geq1}q_{0}^{3\mu}\left\{ \sum_{r|\mu}r^{2}\chi_{-3}(r)\right\} .\]

\end{example}

\subsection{Push-forwards of the construction}

\subsubsection{Eisenstein symbols for other congruence subgroups $\Gamma$}

Recall that this means $\Gamma(N)\subseteq\Gamma\subseteq SL_{2}(\ZZ)$
($N\geq3$), $\{-\text{id}\}\notin\Gamma$; that automatically $\Gamma(N)\trianglelefteq\Gamma$;
and that there are corresponding quotients $(\E^{[\ell]}(N)\m\text{fibers})\rOnto^{\mathcal{P}_{\nicefrac{\Gamma(N)}{\Gamma}}^{[\ell]}}\E_{\Gamma}^{[\ell]}$,
$(Y(N)\m\text{pts.})\rOnto^{\rho_{\nicefrac{\Gamma(N)}{\Gamma}}}Y_{\Gamma}\m\varepsilon_{\Gamma}.$
Our main examples will be\[
\Gamma_{1}(N):=\left\{ \left.\amat\in SL_{2}(\ZZ)\,\right|\, a\emn1\emn d,\, c\emn0\right\} =\left\langle \Gamma(N),\kmat\right\rangle ,\]
\textbf{\[
\Gamma_{1}^{'}(N):=\left\{ \left.\amat\in SL_{2}(\ZZ)\,\right|\, a\emn1\emn d,\, b\emn0\right\} =\left\langle \Gamma(N),\lmat\right\rangle \]
\[
=\mmat\Gamma_{1}(N)\mmat.\]
}Already for $\Gamma_{1}^{(')}(N)$, $N$ not prime, one has type
$I_{m}^{*}$ cusps --- e.g. $\overline{Y}_{1}^{'}(4)$ has cusps $[i\infty]$
($I_{4}$), $[0]$ ($I_{1}$), $[2]$ ($I_{1}^{*}$). (Also, $Y_{1}^{(')}(3)$
has an elliptic point, but for simplicity our notation will ignore
this fact.)

However, we will consider also {}``traditional'' congruence subgroups
that don't fit our convention: e.g.

\textbf{\[
\Gamma_{0}(N):=\left\{ \left.\amat\,\in SL_{2}(\ZZ)\,\right|\, c\emn0\right\} \,(\ni\{-\text{id}\}),\]
}for which one has $\overline{Y}_{\Gamma}$ but no canonically defined
$\E_{\Gamma}^{[\ell]}$ (though when $N=3,4,6$ one can get around
this problem by observing that $SL_{2}(\ZZ)\twoheadrightarrow PSL_{2}(\ZZ)$
sends $\Gamma_{1}(N)\rTo^{\cong}\PP\Gamma_{0}(N)$). We will also
consider (in $\S6.2.2$) \[
\Gamma^{+N}:=\left\langle \Gamma,\iota_{N}:=\nmat\right\rangle \,(\nsubseteq SL_{2}(\ZZ))\]
 for $\Gamma=\Gamma_{0}(N),$$\Gamma_{1}(N).$

We will now (e.g. using $\mathcal{P}_{\nicefrac{\Gamma(N)}{\Gamma}}^{[\ell]}{}_{_{*}}$)
push the $\{\sZ_{\textbf{f}}\}$ constructed in $\S5.3.4$ forward
to cycles (on families) over these new $Y_{\Gamma}$. The aim in doing
this is to produce more Eisenstein symbols (on families of abelian
varieties or $CY$'s) that live over genus $0$ curves, in order to
link up with those cases of the construction of $\S\S1-2$ which are
classically modular. We note that, while $g(\overline{Y}(N))=0$ only
for $N=(2,)\,3,\,4,\,5$, on the other hand $Y_{1}^{(')}(2-10,12)$
and $Y_{0}(2-10,12,13,16,18,25)$ are all rational.

To get a feel for the behavior of cusps under the various $\overline{\rho}_{\nicefrac{\Gamma'}{\Gamma}}$,
consider the maps $\overline{Y}(N)\to\overline{Y}_{1}(N)\to\overline{Y}_{0}(N)\to\overline{Y}_{0}(N)^{+N}$
for $N$ prime, with (resp.) $\frac{N^{2}-1}{N}$ (all $I_{N}$),
$N-1$ (half each of $I_{N},\, I_{1}$), $2$ ($I_{N},I_{1}$), and
$1$ cusp(s). Since $N$ is prime, one has a correspondence $\kappa(N)\cong\frac{(\ZZ/N\ZZ)^{2}\m\{(0,0)\}}{\left\langle \pm\text{id}\right\rangle }$
, and one can picture how these get equated (e.g. for $N=5$) as in
Figure 6.1,%
\begin{figure}

\caption{\protect\includegraphics[scale=0.7]{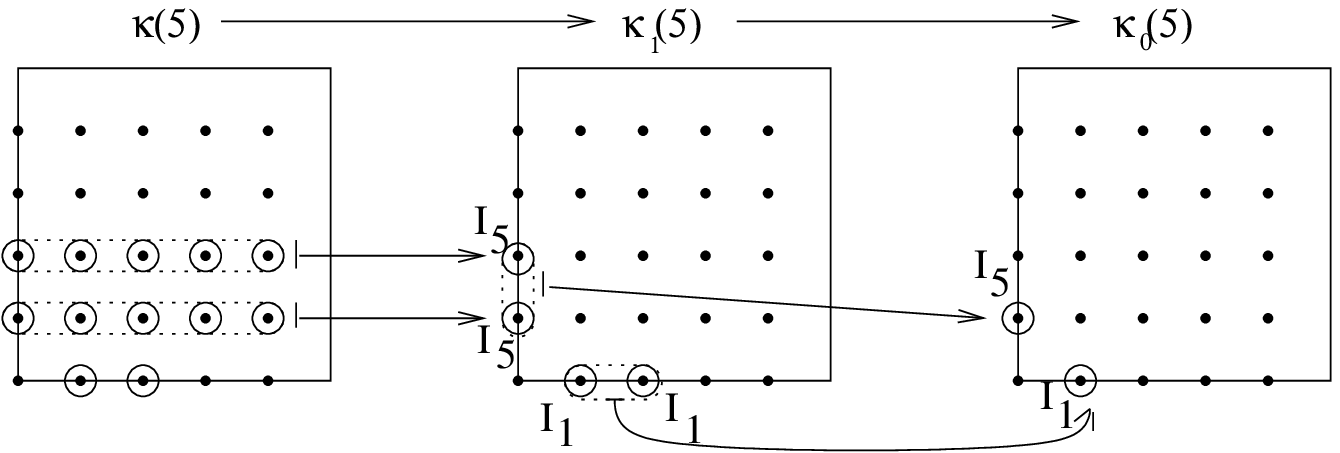}}

\end{figure}
 where circles are chosen representatives of equivalence classes.
Flipping about the diagonal gives the picture for $\kappa(5)\to\kappa_{1}'(5)$.

For $\Gamma'\subset\Gamma$ if index $r$, $\overline{\rho}_{\nicefrac{\Gamma'}{\Gamma}}:\,\overline{Y}_{\Gamma'}\to\overline{Y}_{\Gamma}$
is of degree $r$; if $\Gamma'\trianglelefteq\Gamma$ then $\overline{\rho}_{\nicefrac{\Gamma'}{\Gamma}}$
(omitting cusps/elliptic points and their preimages) is a Galois covering,
so that one has deck transformattions $\{\jmath_{j}\}_{j=1}^{r}$
satisfying $\sum\jmath_{j}^{*}=\rho^{*}\rho_{*}$ (on forms, cycles,
etc.), and corresponding transformations on the Kuga varieties. For
example, one has a diagram ($j=1,\ldots,N$)\\
\xymatrix{& & &
\E^{[\ell]} 
\ar [d]_{\pi^{[\ell]}} 
\ar [rrrr]^{\mathcal{P}^{[\ell]}_{\Gamma_1(N)}} 
\ar [rrdd]|{\mathcal{P}^{[\ell]}_{\Gamma(N)}} 
& & & & 
\E^{[\ell]}_{\Gamma_1(N)} 
\ar [d]^{\pi^{[\ell]}_1(N)} 
\\ & & &
\uhp 
\ar @{..>} [rrrr]^{\rho_{\Gamma_1(N)}} 
\ar [rrdd]|{\rho_{\Gamma(N)}} 
& & & & 
Y_1(N) 
\\ & & &
& & 
{\E^{[\ell]}(N)} 
\ar @(ul,ur)[]|{\J_j}
\ar [d]|{\pi^{[\ell]}(N)} 
\ar [rruu]|{\mathcal{P}^{[\ell]}_{\frac{\Gamma(N)}{\Gamma_1(N)}}} \\ & & &
& & 
Y(N) \ar @(dl,dr)[]|{\jmath_j}
\ar [rruu]|{\rho_{\frac{\Gamma(N)}{\Gamma_1(N)}}}
}\\
\\
(and a similar diagram for $\Gamma_{1}^{'}(N)$) where ${}^{(')}\J_{j}$
and ${}^{(')}\jmath_{j}$ are induced by the action of coset representatives
$\gamma_{j}^{(')}=\omat$ {[}resp. $\pmat$]$\in SL_{2}(\ZZ)$ for
$\nicefrac{\Gamma_{1}^{(')}(N)}{\Gamma(N)}$, on $\E^{[\ell]}$ and
$\uhp$. Now define \[
\mathfrak{Z}_{\textbf{f},1^{(')}}:=\frac{1}{N}\left(\mathcal{P}_{\nicefrac{\Gamma(N)}{\Gamma_{1}^{(')}(N)}}^{[\ell]}\right)_{*}\mathfrak{Z}_{\textbf{f}}\in CH^{\ell+1}(\E_{\Gamma_{1}^{(')}(N)}^{[\ell]},\ell+1);\]
then we have \[
F_{\textbf{f},1^{(')}}:=\theta_{\ell+2}(\Omega_{\mathfrak{Z}_{\textbf{f},1^{(')}}})=\theta_{\ell+2}\left(\left(\mathcal{P}_{\nicefrac{\Gamma(N)}{\Gamma_{1}^{(')}(N)}}^{[\ell]}\right)^{*}\mathfrak{Z}_{\textbf{f},1^{(')}}\right)=\frac{1}{N}\theta_{\ell+2}\left(\sum_{j=1}^{N}{}^{(')}\J_{j}^{*}\Omega_{\mathfrak{Z}_{\textbf{f}}}\right)\]
\[
=\frac{1}{N}\sum_{j=1}^{N}\theta_{\ell+2}(\Omega_{\mathfrak{Z}_{\textbf{f}}})|_{\gamma_{j}^{(')}}^{\ell+2}\,=\,\frac{1}{N}\sum_{j=1}^{N}F_{\textbf{f}}|_{\gamma_{j}^{(')}}^{\ell+2},\]
i.e. \[
F_{\textbf{f},1}(\tau)=\frac{1}{N}\sum_{j=0}^{N-1}F_{\textbf{f}}(\tau+j)\,\,\,\,\,\,\,\text{and}\,\,\,\,\,\,\, F_{\textbf{f},1'}(\tau)=\frac{1}{N}\sum_{j=0}^{N-1}\frac{F_{\textbf{f}}(\frac{\tau}{j\tau+1})}{(j\tau+1)^{\ell+2}}.\]
Writing \begin{equation}
(\rho_*\widehat{\varphi_{\textbf{f}}})(m,n):=\sum_j \widehat{\varphi_{\textbf{f}}}(m,n-mj) \, , \,\,\, ({\rho_*}'\widehat{\varphi_{\textbf{f}}})(m,n):=\sum_j \widehat{\varphi_{\textbf{f}}}(m-nj,n)\\
\end{equation} we get\[
F_{\textbf{f},1^{(')}}(\tau)=\frac{-(\ell+1)}{(2\pi i)^{\ell+2}}\sum_{m,n}{}^{'}\frac{\frac{1}{N}(\rho_{_{*}}^{(')}\widehat{\varphi_{\textbf{f}}})(m,n)}{(m\tau+n)^{\ell+2}}.\]
Using Corollary 6.9(ii) for $\Gamma_{1}^{(')}(N)$ and surjectivity
of $\kappa(N)\to\kappa_{1}^{(')}(N)$, this implies

\begin{prop}
$\left(\mathcal{P}_{\Gamma_{1}^{(')}(N)}^{[\ell]}\right)^{*}$ of
any class in $F^{\ell+1}\cap H^{\ell+1}(\E_{\Gamma_{1}^{(')}(N)}^{[\ell]},\QQ(\ell+1))$
is $(2\pi i)^{\ell+1}\Omega_{F}$ for $F=E_{\varphi}^{[\ell]}$, $\varphi\in\Phi_{2}^{\QQ}(N)$
with $\widehat{\varphi}=\frac{1}{N}\rho_{_{*}}^{(')}\widehat{\varphi}$.
\end{prop}
The effect of $\rho_{*}$ on the $q$-expansion is especially simple:\[
F_{\textbf{f}}(\tau)=\sum_{M\geq0}\alpha_{M}q_{0}^{M}\,\,\implies\,\, F_{\textbf{f},1}(\tau)=\frac{1}{N}\sum_{M\geq0}\alpha_{M}\sum_{j=0}^{N-1}(\xi_{N}(j)\, q_{0})^{M}=\sum_{m\geq0}\alpha_{mN}q^{m},\]
which makes sense since $q$ is the local coordinate at $[i\infty]$
on $\overline{Y}_{1}(N)$.

We are interested in Eisenstein symbols with their only residue at
$[i\infty]$, in analogy to $\S\S1-2$. If $F_{\textbf{f}}=\tilde{E}_{[i\infty]}^{[\ell]}$,
then clearly\[
F_{\textbf{f},1}=\tilde{E}_{[i\infty]}^{[\ell]}\,,\,\,\,\,\,\text{while}\,\,\,\,\, F_{\textbf{f},1'}=\frac{1}{N}\sum_{j=0}^{N-1}\tilde{E}_{[\frac{1}{j}]}^{[\ell]}=\frac{1}{N}\mspace{-20mu}\sum_{\tiny\begin{array}[t]{c}
(\alpha,\beta)\in\ZZ^{2}\\
\gcd(1+N\alpha,\beta)=1\end{array}}\mspace{-20mu}\frac{1}{(1+N\alpha+\beta\tau)^{\ell+2}}.\]
Once $\Gamma$ and $\ell$ are specified, such symbols (or rather,
their cycle-classes) are unique (up to scaling), so for $\Gamma_{1}(N)$
and $\Gamma_{1}^{(')}(N)$ this is it!

\subsubsection{Eisenstein symbols for $K3$ surfaces and $CY$ $3$-fold families}

Given a cycle $\mathfrak{Z}\in CH^{\ell+1}(\E_{\Gamma}^{[\ell]},\ell+1)$
(e.g., $\Gamma=\Gamma(N)$ or $\Gamma_{1}^{(')}(N)$), we have $\Omega_{\mathfrak{Z}}=(2\pi i)^{\ell+1}F_{\mathfrak{Z}}(\tau)\, dz_{1}\wedge\cdots\wedge dz_{\ell}\wedge d\tau$
($F_{\mathfrak{Z}}\in\EE_{\ell+2}^{\QQ}(\Gamma)$), which we assume
$\neq0$. If $\ell=2$, then there is an involution $I:\,(\tau;z_{1},z_{2})\mapsto(\tau;-z_{1},-z_{2})$,
with $I^{*}\Omega_{\mathfrak{Z}}=\Omega_{\mathfrak{Z}}$. Set $\check{\X}_{\Gamma}^{[2]}:=\frac{\E_{\Gamma}^{[2]}}{I}$,
and let $\X_{\Gamma}^{[2]}\to\check{\X}_{\Gamma}^{[2]}$ be the (smooth)
Kummer $K3$ family over $Y_{\Gamma}\m\varepsilon_{\Gamma}$ obtained
by blowing up the $2$-torsion multisections. Using the diagram \begin{equation}
\xymatrix{
& 
\widetilde{\X_\Gamma^{[2]}\times_{\check{\X}^{[2]}_{\Gamma}} \E^{[2]}_{\Gamma}}
\ar @{-->} @/_1pc/ [ldd]_{p_2}
\ar [d]
\ar @{-->} @/^1pc/ [rdd]^{p_1}
\\
& 
\X_{\Gamma}^{[2]} \times_{\check{\X}^{[2]}_{\Gamma}} \E^{[2]}_{\Gamma}
\ar [ld]^{2:1}
\ar [rd]
\\
\X_{\Gamma}^{[2]}
\ar [rd]
& &
\E^{[2]}_{\Gamma}
\ar [ld]^{2:1}
\\
& 
\check{\X}^{[2]}_{\Gamma}
}\\
\end{equation} we define a (nontrivial) cycle by $\mathfrak{Z}_{\X}:=\frac{1}{2}p_{2}{}_{_{*}}p_{1}^{*}\mathfrak{Z}\in CH^{3}(\X_{\Gamma}^{[2]},3)$.
(This will have the same regulator periods and higher normal function
as $\mathfrak{Z}$ by the monodromy argument below. Note also that
if we take $\Gamma=\Gamma_{1}(N)$, then quotienting $\E_{\Gamma}^{[2]}$
by the action of $\Gamma_{0}(N)/\Gamma_{1}(N)$ and blowing up also
yields --- due to the presence of $\qmat$ --- a family of Kummer
$K3$ surfaces over $Y_{0}(N)\m\cdots$ and a nontrivial cycle.) There
is a fiberwise involution $I':\X_{\Gamma}^{[2]}\to\X_{\Gamma}^{[2]}$
induced by $(z_{1},z_{2})\mapsto(z_{1},-z_{2})$ {[}or equivalently
$(-z_{1},z_{2})$], sending $dz_{1}\wedge dz_{2}\mapsto-dz_{1}\wedge dz_{2}$
and fixing the exceptional divisors.

Passing to $\ell=3$, and taking $\mathfrak{Z}\in CH^{4}(\E_{\Gamma}^{[3]},4)$,
we can apply the process above to the first two fiber-factors to obtain
$\mathfrak{Z}'\in CH^{4}(\X_{\Gamma}^{[2]}\times_{_{Y_{\Gamma}\m\varepsilon_{\Gamma}}}\E_{\Gamma},4)$.
Writing $I'':\,\E_{\Gamma}\to\E_{\Gamma}$ {[}$z\mapsto-z$], we have
an involution $I'\times I''$ on $\X_{\Gamma}^{[2]}\times_{_{\cdots}}\E_{\Gamma}$
evidently fixing $\Omega_{\mathfrak{Z}'}$. Blowing up along the singular
set (in each fiber this looks like a disjoint union of $64$ rational
curves) and applying a process similar to the $\ell=2$ case, yields
a family $\X_{\Gamma}^{[3]}$ of Borcea-Voisin ($CY$) $3$-folds
over $Y_{\Gamma}\m\varepsilon_{\Gamma}$, and a nontrivial cycle $\mathfrak{Z}_{\X}\in CH^{4}(\X_{\Gamma}^{[3]},4)$.
(Again, this will have the same regulator periods as $\mathfrak{Z}$.)

Here is a more interesting construction, which yields a $K_{3}$-class
on a $K3$ surface family over $Y_{1}(N)^{+N}$. Recall that the Fricke
involution $\iota_{N}\in SL_{2}(\RR)$ acts on $\uhp$ by $\tau\mapsto-\frac{1}{N\tau}$;
this yields an action of $\Gamma_{1}(N)^{+N}$ on $\uhp^{*}$ with
$\overline{Y}_{1}(N)^{+N}$ as quotient. By normality of $\Gamma_{1}(N)\trianglelefteq\Gamma_{1}(N)^{+N}$,
$\iota_{N}$ also acts on $\overline{Y}_{1}(N)$ with quotient map
$\rho_{+N}:\,\overline{Y}_{1}(N)\twoheadrightarrow\overline{Y}_{1}(N)^{+N}$.

Set $'\E_{1}(N):=\E(N)\times_{\iota_{N}}Y_{1}(N)$, representing points
by $(\tau;\,[z]_{\frac{-1}{N\tau}})$, and consider the relative $N$-isogeny
($not$ an involution!) $J_{N}:{}'\E_{1}(N)\to\E_{1}(N)$ induced
by $(\tau;z)\mapsto(\tau;-N\tau z)$. Writing $'\E_{1}^{[2]}(N):=\E_{1}(N)\times_{_{Y_{1}(N)}}{}'\E_{1}(N)$,
we have $\text{id}\times J_{N}=:J_{N}^{[2]}:\,{}'\E_{1}^{[2]}(N)\to\E_{1}^{[2]}(N)$;
given $F\in M_{4}^{\QQ}(\Gamma_{1}(N))$, ${}'\Omega_{F}:=-\frac{1}{N}(J_{N}^{[2]})^{*}\Omega_{F}=\tau\Omega_{F}.$
Also write $\tilde{J}_{N}^{[2]}:\,\E_{1}^{[2]}(N)\to{}'\E_{1}^{[2]}(N)$
for $(\tau;z_{1},z_{2})\mapsto(\tau;z_{1},\frac{z_{2}}{\tau})$.

Now we are ready to consider the $involution$\\
\xymatrix{
& & & & '\E_1^{[2]}(N) \ar [d]^{\pi} \ar [r]^{I_N^{[2]}} & '\E^{[2]}_1(N) \ar [d]^{\pi} \\ & & & & \uhp \ar [r]^{\iota_N} & \uhp
}\\
\\
induced by exchanging factors: $(\tau;[z_{1}]_{\tau},[z_{2}]_{\frac{-1}{N\tau}})\mapsto(\frac{-1}{N\tau};[z_{2}]_{\frac{-1}{N\tau}},[z_{1}]_{\tau}).$
We have \[
(I_{N}^{[2]})^{*}('\Omega_{F})\,=\,\frac{-1}{N\tau}F\left(\frac{-1}{N\tau}\right)dz_{2}\wedge dz_{1}\wedge d\left(\frac{-1}{N\tau}\right)\]
\[
=\,\tau\left(\frac{1}{N^{2}\tau^{4}}F\left(\frac{-1}{N\tau}\right)\right)dz_{1}\wedge dz_{2}\wedge d\tau\]
\[
=\,'\Omega_{F|_{\iota_{N}}^{4}},\]
where $F|_{\iota_{N}}^{k}(\tau):=\frac{F(\iota_{N}(\tau))}{(\sqrt{N}\tau)^{k}}.$
Set \begin{equation}
F^+:=\frac{1}{2} \left( F+F|^4_{\iota_N}\right) . \\
\end{equation}

Taking the quotient by $I_{N}^{[2]}$\[
\E_{1}^{[2]}(N)^{+N}:=\frac{'\E_{1}^{[2]}(N)\m\pi^{-1}(\nicefrac{i}{\sqrt{N}})}{I_{N}^{[2]}}\lOnto^{\mathcal{P}_{+N}}{}'\E_{1}^{[2]}(N)\m\pi^{-1}(\nicefrac{i}{\sqrt{N}})\]
 and replacing $\E_{\Gamma}^{[2]}$ in $(6.3)$ by this, we get a
family%
\footnote{It may be more desirable to try to construct cycles on a Shioda-Inose
$K3$ family, especially one over $Y_{0}(N)^{+N}$ --- but this seems
difficult to do canonically.%
} $\X_{1}^{[2]}(N)^{+N}$ of (smooth) Kummer $K3$ surfaces over $Y_{1}(N)^{+N}\m\{\nicefrac{i}{\sqrt{N}}\}$.
If $\mathfrak{Z}\in CH^{3}(\E_{1}^{[2]}(N),3)$ with $\theta_{4}(\Omega_{\mathfrak{Z}})=:F_{\mathfrak{Z}}$,
we may define a cycle \begin{equation}
\mathfrak{Z}_{+N}:=\frac{-1}{4N} p_{2*}p_1^*(\mathcal{P}_{+N})_*(J^{[2]}_N)^*\mathfrak{Z} \in CH^3(\X_1^{[2]}(N)^{+N},3). \\
\end{equation} Also take $W\in CH^{3}(\X_{1}^{[2]}(N)^{+N},3)$ to be an arbitrary
cycle.

\begin{prop}
(i) $'\Omega_{F}$ descends to a holomorphic $3$-form with $\QQ(3)$
periods on $\X_{1}^{[2]}(N)^{+N}$ $\Longleftrightarrow$ $F\in M_{4}^{\QQ}\left(\Gamma_{1}(N)^{+N}\right):=[M_{4}^{\QQ}(\Gamma_{1}(N))]^{+}$.

(ii) $\widetilde{W}:=(\tilde{J}_{N}^{[2]})^{*}(\mathcal{P}_{+N})^{*}p_{1}{}_{*}p_{2}^{*}W$
(on $\E_{1}^{[2]}(N)$) has {}``cycle-class'' $\theta_{4}(\Omega_{\tilde{W}})\in$\linebreak
$M_{4}^{\QQ}\left(\Gamma_{1}(N)^{+N}\right)$.

(iii) $\theta_{4}(\Omega_{\widetilde{\mathfrak{Z}_{+N}}})=F_{\mathfrak{Z}}^{+}$.
\end{prop}
Because $'\E_{1}^{[2]}(N)^{+N}$ is not a Kuga variety, we no longer
have that pullbacks $\Omega_{\widetilde{W}}$ to $\E_{1}^{[2]}(N)$
have equal residues at cusps $\in\kappa_{1}(N)$ mapping to the same
cusps $\in\kappa(N)^{+N}$. Consider for simplicity the residues at%
\footnote{Note: the residues of $F$ (hence $F^{+}$) at all $[j]$ ($j\in\ZZ$)
are the same (as the residue at $[0]$).%
} $[0]$ and $[i\infty]$, which are exchanged by the involution on
$\E_{1}^{[2]}(N)$ induced by $\gamma_{0}=\rmat\in SL_{2}(\ZZ)$,
and assume $F\in M_{4}^{\QQ}(\Gamma_{1}(N)^{+N})$ ($\implies$ $N^{-2}\tau^{-4}F\left(\frac{-1}{N\tau}\right)=F(\tau)$).
Then \[
\mathfrak{R}_{[0]}(F)=\lim_{\tau\to i\infty}F|_{\gamma_{0}}^{4}(\tau)=\lim_{\tau\to i\infty}\tau^{-4}F\left(-\frac{1}{\tau}\right)\rEq_{\tau_{0}:=\frac{\tau}{N}}\lim_{\tau_{0}\to i\infty}N^{-4}\tau_{0}^{-4}F\left(\frac{-1}{N\tau_{0}}\right)\]
\[
=N^{-2}\lim_{\tau_{0}\to i\infty}F(\tau_{0})=\frac{\mathfrak{R}_{[i\infty]}(F)}{N^{2}}.\]
If we assume only $F\in M_{4}^{\QQ}(\Gamma_{1}(N))$, then \[
\lim_{\tau\to i\infty}N^{-2}\tau^{-4}F\left(\frac{-1}{N\tau}\right)\rEq_{\tau_{1}:=N\tau}N^{2}\lim_{\tau_{1}\to i\infty}\tau_{1}^{-4}F\left(-\frac{1}{\tau_{1}}\right)=N^{2}\lim_{\tau_{1}\to i\infty}F|_{\gamma_{0}}^{4}(\tau_{1})\]
\[
=N^{2}\mathfrak{R}_{[0]}(F).\]
So \begin{equation}
\mathfrak{R}_{[i\infty]}(F^+)=\frac{1}{2}\left\{ \sR_{[i\infty]}(F) + N^2 \sR_{[0]}(F) \right\} \, , \, \, \, \, \sR_{[0]}(F^+)=\frac{1}{2} \left\{ \frac{1}{N^2}\sR_{[i\infty]}(F)+\sR_{[0]}(F) \right\} . \\
\end{equation} This calculation shows $\left\langle \widetilde{\mathfrak{Z}_{+N}}\right\rangle $
is nontrivial if one picks $\mathfrak{Z}$ so that $\sR_{[i\infty]}(F_{\sZ})\neq-N^{2}\sR_{[0]}(F_{\sZ})$
(obviously possible by $\S6.1.2$).

\begin{rem}
If we replace $I_{N}^{[2]}$ by the order $4$ automorphism $'I_{N}^{[2]}(\tau;[z_{1}]_{\tau},[z_{2}]_{\frac{-1}{N\tau}})=(\frac{-1}{N\tau};[-z_{2}]_{\frac{-1}{N\tau}},[z_{1}]_{\tau}),$
then the corresponding quotient $'\mathcal{P}_{+N}$ yields a family
of singular Kummer surfaces which is then resolved to yield a smooth
$K3$ family $'\X_{1}^{[2]}(N)^{+N}\rOnto^{\pi}Y_{1}(N)^{+N}$. Reworking
this in analogy to $(6.3)$ (so as not to pass through a singular
variety), one constructs a cycle $'\sZ_{+N}$ and most of the exposition
goes through as above with the crucial replacement of $F|_{\iota_{N}}^{4}$
by $-F|_{\iota_{N}}^{4}$ (and $N^{2}$ by $-N^{2}$ in $(6.6)$).
In some sense this is the more natural construction (as the examples
in $\S8$ will suggest). 
\end{rem}

\section{\textbf{Regulator periods and higher normal functions (bis)}}

\subsection{Setup for the fiberwise $AJ$ computation}

We restrict once more to $\Gamma=\Gamma(N)$ and the Kuga modular
varieties $\E^{[\ell]}(N)\rOnto^{\pi^{[\ell]}(N)}Y(N)$, and write
their middle relative cohomology groups: $\HH_{N}^{[\ell]}:=R^{\ell}\pi^{[\ell]}(N)_{_{*}}\ZZ$,
$\H_{N}^{[\ell]}:=\HH_{N}^{[\ell]}\otimes\mathcal{O}_{Y(N)}$, $\H_{N}^{[\ell],\infty}:=\HH_{N}^{[\ell]}\otimes\mathcal{O}_{Y(N)^{\infty}}$,
etc. --- dropping the {}``$N$'' to work on $\E^{[\ell]}/\uhp$,
and flipping super/sub-scripts for homology. One has the subsheaves
of $\G^{*}$($\implies\tilde{\G}^{*}$)-invariants $\text{Sym}^{\ell}\HH_{N,\QQ}^{[1]}\subset\HH_{N,\QQ}^{[\ell]}$,
$\text{Sym}^{\ell}\H_{N}^{[1]}\subset\H_{N}^{[\ell]}$; as well as
$\G^{*}$-coinvariants $\HH_{[\ell]}^{N,\QQ}\twoheadrightarrow\text{Sym}_{\ell}\HH_{[1]}^{N,\QQ}\rTo_{\cong}^{\G^{*}\circ\text{P.D.}}\text{Sym}^{\ell}\HH_{N,\QQ}^{[1]}$.
There are the following well-defined sections $/\uhp$ (multivalued
$/Y(N)$):\[
\alpha=\overrightarrow{[0,1]},\,\beta=\overrightarrow{[0,\tau]}\,\in\Gamma(\uhp,\HH_{[1]})\]
\[
\gamma_{k}^{[\ell]}:=\alpha^{\ell-k}\beta^{k}\in\Gamma(\uhp,\text{Sym}_{\ell}\HH_{[1]}^{\QQ})\]
\[
\tilde{\gamma}_{k}^{[\ell]}:=\G^{*}(\alpha_{1}\times\cdots\times\alpha_{\ell-k}\times\beta_{\ell-k+1}\times\cdots\times\beta_{\ell})\in\Gamma(\uhp,\text{Sym}^{\ell}\HH_{\QQ}^{[1]})\]
\[
\eta_{\ell-k}^{[\ell]}:=\G^{*}(dz_{1}\wedge\cdots\wedge dz_{\ell-k}\wedge d\bar{z}_{\ell-k+1}\wedge\cdots\wedge d\bar{z}_{\ell})\in\Gamma(\uhp,\F^{\ell-k}\text{Sym}^{\ell}\H^{[1],\infty}),\]
where one should think of $\G^{*}$ as reordering the $dz/d\bar{z}$'s
or $\alpha/\beta$'s in all possible ways and dividing by ${\ell \choose k}$.
Writing $[\cdot]_{k}=${}``term of homogeneous degree $k$ in $\tau,\bar{\tau}$'',
\begin{equation}
\left\langle \gamma^{[\ell]}_k ,\eta_{\ell-j}^{[\ell]} \right\rangle = \binom{\ell}{k}^{-1} \left[ (1+\tau)^{\ell-j}(1+\bar{\tau})^j \right]_k = \frac{\sum_{a=0}^k \binom{\ell-j}{a} \binom{j}{k-a}\tau^a\bar{\tau}^{k-a}}{\binom{\ell}{k}} =:\mathfrak{P}^{[\ell]}_{jk} \\
\end{equation} Viewed as the monodromy transformation corresponding to an element
of $\pi_{1}(Y(N))$, $\gamma\in\Gamma(N)$ acts on $(\gamma_{0}^{[\ell]},\ldots,\gamma_{\ell}^{[\ell]})$
from the right, as $\text{Sym}^{\ell}\gamma$; we think of the $\gamma_{i}^{[\ell]}$
as degree-$\ell$ homogeneous polynomials in $\alpha$ and $\beta$,
with $\mu_{i\infty}:=\smat:\,\beta\mapsto\beta+N\alpha,\,\alpha\mapsto\alpha$
and $\mu_{0}:=\tmat:\,\beta\mapsto\beta,\,\alpha\mapsto\alpha+N\beta.$
(Also, $\gamma$ sends $\eta_{\ell-k}^{[\ell]}\mapsto\frac{\eta_{\ell-k}^{[\ell]}}{(c\tau+d)^{\ell-k}(c\bar{\tau}+d)^{k}}$;
note that the $\{\eta_{\ell-k}^{[\ell]}\}$ and $\gamma_{0}^{[\ell]}$
are well-defined over an analytic neighborhood of $[i\infty]$ in
$Y(N)$.)

Now refer to the cycle-construction of $\S5.3.4$, denote the fiberwise
{}``slices'' (pullbacks) of $\left\langle \sZ_{\textbf{f}}\right\rangle $
by $\left\langle \sZ_{\textbf{f}}\right\rangle _{y\text{ (or }\tau\text{)}}$,
etc.; and consider the diagram\\
\xymatrix{
\mathcal{O}^*(U(N))^{\otimes \ell+1}
\ar [d]^{\textbf{f}\mapsto \left\langle \sZ_{\textbf{f}} \right\rangle } 
\ar @/_6pc/ [ddd]_{\R_N^{[\ell]}:=} 
\ar @{->>} [rr]^{\mathsf{H} \circ *^{\ell+1} \circ \otimes^{\ell+1} \div}
& &
\Upsilon^{\QQ}_2(N)
\ar [d]^{Res^{-1}}_{\cong}
\\
\left[ CH^{\ell+1}(\E^{[\ell]}(N),\ell+1) \right]^{\tilde{\G}}
\ar [dd]^{\label}
\ar @{->>} [rr]^{[\cdot]\mspace{50mu}}_{\left\langle \sZ_{\textbf{f}}\right\rangle \mapsto [\Omega_{\sZ_{\textbf{f}}}] \mspace{50mu}}
& & 
\hm\left(\QQ(0),H^{\ell+1}(\E^{[\ell]}(N),\QQ(\ell+1))\right)
\ar @{^(->} [ddd]^{\blabel}_{\alabel}
\\ \\
\Gamma \left( Y(N), \frac{\text{Sym}^{\ell}\H^{[1]}_N}{(\text{Sym}^{\ell}\HH^{[1]}_N)\otimes \QQ(\ell+1)} \right) 
\ar @{^(->} [d]
\\
\Gamma \left( Y(N),\frac{\H^{[\ell]}_N}{\HH^{[\ell]}_{N,\QQ(\ell+1)}} \right)
\ar [rr]^{(-1)^{\ell}\cdot \nabla}
& & 
\Gamma \left( Y(N), \F^{\ell}\H^{[\ell]}_N\otimes \Omega^1_{Y(N)} \right)
}\\
\\
in which the upper square commutes by the proof of Cor. 6.9. Write
simply $\R_{\textbf{f}}(y)$ for the $\R_{N}^{[\ell]}$-image of $\textbf{f}$;
if we pull this back to $\uhp$, we may choose a well-defined lift
$\tilde{\R}_{\textbf{f}}(\tau)\in\Gamma(\uhp,\text{Sym}^{\ell}\H^{[1]})$. 

\begin{lem}
(i) The bottom square commutes.

(ii) $\nabla$ is surjective.
\end{lem}
\begin{proof}
(i) $\left\langle \sZ\right\rangle \in CH^{\ell+1}(\E^{[\ell]}(N),\ell+1)$
has $T_{\sZ}\homeq0$ on $(\pi^{[\ell]}(N))^{-1}(\text{disk})$; so
locally we may write $R_{\sZ}':=R_{\sZ}+(2\pi i)^{\ell+1}\delta_{\d^{-1}T_{\sZ}}$
and compute $\nabla[R_{\sZ}']_{y}=(d[R_{\sZ}'])^{\{1,\ell\}}=\Omega_{\sZ}^{\{1,\ell\}}.$

(ii) follows from irreducibility of the monodromy action on $\text{Sym}^{[\ell]}\HH_{N}^{[1]}$
and consequent vanishing of the space of ($\nabla$-)flat $\G^{*}$-symmetric
normal functions $\Gamma\left(Y(N),\frac{(\text{Sym}^{\ell}\HH_{N}^{[1]})\otimes\CC}{(\text{Sym}^{\ell}\HH_{N}^{[1]})\otimes\QQ(\ell+1)}\right).$
Explicitly, given any $\Gamma=\sum_{k=0}^{\ell}\e_{k}\tilde{\gamma}_{k}^{[\ell]}$
($\{\e_{k}\}\in\CC$), the coefficients of $\tilde{\gamma}_{j}^{[\ell]}$
in $\mu_{i\infty}(\Gamma)-\Gamma=\sum_{j=0}^{\ell-1}\left(\sum_{k=j+1}^{\ell}{k \choose j}\e_{k}N^{k-j}\right)\tilde{\gamma}_{j}^{[\ell]}$
must belong to $\QQ(\ell+1)$; inductively one has $\e_{\ell},\,\e_{\ell-1},\ldots,\e_{1}\in\QQ$.
To show $\e_{0}\in\QQ$, similarly apply $\mu_{0}-\text{id}$.
\end{proof}
\begin{cor}
$\R_{\textbf{\emph{f}}}(y)$ depends only on $\{\mathsf{H}_{\sigma}^{[\ell]}(\varphi_{\textbf{\emph{f}}})\}\in\Upsilon_{2}^{\QQ}(N)$
(or on $\varphi_{\textbf{\emph{f}}}\in\Phi_{2}^{\QQ}(N)$).
\end{cor}
According to $\S\S5.2.4$-$5$, it therefore suffices to compute $\R_{\textbf{f}}$
for $\textbf{f}\in\QQ\left[\mathfrak{F}(N)_{\cmat}^{\times(\ell+1)}\right]$
for {}``each'' $(p,q)$. (In fact, it suffices to do so for $(p,q)=(0,1)$
and $(1,0)$, but it is computationally $convenient$ to consider
at least our choices of $\cmat$ for each $cusp$ $\sigma\in\kappa(N)$.)

For a fixed choice of lift $\tilde{\R}_{\textbf{f}}^{[\ell]}$ (to
be discussed), write \begin{equation}
\tilde{\R}^{[\ell]}_{\textbf{f}}(\tau) =: \sum_{k=0}^{\ell} R^{[\ell]}_{\textbf{f},j}(\tau)\left[ \eta_{\ell-j}^{[\ell]} \right]. \\
\end{equation} We then define regulator periods \begin{equation}
\Psi^{[\ell]}_{\textbf{f},k}(\tau):=\left\langle \gamma^{[\ell]}_k, \tilde{\R}^{[\ell]}_{\textbf{f}}(\tau) \right\rangle \, \, \, \, \, (k=0,\ldots,\ell ) \\
\end{equation} and a higher normal function%
\footnote{It would make more sense on $Y(N)$ to take $V(\tau)=\left\langle \tilde{\R},\mathsf{F}\eta_{\ell}\right\rangle $
for some $\mathsf{F}\in M_{\ell}(\Gamma(N))$; we will essentially
do this later.%
} \begin{equation}
V_{\textbf{f}}^{[\ell]}(\tau):= \left\langle \tilde{\R}^{[\ell]}_{\textbf{f}}(\tau),\eta^{[\ell]}_{\ell} \right\rangle = (-1)^{\binom{\ell+1}{2}}\nu^{\ell}R^{[\ell]}_{\textbf{f},\ell}(\tau). \\
\end{equation} These are the objects which we aim (in the next subsection) to compute
with the \cite{KLM} formula; first we can derive a number of their
properties by {}``pure thought''.

\paragraph*{$\underline{\text{Holomorphicity}}$:}

Since $\nabla_{\d_{\bar{\tau}}}\tilde{\R}_{\textbf{f}}(\tau)=0\in\Gamma(\uhp,\H^{[\ell]})$,
$V_{\textbf{f}}^{[\ell]}$ and the $\{\Psi_{\textbf{f},k}^{[\ell]}\}$
belong to $\mathcal{O}(\uhp)$. The $\{R_{\textbf{f},j}^{[\ell]}\}$
are $not$ holomorphic since the $[\eta_{j}^{[\ell]}]$ aren't (except
for $\eta_{\ell}^{[\ell]}$): \begin{equation}
\nabla \eta^{[\ell]}_j = j\frac{[\eta_{j-1}^{[\ell]}]-[\eta^{[\ell]}_j]}{\nu}\otimes d\tau -(\ell-j)\frac{[\eta_{j+1}^{[\ell]}]-[\eta_j^{[\ell]}]}{\nu}\otimes d\bar{\tau}. \\
\end{equation}

\paragraph*{$\underline{\text{Picard-Fuchs equations}}$:}

Let $\nabla_{\text{PF}}^{\textbf{f}}=\nabla_{\d_{\tau}}^{\ell+1}+\cdots$
denote the PF operator for $\Omega_{\textbf{f}}^{[\ell]}(\tau):=(2\pi i)^{\ell+1}F_{\textbf{f}}(\tau)[\eta_{\ell}^{[\ell]}]\in\Gamma(\uhp,\F^{\ell}\H^{[\ell]}).$
Writing $\bar{\nabla}_{\d_{\tau}}:\nicefrac{\F^{j}}{\F^{j+1}}\to\nicefrac{\F^{j-1}}{\F^{j}}$,
$(7.5)$ $\implies$ $\bar{\nabla}_{\d_{\tau}}\eta_{j}^{[\ell]}=\frac{j}{\nu}[\eta_{j-1}^{[\ell]}]$
$\implies$$\bar{\nabla}_{\d_{\tau}}^{\ell}\eta_{\ell}^{[\ell]}=\frac{\ell!}{\nu^{\ell}}[\eta_{0}^{[\ell]}]$,
which yields the {}``stupid Yukawa coupling''\[
Y_{\tau^{\ell}}(\tau):=\left\langle \eta_{\ell}^{[\ell]},\,\nabla_{\d_{\tau}}^{\ell}\eta_{\ell}^{[\ell]}\right\rangle =(-1)^{{\ell \choose 2}}\frac{\ell!}{\nu^{\ell}}\int dz_{1}\wedge d\bar{z}_{1}\wedge\cdots\wedge dz_{\ell}\wedge d\bar{z}_{\ell}=(-1)^{{\ell \choose 2}}\ell!.\]
Moreover, $\nabla_{\d_{\tau}}^{\ell+1}\eta_{\ell}^{[\ell]}=0$ as
$\eta_{\ell}^{[\ell]}$ has periods $1,\tau,\ldots,\tau^{\ell}$.

\begin{prop}
(i) The $\{\Psi_{\textbf{\emph{f}},k}^{[\ell]}\}$ satisfy the \emph{homogeneous
equation} $(D_{\text{\emph{PF}}}^{\textbf{\emph{f}}}\circ\d_{\tau})(\cdot)=0$.
More precisely, $\frac{d\Psi_{\textbf{\emph{f}},k}^{[\ell]}}{d\tau}=(-1)^{\ell}(2\pi i)^{\ell+1}\tau^{k}F_{\textbf{\emph{f}}}(\tau).$

(ii) $V_{\textbf{\emph{f}}}^{[\ell]}$ satisfies, for any lift $\tilde{\R}_{\textbf{\emph{f}}}$,
the inhomogenous equation \begin{equation}
\d_{\tau}^{\ell+1}(\cdot)=(-1)^{\binom{\ell+1}{2}}(2\pi i)^{\ell+1}\ell!F_{\textbf{\emph{f}}}(\tau) ; \\
\end{equation} i.e. the higher normal function is (const. $\times$) an Eichler
integral of $F_{\textbf{\emph{f}}}$. The various $\{V_{\textbf{\emph{f}}}^{[\ell]}\}$
resulting from the different lifts yield a basis of solutions for
$(7.6)$.
\end{prop}
\begin{proof}
(i) Lemma 7.1(i) says $\nabla_{\d_{\tau}}\tilde{\R}_{\textbf{f}}^{[\ell]}=(-1)^{\ell}\Omega_{\textbf{f}}^{[\ell]}$;
the result follows.

(ii) There are two ways to do this, both instructive:

$\underline{\text{Method I}}$: $\d_{\tau}^{\ell+1}\left\langle \tilde{\R}_{\textbf{f}},\eta_{\ell}\right\rangle =\d_{\tau}^{\ell}\left\langle \tilde{\R}_{\textbf{f}},\nabla_{\d_{\tau}}\eta_{\ell}\right\rangle =\cdots$
\[
\text{[using }\left\langle \eta_{\ell},\nabla_{\d_{\tau}}^{p}\eta_{\ell}\right\rangle =0\,\,\,\,\forall p<\ell\text{]}\]
\[
\cdots\,=\,\d_{\tau}\left\langle \tilde{\R}_{\textbf{f}},\nabla_{\d_{\tau}}^{\ell}\eta_{\ell}\right\rangle \,=\,(-1)^{\ell}(2\pi i)^{\ell+1}\left\langle F_{\textbf{f}}\eta_{\ell},\nabla_{\d_{\tau}}^{\ell}\eta_{\ell}\right\rangle +\left\langle \tilde{\R}_{\textbf{f}},\nabla_{\d_{\tau}}^{\ell+1}\eta_{\ell}\,[=0]\right\rangle \]
\[
=\,(-1)^{\ell}(2\pi i)^{\ell+1}\left\langle F_{\textbf{f}}\eta_{\ell}^{[\ell]},\frac{\ell!}{\nu^{\ell}}\eta_{0}^{[\ell]}+\F^{1}\right\rangle \,=\,(-1)^{\ell+{\ell \choose 2}}(2\pi i)^{\ell+1}\ell!\frac{F_{\textbf{f}}}{\nu^{\ell}}\nu^{\ell}.\]

$\underline{\text{Method II}}$: Note that $\log(\mu_{i\infty})\tilde{\gamma}_{j}^{[\ell]}=j\tilde{\gamma}_{j-1}^{[\ell]}$($=0$
if $j=0$). Taking the priveleged extension basis (single-valued on
$\overline{Y}(N)$, in a neighborhood of $[i\infty]$)\[
\hat{\gamma}_{j}^{[\ell]}:=e^{-\tau\log(\mu_{i\infty})}\tilde{\gamma}_{j}^{[\ell]}\,\,\,\overset{\nabla_{\d_{\tau}}}{\mapsto}\,\,\,-e^{-\tau\log(\mu_{i\infty})}\log(\mu_{i\infty})\tilde{\gamma}_{j}^{[\ell]}=-j\hat{\gamma}_{j-1}^{[\ell]},\]
we write $\tilde{\R}_{\textbf{f}}^{[\ell]}=\sum\hat{\psi}_{j}\hat{\gamma}_{j}^{[\ell]}.$
Applying $\nabla_{\d_{\tau}}$, and using $\hat{\gamma}_{\ell}^{[\ell]}\equiv\eta_{\ell}^{[\ell]}$,
yields\[
\left(\sum_{j=0}^{\ell-1}\left\{ \frac{\d\hat{\psi}_{j}}{d\tau}-(j+1)\hat{\psi}_{j+1}\right\} \hat{\gamma}_{j}^{[\ell]}+\frac{d\hat{\psi}_{\ell}}{d\tau}\hat{\gamma}_{\ell}^{[\ell]}\right)\otimes d\tau=(-1)^{\ell}\Omega_{\textbf{f}}^{[\ell]}\otimes d\tau\]
\[
=(-1)^{\ell}(2\pi i)^{\ell+1}F_{\textbf{f}}\hat{\gamma}_{\ell}^{[\ell]}\otimes d\tau.\]
So \begin{equation}
\left\{ \begin{matrix} \hat{\psi}_{\ell} = (-1)^{\ell} (2\pi i)^{\ell+1}\int F_{\textbf{f}} d\tau \\ \\ \hat{\psi}_j = (j+1)\int \hat{\psi}_{j+1} d\tau \, \, \, \, (j=0,\ldots ,\ell-1); \end{matrix} \right. \\
\end{equation} while $V_{\textbf{f}}^{[\ell]}=\sum\hat{\psi}_{j}\left\langle \hat{\gamma}_{j},\hat{\gamma}_{\ell}\right\rangle =(-1)^{{\ell \choose 2}}\hat{\psi}_{0}.$
To see the {}``basis'' assertion: modifying $\tilde{\R}_{\textbf{f}}$
changes $V_{\textbf{f}}$ by a polynomial in $\tau$ (coefficients
$\in\QQ(\ell+1)$) of degree $\leq\ell$.
\end{proof}
\begin{rem*}
If we notate $\tilde{\R}_{\textbf{f}}^{[\ell]}=\sum\psi_{j}\tilde{\gamma}_{j}$,
then $\left(\begin{array}{c}
\psi_{\ell}\\
\vdots\\
\psi_{0}\end{array}\right)=e^{\tau\log[\mu_{i\infty}]_{\gamma}}\left(\begin{array}{c}
\hat{\psi}_{\ell}\\
\vdots\\
\hat{\psi}_{0}\end{array}\right)$ and this may be used to {}``compute'' $\Psi_{\textbf{f},k}^{[\ell]}=\left\langle \tilde{\gamma}_{k},\tilde{\gamma}_{\ell-k}\right\rangle \psi_{\ell-k}=\frac{(-1)^{k+{\ell \choose 2}}}{{\ell \choose k}}\psi_{\ell-k}.$
\end{rem*}

\paragraph*{$\underline{\text{Monodromy and special values at }[i\infty]}$:}

(This cusp will play a distinguished role later.) If $F_{\textbf{f}}(\tau)\to0$
as $\tau\to i\infty$, then integrating $(-1)^{\ell}(2\pi i)^{\ell}F_{\textbf{f}}(q)\hat{\gamma}_{\ell}^{[\ell]}\otimes\frac{dq}{q}=\nabla\tilde{\R}_{\textbf{f}}^{[\ell]}$
yields on a disk $\Delta\subset Y(N)$ (containing $\{y=0\}=[i\infty]$):
\begin{equation}
(2\pi i)^{\ell+1} \sum_{j=0}^{\ell}\left( Q_j + qP_j(\tau) \right) \tilde{\gamma}^{[\ell]}_j \, , \, \, \, \, Q_j\in\CC \text{ and } P_j\in\mathcal{O}(\Delta)[X]. \\
\end{equation} Since $(\mu_{i\infty}-\text{id})\tilde{\R}_{\textbf{f}}^{[\ell]}$
is of the form $(2\pi i)^{\ell+1}\sum_{j=0}^{\ell}Q_{j}'\tilde{\gamma}_{j}^{[\ell]}$,
we deduce that the $Q_{j}\in\QQ$ for $j\neq0$. A change of lift
$\tilde{\R}_{\textbf{f}}$ merely changes the $\{Q_{j}\}$ (including
$Q_{0}$) by rational numbers.

\begin{prop}
Suppose $\mathsf{H}_{[i\infty]}^{[\ell]}(\varphi_{\textbf{\emph{f}}})=0$,
and set $\mathfrak{K}_{i}:=\lim_{\tau\to i\infty}\Psi_{\textbf{\emph{f}},i}^{[\ell]}(\tau)$.

(i) $\mathfrak{K}_{i}\in\QQ(\ell+1)$ for $0\leq i<\ell$.

(ii) The value of $\mathfrak{K}_{\ell}\in\CC/\QQ(\ell+1)$ is independent
of the lift (i.e. depends only on the other $\{\mathsf{H}_{\sigma}^{[\ell]}(\varphi_{\textbf{\emph{f}}})\}_{(\sigma\neq i\infty)}$).

(iii) Lift $\tilde{\R}_{\textbf{\emph{f}}}^{[\ell]}$ chosen so that
$\{\mathfrak{K}_{i}\}_{i=0}^{\ell-1}$ vanish $\Longleftrightarrow$
$\mathfrak{K}:=\lim_{\tau\to i\infty}V_{\textbf{\emph{f}}}^{[\ell]}(\tau)$
defined. In this case, $\mathfrak{K}=(-1)^{\ell}\mathfrak{K}_{\ell}$
and \begin{equation}
V_{\textbf{\emph{f}}}^{[\ell]}(q)=\mathfrak{K}+(-1)^{\binom{\ell+1}{2}} \ell! \int_0 F_{\textbf{\emph{f}}}(q) \frac{dq}{q} \circ \cdots \circ \frac{dq}{q}. \\
\end{equation}
\end{prop}
\begin{proof}
(i) and (ii) are clear from $(7.8)$. For (iii) (except $(7.9)$),
plug $(7.8)$ into $\left\langle \,\cdot\,,\,\eta_{\ell}^{[\ell]}\right\rangle $.
$(7.9)$ follows from ($\tau\to i\infty$) $\{\Psi_{\textbf{f},i}^{[\ell]}\to0\text{ for }0\leq i<\ell\}\Longleftrightarrow\{\psi_{i}\to0\text{ for }0<i\leq\ell\}\Longleftrightarrow\{\hat{\psi}_{i}\to0\text{ for }0<i\leq\ell\}\Longleftrightarrow$every
$\int$ but the last in $(7.7)$ is taken from $\tau=i\infty$.
\end{proof}
\begin{rem}
(a) $\mathsf{H}_{[i\infty]}^{[\ell]}(\varphi_{\textbf{f}})=0$ means
that (an $AJ$-trivial modification of) $\left\langle \sZ_{\textbf{f}}\right\rangle $
extends across the N\'eron $N$-gon $\hat{E}_{[i\infty]}^{[\ell]}(N)$,
and $\mathfrak{K}_{\ell}$ is essentially $AJ$ of its restriction
(in $H^{\ell}(\hat{E}_{[i\infty]}^{[\ell]}(N),\CC/\QQ(\ell+1))$).
Even with this being well-defined, and even if $\tilde{\R}_{\textbf{f}}^{[\ell]}$
is normalized as in (iii) above, it $need$ $not$ be free of monodromy
about $y=0$! (Of course, when it $is$ monodromy-free, the $\{R_{\textbf{f},k}\},\, V_{\textbf{f}},$
and $\Psi_{\textbf{f},0}$ all follow suit.) This issue has to do
with $\pi^{[\ell]}(N)\left(\left|T_{\mathfrak{Z}_{\textbf{f}}}\right|\right)\subset Y(N)$
and is related to Prop. 2.11.

(b) The lifts used below are chosen for computability rather than
vanishing of $\{\mathfrak{K}_{i}\}.$

(c) One reason we have to do the $AJ$ computation below is to find
$\mathfrak{K}_{\ell}$, if $\mathsf{H}_{[i\infty]}^{[\ell]}(\varphi_{\textbf{f}})=0$
(though we are most interested in the case $\mathsf{H}_{[i\infty]}^{[\ell]}(\varphi_{\textbf{f}})\neq0$).
\end{rem}
For an arbitrary $\textbf{f}$, here is the {}``lift'' we use to
apply KLM:

\begin{itemize}
\item break it up in $\mathcal{O}^{*}(U(N))^{\otimes(\ell+1)}$ into $\sum_{\alpha}\textbf{f}^{\alpha}$,
with each $\varphi_{\textbf{f}^{\alpha}}\in\Phi_{2}^{\QQ}(N)_{(p,q)}^{\circ}$
for some $(p,q)$ as in $\S5.2.4$. This step is not well-defined
w.r.t. the final outcome. Next,
\item break each $\textbf{f}^{\alpha}$ into $\sum_{\beta}\textbf{f}^{\alpha\beta}$,
with each $\textbf{f}^{\alpha\beta}=(f_{1}^{\alpha\beta},\ldots,f_{\ell+1}^{\alpha\beta})\in\mathfrak{F}(N)_{\cmat}^{\times(\ell+1)}$
for some $(-s,r)$ as in $\S5.2.5$; then
\item construct $\tilde{\R}_{\textbf{f}^{\alpha\beta}}$ as in the next
section, and apply KLM.
\end{itemize}
The last 2 steps will yield a well-defined map \[
\Phi_{2}^{\QQ}(N)_{(p,q)}^{\circ}\to\Gamma(\uhp,\text{Sym}^{\ell}\H^{[1]}),\]
 as will be clear from the computations.

\begin{rem*}
$\mathsf{H}_{\sigma}(\varphi_{\textbf{f}^{\alpha}})$ (or $\mathsf{H}_{\sigma}(\varphi_{\textbf{f}^{\alpha\beta}})$)
is $0$ for those $\sigma\longleftrightarrow(-s_{0},r_{0})\in\left\langle (p,q)\right\rangle \subset(\ZZ/N\ZZ)^{2},$
but not necessarily for any other $\sigma\in\kappa(N)$.
\end{rem*}

\subsection{Applying the KLM formula}

This will take place on (subsets of) $\E^{[\ell]}$ rather than $\E^{[\ell]}(N)$;
instead of writing $\mathcal{P}_{N}^{*}$ constantly to pull functions
and cycles back to $\E$ ($\rTo^{\pi}\uhp$), we will take this to
be understood.

Fix a choice of $p,q\in\ZZ$ such that $\left\langle (\bar{p},\bar{q})\right\rangle \cong\ZZ/N\ZZ\subset(\ZZ/N\ZZ)^{2}.$
Taking any $r,s$ {}``completing'' this to an element $M=\cmat\in SL_{2}(\ZZ)$,
we consider $\textbf{f}=(f_{1},\ldots,f_{\ell+1})\in\mathfrak{F}(N)_{M}^{\times(\ell+1)}$,
and compute the $\{R_{\textbf{f},k}^{[\ell]}(\tau)\}$ for a particular
choice of lift $\tilde{\R}_{\textbf{f}}^{[\ell]}(\tau)$ over $(\tau\in)\,\A_{M}$.%
\footnote{$\mathfrak{F}(N)_{M}$ and $\A_{M}$ as in $\S5.2.5$.%
} We then use this to compute the $\Psi_{\textbf{f},j}^{[\ell]}$ over
$\A_{M}$, analytically continue these to $\uhp$, and employ the
result to find the (nonholomorphic) $\{R_{\textbf{f},k}^{[\ell]}(\tau)\}$
over all of $\uhp$.

The choice of lift over $\A_{M}$ must be dealt with in two cases,
according as whether for the Pontryagin product of $(p,q)$-vertical
sets \begin{equation}
0\notin |T_{f_1}|*\cdots *|T_{f_{\ell+1}}| \, \text{ on }\, \pi^{-1}(\A_M)\subset\E . \\
\end{equation} If this is true, then (on all of $\E$) $\{0\}\notin|(f_{1})|*\cdots*|(f_{\ell+1})|$
and (on $\E^{[\ell]}$) we can take $\mathfrak{Z}_{\textbf{f}}:=\,$Zariski
closure of $Z_{\textbf{f}}=\tilde{\G}^{*}\iota^{*}\{\textbf{f}\}$
(see $\S5.3.4$). With this understood, we have

\begin{lem}
$(7.10)$ $\Longleftrightarrow$ $|T_{\mathfrak{Z}_{\textbf{\emph{f}}}}|=\emptyset$
on $\E_{\A_{M}}^{[\ell]}:=(\pi^{[\ell]})^{-1}(\A_{M})\subset\E^{[\ell]}$.
\end{lem}
\begin{proof}
Since $\iota(E_{\tau}^{[\ell]})=\{u_{1}+\cdots+u_{\ell+1}=0\}\subset E_{\tau}^{[\ell+1]}$,
$0\in|T_{f_{1}}|*\cdots*|T_{f_{\ell+1}}|\subset E_{\tau}\Longleftrightarrow$
$0\equiv u_{1}+\cdots+u_{\ell+1}\text{ for some }(u_{1},\ldots,u_{\ell+1})\in|T_{f_{1}}|\cap\cdots\cap|T_{f_{\ell+1}}|\subset E_{\tau}^{[\ell+1]}\Longleftrightarrow$
$\exists(u_{1},\ldots,u_{\ell+1})\in T_{f_{1}}\cap\cdots\cap T_{f_{\ell+1}}\cap\iota(E_{\tau}^{[\ell]})\Longleftrightarrow$
$|T_{\iota^{*}\{\textbf{f}\}}|$ nonempty.
\end{proof}
As a consequence we can take as our lift\[
\tilde{\R}_{\textbf{f}}^{[\ell]}(\tau):=[R_{\mathfrak{Z}_{\textbf{f},\tau}}]\in H^{\ell}(E_{\tau}^{[\ell]},\CC)\,\text{ for }\,\tau\in\A_{M},\]
since (on each fiber) $dR_{\mathfrak{Z}_{\textbf{f},\tau}}=(2\pi i)^{\ell+1}\delta_{T_{\mathfrak{Z}_{\textbf{f},\tau}}}=0.$

\paragraph*{$\underline{\text{Informal remarks on well-definedness}}$:}

Given $\textbf{f}\in\mathfrak{F}(N)_{\cmat}^{\times(\ell+1)}$, $\textbf{g}\in\mathfrak{F}(N)_{\cpmat}^{\times(\ell+1)}$,
with $\varphi_{\textbf{f}}=\varphi_{\textbf{g}}\in\Phi_{2}^{\QQ}(N)_{(p,q)}^{\circ}$
and satisfying $(7.10)$, taking limits along $\A_{M}$ resp. $\A_{M'}$
one finds that $\lim_{\tau\to-\frac{q}{p}}\tilde{\R}_{\textbf{f}}^{[\ell]}$,
$\lim_{\tau\to-\frac{q}{p}}\tilde{\R}_{\textbf{g}}^{[\ell]}$ yield
classes in $H^{\ell}(\hat{E}_{-\frac{q}{p}}^{[\ell]},\CC)$ (the $\{\mathfrak{K}_{i}^{(')}\}_{i=0}^{\ell-1}$
vanish). Also, by Prop. 7.4(ii) these classes are equal up to \linebreak
$H^{\ell}(\hat{E}_{-\frac{q}{p}}^{[\ell]},\QQ(\ell+1))$; hence the
lifts differ at most by $\QQ(\ell+1)\left\langle p[\beta]+q[\alpha]\right\rangle $
on $\uhp.$ That they are in fact equal may be argued from Lemma 6.4,
but the computations below will bear witness to all of this (including
the irrelevancy of $(-s,r)$).$\vspace{2mm}$\\
Now we compute the $\{R_{\textbf{f},j}^{[\ell]}\}$ for our lift.
the diagram $(6.1)$ is replaced for this purpose by%
\footnote{with resp. coordinates $z_{1},\ldots,z_{\ell};\, u_{1},\ldots,u_{\ell+1};\, u$. %
}\[
E_{\tau}^{\ell}\rInto^{\iota}E_{\tau}^{\ell+1}\rOnto^{P}E_{\tau}\,\,,\,\,\,\,\tau\in\A_{M},\]
and the $\pi$'s by integration. Write $\Gamma:=H^{1}(E_{\tau},\ZZ)=\ZZ\left\langle [\alpha],[\beta]\right\rangle ,$
$\gamma=m[\beta]+n[\alpha]=(m,n)\in\Gamma$.

\paragraph*{$\underline{\text{Remarks on currents}}$:}

(i) The fact that $\sZ_{\textbf{f}}=\overline{Z_{\textbf{f}}}$ means
that if $\bar{\mathsf{U}}_{N,\e}\subset E_{\tau}$ denotes the complement
of $\e$-disks about the $N$-torsion points, then $\left\langle [R_{\sZ_{\textbf{f}}}],\eta_{j}^{[\ell]}\right\rangle =\lim_{\e\to0}\int_{\bar{\mathsf{U}}_{N,\e}^{\ell}}R_{Z_{\textbf{f}}}\wedge\eta_{j}^{[\ell]}$
--- but we will just view $R_{Z_{\textbf{f}}}$ as an $L^{1}$-form
on $E_{\tau}^{\ell}$ (rather than write this).

(ii) $R_{\{\textbf{f}\}}=\sum_{j=1}^{\ell+1}(2\pi i)^{j-1}(-1)^{\ell(j-1)}\log f_{j}(u_{j})\dlog f_{j+1}(u_{j+1})\wedge\cdots\wedge\dlog f_{\ell+1}(u_{\ell+1})\cdot\delta_{T_{f_{1}(u_{1})}}\cdot\cdots\cdot\delta_{T_{f_{j-1}(u_{j-1})}}$
is a normal current (of intersection type with respect to $\iota(E_{\tau}^{\ell})$)
on $E_{\tau}^{\ell+1}$, so admits pullback $\iota^{*}R_{\{\textbf{f}\}}=R_{\iota^{*}\{\textbf{f}\}}$
to $E_{\tau}^{\ell}$ (see $\S8$ of \cite{KL}). We also note that
the {}``singularities'' of $P_{*}(R_{\{\textbf{f}\}}\wedge\tilde{\eta}_{j}^{[\ell]})$
are contained in $|T_{f_{1}}|*\cdots*|T_{f_{\ell+1}}|\subset E_{\tau},$
and so are as in Lemma 6.1(ii). Write $\hat{\sum}_{\gamma\in\Gamma}$
for the $\sum_{k}\sum_{j}^{P.V.}$ described there (and depending
on $(p,q)$).$\vspace{2mm}$\\
Writing\[
E_{\tau}^{\ell+1}\rOnto^{\pi_{\widehat{\ell+1}}}E_{\tau}^{\ell}\]
\[
(u_{1},\ldots,u_{\ell},u_{\ell+1})\mapsto(u_{1},\ldots,u_{\ell}),\]
let\[
\tilde{\eta}_{j}^{[\ell]}:=(-1)^{\ell}\pi_{\widehat{\ell+1}}^{*}\eta_{j}^{[\ell]}\,=\,\,(-1)^{\ell}{\ell \choose j}^{-1}\mspace{-20mu}\sum_{\tiny\begin{array}[t]{c}
|J|=j\\
J\subseteq\{1,\ldots,\ell\}\end{array}}\mspace{-20mu}du_{1}^{\{J\}}\wedge\cdots\wedge du_{\ell}^{\{J\}}\in A^{\ell-k,\, k}(E_{\tau}^{\ell+1})\]
where $du_{i}^{\{J\}}:=\left\{ \begin{array}{cc}
du_{i}, & i\in J\\
d\bar{u}_{i}, & i\notin J\end{array}\right..$ We then have $\iota^{*}\tilde{\eta}_{j}^{[\ell]}=\eta_{j}^{[\ell]}$,
and so:\[
\frac{(-1)^{{\ell+1 \choose 2}}(-1)^{\ell-j}\nu^{\ell}}{{\ell \choose j}}R_{\textbf{f},j}^{[\ell]}(\tau)\,\,=\,\, R_{\textbf{f},j}^{[\ell]}(\tau)\int_{E_{\tau}^{\ell}}\eta_{\ell-j}^{[\ell]}\wedge\eta_{j}^{[\ell]}\]
\[
=\,\left\langle \tilde{R}_{\textbf{f}}^{[\ell]},\eta_{j}^{[\ell]}\right\rangle \,\,=\,\,\int_{E_{\tau}^{\ell}}R_{Z_{\textbf{f}}}\wedge\eta_{j}^{[\ell]}\,\,=\,\,\int_{E_{\tau}^{\ell}}\tilde{\G}^{*}R_{\iota^{*}\{\textbf{f}\}}\wedge\tilde{\G}^{*}\eta_{j}^{[\ell]}\]
\[
=\,\int_{E_{\tau}^{\ell}}R_{\iota^{*}\{\textbf{f}\}}\wedge\eta_{j}^{[\ell]}\,\,=\,\,\int_{\iota(E_{\tau}^{\ell})}R_{\{\textbf{f}\}}\wedge\tilde{\eta}_{j}^{[\ell]}\,\,=\,\,\left\{ P_{*}\left(R_{\{\textbf{f}\}}\wedge\tilde{\eta}_{j}^{[\ell]}\right)\right\} (0)\]
\[
=\,\hat{\sum_{\gamma\in\Gamma}}\widehat{P_{_{*}}(R_{\{\textbf{f}\}}\wedge\tilde{\eta}_{j}^{[\ell]})}(\gamma)\,\,=\,\,\nu^{-1}\hat{\sum_{\gamma\in\Gamma}}\int_{E_{\tau}}\overline{\chi_{\gamma}}P_{*}(R_{\{\textbf{f}\}}\wedge\tilde{\eta}_{j}^{[\ell]})\wedge du\wedge d\bar{u}\]
\[
=\,\nu^{-1}\hat{\sum_{\gamma\in\Gamma}}\int_{E_{\tau}^{\ell+1}}P^{*}\overline{\chi_{\gamma}}\cdot R_{\{\textbf{f}\}}\wedge\tilde{\eta}_{j}^{[\ell]}\wedge P^{*}(du\wedge d\bar{u})\]
$=\,\nu^{-1}{\ell \choose j}^{-1}\sum_{j_{0}=1}^{\ell+1}(2\pi i)^{j_{0}-1}(-1)^{\ell j_{0}}$\[
\sum_{\tiny\begin{array}[t]{c}
|J|=j\\
J\subseteq\{1,\ldots,\ell\}\end{array}}\mspace{-20mu}\hat{\sum_{\gamma\in\Gamma}}\int_{E_{\tau}^{\ell+1}}\tiny\begin{array}{c}
P^{*}\overline{\chi_{\gamma}}\cdot\left(\begin{array}{c}
\log f_{j_{0}}\dlog f_{j_{0}+1}\wedge\cdots\wedge\dlog f_{\ell+1}\\
\cdot\delta_{T_{f_{1}}}\cdot\cdots\cdot\delta_{T_{f_{j_{0}-1}}}\end{array}\right)\wedge\\
du_{1}^{\{J\}}\wedge\cdots\wedge du_{\ell}^{\{J\}}\wedge(du_{1}+\cdots+du_{\ell+1})\wedge(d\bar{u}_{1}+\cdots+d\bar{u}_{\ell+1})\end{array}\]
 $=\,\nu^{-1}{\ell \choose j}^{-1}\sum_{j_{0}=1}^{\ell+1}(2\pi i)^{j_{0}-1}(-1)^{(\ell+1)(j_{0}+1)}$\[
\sum_{\tiny\begin{array}[t]{c}
|J_{0}|=j\\
J_{0}\subseteq\{1,\ldots,j_{0}-1\}\end{array}}\mspace{-20mu}\hat{\sum_{\gamma\in\Gamma}}\int_{E_{\tau}^{\ell+1}}\tiny\begin{array}{c}
P^{*}\overline{\chi_{\gamma}}\left(\begin{array}{c}
\log f_{j_{0}}\dlog f_{j_{0}+1}\wedge\cdots\wedge\dlog f_{\ell+1}\\
\cdot\delta_{T_{f_{1}}}\cdot\cdots\cdot\delta_{T_{f_{j_{0}-1}}}\end{array}\right)\wedge\\
du_{1}^{\{J_{0}\}}\wedge\cdots\wedge du_{j_{0}-1}^{\{J_{0}\}}\wedge du_{j_{0}}\wedge d\bar{u}_{j_{0}}\wedge d\bar{u}_{j_{0}+1}\wedge\cdots\wedge d\bar{u}_{\ell+1}\end{array}\]
$=\,(-1)^{{\ell \choose 2}}\nu^{-1}{\ell \choose j}^{-1}\sum_{j_{0}=j+1}^{\ell+1}(2\pi i)^{j_{0}-1}$\[
\sum_{\tiny\begin{array}[t]{c}
|J_{0}|=j\\
J_{0}\subseteq\{1,\ldots,j_{0}-1\}\end{array}}\mspace{-20mu}\hat{\sum_{\gamma\in\Gamma}}\begin{array}[t]{c}
\left(\prod_{m=1}^{j_{0}-1}\int_{T_{f_{m}}}\overline{\chi_{\gamma}}du_{m}^{\{J\}}\right)\left(\int_{E_{\tau}}\overline{\chi_{\gamma}}\log f_{j_{0}}du_{j_{0}}\wedge d\bar{u}_{j_{0}}\right)\\
\times\left(\prod_{m=j_{0}+1}^{\ell+1}\int_{E_{\tau}}\overline{\chi_{\gamma}}\dlog f_{m}\wedge d\bar{u}_{m}\right)\end{array}\]
$\rEq^{\tiny\begin{array}[b]{c}
\text{Lemmas}\\
\text{6.3-4}\end{array}}\,(-1)^{{\ell \choose 2}}\nu^{-1}{\ell \choose j}^{-1}\sum_{j_{0}=j+1}^{\ell+1}(2\pi i)^{j_{0}-1}(-1)^{\ell+1-j_{0}}{j_{0}-1 \choose j}\times$\[
\hat{\sum_{\gamma\in\Gamma}}{}^{'}\,\frac{(p\tau+q)^{j+1}(p\bar{\tau}+q)^{j_{0}-j-1}\nu^{\ell-j_{0}+2}\prod_{m=1}^{\ell+1}\widehat{\varphi_{f_{m}}}(\gamma)}{(2\pi i)^{j_{0}}(mq-np)^{j_{0}}\omega(\gamma)^{\ell-j_{0}+2}}\]
$=\,\frac{(-1)^{{\ell+1 \choose 2}}\nu^{\ell}}{2\pi i{\ell \choose j}}\sum_{j_{0}=j+1}^{\ell+1}(-1)^{j_{0}-1}{j_{0}-1 \choose j}\frac{(p\tau+q)^{j+1}(p\bar{\tau}+q)^{j_{0}-j-1}}{\nu^{j_{0}-1}}\times$\[
\hat{\sum_{\gamma\in\Gamma}}{}^{'}\frac{\widehat{\varphi_{\textbf{f}}}(m,n)}{(m\tau+n)^{\ell-M-j+1}(mq-np)^{M+j+1}},\]
where the primed sum means to omit terms with $mq-np=0.$ Taking $M=j_{0}-j-1$
as summation index, we have therefore \tiny \begin{equation} \boxed{
R^{[\ell]}_{\textbf{f},j}(\tau) = \frac{(-1)^{\ell}}{2\pi i} \sum_{M=0}^{\ell-j} (-1)^M \binom{M+j}{j} \frac{(p\tau+q)^{j+1}(p\bar{\tau}+q)^M}{\nu^{M+j}} \hat{\sum_{(m,n)\in\ZZ^2}} ' \frac{\widehat{\varphi_{\textbf{f}}}(m,n)}{(m\tau+n)^{\ell-M-j+1}(mq-np)^{M+j+1}} .\\
} \end{equation} \normalsize

We now treat the second case, where\[
\{0\}\in\left|T_{f_{1}}\right|*\cdots*\left|T_{f_{\ell+1}}\right|\,\,\,\text{ over }\A_{M}\]
 so that $|T_{\sZ_{\textbf{f}}}|\neq\emptyset$ there. Without loss
of generality the reader can have in mind the case where each $T_{f_{i}}$
(hence $|(f_{i})|$) lies in the connected component of $W_{\tau}^{(p,q)}(N)$
containing $\{0\}$. Let $(\varepsilon_{1},\ldots,\varepsilon_{\ell+1})\in\{|x|<\varepsilon\,|\, x\in\RR\}^{\times(\ell+1)}$
be a very general point in a small polycylinder; we sketch a deformation
argument which shows a lift of $\R_{\textbf{f}}^{[\ell]}(\tau)$ ($\tau\in\A_{M}$)
is still given by $(7.11)$.

Begin by replacing each $f_{j}$ by $f_{j}e^{i\varepsilon_{j}}$ globally
on $\E(N)$, denoting the resulting cycles (from $\S5.3.4$) by $\{\textbf{f}^{\underline{\varepsilon}}\}$,
$Z_{\textbf{f}}^{\underline{\varepsilon}}=\tilde{\G}^{*}\iota^{*}\{\textbf{f}^{\underline{\varepsilon}}\}$;
and note that $\overline{Z_{\textbf{f}}^{\underline{\varepsilon}}}$
is still closed, and now in real good position, on the complement
$\bar{U}^{[\ell]}(N)$ of the $N^{2\ell}$ $N$-torsion sections.
To obtain $\sZ_{\textbf{f}}^{\underline{\varepsilon}}$, we must {}``move
and complete'' $\overline{Z_{\textbf{f}}^{\underline{\varepsilon}}}$;
that is,\[
\left.\sZ_{\textbf{f}}^{\underline{\varepsilon}}\right|_{\bar{U}^{[\ell]}(N)}=\overline{Z_{\textbf{f}}^{\underline{\varepsilon}}}+\db\W_{\textbf{f}}^{\underline{\varepsilon}}\]
for some $\W_{\textbf{f}}^{\underline{\varepsilon}}\in Z_{\RR}^{\ell+1}(\bar{U}^{[\ell]}(N),\ell+2).$
Since obviously $\varphi_{\textbf{f}}=\varphi_{\textbf{f}^{\underline{\varepsilon}}}$,
we have $\Omega_{\sZ_{\textbf{f}}^{\underline{\varepsilon}}}=\Omega_{\sZ_{\textbf{f}}}$
(Theorem 6.6) and therefore $\R_{\textbf{f}^{\underline{\varepsilon}}}^{[\ell]}\equiv\R_{\textbf{f}}^{[\ell]}$
(Corollary 7.2). So it suffices to calculate a lift $\tilde{\R}_{\textbf{f}^{\underline{\varepsilon}}}^{[\ell]}$
for any $\underline{\varepsilon}$, or $\lim_{\underline{\varepsilon}\to\underline{0}}\tilde{\R}_{\textbf{f}^{\underline{\varepsilon}}}$
--- which is in fact what we shall do, working henceforth over a point
$\tau\in\A_{M}$.

Inside $E_{\tau}^{[\ell]}$ we have the open sets\[
\bar{U}_{N,\e}^{[\ell]}\subset\bar{U}_{N}^{[\ell]}:=\,\text{complement of }N^{2\ell}\,\, N\text{-torsion points}\]
\[
\hat{U}_{N,\e}^{[\ell]}\subset\hat{U}_{N}^{[\ell]}:=\,\text{complement of the }\{z_{i}=0,z_{j},-z_{j}\},\]
where the $\e$-subscript denotes removing a closed $\e$-ball/tube
neighborhood. We want to compute (compatible lift-components)\[
\frac{(-1)^{\binom{\ell}{2}+j}\nu^{\ell}}{\binom{\ell}{j}}R_{\textbf{f}^{\underline{\varepsilon}},j}^{[\ell]}(\tau)\,=\,\int_{E_{\tau}^{\ell}}R_{\sZ_{\textbf{f}}^{\underline{\varepsilon}}}\wedge\eta_{j}^{[\ell]}\]
\[
\lim_{\e\to0}\int_{\bar{U}_{N,\e}^{[\ell]}}R_{\sZ_{\textbf{f}}^{\underline{\varepsilon}}}\wedge\eta_{j}^{[\ell]}\,=\,\lim_{\e\to0}\int_{\bar{U}_{N,\e}^{[\ell]}}\left(R_{\overline{Z_{\textbf{f}}^{\underline{\varepsilon}}}}+d[R_{\W_{\textbf{f}}^{\underline{\varepsilon}}}]+(2\pi i)^{\ell+1}\delta_{\mathcal{S}_{\textbf{f}}^{\underline{\varepsilon}}}\right)\wedge\eta_{j}^{[\ell]}\]
\begin{equation}
=\, \lim_{\e \to 0} \int_{\hat{U}^{[\ell]}_{N,\e}} R_{\overline{Z_{\textbf{f}}^{\underline{\varepsilon}}}} \wedge \eta_j^{[\ell]}  \, + \, \lim_{\e\to 0} \int_{\d \bar{U}^{[\ell]}_{N,\e}} R_{\W^{\underline{\varepsilon}}_{\textbf{f}}} \wedge \eta^{[\ell]}_j \, + \, (2\pi i)^{\ell+1} \int_{\mathcal{S}_{\textbf{f}}^{\underline{\varepsilon}}} \eta^{[\ell]}_j ,
\\
\end{equation} where $\mathcal{S}_{\textbf{f}}^{\underline{\varepsilon}}$ is an
$\ell$-chain with $\d(\mathcal{S}_{\textbf{f}}^{\underline{\varepsilon}})=T_{\overline{Z_{\textbf{f}}^{\underline{\varepsilon}}}}+\N$
(with $|\N|\subset N$-torsion points, and nonzero only for $\ell=1$).
One can show that the middle term of $(7.12)$ goes to zero (with
$\e\to0$) at worst like $\e\log^{\kappa}\e$.

Now take the (previously very general) $\varepsilon_{2},\ldots,\varepsilon_{\ell+1}\to\underline{0}$;
then $|T_{\iota^{*}\{\textbf{f}^{\underline{\varepsilon}}\}}|$ limits
into $\{z_{1}\equiv0\}$ and so $|T_{\overline{Z_{\textbf{f}}^{\underline{\varepsilon}}}}|$
limits into $\hat{W}_{N}^{[\ell]}$ (while $R_{\overline{Z_{\textbf{f}}^{\underline{\varepsilon}}}}$
still makes sense on the complement). Since $\overline{Z_{\textbf{f}}^{\underline{\varepsilon}}}$
is $\tilde{\G}^{*}$-invariant by construction, everything else in
$(7.12)$ --- $\W_{\textbf{f}}^{\underline{\varepsilon}}$, $\mathcal{S}_{\textbf{f}}^{\underline{\varepsilon}}$,
etc. --- can be taken to be $\tilde{\G}^{*}$-invariant as well. But
if $\mathcal{S}_{\textbf{f}}^{(\varepsilon_{1},0,\ldots,0)}$ is $\tilde{\G}^{*}$-invariant
and bounds on $\hat{W}_{N}^{[\ell]}$ it must in fact be a cycle on
$E_{\tau}^{\ell}$. This means that in constructing our lift, the
third term of $(7.12)$ can simply be thrown out (which must be done
$(\forall j)$). Finally, taking the limit as $\varepsilon_{1}\to0$
and using $\tilde{\G}^{*}$-invariance of $\eta_{j}^{[\ell]}$, the
first term of $(7.12)$ becomes $\lim_{\e\to0}\int_{\hat{U}_{N,\e}^{[\ell]}}R_{\iota^{*}\{\textbf{f}\}}\wedge\eta_{j}^{[\ell]}$
which puts us back at the start of the computation which led to $(7.11)$.

\subsection{Regulator periods and analytic continuation}

The computations using $(7.11)$ that follow may be justified by appealing
to absolute convergence of the series of the form \begin{equation}
\hat{\sum_{(m,n)\in \ZZ^2}}'  := \sum_{\tiny \begin{matrix} \varkappa\in \ZZ \\ \varkappa \neq 0 \end{matrix} \normalsize} \lim_{J\to\infty} \sum_{\jmath=-J}^J  \mspace{80mu} \tiny \left\{\begin{matrix}m=\jmath p-\varkappa s & & \varkappa=np-mq \\ & \longleftrightarrow \\ n=\jmath q+\varkappa r & & \jmath=ns+mr \end{matrix}\right\} \normalsize \\
\end{equation} if $\pm\jmath$ terms are added first (replacing the {}``$\lim\sum$''
by $\sum_{\jmath\geq0}$). Moreover, the series of this form which
occur do not actually depend on the choice of $(r,s)$.

We start by computing the $\Psi_{\textbf{f},k}^{[\ell]}(\tau)$ for
the lifts $\tilde{\R}_{\textbf{f}}^{[\ell]}(\tau)$ ($\tau\in\A_{M}$)
of the last section. Recycling {}``$\e$'', we let it now denote
a formal variable, and work in $\CC[[\e]]$. Referring to $(7.1)$,
if we write \[
\gamma^{[\ell]}:=\sum_{k=0}^{\ell}\e^{k}{\ell \choose k}\gamma_{k}^{[\ell]},\]
then $\left\langle \gamma^{[\ell]},\eta_{\ell-j}^{[\ell]}\right\rangle =(1+\tau\e)^{\ell-j}(1+\bar{\tau}\e)^{j}$,
so that \[
\sum_{k=0}^{\ell}\Psi_{\textbf{f},k}^{[\ell]}(\tau){\ell \choose k}\e^{k}\,=\,\left\langle \gamma^{[\ell]},\tilde{\R}_{\textbf{f}}^{[\ell]}\right\rangle \]
\begin{equation}
= \, \sum_{j=0}^{\ell}R_{\textbf{f},j}^{[\ell]}(1+\tau \e)^{\ell-j}(1+\bar{\tau}\e)^j \\
\end{equation}\begin{equation} \begin{matrix}
=\, \frac{(-1)^{\ell}}{2\pi i} (1+\tau \e)^{\ell}(p\tau+q) \hat{\sum}_{m,n}^{'} \frac{\widehat{\varphi_{\textbf{f}}}(m,n)}{(m\tau+n)^{\ell+1}(mq-np)}  \times \\ \mspace{120mu} \sum_{j=0}^{\ell}\sum_{M=0}^{\ell-j} \left( \frac{(m\tau+n)(p\bar{\tau}+q)}{(np-mq)\nu } \right)^{M+j} \binom{M+j}{j} \left( -\frac{(1+\bar{\tau}\e)(p\tau+q)}{(1+\tau \e)(p\bar{\tau}+q)} \right)^j. 
\end{matrix} \\ \end{equation} Replacing $M+j$ by $K$ and $\sum_{j}\sum_{M}$ by $\sum_{K=0}^{\ell}\sum_{j=0}^{K}$,
and using\small \[
\sum_{j=0}^{K}\binom{K}{j}\left(-\frac{(1+\bar{\tau}\e)(p\tau+q)}{(1+\tau\e)(p\bar{\tau}+q)}\right)^{j}=\left(1-\frac{(1+\bar{\tau}\e)(p\tau+q)}{(1+\tau\e)(p\bar{\tau}+q)}\right)^{K}=\left(\frac{\nu(p-\e q)}{(1+\tau\e)(p\bar{\tau}+q)}\right)^{K}\]
\normalsize the double-sum in $(7.15)$ becomes\tiny\[
\sum_{K=0}^{\ell}\left(\frac{(m\tau+n)(p-\e q)}{(np-mq)(1+\tau\e)}\right)^{K}=\frac{(np-mq)^{\ell+1}(1+\tau\e)^{\ell+1}-(m\tau+n)^{\ell+1}(p-\e q)^{\ell+1}}{(np-mq)^{\ell}(1+\tau\e)^{\ell}[(np-mq)(1+\tau\e)-(m\tau+n)(p-\e q)]}.\]
\normalsize Simplifying the expression in square brackets to $(p\tau+q)(n\e-m)$,
$(7.15)$ becomes\[
\frac{(-1)^{\ell+1}}{2\pi i}\hat{\sum_{m,n}}{}^{'}\frac{\widehat{\varphi_{\textbf{f}}}(m,n)\left\{ (np-mq)^{\ell+1}(1+\tau\e)^{\ell+1}-(m\tau+n)^{\ell+1}(p-\e q)^{\ell+1}\right\} }{(np-mq)^{\ell+1}(m\tau+n)^{\ell+1}(n\e-m)}\]
--- a {}``zipped'' formula for the $\{\Psi_{\textbf{f},k}^{[\ell]}\}$
which is obviously holomorphic in $\tau$, and hence yields the analytic
continuation to $\uhp$. Since it was substituting $(7.11)$ in $(7.14)$
which yielded this continuation, $(7.11)$ is the correct lift over
$all$ $of$ $\uhp$ (not just $\A_{M}$).

To get explicit formulas for the regulator periods, we reverse the
last step to get $(7.15)=$\[
\frac{(-1)^{\ell+1}}{2\pi i}\hat{\sum_{m,n}}{}^{'}\widehat{\varphi_{\textbf{f}}}(m,n)(p\tau+q)\sum_{\mu=0}^{\ell}\frac{(1+\tau\e)^{\mu}(p-q\e)^{\ell-\mu}}{(np-mq)^{\ell-\mu+1}(n+m\tau)^{\mu+1}},\]
and take coefficients of $\{\e^{k}\}_{k=0}^{\ell}$ (and divide by
${\ell \choose k}$) to find \small \begin{equation} \boxed{
\Psi^{[\ell]}_{\textbf{f},k}(\tau) = \frac{(-1)^{\ell+1}}{2\pi i}(p\tau +q) \hat{\sum_{m,n}}' \widehat{\varphi_{\textbf{f}}}(m,n)  \sum_{\mu=0}^{\ell} \mspace{-3mu} \sum_{a=\text{max}\{0,k-\mu \}}^{\text{min}\{k,\ell-\mu\}} \mspace{-10mu} \frac{(-1)^a\binom{\ell-\mu}{a}\binom{\mu}{k-a}\binom{\ell}{k}^{-1}p^{\ell-\mu-a}q^a \tau^{k-a}}{(np-mq)^{\ell-\mu+1}(m\tau+n)^{\mu+1}}.\\
} \end{equation} \normalsize One can check that this is compatible with Prop. 7.1(i).

Now if we write\[
\mathfrak{F}(N)_{(p,q)}:=\mspace{-20mu}\bigcup_{\tiny\begin{array}[t]{c}
(r,s):\\
\cmat\in SL_{2}(\ZZ)\end{array}}\mspace{-20mu}\mathfrak{F}(N)_{\cmat},\]
then $(7.11)$ and $(7.16)$ extend linearly in an obvious way to
$sums$ of {}``monomials''$\in\mathfrak{F}(N)_{(p,q)}^{\times(\ell+1)}$
(we did this for $\textbf{f}\mapsto\varphi_{\textbf{f}}$ in $\S6.1.2$).

\begin{thm}
Formulas $(7.11)$ and $(7.16)$ yield an abelian group homomorphism
$\tilde{\R}_{(p,q)}^{[\ell]}$ inducing $AJ$ on {}``$(p,q)$-vertical
Eisenstein symbols'', as described in the diagram\\
\xymatrix{
&
\Phi^{\QQ}_2(N)^{\circ}_{(p,q)}
\ar [rdd]^{(7.11)}
\ar [ldd]_{(7.16)}
\\
&
\QQ\left[\mathfrak{F}(N)_{(p,q)}^{\times(\ell+1)}\right]
\ar @/_15pc/ [ddd]_{\textbf{\emph{f}}\mapsto \left\langle  \sZ_{\textbf{\emph{f}}} \right\rangle}
\ar [rd]
\ar [ld]
\ar [d]^{\tilde{\R}^{[\ell]}_{(p,q)}}
\ar [u]_{\textbf{\emph{f}}\mapsto \varphi_{\textbf{\emph{f}}}}
\\
\left( \mathcal{O}_{\uhp} \right)^{\ell+1}
\ar @{->>} [d]
&
\Gamma \left( \uhp, \text{Sym}^{\ell}\H^{[1]} \right)
\ar [l]_{\cong\mspace{20mu}}^{\text{ev}_{\left\{\gamma^{[\ell]}_k\right\}^*}\mspace{20mu}}
\ar @{^(->} [r]_{\mspace{20mu}\text{ev}_{\left\{ \eta^{[\ell]}_{\ell-j}\right\}}}
\ar [d]
&
\left( \mathcal{O}_{\uhp^{\infty}} \right)^{\ell+1}
\ar @{->>} [d]
\\
\frac{\left(\mathcal{O}_{\uhp}\right)^{\ell+1}}{\LL}
&
\Gamma \left( \uhp, \frac{\text{Sym}^{\ell}\H^{[1]}}{(\text{Sym}^{\ell}\HH^{[1]})_{\QQ(\ell+1)}} \right)
\ar [l]_{\cong \mspace{40mu}}
\ar @{^(->} [r]
&
\frac{\left( \mathcal{O}_{\uhp^{\infty}}\right)^{\ell+1}}{\LL_{\infty}}
\\
&
\Gamma \left( \uhp , \underline{\underline{\mathcal{CH}}}\left( \E^{[\ell]}/\uhp , \ell+1 \right)\right) 
\ar [u]_{AJ}
}\\
\\
where {}``ev'' means to write a vector with respect to the given
basis, $\{\,\,\}^{*}$ is the dual basis, while \tiny $\LL=\QQ(\ell+1)\left\langle \left(\begin{array}{c}
1\\
0\\
\vdots\\
0\end{array}\right),\left(\begin{array}{c}
0\\
1\\
0\\
\vdots\\
0\end{array}\right),\ldots,\left(\begin{array}{c}
0\\
\vdots\\
0\\
1\end{array}\right)\right\rangle $\normalsize and \tiny $\LL_{\infty}\rEq^{(7.1)}\QQ(\ell+1)\left\langle \left(\begin{array}{c}
\mathfrak{P}_{00}^{[\ell]}\\
\vdots\\
\mathfrak{P}_{\ell0}^{[\ell]}\end{array}\right),\ldots,\left(\begin{array}{c}
\mathfrak{P}_{0\ell}^{[\ell]}\\
\vdots\\
\mathfrak{P}_{\ell\ell}^{[\ell]}\end{array}\right)\right\rangle .$\normalsize
\end{thm}
The two {}``extreme'' periods are of special interest. For the $\alpha^{\ell}$-period,
$(7.16)$ yields \begin{equation}
\begin{matrix} 
\Psi^{[\ell]}_{\textbf{f},0}(\tau) \,=\, (-1)^{\ell}(2\pi i)^{\ell+1}(\tau+\frac{q}{p})\mathsf{H}^{[\ell]}_{[i\infty]}(\varphi_{\textbf{f}}) \mspace{200mu} 
\\ \mspace{100mu}
+\frac{(-1)^{\ell+1}}{2\pi i}\hat{\sum}_{\tiny \begin{matrix} m,n \\ m\neq 0 \end{matrix}}^{'} \widehat{\varphi_{\textbf{f}}}(m,n)\frac{(m\tau+n)^{\ell+1}p^{\ell+1}-(np-mq)^{\ell+1}}{m(m\tau+n)^{\ell+1}(np-mq)^{\ell+1}}
\end{matrix} \\
\end{equation} if $p\neq0$, and \begin{equation}
\Psi^{[\ell]}_{\textbf{f},0}(\tau)\,=\, \frac{(-1)^{\ell}}{2\pi i} \hat{\sum}^{'}_{m,n} \frac{\widehat{\varphi_{\textbf{f}}}(m,n)}{m(m\tau+n)^{\ell+1}}
\\
\end{equation} if $p=0$ ($q=1$). For the $\beta^{\ell}$-period, we have \begin{equation} \begin{matrix}
\Psi^{[\ell]}_{\textbf{f},\ell} \, = \, (-1)^{\ell+1}(2\pi i)^{\ell+1} (\frac{1}{\tau}+\frac{p}{q})\mathsf{H}^{[\ell]}_{[0]}(\varphi_{\textbf{f}}) \mspace{200mu} \\ \mspace{100mu}
+\frac{(-1)^{\ell+1}}{2\pi i} \hat{\sum}^{'}_{\tiny \begin{matrix} m,n \\ n \neq 0 \end{matrix}}\widehat{\varphi_{\textbf{f}}}(m,n)\frac{(np-mq)^{\ell+1}\tau^{\ell+1}+(-1)^{\ell}(m\tau+n)^{\ell+1}q^{\ell+1}}{n(m\tau+n)^{\ell+1}(np-mq)^{\ell+1}}
\end{matrix} \\
\end{equation} if $q\neq0$ and \begin{equation}
\Psi^{[\ell]}_{\textbf{f},\ell}(\tau ) \, =\, \frac{(-1)^{\ell+1}}{2\pi i}\tau^{\ell+1}\hat{\sum}^{'}_{m,n} \frac{\widehat{\varphi_{\textbf{f}}}(m,n)}{n(m\tau+n)^{\ell+1}} \\
\end{equation} if $q=0$ ($p=1$). We also record the higher normal function for
convenience: using $(7.4)$ and $(7.11)$, this is \begin{equation}
V_{\textbf{f}}^{[\ell]}(\tau) \, = \, \frac{(-1)^{\binom{\ell}{2}}}{2\pi i} (p\tau+q)^{\ell+1} \hat{\sum}^{'}_{(m,n)\in \ZZ^2}\frac{\widehat{\varphi_{\textbf{f}}}(m,n)}{(m\tau+n)(mq-np)^{\ell+1}}.
\\
\end{equation}

By the monodromy argument (Lemma 7.1(ii)) together with $\S6.1.2$,
$AJ$ factors through $\Upsilon_{2}^{\QQ}(N)$. That is, for $any$
$\textbf{f}\in\mathcal{O}^{*}(U(N))^{\otimes(\ell+1)}$ \begin{equation}
\Psi^{[\ell]}_{\textbf{f},k}(\tau) \, = \, \sum_{\sigma \in \kappa(N)} \mathsf{H}_{\sigma}^{[\ell]}(\varphi_{\textbf{f}})\tilde{\Psi}^{[\ell]}_{\sigma,k}(\tau) \, \, \, \, \text{mod} \, \, \QQ(\ell+1) \\
\end{equation} where (using our chosen $\cmat\in SL_{2}(\ZZ)$ for each $\sigma=[\frac{r}{s}]$)
$\tilde{\Psi}_{\sigma,k}^{[\ell]}=\Psi_{\textbf{f}_{\sigma},k}^{[\ell]}$
for some $\textbf{f}_{\sigma}\in\QQ[\mathfrak{F}(N)_{\cmat}^{\times(\ell+1)}]$
satisfying $\mathsf{H}_{\sigma'}^{[\ell]}(\varphi_{\textbf{f}_{\sigma}})=\delta_{\sigma\sigma'}$.
We take $\varphi_{\textbf{f}_{\sigma}}=\frac{1}{N}\pi_{\sigma}^{*}\varphi_{N}^{[\ell]}$,
so that $(7.16)$%
\footnote{choice of $(p,q)$ in $(7.16)$ is different for each $\sigma$. %
} yields \tiny \begin{equation}
\tilde{\Psi}^{[\ell]}_{\sigma,k}(\tau) := \frac{(-1)^{\ell+1}}{\ell+1}(2\pi i)^{\ell+1}(p\tau+q) \mspace{-50mu} \hat{\sum_{\tiny \begin{matrix} \alpha, \beta\in \ZZ^2 \\ \gcd(1+N\alpha,N\beta)=1 \end{matrix}}} \mspace{-40mu} \sum_{\mu=0}^{\ell}\frac{\sum_{a=\text{max}\{0,k-\mu\}}^{\text{min}\{k,\ell-\mu \}}(-1)^a\binom{\ell-\mu}{a}\binom{\mu}{k-a}\binom{\ell}{k}^{-1}p^{\ell-\mu-a}q^a\tau^{k-a}}{(1+N\alpha)^{\ell-\mu+1}\{(1+N\alpha)(r-s\tau)+N\beta(q+p\tau)\}^{\mu+1}},
\\
\end{equation} \normalsize where we have computed as in $\S6.1.3$ with $(m,n)=:\mathfrak{z}(m_{0},n_{0})$,
$(m_{0},n_{0})=:(r+N(\beta q+\alpha r),\,-s+N(\beta p-\alpha r))$
and where $\hat{\sum}$ means to sum $\pm\beta$ first. A similar
result holds for $V_{\textbf{f}}^{[\ell]}(\tau)$, only modulo polynomials
(of degree $\leq\ell$ with $\QQ(\ell+1)$ coefficients).

Also as in $\S6.1.3$ one can do the Fourier expansions in some cases
(and we need these for the examples below). For instance, for $(p,q)=(1,0)$
and $k=0$, $(7.16)$ becomes \begin{equation}
(-1)^{\ell}(2\pi i)^{\ell+1}\tau \mathsf{H}^{[\ell]}_{[i\infty]}(\varphi_{\textbf{f}})+ \frac{(-1)^{\ell+1}}{2\pi i}\hat{\sum}^{'}_{\tiny \begin{matrix} m,n \\ m\neq 0 \end{matrix}} \widehat{\varphi_{\textbf{f}}}(m,n)\frac{(m\tau+n)^{\ell+1}-n^{\ell+1}}{m(m\tau+n)^{\ell+1}n^{\ell+1}}\\
\end{equation} where $\hat{\sum}_{m,n}'$ means $\sum_{\tiny\begin{array}[t]{c}
n\in\ZZ\\
n\neq0\end{array}}\lim_{M\to\infty}\sum_{m=-M}^{M}.$ Assuming additionally that\linebreak $\varphi_{\textbf{f}}(m,n)=\varphi_{\textbf{f}}(m,-n)$
{[}$\Longleftrightarrow\widehat{\varphi_{\textbf{f}}}(m,n)=\widehat{\varphi_{\textbf{f}}}(-m,n)$],
the $\hat{\sum}_{\tiny\begin{array}[t]{c}
m,n\\
m\neq0\end{array}}^{'}\frac{\widehat{\varphi_{\textbf{f}}}(m,n)}{mn^{\ell+1}}=0$ and the second term of $(7.24)$ becomes \begin{equation}
\frac{(-1)^{\ell}}{2\pi i} \sum_{(m,n)\in (\ZZ\m {0})^2}\frac{\widehat{\varphi_{\textbf{f}}}(m,n)}{m(m\tau+n)^{\ell+1}}. \\
\end{equation}

\begin{prop}
If $\varphi_{\textbf{\emph{f}}}=\frac{1}{N}\pi_{[i\infty]}^{*}\varphi$
($\varphi\in\Phi^{\QQ}(N)^{\circ}$) then $\widehat{\varphi_{\textbf{\emph{f}}}}=\iota_{[i\infty]_{*}}\widehat{\varphi}$
and we have \begin{equation} \begin{matrix}
\Psi^{[\ell]}_{\textbf{\emph{f}},0}(\tau) \, = \, \frac{(2\pi i)^{\ell}(\ell+1)N}{(\ell+2)!} \left( \sum_{b=0}^{N-1} \varphi(b) B_{\ell+2}(\frac{b}{N}) \right) \log q_0 
\mspace{100mu} \\ \mspace{200mu} -\frac{(2\pi i)^{\ell}}{\ell ! N^{\ell+1}} \sum_{M\geq 1} \frac{\left( \sum_{r|M} r^{\ell+1} \cdot ^{\ell} \varphi (r) \right)}{M} q_0^{MN} ,
\end{matrix} \\
\end{equation} where ${}^{\ell}\varphi(r)=\varphi(r)+(-1)^{\ell}\varphi(-r)$.
\end{prop}
\begin{proof}
Let $\xi\in\{1,2,\ldots,N-1\}$, and $m_{0}\in\NN$. Using the product
expansion of $\sin(\pi(\alpha+z))$ from \cite[sec. 2.3, Ex. 2]{Ah},
we have \begin{equation}
\frac{d^{\ell+1}}{d\tau^{\ell+1}} \log \left\{ \sin \left( \frac{\pi\xi}{N}+\pi m_0 \tau \right) \right\} \, =
\\
\end{equation}\[
\frac{d^{\ell+1}}{d\tau^{\ell+1}}\left\{ \pi m_{0}\tau\cot\left(\frac{\pi\xi}{N}\right)+\sum_{n_{0}\in\ZZ}\left[\log\left(1+\frac{Nm_{0}\tau}{Nn_{0}+\xi}\right)-\frac{Nn_{0}\tau}{Nn_{0}+\xi}\right]\right\} \,=\]
\[
-\frac{d^{\ell}}{d\tau^{\ell}}\left\{ \sum_{n_{0}\in\ZZ}\frac{N^{2}m_{0}^{2}\tau}{(Nm_{0}+\xi)(Nn_{0}+\xi+Nm_{0}\tau)}\right\} \,=\,\]
\[
(-1)^{\ell}\ell!N^{\ell+1}m_{0}^{\ell+1}\sum_{n_{0}\in\ZZ}\frac{1}{(Nn_{0}+\xi+Nm_{0}\tau)^{\ell+1}}.\]
On the other hand using the Taylor expansion for $\log$, $(7.27)$
becomes \[
\frac{d^{\ell+1}}{d\tau^{\ell+1}}\log\left\{ \frac{1}{2i}\left(e^{\frac{\pi i}{N}\left(\xi+m_{0}N\tau\right)}-e^{-\frac{\pi i}{N}\left(\xi+m_{0}N\tau\right)}\right)\right\} \,=\]
\[
\frac{d^{\ell+1}}{d\tau^{\ell+1}}\log\left(1-e^{2\pi im_{0}\tau}e^{\frac{2\pi i}{N}\xi}\right)\,=\,-\frac{d^{\ell+1}}{d\tau^{\ell+1}}\sum_{r\geq1}\frac{1}{r}e^{\frac{2\pi ir\xi}{N}}e^{2\pi im_{0}r\tau}\,=\]
\[
-(2\pi i)^{\ell+1}m_{0}^{\ell+1}\sum_{r\geq1}r^{\ell}e^{\frac{2\pi ir\xi}{N}}q_{0}^{rm_{0}N};\]
hence we have (for $m_{0}>0$) $\alpha(\xi,m_{0}):=$\[
\sum_{n_{0}\in\ZZ}\frac{1}{(Nn_{0}+\xi+Nm_{0}\tau)^{\ell+1}}\,=\,\frac{(-1)^{\ell+1}(2\pi i)^{\ell+1}}{\ell!N^{\ell+1}}\sum_{r\geq1}r^{\ell}e^{\frac{2\pi i\xi r}{N}}q_{0}^{rm_{0}N}.\]
Substituting $\widehat{\varphi_{\textbf{f}}}=\iota_{[i\infty]_{*}}\widehat{\varphi}$
in $(7.25)$ therefore yields\[
\frac{(-1)^{\ell}}{2\pi i}\sum_{(n,m_{0})\in(\ZZ\m\{0\})^{2}}\frac{\widehat{\varphi}(n)}{Nm_{0}(n+Nm_{0}\tau)^{\ell+1}}\,=\]
\[
\frac{(-1)^{\ell}}{2\pi iN}\sum_{\xi=1}^{N-1}\widehat{\varphi}(\xi)\sum_{m_{0}'\geq1}\frac{1}{m_{0}'}\left\{ \alpha(\xi,m_{0}')+(-1)^{\ell}\alpha(-\xi,m_{0}')\right\} \,=\]
\[
\frac{-(2\pi i)^{\ell}}{\ell!N^{\ell+2}}\sum_{M\geq1}q_{0}^{MN}\sum_{r|M}\frac{r^{\ell+1}}{M}\sum_{\xi\in\nicefrac{\ZZ}{N\ZZ}}\widehat{\varphi}(\xi)\left\{ e^{\frac{2\pi i\xi r}{N}}+(-1)^{\ell}e^{-\frac{2\pi i\xi r}{N}}\right\} \,=\]
\[
\frac{-(2\pi i)^{\ell}}{\ell!N^{\ell+1}}\sum_{M\geq1}q_{0}^{MN}\sum_{r|M}\frac{r^{\ell+1}}{M}{}^{\ell}\varphi(r),\]
where we have reindexed $M=m_{0}'r$. The first term of $(7.26)$
is much easier. 
\end{proof}
We turn briefly to the higher normal function. In analogy to $(7.24)$,
for $(p,q)=(1,0)$ equation $(7.21)$ becomes \begin{equation}
V_{\textbf{f}}^{[\ell]}(\tau) \, = \, \frac{(-1)^{\binom{\ell+1}{2}}(2\pi i)^{\ell+1}}{\ell+1} \tau^{\ell+1}\mathsf{H}^{[\ell]}_{[i\infty]}(\varphi_{\textbf{f}}) - \frac{(-1)^{\binom{\ell+1}{2}}}{2\pi i} \tau^{\ell+1} \hat{\sum}^{'}_{\tiny \begin{matrix}m,n \\ m\neq 0 \end{matrix}} \frac{\widehat{\varphi_{\textbf{f}}}(m,n)}{(m\tau+n)n^{\ell+1}},  \\
\end{equation} and if $\varphi_{\textbf{f}}=\frac{1}{N}\pi_{[i\infty]}^{*}\varphi$
we can calculate its $q_{0}$-expansion as follows. Using\[
\frac{\tau^{\ell+1}}{(Nm_{0}\tau+n)n^{\ell+1}}\,=\,\sum_{j=1}^{\ell}\frac{(-1)^{j-1}\tau^{\ell-j+1}}{(Nm_{0})^{j}m^{\ell-j+2}}\,+\,\frac{(-1)^{\ell}\tau}{(Nm_{0}\tau+n)(Nm_{0})^{\ell}n}\,,\]
the second term of $(7.28)$ becomes\[
\frac{(-1)^{\binom{\ell+1}{2}}}{2\pi i}\sum_{J=1}^{\left\lfloor \frac{\ell}{2}\right\rfloor }\frac{\tau^{\ell-2J+1}}{N^{2J}}\sum_{(m_{0},n)\in(\ZZ\m\{0\})^{2}}\frac{\widehat{\varphi}(n)}{m_{0}^{2J}n^{\ell-2J+2}}\mspace{100mu}\]
\[
\mspace{100mu}-\frac{(-1)^{\binom{\ell}{2}}}{2\pi iN^{\ell+2}}\sum_{\xi=1}^{N-1}\widehat{\varphi}(\xi)\sum_{m_{0}\in\ZZ}{}^{'}\frac{1}{m_{0}^{\ell+1}}\sum_{n_{0}\in\ZZ}\frac{N^{2}m_{0}\tau}{(\xi+Nn_{0})(\xi+Nn_{0}+Nm_{0}\tau)}.\]
For $m_{0}>0$ the $\sum_{n_{0}\in\ZZ}$ is \[
\pi\left(i+\cot\left(\frac{\pi\xi}{N}\right)\right)+2\pi i\sum_{r\geq1}e^{\frac{2\pi ir\xi}{N}}q_{0}^{m_{0}Nr}\]
by an argument like that in the above proof. Writing \[
\Theta_{\ell}(\varphi):=\left\{ \begin{array}{cc}
-\frac{i}{N}\sum_{\xi\in\nicefrac{\ZZ}{N\ZZ}}\widehat{\varphi}(\xi)\cot\left(\frac{\pi\xi}{N}\right), & \ell\text{ odd}\\
\\\varphi(0)\,, & \ell\text{ even}\end{array}\right.\]
and noting $\zeta(2J)=\frac{-(2\pi i)^{2J}}{2(2J)!}B_{2J}$, we eventually
arrive at this expression for the higher normal function (associated
to our lift):%
\footnote{the first big braced expression in $(7.29)$ is a polynomial in $\tau$
with $\QQ(\ell+1)$-coefficients.%
} \begin{equation}
\begin{matrix}
\frac{(-1)^{\binom{\ell}{2}}N^{\ell+1}}{(\ell+2)!}
\left\{
\begin{matrix}
\left( \sum_{a=0}^{N-1} \varphi(a)B_{\ell+2}\left( \frac{a}{N} \right) \right) \log^{\ell+1}q_0 
\\ + \, \sum_{J=1}^{\left\lfloor \frac{\ell}{2} \right\rfloor} \frac{(-2\pi i)^{2J}}{N^{4J}} \binom{\ell+2}{2J}B_{2J}\left( \sum_{a=0}^{N-1} \varphi(a) B_{\ell-2J+2}\left( \frac{a}{N} \right) \right) \log^{\ell-2J+1}q_0 
\end{matrix}
\right\}
\\
-\frac{(-1)^{\binom{\ell}{2}}}{N^{\ell+1}} \left\{ \zeta(\ell+1) \Theta_{\ell}(\varphi) +\sum_{M\geq 1} q_0^{MN}\left( \frac{\sum_{r|M}r^{\ell+1}\cdot ^{\ell}\varphi(r)}{M^{\ell+1}} \right) \right\}.
\end{matrix} \\ 
\end{equation} Both $(7.29)$ and $(7.26)$ check against Proposition 7.3 and Corollary
6.13, as the reader may verify.

Finally, one can evaluate the regulator periods at cusps where $\Omega_{\sZ_{\textbf{f}}}$
has no residue. We demonstrate this for the $\alpha^{\times\ell}$-period.

\begin{prop}
Assume that $\mathsf{H}_{[\frac{r}{s}]}^{[\ell]}(\varphi_{\textbf{\emph{f}}})\,\left[=\frac{-(\ell+1)}{(2\pi i)^{\ell+2}}\tilde{L}(\widehat{\varphi_{\textbf{\emph{f}}}},\ell+2)\right]=0$;
then \[
\lim_{\tau\to\frac{r}{s}}\Psi_{\textbf{\emph{f}},0}^{[\ell]}(\tau)\equiv\frac{-s^{\ell}}{2N}\tilde{L}_{-}(\pi_{[\frac{r}{s}]_{*}}\widehat{\varphi_{\textbf{\emph{f}}}},\ell+1)\,\,\,\,\,\text{mod }\QQ(\ell+1),\]
 where $\tilde{L}_{-}(\phi,\ell+1):=\sum_{m\in\ZZ\m\{0\}}\frac{\phi(m)\cdot\frac{|m|}{m}}{m^{\ell+1}}.$
\end{prop}
\begin{proof}
will proceed by $first$ showing that \begin{equation}
\lim_{\tau\to i\infty} \Psi^{[\ell]}_{\textbf{f},\ell}(\tau) \equiv \frac{-1}{2N} \tilde{L}_{-} (\pi_{[i\infty]_*}\widehat{\varphi_{\textbf{f}}},\ell+1) \\
\end{equation} when $\mathsf{H}_{[i\infty]}^{[\ell]}(\varphi_{\textbf{f}})=0$.
We can write $\varphi_{\textbf{f}}=\varphi_{\textbf{f}'}+\varphi_{\textbf{f}''}$
where $\varphi_{\textbf{f}'}\in\pi_{[0]}^{*}\Phi^{\QQ}(N)^{\circ}\subset\Phi_{2}^{\QQ}(N)_{(0,1)}^{\circ}$
and $\varphi_{\textbf{f}''}\in\Phi_{2}^{\QQ}(N)_{(1,0)}^{\circ}$,
then apply $(7.19)$ {[}with $(p,q)=(0,1)$] resp. $(7.20)$ to conclude
\begin{equation}
\lim_{\tau\to i\infty}\Psi^{[\ell]}_{\textbf{f},\ell}(\tau) \equiv \lim_{\tau\to i\infty}\frac{(-1)^{\ell+1}}{2\pi i}\sum_{(m,n)\in(\ZZ\m {0})^2}\frac{\widehat{\varphi_{\textbf{f}}}(m,n)}{n(m+\frac{n}{\tau})^{\ell+1}} \, \, \, \, \text{mod}\, \, \QQ(\ell+1)\\
\end{equation} after {}``reassembling'' the results. (In $(7.19)$ the sum becomes
\[
\frac{1}{N}\hat{\sum}_{\tiny\begin{array}[t]{c}
m,n_{0}\\
n_{0}\neq0\end{array}}^{'}\left(\frac{\widehat{\varphi_{\textbf{f}}}(m,0)}{n_{0}(m+\frac{Nn_{0}}{\tau})^{\ell+1}}-\frac{\widehat{\varphi_{\textbf{f}}}(m,0)}{n_{0}m^{\ell+1}}\right),\]
 where the $\hat{\sum}$ means to sum $\pm n_{0}$ first, so that
one can delete the second term inside the sum. Then one can remove
the {}``$\widehat{\,\,}$'', in both $(7.19)$ and $(7.20)$%
\footnote{where it means to sum $\pm m$ first.%
}, since the double-sum is now absolutely convergent.) The r.h.s. of
$(7.31)$ is now (summing $\pm n$ first)\tiny \[
\lim_{\tau\to i\infty}\frac{(-1)^{\ell+1}}{2\pi i}\sum_{\xi=0}^{N-1}\sum_{m\in\ZZ}{}^{'}\widehat{\varphi_{\textbf{f}}}(m,\xi)\sum_{n\geq1}\left(\frac{1}{(n_{0}N-\xi)\left(M+\frac{n_{0}N-\xi}{\tau}\right)^{\ell+1}}-\frac{1}{(n_{0}N-\xi)\left(m-\frac{n_{0}N-\xi}{\tau}\right)^{\ell+1}}\right)\]
\normalsize where we have made the (unnecessary) assumption that
$\widehat{\varphi_{\textbf{f}}}(m,-n)=\widehat{\varphi_{\textbf{f}}}(m,n)$
to simplify the exposition. This becomes (writing $\tau=it$)\tiny \[
\frac{2(-1)^{\ell+1}i^{\ell+1}}{2\pi iN}\sum_{m\in\ZZ}{}^{'}\sum_{\xi=0}^{N-1}\widehat{\varphi_{\textbf{f}}}(m,\xi)\sum_{k=0}^{\ell}(-1)^{k}\left\{ \lim_{t\to\infty}\sum_{n_{0}\geq1}\frac{N/t}{\left(\frac{n_{0}N-\xi}{t}+im\right)^{\ell-k+1}\left(\frac{n_{0}N-\xi}{t}-im\right)^{k+1}}\right\} \]
\normalsize where the limit in braces is the Riemann sum for \[
\int_{0}^{\infty}\frac{dX}{(X+im)^{\ell-k+1}(X-im)^{k+1}}=\frac{1}{2}(2\pi i)(-1)^{\ell+k}\frac{|m|}{m}{\ell \choose k}\frac{1}{(2mi)^{\ell+1}}\]
(using residues), and so we get\[
-\frac{\sum_{k=0}^{\ell}{\ell \choose k}}{2^{\ell+1}N}\sum_{m\in\ZZ}{}^{'}\frac{|m|}{m^{\ell+2}}\sum_{\xi=0}^{N-1}\widehat{\varphi_{\textbf{f}}}(m,\xi)\]
which is just the r.h.s. of $(7.30)$.

Now let $\textbf{f}$ be as in the statement of the Proposition:\[
\lim_{\tau\to\frac{r}{s}}\Psi_{\textbf{f},0}^{[\ell]}(\tau)=\left\langle [\alpha^{\times\ell}],\,\lim_{\tau\to\frac{r}{s}}\R_{\textbf{f}}^{[\ell]}(\tau)\right\rangle =\left\langle [\alpha^{\times\ell}],\,\cmat^{*}\R_{\dmat^{*}\textbf{f}}^{[\ell]}(\tau)\right\rangle .\]
By $(7.30)$ this is\[
-\frac{(-1)^{\binom{\ell+1}{2}}}{2N}\tilde{L}_{-}\left(\pi_{[i\infty]_{*}}\dmat^{*}\widehat{\varphi_{\textbf{f}}},\,\ell+1\right)\left\langle [\alpha^{\times\ell}],\,\cmat^{*}[\alpha^{\times\ell}]\right\rangle \]
\[
=-\frac{(-1)^{\binom{\ell+1}{2}}}{2N}\tilde{L}_{-}\left(\pi_{[\frac{r}{s}]_{*}}\widehat{\varphi_{\textbf{f}}},\,\ell+1\right)\left\langle [\alpha^{\times\ell}],\,[(r\alpha-s\beta)^{\times\ell}]\right\rangle \]
which yields the result. 
\end{proof}
\begin{rem*}
In fact, Prop. 7.8 leads to a more general result when combined with
results from previous sections:
\end{rem*}
\begin{cor}
For any $\textbf{\emph{f}}\in\mathcal{O}^{*}(U(N))^{\otimes(\ell+1)}$,\[
\Psi_{\textbf{\emph{f}},0}^{[\ell]}(\tau)\,\,\emql\,\,(-2\pi i)^{\ell}\mathsf{H}_{[i\infty]}^{[\ell]}(\varphi_{\textbf{\emph{f}}})N\log q_{0}\mspace{200mu}\]
\[
\small\mspace{100mu}-\frac{(2\pi i)^{\ell}}{N^{\ell+1}\ell!}\sum_{M\geq1}\frac{q_{0}^{M}}{M}\left\{ \sum_{r|M}r^{\ell+1}\left(\sum_{n_{0}\in\nicefrac{\ZZ}{N\ZZ}}e^{\frac{2\pi in_{0}r}{N}}\cdot{}^{\ell}\widehat{\varphi_{\textbf{\emph{f}}}}\left(\frac{M}{r},n_{0}\right)\right)\right\} .\normalsize\]

\end{cor}
\begin{proof}
Split $\varphi_{\textbf{f}}=\varphi_{\textbf{f}'}+\varphi_{\textbf{f}''}$
with $\varphi_{\textbf{f}'}\in\pi_{[i\infty]}^{*}(\Phi^{\QQ}(N)^{\circ})$,
and $\varphi_{\textbf{f}''}$ $(0,1)$-vertical so that $\mathsf{H}_{[i\infty]}^{[\ell]}(\varphi_{\textbf{f}''})=0$.
By Prop. 7.4(i), $\lim_{\tau\to i\infty}\Psi_{\textbf{f}'',0}^{[\ell]}(\tau)=0$
while the constant and divergent terms (as $\tau\to i\infty$) for
$\Psi_{\textbf{f}',0}^{[\ell]}$ (hence $\Psi_{\textbf{f},0}^{[\ell]}$)
are given by Prop. 7.8. Using this together with Prop. 6.12 and Prop.
7.3(i) (which says that $\Psi_{\textbf{f},0}^{[\ell]}=(-1)^{\ell}(2\pi i)^{\ell+1}\int E_{\varphi_{\textbf{f}}}(\tau)d\tau$)
gives the result.
\end{proof}

\section{\textbf{Toric vs. Eisenstein: comparing constructions}}

In this final section we consider the possible coincidence of (push-forwards
of) Beilinson's Eisenstein symbol over genus zero modular curves,
and the toric symbol on suitably {}``modular'' hypersurface pencils.
This will be done on the level of regulator periods and cycle-classes,
and the general result in $\S8.3$ is followed by many examples. To
whet the reader's appetite we include two motivating examples in $\S8.1,$
which come from extending the computations of regulator periods and
their special values to the cycles considered in $\S6.2$.

\subsection{Regulator periods for other congruence subgroups}

It is worth mentioning a subtlety that enters into computations for
the {}``push-forward cycles'' of $\S6.2.1$ $\sZ_{\textbf{f},1^{(')}}:=\frac{1}{N}\left(\mathcal{P}_{\nicefrac{\Gamma(N)}{\Gamma_{1}^{(')}(N)}}^{[\ell]}\right)_{*}\sZ_{\textbf{f}}\in CH^{\ell+1}(\E_{\Gamma_{1}^{(')}(N)}^{[\ell]},\ell+1)$
(equivalently one can consider $\widetilde{\sZ_{\textbf{f},1^{(')}}}:=\left(\mathcal{P}_{\nicefrac{\Gamma(N)}{\Gamma_{1}^{(')}(N)}}^{[\ell]}\right)^{*}\sZ_{\textbf{f},1^{(')}}$
on $\E^{[\ell]}(N)$). Letting $\Psi_{\textbf{f},1^{(')};k}^{[\ell]}$
denote the period over $\gamma_{k}^{[\ell]}\,(=\alpha^{\ell-k}\beta^{k})$
for an appropriate lift of the fiberwise $AJ$ of $\sZ_{\textbf{f},1^{(')}}$
over $Y_{1}^{(')}(N)$, we have obviously \begin{equation}
\Psi^{[\ell]}_{\textbf{f},1;0}(\tau) = \frac{1}{N}\sum_{j=0}^{N-1} \Psi^{[\ell]}_{\textbf{f},0} (\tau+j)\\
\end{equation} but also \begin{equation}
\Psi^{[\ell]}_{\textbf{f},1;\ell}(\tau)=\frac{1}{N} \sum_{j=0}^{N-1}\sum_{k=0}^{\ell} \binom{\ell}{k}(-j)^{\ell-k} \Psi^{[\ell]}_{\textbf{f},k}(\tau+j) \\
\end{equation} \begin{equation}
\Psi^{[1]}_{\textbf{f},1';0}(\tau) = \frac{1}{N} \sum_{j=0}^{N-1} \left\{ \Psi^{[1]}_{\textbf{f},0} \left( \frac{\tau}{j\tau+1} \right) -j \Psi^{[1]}_{\textbf{f},1}\left( \frac{\tau}{j\tau +1} \right) \right\} \\
\end{equation} since (see $\S6.2.1$) $\J_{j}{}_{_{*}}\beta=\beta-j\alpha$ (resp.
$\J_{j}^{'}{}_{_{*}}\alpha=\alpha-j\beta$). Likewise, for the {}``$K_{3}(K3)$''
cycles $\sZ_{\textbf{f},+N}:=\frac{-1}{4N}(p_{2})_{*}(p_{1})^{*}(\mathcal{P}_{+N})_{*}(J_{N}^{[2]})^{*}\sZ_{\textbf{f},1}\in CH^{3}(\X_{1}^{[2]}(N)^{+N},3)$
(resp. $'\sZ_{\textbf{f},+N}$) of $\S6.2.2$, we find \begin{equation}
{}^{(')}\Psi^{[2]}_{\textbf{f},+N;0}(\tau)=\frac{1}{2}\left\{ \Psi^{[2]}_{\textbf{f},1;0}(\tau) \tiny \begin{matrix} + \\ (-) \end{matrix} \normalsize N\Psi^{[2]}_{\textbf{f},1;2}\left( \frac{-1}{N\tau} \right) \right\} \\
\end{equation} for the periods of $AJ\left(\left\langle {}^{(')}\sZ_{\textbf{f},+N}\right\rangle _{[\tau]\in Y_{1}(N)^{+N}}\right)$
against $(\mathcal{P}_{+N})_{*}(J_{N}^{[2]})_{*}(\alpha\times\alpha)$.
(The latter, it turns out, is divisible by $2N$ in the integral homology
of the $K3$ fibers.) To obtain limiting values of $(8.1\text{-}4)$
at a cusp, one could apply the proof of Prop. 7.9 to each term.

An easier approach is to consider the effect of $\sZ_{\textbf{f}}\mapsto\widetilde{\sZ_{\textbf{f},1^{(')}}}$
(or $'\widetilde{\sZ_{\textbf{f},+N}}$) on the residues of the cycle-class,
transform $\widehat{\varphi_{\textbf{f}}}$ accordingly (cf. $(6.2)$),
and plug the result into Prop. 7.9. We carry this out in two examples
related to toric constructions in this paper.

\begin{example}
($\ell=1$, $N=4$, $\Gamma=\Gamma_{1}^{'}(4)$)\\
Begin with $\textbf{f}$ so that $\varphi_{\textbf{f}}=-\frac{1}{4}\pi_{[i\infty]}^{*}\varphi_{4}^{[1]}$
(see Prop. 5.10) and consider $\sZ_{\textbf{f},1'}$; the corresponding
divisor $\varphi_{\textbf{f},1'}$ has $\widehat{\varphi_{\textbf{f},1'}}=\frac{1}{4}\rho_{*}^{'}\widehat{\varphi_{\textbf{f}}}=-\frac{1}{4}\rho_{*}^{'}\iota_{[i\infty]_{*}}\widehat{\varphi_{4}^{[1]}}=-\frac{1}{4}\pi_{[0]}^{*}\widehat{\varphi_{4}^{[1]}}$
where $\widehat{\varphi_{4}^{[1]}}=0,2^{6}i,0,-2^{6}i$. We have $\pi_{[0]_{*}}\widehat{\varphi_{\textbf{f},1'}}=-\widehat{\varphi_{4}^{[1]}}$
and so \[
\lim_{\tau\to0}\Psi_{\textbf{f},1';0}^{[1]}(\tau)\equiv\frac{1}{8}\tilde{L}_{-}(\widehat{\varphi_{4}^{[1]}},2)=-16iG\,\,\,\,\,\,\,\text{mod}\,\,\QQ(2);\]
this corresponds exactly to the $D5$ example of $\S4.3$.
\end{example}
${}$

\begin{example}
($\ell=2$, $N=6$, $\Gamma=\Gamma_{1}(6)^{+6}$)\\
Start with $\varphi_{\textbf{f}}=-4\pi_{[i\infty]}^{*}\varphi_{6}^{[2]}$,
and consider $'\widetilde{\sZ_{\textbf{f},+6}}$: from $(6.6)$ (and
Remark $6.17$) we know that if $\mathsf{H}_{\sigma}(\varphi_{\textbf{f}})=-24\delta_{\sigma,[i\infty]}$
then $\mathsf{H}_{[i\infty]}('\varphi_{\textbf{f},+6})=-12$ and $\mathsf{H}_{[j]}('\varphi_{\textbf{f},+6})=\frac{1}{3}$
($\forall j\in\ZZ$). As $\widehat{\varphi_{6}^{[2]}}=0,-\frac{6^{4}}{5},0,0,0,-\frac{6^{4}}{5}$,
this leads to\begin{equation}
'\widehat{\varphi_{\textbf{f},+6}}(m,n) = \left\{ 
\begin{matrix}
\frac{2\cdot 6^5}{5}, & (m,n)\ems \pm (0,1)
\\
-\frac{2\cdot 6^3}{5}, & m\ems \pm 1 
\\
0, & \text{otherwise},
\end{matrix}
\right.\\
\end{equation} and $\pi_{[\frac{-1}{2}]_{*}}{}'\widehat{\varphi_{\textbf{f},+6}}=-\frac{8\cdot6^{3}}{5}\cdot\{0,1,-9,1,-9,1;\ldots\}$
so that \[
\lim_{\tau\to-\frac{1}{2}}{}'\Psi_{\textbf{f},+6;0}^{[2]}(\tau)\,\emqt\,-\frac{4}{12}\cdot\frac{-8\cdot6^{3}}{5}\cdot2L(\{0,1,-9,1,-9,1;\ldots\},\,3)\]
\[
=\,\frac{2^{5}\cdot6^{2}}{5}\zeta(3)\cdot\left(1-\frac{10}{2^{3}}+\frac{9}{6^{3}}\right)\,=\,-48\zeta(3).\]
This means that the $AJ$ class of $\left\langle '\widetilde{\sZ_{\textbf{f},+6}}\right\rangle _{\tau}$
limits to $12\zeta(3)\left[(\alpha+2\beta)^{\times2}\right],$ which
is the pullback from the $K3$ family of $2\zeta(3)$ times a vanishing
cycle at $[\frac{-1}{2}]\in\overline{Y}_{1}(6)^{+6}$. This suggests
a link to the Ap\'ery-Beukers higher normal function from the introduction;
the precise relation will be established in $\S8.5$ below.
\end{example}

\subsection{Uniformizing the genus zero case}

Let $\Gamma\subset SL_{2}(\ZZ)$ be a congruence subgroup in the sense
of $\S5.1.1$ ($\{-\text{id}\}\notin\Gamma$, $\Gamma\supset\Gamma(N)$
for some $N\geq3$), and assume $\overline{Y}_{\Gamma}\cong\PP^{1}$.
To fix a uniformizing parameter, note that $\overline{Y}_{\Gamma}$
has local coordinate%
\footnote{e.g. $N_{\Gamma}=N$ for $\Gamma=\Gamma(N)$ or $\Gamma_{1}'(N)$,
while $N_{\Gamma}=1$ for $\Gamma=\Gamma_{1}(N)$ (or $\Gamma_{1}(N)^{+N}$,
though we don't treat this yet).%
} $q_{0}:=q^{\frac{1}{N_{\Gamma}}}=e^{\frac{2\pi i\tau}{N_{\Gamma}}}$
in a neighborhood of $[i\infty]$, and let $H\in\check{M}_{0}(\Gamma)$
be the (unique) Hauptmodul with Fourier expansion $H(q_{0})=\text{const.}\cdot q_{0}+\text{h.o.t.}$.%
\footnote{We will assume $H$ is normalized so that this constant is a root
of unity.%
} Given an {}``Eisenstein symbol'' $\sZ\in CH^{\ell+1}(\E_{\Gamma}^{[\ell]},\ell+1)$
(with $(\mathcal{P}_{\nicefrac{\Gamma(N)}{\Gamma}}^{[\ell]})^{*}\sZ\equiv\sZ_{\textbf{f}}\in CH^{\ell+1}(\E^{[\ell]}(N),\ell+1)$),
writing the data \{$\Omega_{\sZ_{\textbf{f}}},\Psi_{\textbf{f},0}^{[\ell]},V_{\textbf{f}}^{[\ell]},$
PF-equations, etc.\} in terms of $t:=H(\tau)$ yields expressions
resembling those of $\S\S1-2$ arising from the {}``toric symbols''.

While there are intersections between the two constructions (systematically
developed in $\S\S8.3$-$6$), neither one includes the other. Let
$\omega_{\nicefrac{\E}{Y}}^{\Gamma}:=K_{\overline{\E}_{\Gamma}^{[\ell]}}\otimes\overline{\pi}^{-1}(\theta_{\overline{Y}_{\Gamma}}^{1})$
denote the relative dualizing sheaf; if $\deg(\overline{\pi}_{\Gamma}{}_{_{*}}\omega_{\nicefrac{\E}{Y}}^{\Gamma})$
(always $\geq1$) is $>1$, then $\overline{\E}_{\Gamma}$ cannot
be birational to a Fano $n\,(=\ell+1)$-fold $\PP_{\Delta}$. Conversely,
the construction of Theorem 1.7 need not yield a modular family ---
e.g. the $E_{7}$ and $E_{8}$ families of elliptic curves (cf. $\S4.3$)
have marked $non$torsion points (which are used in the construction
of the toric symbol); other examples will be given in $\S\S8.4$-$6$.

To begin {}``uniformizing'' the data, let $\{\sigma_{j}\}\subset\kappa_{\Gamma}$
be the cusps $other$ than $[i\infty]$ where $\sZ$ has nonvanishing
residue, and differentiate the $AJ$ class over $\PP^{1}$ to get\[
\omega_{\textbf{f}}:=\nabla_{\delta_{t}}\tilde{\R}_{\textbf{f}}^{[\ell]}\in\Gamma\left(\overline{Y}_{\Gamma},\,\omega_{\nicefrac{\E}{Y}}^{\Gamma}\otimes\mathcal{O}_{\overline{Y}_{\Gamma}}(\sum\sigma_{j})\right).\]
Pulling this back to $(\E^{[\ell]}\to)\,\uhp$ yields\[
(-2\pi i)^{\ell}A_{\textbf{f}}(\tau)\eta_{\ell}^{[\ell]}\,,\,\,\,\,\,\, A_{\textbf{f}}(\tau)\in\check{M}_{\ell}(\Gamma);\]
here $A_{\textbf{f}}$ may have {}``poles'' (as an automorphic form)
at elliptic points, non-unipotent cusps, and the $\{\sigma_{j}\}$.
Similarly, writing $H':=\frac{dH}{dq_{0}}$, $\frac{dt}{t}$ pulls
back to $2\pi iB_{\textbf{f}}(\tau)d\tau$, where\[
B_{\textbf{f}}(\tau):=\frac{\dlog t}{\dlog q}=\frac{q_{0}}{N_{\Gamma}}\cdot\frac{H'}{H}\in M_{2}^{\QQ}(\Gamma).\]
Pulling back the cycle-class $\Omega_{\sZ_{\textbf{f}}}=(-1)^{\ell}\nabla_{\delta_{t}}\tilde{\R}_{\textbf{f}}^{[\ell]}\wedge\frac{dt}{t}$,
we see that\[
F_{\textbf{f}}(\tau)=A_{\textbf{f}}(\tau)\cdot B_{\textbf{f}}(\tau)\,\,(\in M_{\ell+2}^{\QQ}(\Gamma)).\]

Now we can write down a power-series expansion for the period of $\omega_{\textbf{f}}$
over the (locally defined) family of topological cycles $\alpha^{\times\ell}\in H_{\ell}(E_{\Gamma,t}^{[\ell]},\ZZ)$
vanishing at $t=0$. Using Prop. 6.12 and inverting the Fourier expansion
of $H$, one has\[
\int_{\alpha^{\times\ell}}\omega_{\textbf{f}}(t)\,=\,(-2\pi i)^{\ell}\left(\frac{F_{\textbf{f}}}{B_{\textbf{f}}}\circ H^{-1}\right)(t)\,=\,(-2\pi i)^{\ell}N_{\Gamma}\frac{t(H^{-1})'(t)}{H^{-1}(t)}\cdot F_{\textbf{f}}(H^{-1}(t))\]
\[
=:\,(2\pi i)^{\ell}\sum_{m\geq0}a_{m}t^{m},\]
 where $(H^{-1})'=\frac{dq_{0}}{dt}$. Moreover $a_{0}=(-1)^{\ell}N_{\Gamma}\cdot\mathsf{H}_{[i\infty]}^{[\ell]}(\varphi_{\textbf{f}})$,
and\[
\Psi_{\textbf{f}}^{\Gamma}(t):=\int_{\alpha^{\times\ell}}R_{\left(\sZ|_{E_{\Gamma,t}^{[\ell]}}\right)}\emql\Psi_{\textbf{f},0}^{[\ell]}(H^{-1}(t))=(2\pi i)^{\ell}\left\{ a_{0}\log t+\sum_{m\geq1}\frac{a_{m}}{m}t^{m}\right\} \]
(compare Theorem 2.2).

A key observation is that $A_{\textbf{f}}(\tau)\eta_{\ell}^{[\ell]}$
descends to $\E_{\Gamma}$, whereas the relative differentials ($\eta_{\ell}^{[\ell]}$
or $F_{\textbf{f}}(\tau)\eta_{\ell}^{[\ell]}$) used in previous sections
did not. This leads to a higher normal function and PF equations which
make sense over $Y_{\Gamma}$. Recalling $\nabla_{\text{PF}}^{\textbf{f}}=\nabla_{\d_{\tau}}^{\ell+1}+\text{l.o.t.}$
from $\S7.1$,\[
\nabla_{\text{PF}}^{\omega}:=\frac{1}{(2\pi iB_{\textbf{f}}(\tau))^{\ell+2}}\circ\nabla_{\text{PF}}^{\textbf{f}}\circ(2\pi iB_{\textbf{f}}(\tau))\,=\,\nabla_{\delta_{t}}^{\ell+1}+\text{l.o.t.}\]
descends to $\PP^{1}$, yielding the homogeneous equation\[
(D_{\text{PF}}^{\omega}\circ\delta_{t})\Psi_{\textbf{f}}^{\Gamma}=0.\]
Writing \[
\nu_{\textbf{f}}(\tau):=\left\langle \tilde{\R}_{\textbf{f}}^{[\ell]},\omega_{\textbf{f}}\right\rangle =(-2\pi i)^{\ell}V_{\textbf{f}}^{[\ell]}(\tau)\cdot A_{\textbf{f}}(\tau),\]
 we have the inhomogeneous equation\[
D_{\text{PF}}^{\omega}\nu_{\textbf{f}}=\left\langle \nabla_{\delta_{t}}\tilde{\R}_{\textbf{f}}^{[\ell]},\nabla_{\delta_{t}}^{\ell}\omega_{\textbf{f}}\right\rangle =\left\langle \omega_{\textbf{f}},\nabla_{\delta_{t}}^{\ell}\omega_{\textbf{f}}\right\rangle =:\Y_{\textbf{f}}^{[\ell]}(t),\]
where the Yukawa coupling\[
\Y_{\textbf{f}}^{[\ell]}(H(\tau))=(-2\pi i)^{2\ell}A_{\textbf{f}}^{2}(\tau)\left\langle \eta_{\ell}^{[\ell]},\frac{1}{(2\pi iB_{\textbf{f}})^{\ell}}\nabla_{\d_{\tau}}^{\ell}\eta_{\ell}^{[\ell]}\right\rangle =(2\pi i)^{\ell}\frac{A_{\textbf{f}}^{2}}{B_{\textbf{f}}^{\ell}}Y_{\tau^{\ell}}(\tau)\]
\[
=(-1)^{\binom{\ell}{2}}\ell!\left(\frac{2\pi i}{B_{\textbf{f}}(\tau)}\right)^{\ell}(A_{\textbf{f}}(\tau))^{2}.\]
Obviously the weights cancel so that $\Y_{\textbf{f}}^{[\ell]}\circ H\in\check{M}_{0}(\Gamma)$,
i.e. $\Y_{\textbf{f}}^{[\ell]}$ yields a rational function on $\PP^{1}$.

Suppose $\mathsf{H}_{[i\infty]}^{[\ell]}(\varphi_{\textbf{f}})\neq0$
and $|\kappa_{\Gamma}^{[\ell]}|>1$, so that one can choose $\textbf{g}\in\Phi_{2}^{\QQ}(N)^{\circ}$
(such that $\sZ_{\textbf{g}}$ also descends to $\E_{\Gamma}^{[\ell]}$)
with $\mathsf{H}_{[i\infty]}^{[\ell]}(\varphi_{\textbf{g}})=0$ but
$\mathsf{H}_{\sigma}^{[\ell]}(\varphi_{\textbf{g}})\neq0$ (for some
$\sigma\neq[i\infty]$). Then one can consider $A_{\textbf{f}}\cdot V_{\textbf{g}}^{[\ell]}=\frac{1}{(-2\pi i)^{\ell}}\left\langle \tilde{\R}_{\textbf{g}}^{[\ell]},\omega_{\textbf{f}}\right\rangle $,
where $\tilde{\R}_{\textbf{g}}^{[\ell]}$ is a lift with all $\mathfrak{K}_{\textbf{g},i}=0$
($0\leq i<\ell).$%
\footnote{cf. Prop. 7.4: in this case $\mathfrak{K}_{\textbf{g}}:=\lim_{\tau\to i\infty}V_{\textbf{g}}^{[\ell]}(\tau)=(-1)^{\ell}\lim_{\tau\to i\infty}V_{\textbf{g}}^{[\ell]}(\tau).$%
} This is the more general type of higher normal function implicit
in the Ap\'ery-Beukers irrationality proofs (cf. Introduction). (The
general idea is this: one must show the radius of convergence of its
$t$-series expansion to be {}``much larger'' than that for either
$A_{\textbf{f}}$ or $A_{\textbf{f}}\cdot(V_{\textbf{g}}^{[\ell]}-\mathfrak{K}_{\textbf{g}})$,
while the latter expansions must satisfy certain integrality properties.)
The story will be related from a less {}``modular'' perspective
in \cite{Ke2}.

\subsection{Identifying pullbacks of toric symbols}

If (in oversimplified terms) the idea of $\S8.2$ was to pull back
the Eisenstein construction along $H^{-1}$ (when it exists), here
we pull back a given toric symbol (if possible) along some $H$, and
try to recognize the result as an Eisenstein symbol. This leads to
motivic proofs of several of the Mahler measure computations in \cite{S,Be1,Be2}.

We begin with an {}``anticanonical pencil'' $\tilde{\X}=\overline{\{1-t\phi(\underline{x})=0\}}\subset\PP^{1}\times\PP_{\tilde{\Delta}}$
satisfying the assumptions of Theorem 1.7, with its attendant cycle
$\tilde{\Xi}\in H_{\M}^{n}(\TXM,\QQ(n))$ for $n=2,3,4$. We also
require $\phi$ to have root-of-1 vertex coefficients so that theorem
$2.2$ holds. Set $\ell:=n-1$, and restrict/refine this family in
several steps:

\begin{itemize}
\item \textbf{(1)} $\ell=3$: assume that $\PP_{\tilde{\Delta}}$ is smooth
(so that $t=0$ is a point of maximal unipotent monodromy).
\item \textbf{(2)} If $\phi$ is regular, define%
\footnote{preferring inconsistent notation to writing everywhere $\widetilde{\tilde{\X}}$.
We retain this convention for the rest of the paper.%
} $\X\,(\mapg\PP^{1})$ to be the (smooth) proper transform of $\tilde{\X}$
under successive blow-up of the components of the base locus%
\footnote{$X_{\eta}$ denotes a very general fiber.%
} $\PP^{1}\times(\tilde{X}_{\eta}\cap\tilde{\DD})\subset\PP^{1}\times\PP_{\tilde{\Delta}}$;
this accomplishes semistable reduction at $t=0$. When $\phi$ is
$not$ regular this must be combined with the desingularization of
$\TXM$ from the proof of Theorem 1.7 (to produce $\X$). Denote that
pulled-back cycle by $\Xi\in CH^{\ell+1}(\X\m X_{0},\ell+1)$.
\end{itemize}
(In what follows, one could also replace $\X$ by a {[}desingularized]
quotient --- if one exists --- over a $t\mapsto t^{\kappa}$ quotient
of the base preserving unipotency at $t=0$, and $\Xi$ by the push-forward
cycle.)

\begin{itemize}
\item \textbf{(3)} $\begin{array}[t]{cc}
\ell=2: & \text{assume rk}(Pic(X_{\eta}))=19\\
\ell=3: & \text{assume }h^{2,1}(X_{\eta})=1\text{, and that}\\
 & \text{the VHS has no "instanton}\\
 & \text{corrections" (cf. [Do1])}\end{array}$ \\
Then $H^{\ell}(X_{t})$ (or $H_{tr}^{2}(X_{t})$ for $\ell=2$) is
the symmetric $\ell^{th}$ power of a weight $1$ (rank $2$) VHS;
likewise for the PF equation of the section of $\omega_{\nicefrac{\X}{\PP^{1}}}:=K_{\X}\otimes\pi^{-1}\theta_{\PP^{1}}^{1}$
given by $\omega:=\nabla_{\delta_{t}}\R_{t}$ (cf. $\S\S2.2$-$3$).
\end{itemize}
In fact, $\omega$ is (up to scaling) the unique section of $\omega_{\nicefrac{\X}{\PP^{1}}}\otimes\mathcal{O}_{\PP^{1}}(-[\infty])\cong\mathcal{O}_{\PP^{1}}$.

Now let $U\subset\PP^{1}$ be a small neighborhood of $t=0$. Working
over $U^{*}$, denote by $W_{\bullet}$ the weight monodromy filtration
on $H^{\ell}(X_{t},\QQ)$ ($H_{tr}^{2}$ if $\ell=2$) and set $W_{\bullet}^{\ZZ}:=W_{\bullet}\cap H_{(tr)}^{\ell}(X_{t},\ZZ)$.
There are unique generating sections $\varphi_{0}\in\Gamma(U,W_{0}^{\ZZ})$,
$-\overline{\varphi_{1}}\in\Gamma(U^{*},W_{2}^{\ZZ}/W_{0}^{\ZZ})$
positively oriented as topological cycles; the latter lifts to a multivalued
section of $W_{2}^{\ZZ}$ with monodromy $\varphi_{1}\mapsto\varphi_{1}+N_{\X}\varphi_{0}$.
The mirror map \begin{equation}
(q=)\, \M(t)=\exp \left\{ 2\pi i \frac{\int_{\varphi_1(t)}\omega_t}{\int_{\varphi_0(t)}\omega_t} \right\} \\
\end{equation} is well-defined on $U^{*}$; its logarithm $\mu=\frac{\log\M}{2\pi i}$
extends to a multivalued map $\PP^{1}\rightsquigarrow\uhp^{*}$. Recall
$A(t):=\int_{\varphi_{0}(t)}\omega_{t}$, $\Psi(t):=\int_{\varphi_{0}(t)}\R_{t}$
(with $\d_{t}\Psi=A$).

\begin{itemize}
\item \textbf{(4)} Assume the mirror map is {}``modular'': that is, $\exists\,\tilde{N}\geq3$
such that $\mu^{-1}=:\tilde{H}(\tau)$ is a well-defined automorphic
function for $\Gamma(\tilde{N})$ ($H\in\check{M}_{0}(\Gamma(\tilde{N}))$);
for odd $\ell$, we also demand that $\{-\text{id}\}\notin$monodromy
group of $R^{\ell}\pi_{*}\ZZ$. (Obviously this implies $N_{\X}|\tilde{N}$
and $\tilde{H}(\tau)=C\cdot\tilde{q}_{0}+\text{h.o.t.}$ where $\tilde{q}_{0}=q^{\frac{1}{N_{\X}}}$.)
Then \[
A(\tilde{H}(\tau))\in\check{M}_{\ell}(\Gamma(\tilde{N})),\]
where the {}``poles'' come from non-unipotent singular fibers and
are cancelled by $\tilde{H}^{*}\frac{dt}{t}$ to yield\[
\mathsf{F}(t):=\frac{(-1)^{\ell}}{(2\pi i)^{\ell+1}}\d_{\tau}\Psi(\tilde{H}(\tau))=(-1)^{\ell}\frac{d\log\tilde{H}}{d\tau}\cdot\frac{A(\tilde{H}(\tau))}{(2\pi i)^{\ell+1}}\in M_{\ell+2}(\Gamma(\tilde{N})).\]

\end{itemize}
Now we want to force $\mathsf{F}$ to be an Eisenstein series; the
following stronger assumption (which for $\ell=1$ follows from the
previous) does the job after a slight adjustment to $\tilde{H}$ (and
$\tilde{N}$).

\begin{itemize}
\item \textbf{(5)} Assume $\X$ is {}``modular'': i.e. in addition to
assumptions (1)-(3), $\exists\, N\geq3$, $H\in\check{M}_{0}(\Gamma(N))$,
and a (surjective) rational map $\theta:\,\overline{\E}^{[\ell]}(N)\dashrightarrow\X$
over $H:\,\overline{Y}(N)\twoheadrightarrow\PP_{t}^{1}$ (which can
include e.g. a fiberwise Kummer- or Borcea-Voisin- type construction).%
\footnote{While there are plenty of examples for $\ell=1,2$, we will see that
for $\ell=3$ there are $no$ modular anticanonical families of this
form; the problem already arises in hypothesis (3). However, there
are relaxations of the hypothses that $are$ likely to produce examples.
See $\S8.6$. %
} Define $\theta^{*}\Xi\in CH^{\ell+1}(\E^{[\ell]}(N),\ell+1)$ by
pulling back (to an appropriate blow-up of $\E^{[\ell]}(N)$) and
pushing forward. Then \begin{equation}
\Omega_{\theta^*\Xi}=(2\pi i)^{\ell+1}F_{\theta^*\Xi}(\tau)\eta_{\ell}^{[\ell]}\wedge d\tau \in F^{\ell+1}\cap H^{\ell+1}(\E ^{[\ell]}(N),\QQ(\ell+1)),\\
\end{equation} where $F_{\theta^{*}\Xi}\in M_{\ell+2}^{\QQ}(\Gamma(N)).$ If we
know the divisor \begin{equation}
\theta^*(X_0)=:(-1)^{\ell}\sum_{\sigma \in \kappa(N)} r_{\sigma}(\Xi)\cdot \overline{\pi}^{-1}_{\Gamma(N)}(\sigma),\\
\end{equation} then taking $\textbf{f}\in\mathcal{O}^{*}(U(N))^{\otimes(\ell+1)}$
with $\mathsf{H}_{\sigma}^{[\ell]}(\varphi_{\textbf{f}})=r_{\sigma}(\Xi)$
($\forall\sigma\in\kappa(N)$), $\Omega_{\sZ_{\textbf{f}}}$ and $\Omega_{\theta^{*}\Xi}$
have the same residues. By $\S5.1.5$ they are equal (i.e. $F_{\theta^{*}\Xi}=F_{\textbf{f}}$)
hence (by Lemma 7.1(ii)) so are the fiberwise $AJ$ classes.
\end{itemize}
To compute further we need precise information about $\theta$: consider
the positive integers $M_{\theta}:=\deg(\theta)$, $m_{0}:=\frac{\theta_{*}(\alpha^{\ell})}{\varphi_{0}}$,
$m_{1}:=\frac{\theta_{*}(\G^{*}(\alpha^{\ell-1}\beta))}{\varphi_{1}}$
(see $\S7.1$), $m_{\theta}:=\frac{m_{0}}{m_{1}}$, and (in suggestive
notation) $N_{\Gamma}:=\frac{N_{\X}}{m_{\theta}}$.%
\footnote{For $\ell=1$ we just have $m_{0}=m_{1}=m_{\theta}=1$ ($\implies\, N_{\Gamma}=N_{\X}$),
$M_{\theta}=\kappa$.%
} One easily checks that $H(\tau)=\tilde{H}(m_{\theta}\tau)=C_{0}\cdot q_{0}+\text{h.o.t.}$,
when $q_{0}:=q^{\frac{1}{N_{\Gamma}}}$ (by abuse of notation we will
write this $H(q_{0})$, and $H'(q_{0}):=\frac{dH}{dq_{0}}$). We then
have\[
\theta^{*}\omega=m_{0}A(H(q_{0}))\eta_{\ell}^{[\ell]}\in\Gamma(\overline{Y}(N),\omega_{\nicefrac{\E}{Y}}^{\Gamma(N)}),\]
\[
H^{*}\frac{dt}{t}=\frac{2\pi i}{N_{\Gamma}}\frac{q_{0}}{H(q_{0})}H'(q_{0})d\tau\in\Omega^{1}(\overline{Y}(N))\left\langle \log(H^{-1}(0)\cup H^{-1}(\infty))\right\rangle ,\]
\[
\theta^{*}\Omega_{\Xi}=\theta^{*}\left(\frac{dt}{t}\wedge\nabla_{\delta_{t}}\R_{t}\right)=(-1)^{\ell}\theta^{*}\omega\wedge H^{*}\frac{dt}{t}\]
\[
=(-1)^{\ell}\frac{2\pi im_{0}}{N_{\Gamma}}\frac{q_{0}}{H(q_{0})}H'(q_{0})A(H(q_{0}))\eta_{\ell}^{[\ell]}\wedge d\tau\in\Omega^{\ell+1}(\overline{\E}^{[\ell]}(N))\left\langle \log\theta^{*}(X_{0})\right\rangle .\]
Under pullback the regulator period becomes (for $f$ as above) \begin{equation}
\Psi(H(\tau ))=\int_{\varphi_0(H(\tau))}\tilde{\R}_{H(\tau)} = \frac{1}{m_0}\int_{\alpha^{\ell}(\tau )}\tilde{\R}_{\theta^*\Xi}(\tau)=\frac{1}{m_0} \int_{\alpha^{\ell}}\tilde{\R}^{[\ell]}_{\textbf{f}}(\tau) = \frac{1}{m_0}\Psi^{[\ell]}_{\textbf{f},0}(\tau), \\
\end{equation} so that (by Prop. 7.3(i)) \[
\d_{\tau}\Psi(H(\tau))=(-1)^{\ell}(2\pi i)^{\ell+1}m_{0}^{-1}F_{\theta^{*}\Xi}(\tau).\]
That \[
\Psi(H(\tau))\text{ is of the form }(7.17)\]
is of fundamental importance; if one divides by $(2\pi i)^{\ell}$
and takes the real parts it essentially says the real regulator period
(or Mahler measure, in the region described in Cor. 2.7) pulls back
to an Eisenstein-Kronecker-Lerch series (noticed in examples by \cite{Be1,Be2,RV}).
Furthermore, this allows us to use Prop. 7.9 to compute its special
values at $H\{\text{unipotent cusps}\}$, which therefore must be
a sum (with coefficients $\in\QQ(e^{\frac{2\pi i}{N}})$) of $(\ell+1)^{\underline{\text{st}}}$
special values of Dirichlet $L$-functions.%
\footnote{This is similar to the case in $\S4$ of $L/\QQ$ abelian (which however
does not imply modularity).%
}

Our last object of interest is the Yukawa coupling $Y(t)=\left\langle \omega_{t},\nabla_{\delta_{t}}^{\ell}\omega_{t}\right\rangle ,$
which becomes \[
Y(H(q_{0}))\,=\, M_{\theta}^{-1}\left\langle \theta^{*}\omega,\theta^{*}\nabla_{\delta_{t}}^{\ell}\omega\right\rangle \]
\[
=\,\frac{N_{\Gamma}^{\ell}}{(2\pi i)^{\ell}M_{\theta}}\cdot\frac{1}{\{H'(q_{0})\}^{\ell}}\left\langle \theta^{*}\omega,\nabla_{\d_{\tau}}^{\ell}\theta^{*}\omega\right\rangle \]
\[
=\,\frac{N_{\Gamma}^{\ell}m_{0}^{2}}{(2\pi i)^{\ell}M_{\theta}}\cdot\frac{\{A(H(q_{0}))\}^{2}}{\{H'(q_{0})\}^{\ell}}\left\langle \eta_{\ell}^{[\ell]},\nabla_{\d_{\tau}}^{\ell}\eta_{\ell}^{[\ell]}\right\rangle \]
\begin{equation}
= \frac{(-1)^{\binom{\ell}{2}}\ell ! N_{\Gamma}^{\ell}m_0^2}{(2\pi i)^{\ell}M_{\theta}}\cdot \frac{\left\{ A(H(q_0))\right\}^2 }{\left\{ H'(q_0) \right\}^{\ell} }, \\
\end{equation} a rational function on $\overline{Y}(N)$. Noting $A(0)=(2\pi i)^{\ell}$
and using $(8.7)$ and Prop. 7.9 gives

\begin{thm}
Assuming modularity of a family of $CY$ $\ell$-folds $\X$ arising
(as described) from the toric construction, we have \begin{equation} 
\begin{matrix}
\frac{(-1)^{\ell}m_0}{(2\pi i)^{\ell}N_{\Gamma}}\delta_{q_0}\Psi(H(q_0)) \, = \, \frac{(-1)^{\ell}m_0}{(2\pi i)^{\ell}N_{\Gamma}}\frac{q_0}{H(q_0)}H'(q_0)A(H(q_0))
\\
= \, F_{\theta^*\Xi}(q_0) \, = \, \sum_{\sigma \in \kappa(N)}r_{\sigma}(\Xi)\tilde{E}_{\sigma}^{[\ell]}(q_0)
\end{matrix} \\
\end{equation} for the pulled-back cycle-class of the toric symbol, and also \begin{equation}
\frac{Y(0)}{(2\pi i)^{\ell}} \, = \, \frac{(-1)^{\binom{\ell}{2}}\ell ! N_{\Gamma}^{\ell}m_0^2}{M_{\theta}C_0^{\ell}}\in \QQ(C_0). \\
\end{equation} Finally, if $X_{t_{0}\neq0}$ is a maximally unipotent singular fiber,
then%
\footnote{The specific choice of representative $\frac{r_{0}}{s_{0}}$ of the
cusp $\mu(t_{0})$ depends on the path along which $\Psi(t)$ has
been continued prior to taking $\lim_{t\to t_{0}}$.%
} $\mu(t_{0})\equiv[\frac{r_{0}}{s_{0}}]\in\kappa(N)$ and \begin{equation}
\lim_{t\to t_0}\Psi(t) \emql \frac{(-1)^{\ell+1}}{2N} \sum_{\tiny \begin{matrix} [\frac{r}{s}]\in \kappa(N) \\ [\frac{r}{s}] \nequiv [\frac{r_0}{s_0}] \end{matrix}} s^{\ell}r_{[\frac{r}{s}]}(\Xi) \tilde{L}_{-}\left( \pi_{[\frac{r_0}{s_0}]_*}\iota_{[\frac{r}{s}]_*} \widehat{\varphi_N^{[\ell]}}, \ell+1 \right) .\\
\end{equation}
\end{thm}
By comparing values at $[i\infty]$ (i.e. $q_{0}=0$) in $(8.11)$,
we have the interesting

\begin{cor}
$r_{[i\infty]}(\Xi)=(-1)^{\ell}\frac{m_{0}}{N_{\Gamma}}.$
\end{cor}
\begin{rem*}
If the $r_{\sigma}(\Xi)$ are known but the series expansion $t=H(q_{0})=C_{0}q_{0}+\cdots$
for the mirror map is not, one can in principle determine the latter
from \[
\Psi(H(\tau))=\frac{1}{m_{0}}\Psi_{\textbf{f},0}^{[\ell]}(\tau)\]
(cf. $(8.9)$), by using $(2.6)$ for the l.h.s. and Cor. 7.10 for
the r.h.s. (In the computations below, we have preferred to take $H$
from other sources, in order to partially vet our formulas.) Since
the {}``log + constant'' terms of both sides must agree (mod $\QQ(n)$),
an immediate consequence is
\end{rem*}
\begin{cor}
$C_{0}$ (hence $\frac{Y(0)}{(2\pi i)^{\ell}}$) is a root of unity.
\end{cor}
Clearly one can normalize $\phi$ (retaining the assumption on vertex
coefficients) so that $Y(0)\in\QQ(\ell)$.

\subsection{The elliptic curve case}

Start with a reflexive tempered Laurent polynomial $\phi\in\bar{\QQ}[x^{\pm1},y^{\pm1}]$
defining a family of (generically smooth) elliptic curves, $\tilde{\X}\subset\PP_{t}^{1}\times\PP_{\widetilde{\Delta_{\phi}}}$.
Possibly after a finite $(t\mapsto t^{\kappa})$ quotient%
\footnote{again preserving unipotency at $t=0$%
} we desingularize this and blow down all $(-1)$-curves contained
in fibers. The resulting elliptic surface is denoted $\X$, and is
relatively minimal in the sense that $\omega_{\nicefrac{\X}{\PP^{1}}}\cong\pi^{*}\pi_{*}\omega_{\nicefrac{\X}{\PP^{1}}}$;
the singular fibers are therefore of the types described by Kodaira
\cite{Ko}. Clearly $\chi(\X)=12\cdot\deg(\pi_{*}\omega_{\nicefrac{\X}{\PP^{1}}})$
is $12$, either by looking at zeroes of $\omega=\nabla_{\delta_{t}}\R_{t}\in\Gamma(\pi_{*}\omega_{\nicefrac{\X}{\PP^{1}}})$
or the fact that $\X$ is birational to $\PP_{\Delta_{\phi}}$hence
to $\PP^{2}$. This constrains the possible combinations of singular
fibers in light of the table:\\
\begin{tabular}{|c|c|c|c|}
\hline 
sing. fiber type & contrib. to $\chi(\X)$ & ord. of monodromy & no. of components\tabularnewline
\hline
\hline 
$\begin{array}{c}
I_{n\geq1}\\
I_{n\geq0}^{*}\end{array}$ & $\begin{array}{c}
n\\
n+6\end{array}$ & $\begin{array}{c}
\infty\\
\infty\end{array}$ & $\begin{array}{c}
n\\
n+5\end{array}$\tabularnewline
\hline 
$\begin{array}{c}
II\\
IV^{*}\end{array}$ & $\begin{array}{c}
2\\
8\end{array}$ & $\begin{array}{c}
6\\
3\end{array}$ & $\begin{array}{c}
1\\
7\end{array}$\tabularnewline
\hline 
$\begin{array}{c}
III\\
III^{*}\end{array}$ & $\begin{array}{c}
3\\
9\end{array}$ & $\begin{array}{c}
4\\
4\end{array}$ & $\begin{array}{c}
2\\
8\end{array}$\tabularnewline
\hline 
$\begin{array}{c}
IV\\
II^{*}\end{array}$ & $\begin{array}{c}
4\\
10\end{array}$ & $\begin{array}{c}
3\\
6\end{array}$ & $\begin{array}{c}
3\\
9\end{array}$\tabularnewline
\hline
\end{tabular}\\
\\
where we have paired those types related by a quadratic transformation
({}``adding a $*$''). We identify families by the set of fiber
types, e.g. $I_{1}^{4}/I_{4}^{*}$ means $4$ $I_{1}$'s and $1$
$I_{4}^{*}$.

Now referring to $(8.6)$, we make a precise

\begin{defn}
$\M$ is \textbf{weakly modular} $\Longleftrightarrow$ $\mu^{-1}(=:H)$
is a Hauptmodul for $\Gamma\subset SL_{2}(\ZZ)$ of finite index.
We say $\M$ is \textbf{modular} if in addition $\{-\text{id}\}\notin\Gamma$
and $\Gamma\supset\Gamma(N)$ for some $N\geq3$.
\end{defn}
Obviously if $\M$ is modular then one has a canonical quotient $\overline{\E}_{\Gamma(N)}^{[1]}\thda\overline{\E}_{\Gamma}^{[1]}\cong\X$
and $\X$ is modular in the sense of $\S8.3$.

\begin{lem}
\cite[Prop. 2]{Do1} $\M$ is weakly modular $\Longleftrightarrow$
the $J$-invariant $J(\mu(t))$ ramifies only over $J=0$ (to order
$1$ or $3$), $J=1$ (to order $1$ or $2$), and $J=\infty$ (to
any order).
\end{lem}
The point is that $\mu^{-1}$ cannot possibly be single-valued if
$J\circ\mu$ has {}``excess ramification'' (which explains why we
wanted to allow order-$\kappa$ quotients of the base in constructing
$\X$). It folllows (cf. \cite{Do1}) that fiber types $II^{*}$ and
$IV$ are not permitted (so no $I_{1}^{2}/II^{*}$), and neither are
certain other combinations (e.g. $I_{1}^{6}/I_{6}$); in \cite[Thm. 4.12]{Do2}
the remaining possibilities are listed (up to {}``transfer of $*$'').
Disallowing those fiber types left which contain $-\text{id}$ in
their local monodromy group ($II,\, III,\, III^{*}$), and checking
for $-\text{id}$ also in global monodromy, one arrives at the list
below.

\begin{prop}
Suppose the singular fiber configuration of $\X$ is one of those
shown in the table, with fiber $I_{n_{\X}}$ at $t=0$. (This gives
an additional degree of freedom.) Then $\M$ is modular, $\X\cong\overline{\E}_{\Gamma}$
(for $\Gamma\supset\Gamma(N)$ as displayed), and%
\footnote{here $q_{0}=q^{\frac{1}{n_{\X}}}$%
} \begin{equation}
-\frac{1}{n_{\X}}\sum_{\sigma \in |H^{-1}(0)|\subset \kappa(N)} \tilde{E}^{[1]}_{\sigma}(q_0) = F_{\theta^*\Xi}(q_0)\\
\end{equation} where $|H^{-1}(0)|$ is $not$ counted with multiplicity. Finally,
all the configurations below occur in the toric construction.
\end{prop}
$\mspace{100mu}$\begin{tabular}{|c|c|c|}
\hline 
configuration & $\Gamma$ & $N$\tabularnewline
\hline
\hline 
$I_{3}^{4}$ & $\Gamma(3)$ & 3\tabularnewline
\hline 
$I_{1}/I_{3}/IV^{*}$ & $\Gamma^{(')}(3)$ & 3\tabularnewline
\hline 
$I_{1}/I_{1}^{*}/I_{4}$ & $\Gamma_{1}^{(')}(4)$ & 4\tabularnewline
\hline 
$I_{2}^{2}/I_{4}^{2}$ & $\left\langle \Gamma(4),\umat\right\rangle $ & 4\tabularnewline
\hline 
$I_{2}^{2}/I_{2}^{*}$ & $\widetilde{\Gamma(2)}:=\left\langle \umat,\vmat\right\rangle $ & 4\tabularnewline
\hline 
$I_{1}^{2}/I_{5}^{2}$ & $\Gamma_{1}^{(')}(5)$ & 5\tabularnewline
\hline 
$I_{1}/I_{2}/I_{3}/I_{6}$ & $\Gamma_{1}^{(')}(6)$ & 6\tabularnewline
\hline 
$I_{1}^{2}/I_{2}/I_{8}$ & $\left\langle \Gamma_{1}^{'}(8),\wmat\right\rangle $ & 8\tabularnewline
\hline 
$I_{1}^{2}/I_{4}^{*}$ & $\left\langle \Gamma_{1}^{'}(8),\wmat,\xmat\right\rangle $ & 8\tabularnewline
\hline 
$I_{1}^{3}/I_{9}$ & $\left\langle \Gamma_{1}^{'}(9),\ymat\right\rangle $ & 9\tabularnewline
\hline
\end{tabular}\\
\\
For computations it is desirable to replace $-\frac{1}{n_{\X}}\sum\tilde{E}_{\sigma}^{[1]}$
by $F_{\textbf{f}}$ with $\varphi_{\textbf{f}}$ chosen to have $\mathsf{H}_{\sigma}(\varphi_{\textbf{f}})=\left\{ \begin{array}{cc}
\frac{-1}{n_{\X}}, & \sigma\in|H^{-1}(0)|\\
0, & \text{otherwise}\end{array}\right..$ Note that by $(8.9)$, for $\tau\in\uhp$ \begin{equation}
\Psi(H(\tau)) \equiv \Psi^{[1]}_{\textbf{f},0} (\tau) \, \, \, \, \, \text{mod}\, \QQ(2). \\
\end{equation} The two {}``$E_{6}$'' examples below both correspond to the second
row of the table, and their difference illustrates a technical subtlety.
The first computation is essentially that in \cite[Ex. 3]{S}; Ex.
4,5,6 in op. cit. also fall under Prop. 8.7's aegis, and correspond
to lines 3,6,7 (resp.) in the table.

\begin{example}
$\phi=x^{2}y^{-1}+x^{-1}y^{2}+x^{-1}y^{-1}$, $\kappa=3$ (quotient).\\
This yields $\X$ with fibers $X_{t}\cong\overline{\{1-t^{\frac{1}{3}}\phi=0\}}\subset\PP^{2}$,
$\Gamma=\Gamma_{1}(3)$, and $n_{\X}=1$. (This is just the Hesse
pencil, which appears as Ex. 1 in \cite{RV} and Ex. 3 in \cite{S}.)
The singular fibers occur at $t=0\,(I_{1})$, $\frac{1}{3^{3}}\,(I_{3})$,
$\infty\,(IV^{*})$; whereas if we had not taken the quotient ($\kappa=1$),
there would be $4$ $I_{1}$'s (at $t=0,\frac{1}{3},\frac{\zeta_{3}}{3},\frac{\zeta_{3}^{2}}{3}$)
with $\Gamma=\Gamma(3).$ This will be useful in Example 8.9.

From \cite{S},\[
H(q)=H_{\Gamma_{1}(3)}(q):=\left(27+\frac{\eta(q)^{12}}{\eta(q^{3})^{12}}\right)^{-1}=q(1-15q+171q^{2}-1679q^{3}+\cdots)\]
where of course $\eta(q)=q^{\frac{1}{24}}\prod_{n\geq1}(1-q^{n})$,
and we have\[
A(t)=2\pi i\sum_{m\geq0}\frac{(3m)!}{(m!)^{3}}t^{m}=2\pi i(1+6t+90t^{2}+1680t^{3}+\cdots).\]
Since $|H^{-1}(0)|=\{[i\infty]\}$, we put $\varphi_{\textbf{f}}:=-\frac{1}{3}\pi_{[i\infty]}^{*}\varphi_{3}^{[1]}$;
by Example 6.14\[
F_{\textbf{f}}(q)=-1+9\sum_{K\geq1}q^{K}\sum_{r|K}r^{2}\chi_{-3}(r)=-1+9q-27q^{2}+9q^{3}+\cdots.\]
The Proposition says this equals\[
\frac{-q}{H(q)}H'(q)\frac{A(H(q))}{2\pi i}\,\,=\]
\[
-(1+15q+54q^{2}-76q^{3}+\cdots)(1-30q+513q^{2}-6716q^{3}+\cdots)(1+6q+6q^{3}+\cdots),\]
which is clearly plausible from the first 3 terms of the series. From
$(7.26)$ we have \[
\Psi_{\textbf{f},0}^{[1]}(q)=2\pi i\left\{ \log q-9\sum_{K\geq1}\left(\frac{\sum_{r|K}r^{2}\chi_{-3}(r)}{K}\right)q^{K}\right\} \]
while $\Psi(t)=2\pi i\left\{ \log t+\sum_{m\geq1}\frac{(3m)!}{(m!)^{3}}t^{m}\right\} $;
computation again suggests that $\Psi(H(q))=\Psi_{\textbf{f},0}^{[1]}(q)$,
which is (mod $\QQ(2)$) exactly what $(8.15)$ asserts.
\end{example}
${}$

\begin{example}
$\phi=x+y+x^{-1}y^{-1}$, $\kappa=3$.\\
This gives $\X$ with $\Gamma=\Gamma_{1}^{'}(3)$, $n_{\X}=3$, and
singular fibers at $t=0\,(I_{3})$, $\frac{1}{3^{3}}\,(I_{1})$, $\infty\,(IV^{*})$;
before the quotient these are $t=0\,(I_{9})$ and $t=\frac{1}{3},\frac{\zeta_{3}}{3},\frac{\zeta_{3}^{2}}{3}\,(I_{1})$.
Put $\mathfrak{g}(u)=1-\left(\frac{1-3u}{1+6u}\right)^{3}$; by considering
locations of singular fibers one deduces\[
H(q_{0})=H_{\Gamma_{1}'(3)}(q_{0})=\frac{1}{3^{3}}\mathfrak{g}\left(H_{\Gamma(3)}(q_{0})\right)=\frac{1}{3^{3}}\mathfrak{g}\left[\left(H_{\Gamma_{1}(3)}(q_{0}^{3})\right)^{\frac{1}{3}}\right]\]
\[
=q_{0}(1-15q_{0}+171q_{0}^{2}-5q_{0}^{3}+\cdots).\]
This is so similar to the previous example that the $A(t)$'s are
the same, and\[
-\frac{1}{3}\frac{q_{0}}{H(q_{0})}H'(q_{0})\frac{A(H(q_{0}))}{2\pi i}=-\frac{1}{3}+3q_{0}-9q_{0}^{2}+\cdots.\]
We want $\widehat{\varphi_{\textbf{f}}}=-\frac{1}{3}\rho'_{*}(\iota_{[i\infty]_{*}}\widehat{\varphi_{3}^{[1]}})=-\frac{1}{3}\pi_{[0]}^{*}\widehat{\varphi_{3}^{[1]}}$
($\implies$ $\varphi_{\textbf{f}}=\frac{1}{3}\iota_{[0]_{*}}\varphi_{3}^{[1]}$)
since $|H^{-1}(0)|=\{[i\infty],[1],[\frac{1}{2}]\}.$ Using Prop.
6.12\[
F_{\textbf{f}}(q_{0})=-\frac{1}{3}+3\sum_{K\geq1}q_{0}^{K}\sum_{r|K}r^{2}\chi_{-3}(r),\]
in agreement with the above.
\end{example}
It is interesting to explain why the {}``$E_{8}$'' family (\cite[Ex. 1]{S},
\cite[Ex. 3]{RV})\[
\phi=xy^{-1}+x^{-1}y^{2}+x^{-1}y^{-1},\,\,\,\,\kappa=6,\,\,\,\, I_{1}^{2}/II^{*}\]
and {}``$E_{7}$'' family\[
\phi=xy^{-1}+x^{-1}y^{3}+x^{-1}y^{-1},\,\,\,\,\kappa=4,\,\,\,\, I_{1}/I_{2}/III^{*}\]
fail to yield Eisenstein series (despite nontriviality of $\Xi\in CH^{2}(\X\m X_{0},2)$).
More to the point,\begin{equation}
\frac{q}{\mu^{-1}(q)}(\mu^{-1})'(q)\frac{A(\mu^{-1}(q))}{2\pi i} =: \sum_{m\geq 0}\alpha_m q^m \\
\end{equation} does not even yield a modular form (of any level) since $\limsup_{M\to\infty}\sqrt[M]{|\alpha_{M}|}=:\gamma>1$.
(At least one infers this from the data $\{b_{n}\}$ in \cite{S}.)
It is insufficient to say that the divisors of $\{x|_{X_{t}},y|_{X_{t}}\}$
are not supported on torsion (perhaps this could be fixed by an $AJ$-equivalence),
although this is probably required for instances where Prop. 8.7 fails.

In the $E_{8}$ case, $J(\mu(t))$ vanishes to order $2$ at $t=\infty$
(the $II^{*}$ fiber), so that $\mu^{-1}$ is multivalued at $\tau=e^{\frac{2\pi i}{6}}$.
As a result $(8.16)$ both is multivalued and blows up there.

According to Lemma 8.6, for the $E_{7}$ family $\mu^{-1}$ is a Hauptmodul.
However, the fact that $\Gamma=\Gamma_{1}(2)\ni\{-\text{id}\}$ manifests
itself in $(\pm)$ multivaluedness of $A\circ H$ about $\tau=\frac{1+i}{2}$
(where $J=1$ and $t=\infty$).

In neither case does one have $\theta:\,\E_{\Gamma(N)}^{[1]}\dashrightarrow\X$
along which to pull back $\Xi$. Perhaps this suggests a study of
{}``generalized Eisenstein symbols'' on families over finite covers
of $\uhp$, with additional (nontorsion) marked structure; the elliptic
Bloch groups of Wildeshaus \cite{Wi} seem quite suitable for this
purpose.

\subsection{Examples in the $K3$ case}

Up to unimodular transformation, there are $4319$ reflexive polytopes
in $\RR^{3}$ \cite{KS}; according to Cor. 1.9ff we immediately get
(at least) 358 examples for $\ell=2$ where the toric symbol completes
by taking $\phi=$ characteristic polynomial of vertices. (Putting
{}``random'' roots of unity instead of {}``1'' on each vertex
renders all 1071 polytopes from Remark 1.11 usable.) For each $\X$/$\Xi$
to be a candidate for modularity/Eisenstein-ness, we must have $\text{rk}(Pic(X_{\eta}))=19$,
in which case $X_{\eta}$ has the Shioda-Inose structure \cite{Mo}
(and one can then ask whether the underlying family of elliptic curves
is suitably modular). Such candidates are nontrivial to produce, but
{}``non-candidates'' seem much more elusive.

\begin{problem*}
Does Theorem 1.7 produce any families of $K3$'s with generic Picard
rank $\leq18$? Or does the tempered condition indirectly furnish
enough additional divisors to preclude this possibility?
\end{problem*}
Here are eight Laurent polynomials which satisfy Theorem 1.7 and produce
(after desingularization%
\footnote{see $\S8.3$ for the definition of $\X$%
}) $1$-parameter $K3$ families $\X$ provably of generic Picard rank
19 (together with the method of proof).\Tiny\\
\begin{tabular}{|c|c|c|c|c|}
\hline 
 & family & $\phi(x,y,z)$ & $\frac{A(t)}{(2\pi i)^{2}}$ & method\tabularnewline
\hline
\hline 
1  & Fermat quartic & $\frac{1+x^{4}+y^{4}+z^{4}}{xyz}$ & $\sum_{m\geq0}\frac{(4m)!}{(m!)^{4}}t^{4m}$ & $\begin{array}{c}
\text{symmetry}\\
\mathfrak{G}\cong(\ZZ/4\ZZ)^{2}\end{array}$\tabularnewline
\hline 
2 & quartic mirror & $x+y+z+\frac{1}{xyz}$ & same & restrict from $\PP_{\tilde{\Delta}}$\tabularnewline
\hline 
3 & $\begin{array}{c}
\mathbb{W}\PP(1,1,1,3)\\
\text{"Fermat"}\end{array}$ & $\frac{1+x^{6}+y^{6}+z^{2}}{xyz}$ & $\sum_{m\geq0}\frac{(6m)!}{(m!)^{3}(3m)!}t^{6m}$ & $\begin{array}{c}
\text{symmetry}\\
\mathfrak{G}\cong\ZZ/6\ZZ\times\ZZ/2\ZZ\end{array}$\tabularnewline
\hline 
4 & $\begin{array}{c}
\mathbb{W}\PP(1,1,1,3)\\
\text{mirror}\end{array}$ & $x+y+z+\frac{1}{xyz^{3}}$ & same & restrict from $\PP_{\tilde{\Delta}}$\tabularnewline
\hline 
5 & {}``box'' & $\frac{(x-1)^{2}(y-1)^{2}(z-1)^{2}}{xyz}$ & $\sum_{m\geq0}\binom{2m}{m}^{3}t^{m}$ & Shioda\tabularnewline
\hline 
6 & Fermi \cite{PS} & $x+\frac{1}{x}+y+\frac{1}{y}+z+\frac{1}{z}$ & $\sum_{m\geq0}t^{2m}\binom{2m}{m}\sum_{k=0}^{m}\binom{m}{k}^{2}\binom{2k}{k}$ & $\begin{array}{c}
\text{"double cover"}\\
\text{of Apery}\end{array}$\tabularnewline
\hline 
7 & Ap\'ery & $\frac{(x-1)(y-1)(z-1)[(x-1)(y-1)-xyz]}{xyz}$ & $\sum_{m\geq0}t^{m}\sum_{k=0}^{m}\binom{m}{k}^{2}\binom{m+k}{k}^{2}$ & Shioda\tabularnewline
\hline 
8 & Verrill \cite{Ve} & $\frac{(1+x+xy+xyz)(1+z+zy+zyx)}{xyz}$ & $\sum_{m\geq0}t^{m}\sum_{p+q+r+s=m}\left(\frac{m!}{p!q!r!s!}\right)^{2}$ & intersection form\tabularnewline
\hline
\end{tabular}\normalsize\\
\\
(The {}``Ap\'ery'' family is birational to the one studied in \cite{Bk,BP,Pe}.)
Families \#1-4 and 6 are instances of Example 1.10 (with Remark 1.11
for \#1 and \#3). The other three $\phi$'s are not regular and need
Theorem 1.7 with $K=\QQ$ (for \#5 and \#7) or Remark 1.8(iv) (for
\#8).%
\footnote{applied to the equivalent symbol $\{xy,y,z\}$%
}

We quickly summarize the {}``methods'' in the r.h. column; a study
(including most of these examples) can be found in \cite{Wh}. If
$\tilde{X}_{\eta}$ is nonsingular ($=X_{\eta}$) then \cite{Ro}\[
\text{rk}(Pic(X_{\eta}))\,\,\geq\,\,\text{rk}\{\text{im}(Pic(\PP_{\tilde{\Delta}})\to Pic(X_{\eta}))\}\,\,=\,\,\ell(\Delta^{\circ})-\sum_{\sigma\in\Delta^{\circ}(1)}\ell^{*}(\sigma)-4,\]
which $=19$ for families \#2 and \#4 and $=1$ for \#1 and \#3. For
the latter cases, the action on $X_{\eta}$ by a finite subgroup $\mathfrak{G}\subset(\CC^{*})^{3}$
augments the Picard rank by \[
\text{rk}\left[\left(H^{2}(X_{\eta},\ZZ)^{\mathfrak{G}}\right)^{\perp}\right]\]
 (\cite{Ni,Wh}), which turns out to be $18$. For \#5 (resp \#7),
$X_{\eta}$ is obtained from $\tilde{X}_{\eta}$ (remember $X_{\eta}$
is really $\widetilde{\tilde{X}_{\eta}}$) by blowing up the $12$
(resp. $7$) $A_{1}$ singularities. The elliptic fibration $X_{\eta}\to\PP_{z}^{1}$
has singular fibers $(I_{1}^{*})^{2}/I_{8}/I_{1}^{2}$ (resp. $I_{1}^{*}/I_{5}/I_{8}/I_{1}^{4}$).
By \cite{Sd} \[
\text{rk}(Pic(X_{\eta}))=2+r+\sum(\mathfrak{M}_{i}-1),\]
where $r=\,\text{rank of group of sections}\,=0$ (resp. $1$%
\footnote{the existence of a nontorsion section is demonstrated in \cite{BP}%
}) and $\mathfrak{M}_{i}=$ \# of fiber components in each singular
fiber; this yields $19$. This result is transferred to the Fermi
family by observing that its pullback $\{1-\frac{1}{u+u^{-1}}\phi_{\text{Fermi}}=0\}$
has a $2:1$ rational map (over $u\mapsto u^{2}=t$) onto the Ap\'ery
family $\{1-t\phi_{\text{Apery}}=0\}$ (see \cite{PS}). Finally,
to deal with \#8, \cite{Ve} adds some lines to the components of
$D\subset X_{t}$ and shows the rank of the resulting intersection
form is $19$.

The Fermi, Ap\'ery and Verrill pencils (which are modular) yield
an instructive set of examples for Theorem 8.3: $N=6$ in all three
cases but the $\{r_{\sigma}(\Xi)\}$, hence $\{F_{\theta^{*}\Xi}\}$,
are all different.

\begin{example}
By \cite{Pe}, the Ap\'ery pencil's $\ZZ$-PVHS is equivalent to
that coming from the construction of of Remark 6.17 for $N=6$ (and
we will assume the 2 $\X$'s birational). This gives%
\footnote{See below for $C_{0}.$ Singularities: monodromy is maximally unipotent
about $0,\,\infty$($=t$), finite (order $2$) about $(\sqrt{2}+1)^{4},\,(\sqrt{2}-1)^{4}$. %
} (with $\Gamma=\Gamma_{1}(6)^{+6}$)\[
m_{0}=-12,\, m_{1}=1,\, N_{\Gamma}=1,\, M_{\theta}=24\,\,\,\implies\,\,\,\frac{Y(0)}{(2\pi i)^{2}}=-12;\]
moreover, $\widehat{\varphi_{\textbf{f}}}$ should be a constant multiple
of $(8.5)$. Since (by Cor. 8.4) $r_{[i\infty]}(\Xi)=-12$, we take\[
\widehat{\varphi_{\textbf{f}}}:=(8.5)\,=\,\]
\[
-\frac{2\cdot6^{3}}{5}\{\widehat{\varphi_{\{1,1\}}}-\widehat{\varphi_{\{2,1\}}}-\widehat{\varphi_{\{3,1\}}}+\widehat{\varphi_{\{6,1\}}}\}\,+\,\frac{2\cdot6^{5}}{5}\{\widehat{\varphi_{\{6,1\}}}-\widehat{\varphi_{\{6,2\}}}-\widehat{\varphi_{\{6,3\}}}+\widehat{\varphi_{\{6,6\}}}\}\]
where $\widehat{\varphi_{\{a,b\}}}(m,n):=\left\{ \begin{array}{cc}
1, & a|m\text{ and }b|n\\
0, & \text{otherwise}\end{array}\right.$. (See Figure 8.1%
\begin{figure}
\caption{\protect\includegraphics[scale=0.7]{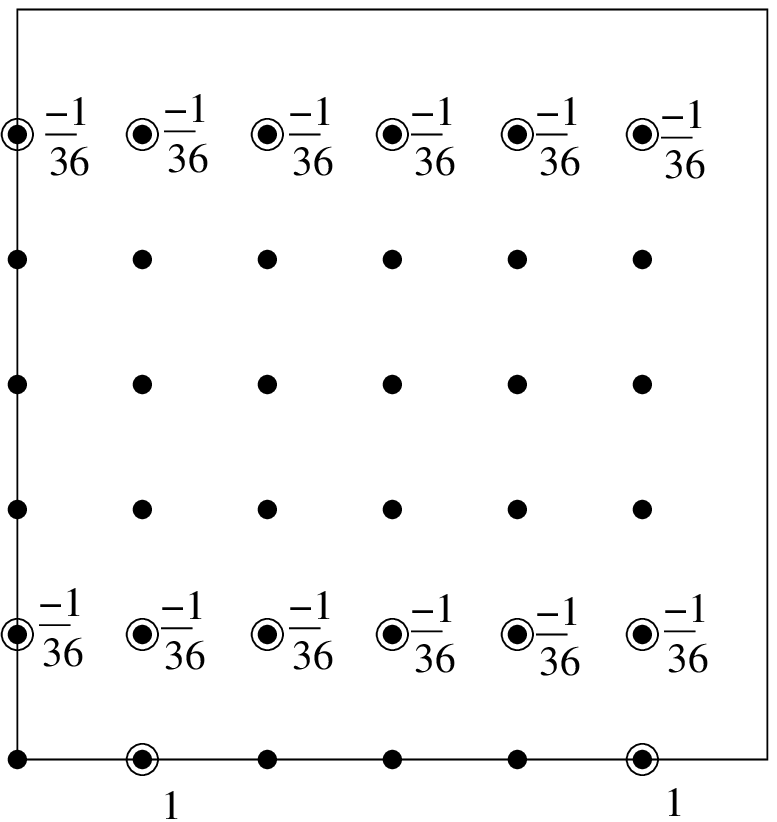}}

\end{figure}
 for a depiction of $\frac{5}{2\cdot6^{5}}\widehat{\varphi_{\textbf{f}}}$;
any places where it takes the value $0$ are simply left blank.) By
Prop. 6.12, $E_{\varphi_{\{a,b\}}}^{[2]}(q)\,=$\small \[
=\,\frac{-3}{(2\pi i)^{4}}\tilde{L}\left(\iota_{[i\infty]}^{*}\widehat{\varphi_{\{a,b\}}},\,4\right)-\frac{1}{6^{4}}\sum_{M\geq1}q^{\frac{M}{6}}\left\{ \sum_{r|M}r^{3}\left(\sum_{n_{0}\in\ZZ/6\ZZ}e^{\frac{2\pi in_{0}r}{6}}\widehat{\varphi_{\{a,b\}}}\left(\frac{M}{r},n_{0}\right)\right)\right\} \]
\normalsize \[
=\,\frac{-1}{240b^{4}}-\frac{1}{b^{4}}\sum_{K\geq1}q^{\frac{a}{b}K}\left\{ \sum_{\mathfrak{r}|K}\mathfrak{r}^{3}\right\} \]
\[
=\,\frac{-1}{240b^{4}}E_{4}(q^{\frac{a}{b}}),\]
using substitutions $M=6\frac{a}{b}K$ and $r=\frac{6}{b}\mathfrak{r}$.
So we have, with $E_{4}(q)=1+240(q+9q^{2}+28q^{3}+73q^{4}+\cdots)$,\small\[
E_{\varphi_{\textbf{f}}}^{[2]}(q)\,=\,\frac{-12}{240\cdot5}\left\{ (1-6^{2})E_{4}(q)+(6^{2}-2^{4})E_{4}(q^{2})+(6^{2}-3^{4})E_{4}(q^{3})+(6^{4}-6^{2})E_{4}(q^{6})\right\} \]
\normalsize \[
=\,\frac{7}{20}E_{4}(q)-\frac{1}{5}E_{4}(q^{2})+\frac{9}{20}E_{4}(q^{3})-\frac{63}{5}E_{4}(q^{6})\]
\[
=\,-12+84q+708q^{2}+2460q^{3}+\cdots.\]
On the other hand, from \cite{Be2} $u=\frac{\eta(\tau)^{6}\eta(6\tau)^{6}}{\eta(2\tau)^{6}\eta(3\tau)^{6}}$
$\implies$\[
H(q)=u^{2}=q(1-12q+66q^{2}-220q^{3}+\cdots),\]
while from the table\[
A(t)=(2\pi i)^{2}(1+5t+73t^{2}+1445t^{3}+\cdots);\]
therefore (from Theorem 8.3)\[
F_{\theta^{*}\Xi}\,=\,\frac{m_{0}}{(2\pi i)^{2}N_{\Gamma}}\frac{q}{H(q)}H'(q)A(H(q))\,=\,-12+84q+708q^{2}+2460q^{3}+\cdots.\]
So here we were able to correctly predict the Eisenstein series; in
the remaining examples (where obviously Thm. 8.3 predicts (8.11) is
$an$ Eisenstein series) we have found $\varphi_{\textbf{f}}$ essentially
by solving for the correct combination of $\varphi_{\{a,b\}}$'s.
\end{example}
${}$

\begin{example}
(Compare \cite[Ex. 1]{Be2}.) For the Fermi family, one deduces from
Ap\'ery (and the relationship between the two) that\[
m_{0}=-12,\, m_{1}=1,\, C_{0}=1,\, N_{\Gamma}=2,\, M_{\theta}=24\,\,\,\,\]
\[
\implies\,\,\, r_{[i\infty]}(\Xi)=-6,\,\frac{Y(0)}{(2\pi i)^{2}}=-48;\]
so $q_{0}=q^{\frac{1}{2}}$ and \[
H(q_{0})=\frac{1}{u+\frac{1}{u}}=q_{0}(1-7q_{o}^{2}+34q_{0}^{4}-204q_{0}^{6}+\cdots).\]
(The family has order $2$ monodromy about $t=\pm\frac{1}{2},\pm\frac{1}{6}$
and maximally unipotent monodromy about $t=0$.) From the table $A(t)=(2\pi i)^{2}(1+6t{}^{2}+90t^{4}+1860t^{6}+\cdots),$
and by Theorem 8.3\[
F_{\theta^{*}\Xi}(q_{0})=-6\frac{q_{0}}{H(q_{0})}H'(q_{0})\frac{A(H(q_{0}))}{(2\pi i)^{2}}=-6+48q_{0}^{2}+240q_{0}^{4}+1776q_{0}^{6}+\cdots.\]
An educated guess for $\widehat{\varphi_{\textbf{f}}}(m,n)$ is $\frac{6^{5}}{5}$
times Figure 8.2 %
\begin{figure}
\caption{\protect\includegraphics[scale=0.7]{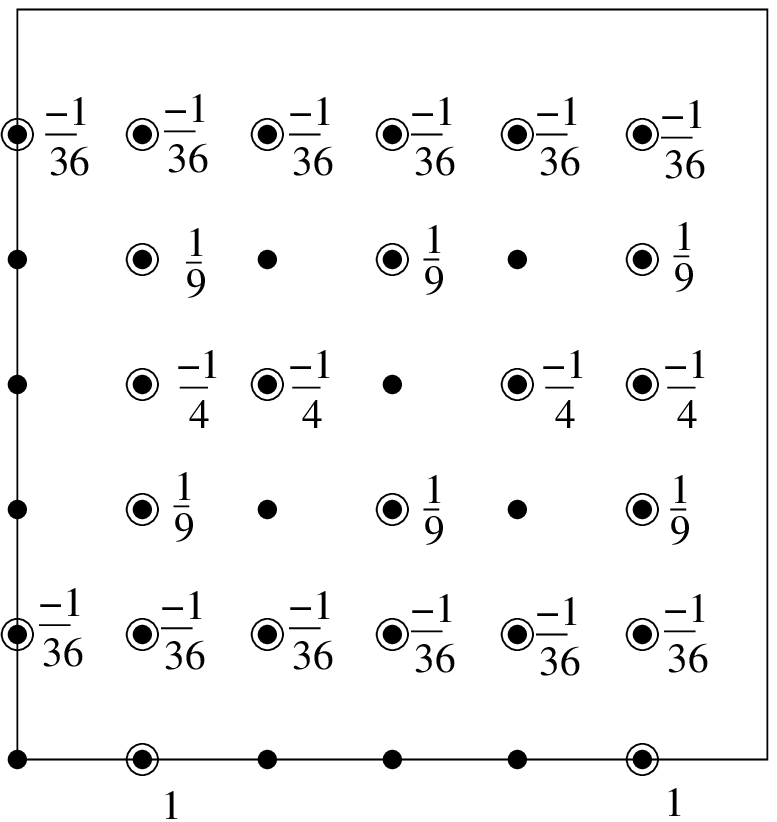}}

\end{figure}
 \[
=\,\left(\widehat{\varphi_{\{6,1\}}}-\widehat{\varphi_{\{6,2\}}}-\widehat{\varphi_{\{6,3\}}}+\widehat{\varphi_{\{6,6\}}}\right)-\frac{1}{36}\left(\widehat{\varphi_{\{1,1\}}}-\widehat{\varphi_{\{2,1\}}}-\widehat{\varphi_{\{3,1\}}}+\widehat{\varphi_{\{6,1\}}}\right)\]
\[
+\frac{1}{9}\left(\widehat{\varphi_{\{2,1\}}}-\widehat{\varphi_{\{2,2\}}}-\widehat{\varphi_{\{6,1\}}}+\widehat{\varphi_{\{6,2\}}}\right)-\frac{1}{4}\left(\widehat{\varphi_{\{3,1\}}}-\widehat{\varphi_{\{3,3\}}}-\widehat{\varphi_{\{6,1\}}}+\widehat{\varphi_{\{6,3\}}}\right),\]
which yields\[
E_{\varphi_{\textbf{f}}}^{[2]}(q)\,=\,\frac{1}{5}E_{4}(q)-\frac{4}{5}E_{4}(q^{2})+\frac{9}{5}E_{4}(q^{3})-\frac{36}{5}E_{4}(q^{6})\]
\[
=\,-6+48q+240q^{2}+1776q^{3}+\cdots\]
in agreement with the above.
\end{example}
${}$

\begin{example}
Verrill's pencil has order $2$ monodromy at $t=\frac{1}{16},\,\frac{1}{4}$
and maximal unipotent monodromy at $0,\,\infty$; it is modular with
$\Gamma=\Gamma_{1}(6)^{+3}$, and presumably a construction analogous
to that in Remark $6.17$ (with $\iota_{3}$ replacing $\iota_{6}$)
yields the total space (up to birational equivalence). This implies\[
m_{0}=-6,\, m_{1}=1,\, N_{\Gamma}=1,\, M_{\theta}=12\,\,\implies\,\, r_{[i\infty]}=-6,\,\frac{Y(0)}{(2\pi i)^{2}}=-6.\]
Verrill's $\Lambda=-\frac{\eta(\tau)^{6}\eta(3\tau)^{6}}{\eta(2\tau)^{6}\eta(6\tau)^{6}}-4$
$\implies$ our $t=$\[
H(q)=\frac{1}{\Lambda+4}=-\frac{\eta(2\tau)^{6}\eta(6\tau)^{6}}{\eta(\tau)^{6}\eta(3\tau)^{6}}=-9(1+6q+21q^{2}+68q^{3}+198q^{4}+\cdots);\]
together with $\frac{A(t)}{(2\pi i)^{2}}=1+4t+28t^{2}+256t^{3}=\cdots$,
this gives\[
F_{\theta^{*}\Xi}=-6\frac{q}{H(q)}H'(q)\frac{A(H(q))}{(2\pi i)^{2}}=-6-12q+84q^{2}-228q^{3}+\cdots.\]
Put $\widehat{\varphi_{\textbf{f}}}:=\frac{6^{5}}{5}$ times Figure
8.3%
\begin{figure}
\caption{\protect\includegraphics[scale=0.7]{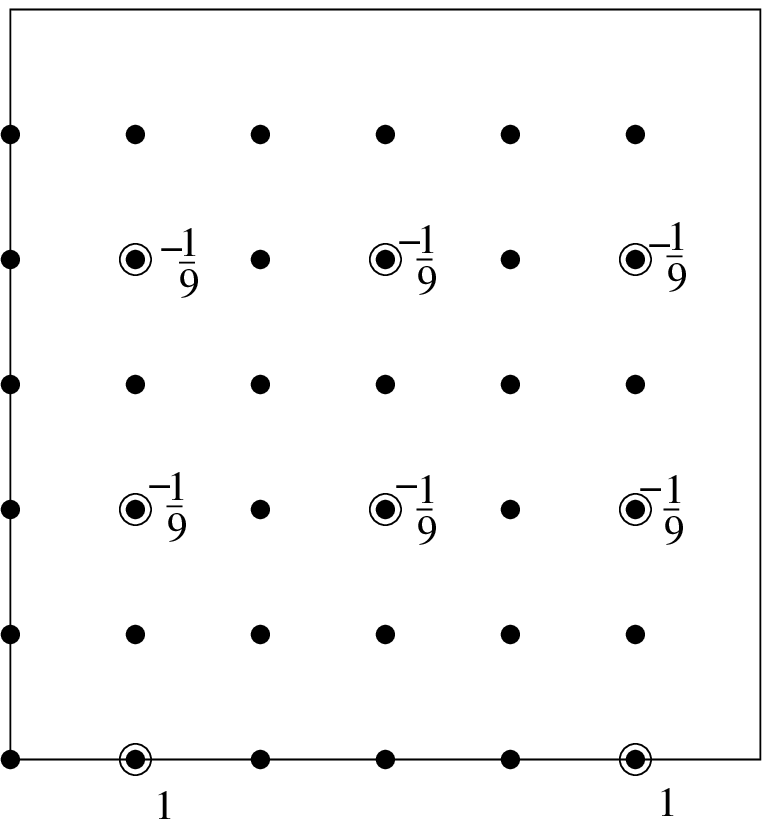}}

\end{figure}
\[
=\left(\widehat{\varphi_{\{6,1\}}}-\widehat{\varphi_{\{6,2\}}}-\widehat{\varphi_{\{6,3\}}}+\widehat{\varphi_{\{6,6\}}}\right)-\frac{1}{9}\left(\widehat{\varphi_{\{2,1\}}}-\widehat{\varphi_{\{2,2\}}}-\widehat{\varphi_{\{6,1\}}}+\widehat{\varphi_{\{6,2\}}}\right);\]
then indeed\[
E_{\varphi_{\textbf{f}}}^{[2]}(q)=-\frac{1}{20}E_{4}(q)+\frac{4}{5}E_{4}(q^{2})+\frac{9}{20}E_{4}(q^{3})-\frac{36}{5}E_{4}(q^{6})\]
\[
=-6-12q+84q^{2}-228q^{3}+\cdots.\]

\end{example}

\subsection{Remarks on the $CY$ $3$-fold case}

In this subsection we present no further examples of Theorem 8.3,
because there aren't any (Prop. 8.15). To illustrate what the problem
is, we begin by describing a \emph{local} modularity criterion for
$\pi:\X\to\PP^{1}$ in terms of the associated limit mixed Hodge structure
at $t=0$. This is a necessary condition for applying that result,
and it fails dramatically for the celebrated quintic mirror family
(as we shall see).

Let $(H_{\ZZ},\H,\F^{\bullet})$ be a weight $3$ rank $4$ polarized
$\ZZ$-VHS over a punctured disk $U=D_{\e}^{*}(0)$ with maximal unipotent
monodromy $T\in Aut(H_{\ZZ})$ about $t=0$. The weight monodromy
filtration $W_{\bullet}$ can be defined on $H_{\ZZ}$, with adapted
symplectic $\ZZ$-basis $\{\varphi_{i}\}_{i=0}^{3}$:\[
Gr^{W}\varphi_{i}\in\Gamma\left(U,\,\frac{W_{2i}}{W_{2i-2}}H_{\ZZ}\right)\,,\,\,\,\,\,\,\,\left[\left\langle \varphi_{i},\varphi_{j}\right\rangle \right]=\tiny\left(\begin{array}{cccc}
 &  &  & 1\\
 &  & 1\\
 & -1\\
-1\end{array}\right).\normalsize\]
Moreover, there is a unique ($\mathcal{O}_{U}$-) basis $\{\omega_{i}\}_{i=0}^{3}$
for $\H$ adapted to the Hodge filtration ($\omega_{i}\in\Gamma(U,\F^{i})$)
and satisfying\[
Gr^{W}\omega_{i}=Gr^{W}\varphi_{i}\in\Gamma\left(U,\,\frac{W_{2i}}{W_{2i-2}}\H\right).\]
Replacing $t$ by $q:=\exp\left(2\pi i\frac{\left\langle \varphi_{1},\omega_{3}\right\rangle }{\left\langle \varphi_{0},\omega_{3}\right\rangle }\right)$,
an {}``integral'' basis for the LMHS \linebreak $(H_{\ZZ}^{\lim},W_{\bullet},\H_{\lim},\F_{\lim})$
is $\{e_{i}:=\tilde{\varphi}_{i}(0)\}$, where $\tilde{\varphi}_{i}(q):=\exp\left(-\frac{\log q}{2\pi i}\log T\right)\varphi_{i}(q)$.

The \emph{period matrix} $\Omega$ of $\H_{\lim}$ is given by writing
the $\omega_{i}(0)\in\F_{\lim}^{i}$ as vectors w.r.t. the basis $\{e_{i}\}$.
If $\H=\text{Sym}^{3}\H^{[1]}$ as in the beginning of $\S7.1$, then
since $\tilde{\beta}=\beta-\frac{\log q}{2\pi i}\alpha$ and $[\tilde{\beta}(0)]=\lim_{q\to0}[dz]\in\H_{\lim}^{[1]}$,
$\Omega=\text{Sym}^{3}\Omega^{[1]}=$ identity (up to unimodular transformations
preserving $W_{\bullet}$). This leads to (ii) in the following

\begin{prop}
(i) \cite{GGK2} In the above situation,\[
\Omega=\left(\begin{array}{cccc}
1 & 0 & \frac{f}{2a} & \xi\\
 & 1 & \frac{e}{a} & \frac{f}{2a}\\
 &  & 1 & 0\\
 &  &  & 1\end{array}\right)\,\,\,\,\,\,\text{with }a,e,f\in\ZZ\,\text{ (but }\xi\in\CC).\]

(ii) If $\H=R^{3}\pi_{*}\CC\otimes\mathcal{O}_{U}$ comes from a modular
family $\pi:\,\X\to\PP^{1}$ of CY 3-folds (in the sense of $\S8.3$),
then $\xi\in\QQ$.
\end{prop}
In the language of \cite{Do1}, $\xi\in\CC/\QQ$ detects the presence
of \emph{instanton corrections}: in fact $\xi$ is nothing but $-\frac{1}{2}F(0)$
where $F$ is the prepotential. This is considered in \cite{CdOGP}
for the quintic mirror, which in our setup is\[
\phi=x+y+z+w+\frac{1}{xyzw}.\]
(Obviously this satisfies Cor. 1.9 for $n=4$.) Indeed, for this most
fundamental example (by \cite{GGK2})\[
\Omega=\left(\begin{array}{cccc}
1 & 0 & \frac{25}{12} & -\frac{200\zeta(3)}{(2\pi i)^{3}}\\
 & 1 & \frac{-11}{2} & \frac{25}{12}\\
 &  & 1 & 0\\
 &  &  & 1\end{array}\right)\]
 tells us that $\X$ is $not$ modular.

Now consider the $5$ Laurent polynomials\\

\begin{tabular}{|c|c|}
\hline 
$\phi(\underline{x})$ & corresponding CY family $\{\tilde{X}_{t}\}$\tabularnewline
\hline
\hline 
$x_{1}+x_{2}+x_{3}+x_{4}+\frac{1}{x_{1}x_{2}x_{3}x_{4}}$ & quintic mirror\tabularnewline
\hline 
$x_{1}+x_{2}+x_{3}+x_{4}+\frac{1}{x_{1}^{2}x_{2}x_{3}x_{4}}$ & sextic mirror\tabularnewline
\hline 
$x_{1}+x_{2}+x_{3}+x_{4}+\frac{1}{x_{1}^{4}x_{2}x_{3}x_{4}}$ & octic mirror\tabularnewline
\hline 
$x_{1}+x_{2}+x_{3}+x_{4}+\frac{1}{x_{1}^{5}x_{2}^{2}x_{3}x_{4}}$ & dectic mirror\tabularnewline
\hline 
$x_{1}+x_{2}+x_{3}+x_{1}x_{2}^{2}x_{3}^{3}x_{4}^{5}+\frac{1}{x_{1}^{2}x_{2}^{3}x_{3}^{4}x_{4}^{5}}$ & quintic twin mirror\tabularnewline
\hline
\end{tabular}\\
\\
all of which fall under the aegis of Corollary 1.9 ($n=4$). These
are the $only$ families%
\footnote{In particular, the corresponding polytopes $\Delta$ have only $6$
integral points, so the anticanonical hypersurfaces in $\PP_{\tilde{\Delta}}$
have one modulus and modifying the monomial coefficients yields isomorphic
families.%
} of smooth $h^{2,1}=1$ Calabi-Yau anticanonical hypersurfaces in
Gorenstein toric Fano fourfolds, and their Picard-Fuchs equations
are all classical generalized hypergeometric equations \cite{DM1}.
Moreover, none of these is a symmetric cube of a second-order ODE
whose projective normal form is the uniformizing differential equation
for a modular curve \cite{Do1}. We conclude:

\begin{prop}
There are no anticanonical toric modular families of CY $3$-folds
in the precise sense of $(5)$ from $\S8.3$.
\end{prop}
There are a couple of ways to relax the toric hypotheses that would
likely lead to modular examples. What does $not$ work is relaxing
the rank $4$ ($h^{2,1}=1$) hypothesis on $H^{3}(X_{t})$ (e.g. to
$H^{3}$ having a rank $4$ level $3$ $sub$-Hodge-structure), since
the geometric information of $\theta:\,\overline{\E}^{[\ell]}(N)\dashrightarrow\X$
is crucial and birational (smooth) CY's have equal Hodge numbers \cite{Ba2}.

One possibility is to consider a toric $4$-fold $\PP_{\tilde{\Delta}}$
whose anticanonical hypersurfaces have multiple moduli, and choose
our $1$-parameter family ($1-t\phi=0$) to have (fiberwise) crepant
singularities on its generic member. Resolving the singularities would
then yield a family of CY's with $h^{p,q}$'s distinct from those
of the generic (smooth) anticanonical hypersurface. This approach
will require a generalization of Theorem 1.7 to treat such singularities.
Alternately, one could try to extend the construction of motivic cohomology
classes from $\S1$ to families of complete intersections in toric
$\geq5$-folds. The generation of such families by way of nef-partitions
of polytopes \cite{BB} yields an as-yet unknown number of $h^{2,1}=1$
examples.

\address{\noun{Department of Mathematical and Statistical Sciences,}}

\address{\noun{University of Alberta}}

\email{doran@math.ualberta.ca\\
}

\address{\noun{Department of Mathematical Sciences, University of Durham}}

\email{matthew.kerr@durham.ac.uk}

\begin{thebibliography}{CdOGP}
\bibitem[Ah]{Ah}L. Ahlfors, {}``Complex Analysis (2nd Ed.)'', McGraw-Hill,
New York, 1966.

\bibitem[BPV]{BPV}W. Barth, C. Peters, A. Van de Ven, {}``Compact
Complex Surfaces'', Springer-Verlag, Berlin, 1984.

\bibitem[Ba1]{Ba1}V. Batyrev, \emph{Dual polyhedra and mirror symmetry
for Calabi-Yau hypersurfaces in toric varieties}, J. Algebraic Geom.
3 (1994), no. 3, 493-535. 

\bibitem[Ba2]{Ba2}---------, \emph{Birational Calabi-Yau $n$-folds
have equal Betti numbers}, in {}``New Trends in Algebraic Geometry
(Warwick, 1996)'', LMS Lec. Not. Ser. 264 (1999), Cambridge Univ.
Press, 1-32.

\bibitem[BB]{BB}V. Batyrev and L. Borisov, \emph{On Calabi-Yau complete
intersections in toric varieties}, in {}``Higher-dimensional complex
varieties (Trento, 1994)'', de Gruyter, Berlin (1996), 39-65.

\bibitem[B1]{B1}A. Beilinson, \emph{Higher regulators and values
of $L$-functions}, J. Soviet Math. 30 (1985), 2036-2070.

\bibitem[B2]{B2}---------, \emph{Higher regulators of modular curves},
in {}``Applications of algebraic $K$-theory to Algebraic Geometry
and Number Theory (Boulder, CO, 1983)'', Contemp. Math. 55, AMS,
Providence, RI, 1986, 1-34.

\bibitem[Be1]{Be1}M.-J. Bertin, \emph{Mahler's measure and $L$-series
of }K3\emph{ hypersurfaces}, in {}``Mirror Symmetry V'' (Lewis,
Yau, Yui, Eds.), AMS/IP Stud. Adv. Math. 38, 2006, 3-18.

\bibitem[Be2]{Be2}---------, {}``Mahler measure of \emph{K3} hypersurfaces'',
talk at Banff conference, 2003.

\bibitem[Bl1]{Bl1}S. Bloch, \emph{The moving lemma for higher Chow
groups}, J. Algebraic Geom. 3 (1993), no. 3, 537-568.

\bibitem[Bl2]{Bl2}---------, {}``Higher regulators, algebraic $K$-theory,
and zeta functions of elliptic curves'', CRM Monograph Ser. 11, AMS,
Providence, RI, 2000.

\bibitem[Bl3]{Bl3}---------, {}``Lectures on algebraic cycles'',
Duke Univ. Math. Ser. IV, Duke University, Math. Dept., Durham, NC,
1980.

\bibitem[Bl4]{Bl4}---------, \emph{Alg. cycles and higher $K$-theory},
Adv. Math. 61 (1986), No. 3, 267-304.

\bibitem[Bk]{Bk}F. Beukers, \emph{Irrationality of $\zeta(2)$, periods
of an elliptic curve and $\Gamma_{1}(5)$}, in {}``Diophantine approximations
and transcendental numbers (Luminy, 1982)'', Prog. Math. 31, Birkh\"auser,
Boston, MA, 1983, 47-66.

\bibitem[BP]{BP}F. Beukers and C. Peters, \emph{A family of $K3$
surfaces and $\zeta(3)$}, J. Reine Angew. Math. 351 (1984), 42-54.

\bibitem[BS]{BS}F. Beukers and J. Stienstra, \emph{On the Picard-Fuchs
equation and the formal Brauer group of certain elliptic $K3$ surfaces},
Math. Ann. 271 (1985), no. 2, 269-304.

\bibitem[BSk]{BSk}A. Bouchard and H. Skarke, \emph{Affine Kac-Moody
algebras, CHL strings and the classification of tops}, Adv. Theor.
Math. Phys. 7 (2003), no. 2, 205-232.

\bibitem[Bo]{Bo}D. Boyd, \emph{Mahler's measure and special values
of $L$-functions}, Experiment. Math. 7 (1998), 37-82.

\bibitem[CdOGP]{CdOGP}P. Candelas, X. de la Ossa, Greene, and Parkes,
\emph{A pair of manifolds as an exactly solvable superconformal theory},
Nucl. Phys. B359 (1991), 21-74.

\bibitem[CKYZ]{CKYZ}T. M. Chiang, A. Klemm, S.-T. Yau and E. Zaslow,
\emph{Local mirror symmetry: calculations and interpretations}, Adv.
Theor. Math. Phys. 3 (1999), no. 3, 495-565.

\bibitem[Co]{Co}A. Collino, \emph{Griffiths's infinitesimal invariant
and higher $K$-theory on hyperelliptic Jacobians}, J. Algebraic Geom.
6 (1997), no. 3, 393-415.

\bibitem[CK]{CK}Cox and Katz, {}``Mirror symmetry and algebraic
geometry'', Math. Surveys and Monographs 68, AMS, Providence, RI,
1999.

\bibitem[dAM1]{dAM1}P. del Angel and S. M\"uller -Stach, \emph{The
transcendental part of the regulator map for $K_{1}$ on a mirror
family of $K3$ surfaces}, Duke Math. J. 112 (2002), no. 3, 581-598.

\bibitem[dAM2]{dAM2}---------, \emph{Differential equations associated
to families of algebraic cycles}, to appear in Ann. Inst. Fourier
58 (2008).

\bibitem[dAM3]{dAM3}---------, \emph{Picard-Fuchs equations, integrable
systems and higher algebraic $K$-theory}, in {}``Calabi-Yau varieties
and mirror symmetry (Toronto, ON, 2001)'', 43-55, Fields Inst. Commun.
38, AMS, Providence, RI, 2003.

\bibitem[dOFS]{dOFS}X. de la Ossa, B. Florea, and H. Skarke, \emph{$D$-branes
on noncompact Calabi-Yau manifolds: $K$-theory and monodromy}, Nuclear
Phys. B 644 (2002), no. 1-2, 170-200.

\bibitem[De1]{De1}C. Deninger, \emph{Higher regulators and Hecke
$L$-series of imaginary quadratic fields I}, Invent. Math. 96 (1989),
no. 1, 1-69.

\bibitem[De2]{De2}C. Deninger, \emph{Deligne periods of mixed motives,
$K$-theory and the entropy of certain $\ZZ_{n}$-actions}, J. Amer.
Math. Soc. 10 (1997), no. 2, 259-281.

\bibitem[DS]{DS}C. Deninger and A. Scholl, \emph{The Beilinson conjectures},
in {}``$L$-functions and arithmetic (Durham, 1989)'', London Math.
Soc. Lect. Note Ser. 153, Cambridge Univ. Press, Cambridge, 1991,
173-209.

\bibitem[Fr]{Fr}L. Frennemo, \emph{A generalization of Littlewood's
Tauberian theorem for the Laplace transform}, preprint, 1999.

\bibitem[Do1]{Do1}C. Doran, \emph{Picard-Fuchs uniformization and
modularity of the mirror map}, Comm. Math. Phys. 212 (2000), 625-647.

\bibitem[Do2]{Do2}---------, \emph{Picard-Fuchs uniformization: modularity
of the mirror map and mirror-moonshine}, CRM Proc. Lect. Not. 24 (2000),
257-281.

\bibitem[DM1]{DM1}C. Doran and J. Morgan, \emph{Mirror symmetry and
integral variations of Hodge structure underlying one-parameter families
of Calabi-Yau threefolds}, in {}``Mirror Symmetry V: Proceedings
of the BIRS Workshop on CY Varieties and Mirror Symmetry (Dec. 2003)'',
AMS/IP Stud. Adv. Math. 38 (2006), 517-537.

\bibitem[DM2]{DM2}---------, \emph{Algebraic topology of Calabi-Yau
threefolds in toric varieties}, Geometry \& Topology 11 (2007), 597-642.

\bibitem[DvK]{DvK}J. J. Duistermaat and W. van der Kallen, \emph{Constant
terms in powers of a Laurent polynomial}, Indag. Math. N.S. 9 (1998),
221-231.

\bibitem[Fu]{Fu}W. Fulton, {}``Intersection Theory'' (2nd Ed.),
Springer-Verlag, Berlin, 1998.

\bibitem[GG]{GG}M. Green and P. Griffiths, \emph{Algebraic cycles
and singularities of normal functions}, in {}``Algebraic cycles and
motives'', London Math. Soc. Lect. Not. Ser. 343, Cambridge Univ.
Press, Cambridge, 2007, 206-263.

\bibitem[GGK1]{GGK1}M. Green, P. Griffiths and M. Kerr, \emph{Neron
models and limits of Abel-Jacobi mappings}, preprint, 2008. Avaliable
from http://www.maths.dur.ac.uk/dma0mk/GGK1.pdf.

\bibitem[GGK2]{GGK2}---------, \emph{Neron models and boundary components
for degenerations of Hodge structure of mirror quintic type}, in {}``Curves,
abelian varieties, and their interactions (Athens, GA, 2007)'', V.
Alexeev et al (Eds.), AMS, 2008.

\bibitem[G]{G}P. Griffiths, \emph{On the periods of certain rational
integrals: II}, Ann. of Math. Ser. 2, Vol. 90, No. 3 (1969), 496-541.

\bibitem[Go]{Go}A. Goncharov, \emph{Chow polylogarithms and regulators},
Math. Res. Lett. 2 (1995), no. 1, 95-112.

\bibitem[Gu]{Gu}R. Gunning, {}``Lectures on modular forms'', Annals
of Math. Stud. 48, Princeton Univ. Press, 1962.

\bibitem[Ho]{Ho}S. Hosono, \emph{Central charges, symplectic forms,
and hypergeometric series in local mirror symmetry} ,in {}``Mirror
Symmetry V'' (Lewis, Yau, Yui, Eds.), AMS/IP Stud. Adv. Math. 38,
2006, 405-440.

\bibitem[Ke1]{Ke1}M. Kerr, \emph{A regulator formula for Milnor $K$-groups},
$K$-Theory 29 (2003), no. 3, 175-210.

\bibitem[Ke2]{Ke2}---------, \emph{Motivic irrationality proofs},
in preparation.

\bibitem[KL]{KL}M. Kerr and J. Lewis, \emph{The Abel-Jacobi map for
higher Chow groups, II,} Invent. Math. 170 (2007), 355-420.

\bibitem[KLM]{KLM}M. Kerr, J. Lewis, and S. M\"uller-Stach, \emph{The
Abel-Jacobi map for higher Chow groups}, Compos. Math. 142 (2006),
no. 2, 374-396.

\bibitem[Ko]{Ko}K. Kodaira, \emph{On compact analytic surfaces I,
II, III}, Ann of Math. (2) 71 (1960), 111-152; 77 (1963), 563-626;
78 (1963), 1-40.

\bibitem[KS]{KS}M. Kreuzer and H. Skarke,\emph{ Classification of
reflexive polyhedra in three dimensions}, Adv. Theor. Math. Phys.
2 (1998) 847, hep-th/9805190. 

\bibitem[L1]{L1}M. Levine, {}``Mixed Motives'', AMS, Providence,
RI, 1998.

\bibitem[L2]{L2}---------, \emph{Bloch's higher Chow groups revisited},
in {}``$K$-theory (Strasbourg, 1992)'', Asterisque No. 226 (1994),
10, 235-320.

\bibitem[LTY]{LTY}B. Lian, A. Todorov and S.-T. Yau, \emph{Maximal
unipotent monodromy for comlplete intersection CY manifolds}, Amer.
J. Math. 127 (2005), no. 1, 1-50.

\bibitem[Li]{Li}C.-C. M. Liu, \emph{Formulae of one-partition and
two-partition Hodge integrals}, Geometry and Topology Monographs 8
(2006), 105-128.

\bibitem[Mi]{Mi}G. Mikhalkin, \emph{Enumerative tropical algebraic
geometry in $\RR^{2}$}, JAMS 18 (2005), 313-377.

\bibitem[My]{My}T. Miyake, {}``Modular Forms'', Springer-Verlag,
Berlin, 1989.

\bibitem[MOY]{MOY}K. Mohri, Y. Onjo, and S.-K. Yang, \emph{Closed
sub-monodromy problems, local mirror symmetry and branes on orbifolds},
Rev. Math. Phys. 13 (2001), no. 6, 675-715.

\bibitem[Mo]{Mo}D. Morrison, \emph{On $K3$ surfaces with large Picard
number}, Invent. Math. 75 (1984), 105-121.

\bibitem[MW]{MW}D. Morrison and J. Walcher, \emph{$D$-branes and
normal functions}, to appear in Adv. Theor. Math. Phys.

\bibitem[Ne1]{Ne1}J. Neukirch, \emph{The Beilinson conjecture for
algebraic number fields}, in {}``Beilinson's conjectures on special
values of $L$-functions'' (Rapoport et al, Eds.), Academic Press,
1988.

\bibitem[Ne2]{Ne2}---------, {}``Algebraic number theory'', Springer-Verlag,
Berlin, 1999.

\bibitem[Ni]{Ni}V. Nikulin, \emph{Finite automorphism groups of K\"ahlerian
surfaces of type $K3$}, Trans. Moscow Math. Soc. 38 (1980), no. 2,
71-135.

\bibitem[No]{No}A. Novoseltsev, computations using {}``PALP in Sage'';
M. Kreuzer and H. Skarke, {}``PALP (Package for Analyzing Lattice
Polytopes)''; W. Stein and D. Joyner, {}``Sage'' (mathematical
software, http://www.sagemath.org), v. 3.1.2.

\bibitem[Pe]{Pe}C. Peters, \emph{Monodromy and Picard-Fuchs equations
for families of $K3$ surfaces and elliptic curves}, Ann. ENS 19 (1986),
583-607.

\bibitem[PS]{PS}C. Peters and J. Stienstra, A pencil of $K3$ surfaces
related to Ap\'ery's recurrence for $\zeta(3)$ and Fermi surfaces
for potential zero, in {}``Arithmetic of complex manifolds (Erlangen,
1988)'', Springer Lect. Not. Math. 1399, 1989, 110-127.

\bibitem[RV]{RV}F. Rodriguez-Villegas, \emph{Modular Mahler measures},
in {}``Topics in number theory University Park, PA, 1997)'', 17-48,
Kluwer Acad. Publ., Dordrecht, 1999.

\bibitem[Ro]{Ro}F. Rohsiepe, \emph{Lattice polarized toric $K3$
surfaces}, arXiv:hep-th/0409290 v1, 2004.

\bibitem[Sc]{Sc}A. Scholl, {}``$L$-functions of modular forms and
higher regulators'', http://www.dpmms.cam.ac.uk/\textasciitilde{}ajs1005/mono/index.html.

\bibitem[Sh]{Sh}G. Shimura, {}``Introduction to the arithmetic theory
of automorphic functions'', Princeton University Press, 1994.

\bibitem[Sd]{Sd}T. Shioda, \emph{On elliptic modular surfaces}, J.
Math. Soc. Japan 24 (1972), 20-59.

\bibitem[So]{So}S. Shokurov, \emph{Holomorphic forms of highest degree
on Kuga's modular varieties}. (Russian) Mat. Sb. (N.S.) 101 (143)
(1976), no. 1, 131-157, 160.

\bibitem[Si]{Si}J. Silverman, {}``Advanced topics in the arithmetic
of elliptic curves'', GTM 151, Springer-Verlag, Berlin, 1994.

\bibitem[S]{S}J. Stienstra, \emph{Mahler measure variations, Eisenstein
series and instanton expansions}, in {}``Mirror Symmetry V'' (Lewis,
Yau, Yui, Eds.), AMS/IP Stud. Adv. Math. 38, 2006, 139-150.

\bibitem[St]{St}P. Stiller, \emph{A note on automorphic forms of
weight one and weight three}, Trans. Amer. Math. Soc. 291 (1985),
no. 2, 503-518.

\bibitem[Ty]{Ty}I. Tyomkin, \emph{On Severi varieties on Hirzebruch
surfaces}, IMRN (2007) Vol. 2007: article ID rnm109, 31 pp.

\bibitem[Ve]{Ve}H. Verrill, \emph{Root lattices and pencils of varieties},
J. Math. Kyoto Univ. 36 (2) (1996), 423-446.

\bibitem[Wh]{Wh}U. Whitcher, \emph{Symplectic automorphisms and the
Picard group of a $K3$ hypersurface}, preprint, 2008, avaliable at
http://www.math.washington.edu/\textasciitilde{}ursula/notes.

\bibitem[Wi]{Wi}J. Wildeshaus, \emph{On the Eisenstein symbol}, preprint,
available at http://www.math.uiuc.edu/K-theory/0425.

\bibitem[WM]{WM}H. Wilcox and D. L. Myers, {}``An introduction to
Lebesgue integration and Fourier series'', Dover, NY, 1994.

\bibitem[Z]{Z}D. Zagier, \emph{Integral solutions of Ap\'ery-like
recurrence equations}, preprint.\\


\end{thebibliography}
\end{document}